\documentclass[reqno,11pt]{scrartcl}
\setkomafont{disposition}{\normalfont\bfseries}

\usepackage{a4wide}
\usepackage[pdfborder={0 0 0}]{hyperref}  
\usepackage{color,eucal,enumerate,mathrsfs}
\usepackage[normalem]{ulem}
\usepackage{amsmath,amssymb,epsfig,bbm}
\numberwithin{equation}{subsection}
\usepackage{dsfont}
\usepackage{tikz}

\usepackage[latin1]{inputenc}

\newcommand{\E}{\mathbb{E}}

\renewcommand{\H}{\mathbb{H}}

\newcommand{\N}{\mathbb{N}}
\newcommand{\Q}{\mathbb{Q}}
\newcommand{\R}{\mathbb{R}}

\newcommand{\mm}{{\mbox{\boldmath$m$}}}

\newcommand{\ppi}{{\mbox{\boldmath$\pi$}}}

\newcommand{\sppi}{{\mbox{\scriptsize\boldmath$\pi$}}}

\newcommand{\sfd}{{\sf d}}

\newcommand{\Id}{{\rm Id}}                          
\newcommand{\Kliminf}{K\kern-3pt-\kern-2pt\mathop{\rm lim\,inf}\limits}  
\newcommand{\supp}{\mathop{\rm supp}\nolimits}   
\newcommand{\Lip}{\mathop{\rm Lip}\nolimits}          
\renewcommand{\d}{{\mathrm d}}

\newcommand{\restr}[1]{\lower3pt\hbox{$|_{#1}$}}
\newcommand{\la}{{\langle}}                  
\newcommand{\ra}{{\rangle}}
\newcommand{\eps}{\varepsilon}  
\newcommand{\nchi}{{\raise.3ex\hbox{$\chi$}}}
\newcommand{\weakto}{\rightharpoonup}

\newcommand{\limi}{\varliminf}
\newcommand{\lims}{\varlimsup}

\newcommand{\fr}{\penalty-20\null\hfill$\blacksquare$}                      

\newcommand{\prob}[1]{\mathscr P(#1)}                   
\newcommand{\probt}[1]{\mathscr P_2(#1)}                   
\newcommand{\e}{{\rm{e}}}                           

\renewcommand{\mm}{\mathfrak m}                                

\newcommand{\weakgrad}[1]{|\nabla #1|_w} 

\newenvironment{proof}{\removelastskip\par\medskip   
\noindent{\em proof} \rm}{\penalty-20\null\hfill$\square$\par\medbreak}

\newtheorem{theorem}{Theorem}[subsection]

\newtheorem{corollary}[theorem]{Corollary}

\newtheorem{lemma}[theorem]{Lemma}
\newtheorem{proposition}[theorem]{Proposition}
\newtheorem{problem}[theorem]{Open Problem}

\newtheorem{definition}[theorem]{Definition}

\newtheorem{example}[theorem]{Example}
\newtheorem{remark}[theorem]{Remark}

\newcommand\natop[2]{\genfrac{}{}{0pt}{}{#1}{#2}}

\newcommand{\bd}{{\mathbf\Delta}}

\newcommand{\s}{{\rm S}}

\newcommand{\X}{{\rm M}}
\newcommand{\TX}{{\rm TM}}

\newcommand{\Bo}{\mathcal B}
\newcommand{\BB}{\mathcal B(\X)}
\newcommand{\BBB}[1]{\mathcal B(#1)}
\newcommand{\ed}{\mathcal E_2}
\newcommand{\ec}{\mathcal E_C}
\newcommand{\ect}{\tilde{\mathcal E}_C}
\newcommand{\eht}{\tilde{\mathcal E}_\Ho}
\newcommand{\eh}{\mathcal E_\Ho}
\newcommand{\edd}{\mathcal E_\d}
\newcommand{\intmap}{{\text{\sc  Int}}}
\newcommand{\Hom}{{\text{\sc Hom}}}
\newcommand{\h}{{\sf h}}
\newcommand{\Ho}{{\rm H}}
\newcommand{\vol}{{\rm vol}}
\newcommand{\ric}{{\mathbf{Ric}}}

\renewcommand{\weakgrad}[1]{|{\rm D} #1|}
\newcommand{\weakgradp}[2]{|{\rm D} #1|_{#2}}
\newcommand{\Ggamma}{{\mbox{\boldmath$\Gamma$}}}

\renewcommand{\div}{{\rm div}}
\newcommand{\vsm}{{\rm TestV}(\X)}
\newcommand{\fsm}{{\rm TestF}(\X)}
\newcommand{\fsms}[1]{{\rm TestF}(#1)}
\newcommand{\ffsm}[1]{{\rm TestForm}_{#1}(\X)}
\newcommand{\closed}[1]{{\rm C}_{#1}(\X)}
\newcommand{\exact}[1]{{\rm E}_{#1}(\X)}
\newcommand{\exacto}[1]{\overline{\rm E}_{#1}(\X)}
\newcommand{\exactos}[2]{\overline{\rm E}_{#1}(#2)}

\newcommand{\harm}[1]{{\rm Harm}_{#1}(\X)}
\newcommand{\HS}{{\lower.3ex\hbox{\scriptsize{\sf HS}}}}
\renewcommand{\H}[1]{{\rm Hess}(#1)}
\newcommand{\MM}{{\mathscr M}}
\newcommand{\NN}{{\mathscr N}}
\newcommand{\HH}{{\mathscr H}}
\newcommand{\mes}{{\sf Meas}}

\newcommand{\RCD}{{\sf RCD}}
\newcommand{\CD}{{\sf CD}}
\renewcommand{\ae}{{\textrm{\rm{-a.e.}}}}

\newcommand{\lf}{{\sf L}}
\newcommand{\cf}{{\sf C}}

\newcommand{\mms}{\text{\bf Mms}}

\newcommand{\ent}{{\rm Ent}}

\renewcommand{\E}{{\sf E}}

\DeclareMathOperator*{\esssup}{\rm ess-sup}
\DeclareMathOperator*{\essinf}{\rm ess-inf}

\title{Nonsmooth differential geometry}
\subtitle{An approach tailored  for spaces with Ricci curvature bounded from below}
\begin{document}

\author{Nicola Gigli \thanks{Institut de math\'ematiques de Jussieu - UPMC. email: \textsf{nicola.gigli@imj-prg.fr}}
   }

\maketitle

\begin{abstract}  
We discuss in which sense general metric measure spaces possess a first order differential structure. Building on this, we then see that on spaces with Ricci curvature bounded from below a second order calculus can be developed, permitting to define Hessian, covariant/exterior derivatives and  Ricci curvature.

\end{abstract}

\tableofcontents
\section*{Introduction}
\addcontentsline{toc}{section}{\protect\numberline{} Introduction}
\subsection*{Aim and key ideas}
\addcontentsline{toc}{subsection}{\protect\numberline{} Aim and key ideas}

The first problem one encounters when studying metric measure spaces is the lack of all the vocabulary available in the smooth setting which allows to `run the necessary computations'. This issue is felt as particularly strong in considering structures carrying geometric information like Alexandrov or $\RCD$ spaces: here  to make any serious use of the curvature assumption, which is of second order in nature, one needs sophisticate calculus tools. 

In the context of Alexandrov spaces with curvature bounded from below, this problem  is addressed mainly using  the concavity properties of the distance function which come with the definition of the spaces themselves. This regularity information is sufficient to create the basis of non-trivial second order calculus, for instance allowing to state and prove second order differentiation formulas. We refer to the work in progress \cite{AKP} for an overview on the topic.

Things are harder and less understood in  the more general setting of spaces with Ricci curvature bounded from below, where even basic questions like `what is a tangent vector?' do not have a clear answer.

\bigskip

The aim of this paper, which continues the analysis done in \cite{Gigli12},  is to make a proposal in this direction by showing that every metric measure space possesses a first order differential structure and that a second order one arises when  a lower Ricci bound is imposed. Our constructions are analytic in nature, in the sense that they provide tools to make computations on metric measure spaces, without having an a priori relation with their geometry. For instance, our definition of `tangent space'  has nothing to do with pointed-measured-Gromov-Hausdorff limits of rescaled spaces, and the two notions can have little in common on irregular spaces.

The expectation/hope is then that with these tools one can obtain new information about the shape, in a broad sense, of   $\RCD$ spaces. This part of the plan is not pursued in this paper, which therefore is by nature incomplete. Still, we believe there are reasons to be optimistic about applications of the language we propose, because:
\begin{itemize}
\item[i)] The constructions made here are compatible with all the analytic tools developed so far for the study of $\RCD$ spaces, which in turn already produced non-trivial geometric consequences like the Abresch-Gromoll inequality \cite{Gigli-Mosconi12}, the splitting theorem \cite{Gigli13}, the maximal diameter theorem \cite{Ketterer13} and rectifiability results \cite{Mondino-Naber14}.
\item[ii)] The picture which emerges is coherent and quite complete, in the sense that most of the basic differential operators appearing in the smooth context of Riemannian manifolds have a counterpart in the non-smooth setting possessing the expected properties. In particular, we shall provide the notions of Hessian, covariant and exterior differentiation, connection and Hodge Laplacian and a first glance on the Ricci curvature tensor.
\end{itemize}

Our constructions are based on 3 pillars:
\begin{itemize}
\item[1)] The concept of $L^\infty(\mm)$-module, introduced in this context by Weaver in \cite{Weaver01}, who in turn was inspired by the papers \cite{Sauvageot89}, \cite{Sauvageot90} of Sauvageot dealing with the setting  of Dirichlet forms. Such concept  allows to give an answer to the question
\[
\text{what is a (co)tangent vector field?}
\]
the answer being
\[
\text{an element of the  (co)tangent module.}
\]
Shortly said, an $L^\infty(\mm)$ module is a Banach space whose elements can be multiplied by functions in $L^\infty(\mm)$. The analogy with the smooth case is in the fact that the space of smooth sections of a vector bundle on a smooth manifold $M$ can be satisfactorily described via its structure as   module over the space $C^{\infty}(M)$ of smooth functions on $M$. Replacing the smoothness assumption with an  integrability condition we see that the space of, say, $L^p$ sections of a normed vector bundle on $M$ can be described as module over the space of $L^\infty$ functions on $M$.

We shall therefore adopt this point of view and declare that tensor fields on a metric measure space $(\X,\sfd,\mm)$ are   $L^\infty(\mm)$-modules. Notice that this implies that tensors will never be defined pointwise, but only given $\mm$-almost everywhere, in a sense.

As said, the idea of using $L^\infty(\mm)$-modules to provide an abstract definition of vector fields in non-smooth setting  has been proposed by Weaver in \cite{Weaver01}. Here we modify and adapt his approach to tailor it to our needs. Technicalities apart, the biggest difference is that we base our constructions on Sobolev functions, whereas in \cite{Weaver01} Lipschitz ones have been used.

\item[2)] The self-improving properties of Bochner inequality as obtained by Bakry \cite{Bakry83} in the context of abstract $\Gamma_2$-calculus and adapted by Savar\'e \cite{Savare13} in the one of $\RCD$ spaces.  A refinement of these estimates will allow to pass from 
\[
\Delta\frac{|\nabla f|^2}2\geq\la \nabla f,\nabla\Delta f\ra+K|\nabla f|^2,
\]
which in the weak form is already known (\cite{Gigli-Kuwada-Ohta10}, \cite{AmbrosioGigliSavare11-2}) to be valid on $\RCD(K,\infty)$ spaces, to
\[
\Delta\frac{|\nabla f|^2}2\geq|\H f|_\HS^2+\la \nabla f,\nabla\Delta f\ra+K|\nabla f|^2,
\]
which is key to get an $L^2$ control on the Hessian of functions and will be - in the more general form written for vector fields - the basis of all the second-order estimates in the non-smooth setting.

Some of the arguments that we use also appeared in the recent paper of Sturm \cite{Sturm14}: there the setting is technically simpler but also covers finite dimensional situations, a framework which we will not analyze.

\item[3)] The link between `horizontal and vertical derivatives', i.e.\ between Lagrangian and Eulerian calculus, i.e.\ between $W_2$ and $L^2$ analysis. In the context of smooth Finsler manifolds this amounts to the simple identity
\[
\lim_{t\to 0}\frac{f(\gamma_t)-f(\gamma_0)}{t}=\lim_{\eps\to0}\frac{|\d (g+\eps f)|_*^2+|\d g|_*^2}{2\eps}(\gamma_0),\qquad\text{ provided }\gamma_0'=\nabla g(\gamma_0),
\]
which relates the `horizontal' derivative $\lim_{t\to 0}\frac{f(\gamma_t)-f(\gamma_0)}{t}$ of $f$, so called because the perturbation is at the level of the independent variable, to the `vertical' derivative $\lim_{\eps\to0}\frac{|\d (g+\eps f)|_*^2+|\d g|_*^2}{2\eps}(\gamma_0)$ of $|\d g|_*^2$, so called because here perturbation is on  the dependent variable (in using this terminology we are imagining the graphs of the functions drawn in the `Cartesian plane').

It has been first realized in \cite{AmbrosioGigliSavare11-2} and then better understood in \cite{Gigli12} that the same identity can be stated and proved in the non-smooth setting, this being a crucial fact in the development of a differential calculus on metric measure spaces suitable to obtain geometric information on the spaces themselves. 

In this paper we shall see the effects of this principle mainly via the fact that `distributional solutions' of the continuity equation
\begin{equation}
\label{eq:contintro}
\frac{\d}{\d t}\mu_t+\nabla\cdot(X_t\mu_t)=0,
\end{equation}
completely characterize $W_2$-absolutely continuous curves of measures provided one assumes that $\mu_t\leq C\mm$ for every $t\in[0,1]$ and some $C>0$, in analogy with the result valid in  the Euclidean space \cite{AmbrosioGigliSavare08}. This means that for a $W_2$-absolutely continuous curve $(\mu_t)$ with $\mu_t\leq C\mm$ there are vector fields $X_t$ such that for every  Sobolev function $f$  the first order differentiation formula
\[
\frac{\d}{\d t}\int f\,\d\mu_t=\int \d f(X_t)\,\d\mu_t,\qquad \text{a.e. }t\in[0,1],
\]
holds, see Theorem \ref{thm:conteq} for the rigorous statement and notice that this is  in fact    a reformulation of the result proved in \cite{GigliHan13} in the language developed here.

Then we further push in this direction by providing a second-order analogous of the above. As said, the crucial inequality on the Hessian is obtained via tools related to  Dirichlet forms, which are  `vertical' in nature as tied to the structure of the space $L^2(\mm)$, where distances between functions are measured `vertically'. It is then natural to ask if such notion controls the `horizontal' displacement of a functions, which amounts to ask whether we can compute the second derivative of $t\mapsto \int f\,\d\mu_t$ for a given function $f\in W^{2,2}(\X)$ and $W_2$-absolutely continuous curve $(\mu_t)$ satisfying appropriate regularity assumptions.

The answer is positive and  the expected formula
\[
\frac{\d^2}{\d t^2}\int f\,\d\mu_t=\int \H f(X_t,X_t)+\la\nabla f, \partial_t X_t\ra+\la\nabla_{X_t}X_t, \nabla f\ra\,\d\mu_t,
\]
holds, see Theorem   \ref{thm:secondder}  for the precise  formulation.

\end{itemize}

\subsection*{Overview of the content}
\addcontentsline{toc}{subsection}{\protect\numberline{} Overview of the content}

Fix  a metric measure space $(\X,\sfd,\mm)$ which is complete, separable and equipped with a non-negative Radon measure.

The crucial object in our analysis is the notion of $L^\infty(\mm)$-module and in particular of $L^p(\mm)$-normed module, which in this introduction we consider for the case   $p=2$ only.

An {\bf $L^2(\mm)$-normed module} is the structure $(\MM,\|\cdot\|_\MM,\cdot,|\cdot|)$ where:  $(\MM,\|\cdot\|_\MM)$ is a Banach space,  $\cdot$ is a multiplication of elements of $\MM$ with $L^\infty(\mm)$ functions satisfying 
\[
\text{$f(gv)=(fg)v$\qquad and\qquad ${\mathbf 1}v=v$\qquad for every\qquad $f,g\in L^\infty(\mm)$, $v\in\MM$,}
\]
where ${\mathbf 1}$ is the function identically equal to 1, and $|\cdot|:\MM\to L^2(\mm)$ is the `pointwise norm', i.e.\ a map assigning to every $v\in\MM$ a non-negative function in $L^2(\mm)$ such that
\[
\begin{split}
\|v\|_\MM&=\||v|\|_{L^2(\mm)},\\
|f v|&=|f| |v|,\quad\mm\ae.
\end{split}
\qquad\qquad\qquad\text{for every $f\in L^\infty(\mm)$ and $v\in\MM$,}
\]
so  that in particular we have
\[
\|f v\|_\MM\leq \|f\|_{L^\infty(\mm)}\|v\|_\MM, \qquad \qquad \text{for every $f\in L^\infty(\mm)$ and $v\in\MM$.}
\]

The basic example of $L^2(\mm)$-normed module is the space of $L^2$ (co)vector fields on a Riemannian/Finslerian manifold: here the norm $\|\cdot\|_\MM$ is the $L^2$ norm and the multiplication with an $L^\infty$ function and the pointwise norm are defined in the obvious way.

The job that the notion of $L^2(\mm)$-normed module does is to revert this procedure and give the possibility of speaking about $L^2$ sections of a vector bundle without really having the bundle. 

This fact is of help when trying to build a differential structure on metric measure spaces, because it relieves from the duty of defining a tangent space at every, or almost every, point, allowing one to concentrate on the definition of $L^2$ (co)vector field. 

To present the construction, we briefly recall the definition of the Sobolev class $\s^2(\X)$. One says that a probability measure $\ppi\in\prob{C([0,1],\X)}$ is a test plan provided there is $\cf(\ppi)>0$ such that $(\e_t)_*\ppi\leq \cf(\ppi)\mm$ for every $t\in[0,1]$ and $\iint_0^1|\dot\gamma_t|^2\,\d t\d\ppi(\gamma)<\infty$, where $\e_t:C([0,1],\X)\to\X$ is the evaluation map defined by $\e_t(\gamma):=\gamma_t$ and $| \dot\gamma_t|$ is the metric speed of the curve. Then a Borel function $f:\X\to\R$ is said to belong to $\s^2(\X)$ provided there is a non-negative $G\in L^2(\mm)$ such that
\[
\int|f(\gamma_1)-f(\gamma_0)|\,\d\ppi(\gamma)\leq \iint_0^1G(\gamma_t)|\dot\gamma_t|\,\d t\,\d\ppi(\gamma),\qquad\forall\ppi\text{ test plan.}
\]
It turns out that for functions in $\s^2(\X)$ there is a minimal $G$ in the $\mm$-a.e.\ sense satisfying the above: we shall denote such minimal function as $\weakgrad f$. Notice that for the moment the notation is purely formal as we didn't define yet who is the differential of $f$, so that at this stage $\weakgrad f$ is not the modulus of something. The basic calculus rules for $\weakgrad f$ are:
\begin{align*}
\weakgrad f&=0,\quad\mm\ae\text{ on }\{f=0\}&&\forall f\in \s^2(\X),\\
\weakgrad{(\varphi\circ f)}&=|\varphi'|\circ f\weakgrad f,&&\forall f\in\s^2(\X),\ \varphi\in C^1(\R), \\
\weakgrad{(fg)}&\leq |f|\weakgrad g+|g|\weakgrad f,&&\forall f,g\in\s^2\cap L^\infty(\X).
\end{align*}
The idea to define the {\bf cotangent module }is then to pretend that it exists and that for each $f\in\s^2(\X)$ and Borel set $E\subset \X$ the abstract object $\nchi_E\d f$ is an element of such module. The definition then comes via explicit construction. We introduce the set  `Pre-cotangent module' ${\rm Pcm}$ as 
\[
\begin{split}
{\rm Pcm}:=\Big\{\{(f_i,A_i)\}_{i\in\N}\ :&\  (A_i)_{i\in\N}\text{ is a Borel  partition of $\X$,} \\
&\ f_i\in\s^2(\X)\ \forall i\in\N, \text{ and }\sum_{i\in\N}\int_{A_i}\weakgrad f^2\,\d\mm<\infty \Big\}
\end{split}
\]
and an equivalence relation on it via
\[
\{(f_i,A_i)\}_{i\in\N}\sim \{(g_j,B_j)\}_{j\in\N}\qquad\text{ provided }\qquad \weakgrad{(f_i-g_j)}=0,\quad\mm\ae\text{ on }A_i\cap B_j\ \forall i,j\in\N.
\]
Denoting by $[(f_i,A_i)]$ the equivalence class of $\{(f_i,A_i)\}_{i\in\N}$,  the operations of addition, multiplication by a scalar and by a simple function (i.e.\ taking only a finite number of values) and the one of taking the pointwise norm can be introduced as
\[
\begin{split}
[(f_i,A_i)]+[(g_j,B_j)]&:=[(f_i+g_j,A_i\cap B_j)]\\
\lambda[(f_i,A_i)]&:=[(\lambda f_i,A_i)]\\
\Big(\sum_j\alpha_j\nchi_{B_i}\Big)\cdot[(f_i,A_i)]&:=[(\alpha_jf_i,A_i\cap B_j)],\\
\big|[(f_i,A_i)]\big|&:=\sum_i\nchi_{A_i}\weakgrad{f_i},
\end{split}
\]
and it is not difficult to see that these are continuous on ${\rm Pcm}/\sim$ w.r.t.\ the norm $\|[f_i,A_i]\|:=\sqrt{\int |[(f_i,A_i)]|^2\,\d\mm }$ and the $L^\infty(\mm)$-norm on the space of simple functions. Thus they all can be continuously extended to the completion of $({\rm Pcm}/\sim,\|\cdot\|)$: we shall call such completion together with these operation the cotangent module and denote it by $L^2(T^*\X)$. When applied to a smooth Riemannian/Finslerian manifold, this abstract construction is canonically identifiable with the space of $L^2$ sections of the   cotangent bundle $T^*\X$, whence the notation chosen.

Given a Sobolev function $f\in\s^2(\X)$, its {\bf differential }$\d f$ is a well defined element of $L^2(T^*\X)$, its definition being
\[
\d f:=[(f,\X)],
\]
and from the properties of Sobolev functions one can verify that the differential is a closed operator. Directly from the definition we see that $|\d f|=\weakgrad f$ $\mm$-a.e., and with little work one can check that the calculus rules for $\weakgrad f$ can be improved to:
\begin{align*}
\d f&=0,\qquad\mm\ae\text{ on }\{f=0\}&&\forall f\in \s^2(\X),\\
\d(\varphi\circ f)&=\varphi'\circ f\d f,&&\forall f\in\s^2(\X),\ \varphi\in C^1(\R), \\
\d(fg)&=f\d g+g\d f,&&\forall f,g\in\s^2\cap L^\infty(\X),
\end{align*}
where thanks to the $L^\infty$-module structure the chain and Leibniz rules both make sense and the locality condition is interpreted as $\nchi_{\{f=0\}}\d f=0$.

Once the notion of cotangent module is given, the tangent module $L^2(T\X)$ can be introduced by duality: it is the space of linear continuous maps $L:L^2(T^*\X) \to L^1(\mm)$ satisfying
\[
L(f\omega)=fL(\omega),\qquad\forall \omega\in L^2(T^*\X),\ f\in L^\infty(\mm),
\]
and it is not hard to see that it carries a canonical structure of $L^2(\mm)$-normed module as well, so that in particular for any vector field $X$, i.e.\ every element of the tangent module $L^2(T\X)$, the pointwise norm $|X|$ is a well defined function in $L^2(\mm)$.

Based on these grounds, a general first-order differential theory can be developed on arbitrary metric measure spaces. Properties worth of notice are:
\begin{itemize}
\item[-] In the smooth setting, for every smooth curve $\gamma$ the tangent vector $\gamma_t'$ is well defined for any $t$ and its norm coincides with the metric speed of the curve.

Similarly,  in the context of metric measure spaces for any given test plan $\ppi$ we have that for a.e.\ $t\in[0,1]$ and $\ppi$-a.e.\ $\gamma$ the tangent vector $\gamma_t'$ is well defined and its norm coincides with the metric speed $|\dot\gamma_t|$ of the curve $\gamma$ at time $t$ (Theorem \ref{thm:spppi}). 

The rigorous meaning of this statement is given via the notion of {\bf pullback} of a module (Section \ref{se:pullback}).

\item[-] (co)vector fields are transformed via `regular' maps between metric measure spaces as  in the smooth setting, i.e.\ we can speak of pullback of forms and these regular maps possess a differential acting on vector fields (Section \ref{se:mbd}).

Here the relevant notion of  regularity for a map $\varphi$ from $(\X_2,\sfd_2,\mm_2)$ to $(\X_1,\sfd_1,\mm_1)$ is to be  Lipschitz and  such that $\varphi_*\mm_2\leq C\mm_1$ for some $C>0$. We will call maps of this kind of {\bf bounded deformation}.

\item[-]  The gradient of a Sobolev function is in general not uniquely defined and even if so it might not linearly depend  on the function,  as it happens on smooth Finsler manifolds. Spaces where the gradient $\nabla f\in L^2(T\X)$ of a Sobolev function $f\in\s^2(\X)$ is unique and linearly depends on $f$ are those which, from the Sobolev calculus point of view, resemble Riemannian manifolds among the more general Finsler ones and can be characterized as those for which the energy $\E:L^2(\mm)\to[0,+\infty]$ defined as 
\[
\E(f):=\left\{\begin{array}{ll}
\displaystyle{\frac12\int\weakgrad f^2\,\d\mm},&\qquad\text{ if }f\in\s^2(\X),\\
+\infty,&\qquad\text{ otherwise}.
\end{array}\right.
\]
 is a Dirichlet form.  Following the terminology introduced in \cite{Gigli12} we shall call these spaces {\bf infinitesimally Hilbertian}. On such spaces, the tangent module (and similarly the cotangent one) is, when seen as Banach space, an Hilbert space and its  pointwise norm satisfies a pointwise parallelogram identity. Thus by polarization it induces  a pointwise scalar product
\[
L^2(T\X)\ni X,Y\qquad\mapsto \qquad\la X,Y\ra\in L^1(\mm),
\]
which we might think of as the `metric tensor' on our space. It  can then be verified that for $f,g\in L^2\cap\s^2(\X)$ the scalar product $\la\nabla f,\nabla g\ra$ coincides with the Carr\'e du champ $\Gamma(f,g)$ induced by the Dirichlet form $\E$.

\end{itemize}

\bigskip

With the basis provided by the general first order theory, we can then study the {\bf second order differential structure of $\RCD(K,\infty)$ spaces}. Recall that these structures, introduced in \cite{AmbrosioGigliSavare11-2} (see also the axiomatization given in \cite{AmbrosioGigliMondinoRajala12}), are given by infinitesimally Hilbertian spaces on which a curvature condition is imposed in the sense of Lott-Sturm-Villani (\cite{Lott-Villani09} and \cite{Sturm06I}).

To the best of our knowledge, attempts to define higher order Sobolev spaces on metric measure spaces have been done only in \cite{LLW02} and \cite{AMV12}. Our approach is structurally different from the one of these references, being  intrinsically based on the $\RCD$ condition, an assumptions which was  not present  in \cite{LLW02} and \cite{AMV12}.

To begin with,  consider the following 3 formulas valid in a smooth Riemannian manifold:
\begin{equation}
\label{eq:perdef2}
\begin{split}
2\H f(\nabla g_1,\nabla g_2)&=\big<\nabla\la \nabla f,\nabla g_1\ra,\nabla g_2\big>+\big<\nabla\la \nabla f,\nabla g_2\ra,\nabla g_1\big>-\big<\nabla f,\nabla \la \nabla g_1,\nabla g_2\ra\big>,\\
\la\nabla_{\nabla g_2}X,\nabla g_1\ra&=\big<\nabla\la X,\nabla g_1\ra,\nabla g_2\big>-\H{g_2}(\nabla g_1,X),\\
\d \omega(X_1,X_2)&=X_1(\omega(X_2))-X_2(\omega(X_1))-\omega(\nabla_XY-\nabla_YX).
\end{split}
\end{equation}

The first completely characterizes the Hessian of the function $f$ in terms of the scalar product of gradients only, the second the Levi-Civita connection in terms of the Hessian  and the scalar product of gradients and analogously the third the exterior differentiation of a 1-form (a similar formula being valid for $k$-forms) via previously defined objects. Thus one can use them to actually define the Hessian and the covariant/exterior derivative. For instance, one could use the first above to define the Sobolev space $W^{2,2}(M)$ on a smooth Riemannian manifold $M$ by declaring that a function $f\in W^{1,2}(M)$ is in $W^{2,2}(M)$ if there is a $(0,2)$-tensor field $\H f$  in $L^2$ such that for any $g_1,g_2,h\in C^\infty_c(M)$ it holds
\begin{equation}
\label{eq:defhessintro}
\begin{split}
2\int h\H f&(\nabla g_1,\nabla g_2)\,\d\mm\\
&=\int -\la \nabla f,\nabla g_1\ra\,\div(h\nabla g_2)-\la \nabla f,\nabla g_2\ra\,\div(h\nabla g_1)-h\big< \nabla f,\nabla \la \nabla g_1,\nabla g_2\ra\big>\,\d\mm.
\end{split}
\end{equation}
The integration by parts here is useful to have a right hand side where only the first order derivative of $f$ appears and the multiplication by the smooth function $h$ to ensure that the integrated identity is still sufficient to recover the pointwise value of the Hessian. It is then easy to see that this notion coincides with the one given via the use of charts.

On $\RCD(K,\infty)$ spaces, as in every metric measure structure, the space $W^{1,2}(\X)$ is defined as $L^2\cap\s^2(\X)$ and is equipped with the norm $\|f\|_{W^{1,2}(\X)}^2:=\|f\|_{L^2(\mm)}^2+\|\weakgrad f\|^2_{L^2(\mm)}$. Given that  $\RCD(K,\infty)$ spaces are infinitesimally Hilbertian, gradients of Sobolev functions are well defined and so is their pointwise scalar product. Thus we can adopt the approach just described  to define  $W^{2,2}(\X)$. There are 2 things to do to ensure that this provides a meaningful definition:
\begin{itemize}
\item[1)] To explain what it is  a `$(0,2)$-tensor field in $L^2$'
\item[2)] To ensure that there are sufficiently many `test functions' $g_1,g_2,h$ for which the right hand side of \eqref{eq:defhessintro} makes sense so that this formula really identifies the object $\H f$. The term difficult to handle is $\nabla \la \nabla g_1,\nabla g_2\ra$ as it requires   a Sobolev regularity for $ \la \nabla g_1,\nabla g_2\ra$.
\end{itemize}
Getting this two points would allow to provide the definition of $W^{2,2}(\X)$, but such definition would be empty unless we 
\begin{itemize}
\item[3)] Prove that there are many $W^{2,2}(\X)$ functions.
\end{itemize}
We briefly see how to handle these issues.  Point $(1)$ is addressed via the general construction of {\bf tensor product of Hilbert modules}, so that the Hessian will be an element of the tensor product $L^2((T^*)^{\otimes 2}\X)$ of the cotangent module $L^2(T^*\X)$ with itself. The concept captured by this notion is the following: the tensor product of the module of $L^2$ covector fields on a Riemannian manifold with itself is the space of $L^2$ $(0,2)$-tensors on the manifold, the pointwise norm being the Hilbert-Schmidt one. The choice of the Hilbert-Schmidt norm in the construction is motived by the fact that this is the norm of the Hessian appearing in Bochner inequality, so that ultimately this is the kind of norm of the Hessian for which we will gain a control.

As for the cotangent module $L^2(T^*\X)$, the definition of the tensor product $L^2((T^*)^{\otimes 2}\X)$ comes via explicit construction. We firstly introduce the algebraic tensor product $L^2(T^*\X)\otimes^{\rm Alg}_{L^\infty}L^2(T^*\X)$  of $L^2(T^*\X)$ with itself as $L^\infty(\mm)$ module, so that this is the space of formal finite sums of objects of the kind $\omega_1\otimes\omega_2$ with $\omega_1,\omega_2\in L^2(T^*\X)$ having the standard bilinearity property and satisfying $f(\omega_1\otimes\omega_2)=(f\omega_1)\otimes\omega_2=\omega_1\otimes(f\omega_2)$.  On $L^2(T^*\X)\otimes^{\rm Alg}_{L^\infty}L^2(T^*\X)$ we then define a bilinear form $:$ with values in $L^0(\mm)$ (i.e.\ the space of Borel real valued functions equipped with the topology of $\mm$-a.e.\ convergence) by putting
\[
(\omega_1\otimes\omega_2):(\tilde\omega_1\otimes\tilde\omega_2):=\la\omega_1,\tilde \omega_1\ra\la\omega_2,\tilde \omega_2\ra, \qquad \qquad\forall \omega_1,\omega_2,\tilde\omega_1,\tilde\omega_2\in L^2(T^*\X),
\]
and extending it by bilinearity. Then the space $L^2((T^*)^{\otimes 2}\X)$ is defined as the completion of the space of $A$'s in $L^2(T^*\X)\otimes^{\rm Alg}_{L^\infty}L^2(T^*\X)$ such that $A:A\in L^1(\mm)$ equipped with the norm $\|A\|_{L^2((T^*)^{\otimes 2}\X)}:=\sqrt{\int A:A\,\d\mm}$. It is not hard to check that  $L^2((T^*)^{\otimes 2}\X)$ has a canonical structure of $L^2(\mm)$-normed module, see Section \ref{se:tensorproduct} for the details and notice that although the description we gave here  is slightly different from the one given there, the two constructions are in fact canonically equivalent.

In handling point $(2)$ we see for the first time the necessity of working on spaces with Ricci curvature bounded from below, as on arbitrary spaces it is unclear whether there are non-constant functions $g\in\s^2(\X)$ such that $|\nabla g|^2\in\s^2(\X)$. In presence of a lower Ricci curvature bound, instead, such regularity is ensured for bounded Lipschitz functions in $W^{1,2}(\X)$ whose Laplacian is also in $W^{1,2}(\X)$. This is due to a Caccioppoli-type inequality firstly observed in this setting by Bakry \cite{Bakry83} and proved in the generality we are now by Savar\'e \cite{Savare13}. Formally, the idea is to multiply by $|\nabla g|^2$ the two sides of  Bochner inequality
\begin{equation}
\label{eq:Bo1}
\Delta\frac{|\nabla g|^2}{2}\geq \la \nabla g,\nabla\Delta g\ra+K|\nabla g|^2,
\end{equation}
integrate and then integrate by parts the left hand side to obtain the estimate
\[
\int|\nabla|\nabla g|^2|^2\,\d\mm\leq 2 \Lip^2(g)\int |\nabla g||\nabla\Delta g|+K|\nabla g|^2\,\d\mm,
\]
which grants the required Sobolev regularity for $|\nabla g|^2$ provided the right hand side is finite. Since by regularization via the heat flow it is easy to produce functions $g$ for which indeed the above right hand side is finite, we have   at disposal a large class of test functions.

Point $(3)$ is the technically most delicate. Here again the deep reason which ensures the existence of many $W^{2,2}(\X)$ is the Bochner inequality: in the smooth setting we have
\begin{equation}
\label{eq:Bo2}
\Delta\frac{|\nabla f|^2}{2}\geq|\H f|_\HS^2+ \la \nabla f,\nabla\Delta f\ra+K|\nabla f|^2,
\end{equation}
which in particular after integration gives
\[
\int|\H f|_\HS^2\,\d\mm\leq \int(\Delta f)^2-K|\nabla f|^2\,\d\mm,
\]
for, say, smooth and compactly supported functions.
Recalling that $\RCD(K,\infty)$ spaces come with the Bochner inequality stated as (a proper reformulation that handles  the lack of smoothness of) inequality  \eqref{eq:Bo1} (see \cite{Gigli-Kuwada-Ohta10} and \cite{AmbrosioGigliSavare11-2}),  the question is whether it can be improved to inequality \eqref{eq:Bo2}. As mentioned in the previous section the answer is positive and comes building on top  of some intuitions of Bakry \cite{Bakry83} adapted by Savar\'e \cite{Savare13} to the non-smooth setting. Very shortly said, the basic strategy is to write inequality \eqref{eq:Bo1} for $f$ which is an appropriate polynomial function of other test functions and then optimizing in the coefficients of the chosen polynomial. See Section \ref{se:whymanyw22} for the details.

Having clarified these three points, we have at disposal the space $W^{2,2}(\X)$, which is a dense subset of $W^{1,2}(\X)$, and for each $f\in W^{2,2}(\X)$ a well defined Hessian $\H f\in L^2((T^*)^{\otimes 2}\X)$. It is then possible to establish the expected calculus rules
\[
\begin{split}
\H{fg}&=g\H f+f\H g+\d f\otimes \d g+\d g\otimes\d f,\\
\H{\varphi\circ f}&=\varphi'\circ f\H f+\varphi''\circ f\d f\otimes\d f,\\
\d\la\nabla f,\nabla g\ra&=\H f(\nabla g,\cdot)+\H g(\nabla f,\cdot),
\end{split}
\]
in a reasonable generality, see Section \ref{se:calchess}.

In a similar way, we can use the second formula in \eqref{eq:perdef2} to introduce {\bf Sobolev vector fields and their covariant derivative}. We remark that while for functions there are  several possible notions of regularity (continuity, H\"older/Lipschitz estimates, Sobolev regularity etc.), for vector fields Sobolev regularity is the only we have at disposal.

Thus we say that a vector field $X\in L^2(T\X)$ belongs to the Sobolev space ${  W^{1,2}_C(T\X)}$ provided there is an element $\nabla X$, called covariant derivative of $X$, of the tensor product  $L^2(T^{\otimes 2}\X)$ of the tangent module $L^2(T\X)$ with itself such that for every $g_1,g_2,h$ test functions as before we have
\[
\int h\, \nabla X:  (\nabla g_1\otimes \nabla g_2)\,\d\mm=\int-\la X,\nabla g_2\ra\,\div(h\nabla g_1)-h\H{ g_2}(X,\nabla g_1)\,\d\mm.
\]
Starting from the results for $W^{2,2}(\X)$, it will then be not hard to see that $W^{1,2}_C(\TX)$ is a dense subspace of $L^2(T\X)$ and that the basic calculus rules that identify $\nabla X$ as the Levi-Civita connection hold, under the natural regularity assumptions, in this setting. See Sections \ref{se:w12c} and \ref{se:lc}.

Having clarified what are Sobolev vector fields, the {\bf second order differentiation formula} comes out quite easily by iterating the first order one, see Theorem \ref{thm:secondder} for the rigorous statement and the proof.

Another construction which comes more or less for free from the language developed so far is that of {\bf connection Laplacian}: imitating the definition of diffusion operator induced by a Dirichlet form, we can say that a vector field $X\in W^{1,2}_C(T\X)$ has a connection Laplacian provided there is $Y\in L^2(T\X)$ such that 
\[
\int \la Y,Z\ra\,\d\mm=-\int \nabla X:\nabla Z\,\d\mm,\qquad \forall Z\in W^{1,2}_C(T\X),
\]
and in this case we put $\Delta_CX:=Y$. Then the heat flow of vector fields can be introduced and studied as well, see Section \ref{se:clap} for the details and notice that in fact the definition of $\Delta_C$ that we shall adopt is, for technical reasons, slightly different than the one presented here.

The very same ideas and the last formula in \eqref{eq:perdef2} - and its analogous for $k$-forms - allow to introduce the  space  of {\bf Sobolev differential forms and their  exterior differential}. Then the codifferential $\delta$ can be introduced as the adjoint of the exterior differential $\d$ and the Hodge Laplacian $\Delta_\Ho$ along the same lines used for the connection Laplacian replacing the `connection energy' $X\mapsto\frac12\int|\nabla X|^2\,\d\mm$ with the `Hodge energy' given by
\[
\omega\qquad\mapsto\qquad\frac12\int|\d \omega|^2+|\delta\omega|^2\,\d\mm,
\]
see Sections \ref{se:extd}, \ref{se:dr}  for the precise definitions. The nature of our approach grants that the exterior differential is a closed operator, which  together with the identity $\d^2=0$ allows for a reasonable definition of the {\bf de Rham cohomology} groups. It should be noticed, however, that here the space of Sobolev forms is not known to be compactly embedded in the one of $L^2$ forms, a fact that creates some issues when trying to produce the Mayer-Vietoris sequence, see Remark \ref{rem:mv}. A similar issue occurs  in searching for a form of the Hodge decomposition theorem: compare Theorem \ref{thm:hodge} with the classical decomposition of $L^2$ forms available in the smooth setting. As Kapovitch pointed out to me, examples by Anderson and Perelman - see \cite{Perelman97} - show that it is unreasonable to expect the compact embedding of Sobolev $k$-forms to be present in the setting of $\RCD(K,N)$ spaces for arbitrary $k\in\N$.

With all this machinery at disposal we will then be able to define the {\bf Ricci curvature}. The starting point here is that for vector fields of the kind $X=\sum g_i\nabla f_i$ where $f_i,g_i$ are test functions as discussed before, we have that:
\begin{itemize}
\item[-] the function $|X|^2$ admits a measure valued distributional Laplacian $\bd|X|^2$,
\item[-] $X$ belongs to $W^{1,2}_C(T\X)$ and thus it has a covariant derivative,
\item[-] up to identifying  vector and covector fields, $X$ is in the domain of the Hodge Laplacian.
\end{itemize}
Thus for such $X$ all the terms appearing in Bochner identity are well defined except for the Ricci curvature. We can therefore define the Ricci curvature as the object realizing the identity in Bochner's formula, i.e. as the measure $\ric(X,X)$ given by
\[
\ric(X,X):=\bd\frac{|X|^2}2+\big(\la X,\Delta_\Ho X\ra-|\nabla X|_\HS^2\big)\mm.
\]
It turns out that the so-defined Ricci curvature can be continuously extended to more general vector fields possessing a certain amount of Sobolev regularity and that  the same kind of estimates that produced the Bochner inequality \eqref{eq:Bo2} give the expected bound
\[
\ric(X,X)\geq K|X|^2\mm,
\]
on our $\RCD(K,\infty)$ space.

\subsection*{Some open problems}
\addcontentsline{toc}{subsection}{\protect\numberline{} Some open problems}

As we show through the text, the objects we define have many of the properties one expects by the analogy with the smooth world. Still, in order to consider the theory to be reasonably satisfactory, a number of open problems  remain to be settled. We collect here the key ones with some informal comments, others, more technical in nature, are highlighted in the body  of the work.

\bigskip

\noindent A  first basic problem concerns the link between the abstract definition of tangent module and the more geometric notion of tangent space as introduced by Gromov:
\begin{quote}
Is it true that on $\RCD^*(K,N)$ spaces there is a (canonical?) isometry between the tangent module and the collection of pmGH limits of the rescaled space?
\end{quote}
In fact, the positive answer to this question might be not far away: in the recent work \cite{Mondino-Naber14} Mondino-Naber built, for every $\eps>0$, a countable collection $(E_i)$ of Borel subsets of $\X$ and of biLipschitz maps $\varphi_i$ from $E_i$ to $\R^{n_i}$ such that the Lipschitz constants of both $\varphi_i$ and its inverse are bounded above by $1+\eps$ for every $i$. If one could also prove that the measures $(\varphi_i)_*\mm\restr{E_i}$ are close to $\mathcal L^{n_i}\restr{\varphi_i(E_i)}$ in the total variation distance, which according to the authors should be achievable with minor modifications of their arguments, then the properties of the pullback of 1-forms as given in Proposition \ref{prop:pull1form} would immediately give the conclusion.

\bigskip

\noindent Another question concerning the structure of the tangent module is:
\begin{quote}
On $\RCD^*(K,N)$ spaces, is the dimension of the tangent module constant through the space?
\end{quote}
By analogy with the Ricci-limit case and from the paper \cite{ColdingNaber12} of Colding-Naber one expects the answer to be positive, but with the current technology the replication of the arguments in  \cite{ColdingNaber12}  in the non-smooth setting still presents some difficulties.

\bigskip

\noindent Alternatively, the constant dimension of the tangent space might be proven by positively answering to:
\begin{quote}
Can we give a definition of parallel transport along a $W_2$-geodesic made of measures with uniformly bounded density granting both  existence and uniqueness? What if geodesics are replaced by more general absolutely continuous curves?
\end{quote}
Assuming that the isometry between the tangent module and pmGH limits of rescaled spaces has been produced, an answer to this question would also shed new light on Petrunin's construction of parallel transport on Alexandrov spaces \cite{Petrunin98}, for which as of today there is no uniqueness result.

\bigskip

\noindent More generally:
\begin{quote}
Can the framework proposed here be of any help in Alexandrov geometry? In particular, is it possible to produce a `working definition' of sectional curvature?
\end{quote}

\bigskip

\noindent Still concerning $W_2$-geodesics, a very important problem in terms of applicability of the second order calculus is:
\begin{quote}
Can we establish the second order differentiation formula along geodesics?
\end{quote}
This crucial case is not covered by Theorem \ref{thm:secondder}, due to the lack of regularity of geodesics themselves.

\bigskip

\noindent On a different direction, we remark that all our definitions are intrinsic in nature, in the sense that they do not rely on calculus in charts. While this is a key feature that allows for the calculus to be developed, because it ties such calculus to the regularity of the space and not to that of the charts that one is able to produce, it has as  collateral effect that questions which are trivially addressable in the classical context via a partition of the unit and reduction to the Euclidean case, become unexpectedly difficult here. One of these is:
\begin{quote}
Is $H=W$?
\end{quote}
By this we mean the following. In analogy with  what happens in the classical framework, when defining a Sobolev space  we can either proceed via integration by parts, producing  the `$W$' version of the space, or by taking the closure of smooth objects w.r.t.\ the appropriate Sobolev norm, this giving the `$H$' space. By nature of the definitions we always have $W\supset H$, but in no non-trivial case we are able to establish equality. 

\bigskip

\noindent A topic which is not studied  in this paper is the effect of the assumption of finite dimensionality. In this direction, in recent paper \cite{Sturm14} Sturm provided an abstract definition of the `dimensional' Ricci curvature tensor  ${\rm Ric}_N$ in the abstract setting of diffusion operators satisfying a curvature-dimension condition. The smoothness assumptions done in \cite{Sturm14} are apparently too strong to be satisfied on $\RCD^*(K,N)$, the question is thus:
\begin{quote}
Can we justify  Sturm's computations in the non-smooth setting of $\RCD^*(K,N)$ spaces?
\end{quote}
Due to the nature of the arguments in \cite{Sturm14} and given  the technical tools provided by Savar\'e in \cite{Savare13} and further refined here, the answer is expected to be affirmative, but details have to be checked carefully.

\bigskip

\noindent Finally, a vaguely formulated problem  is:
\begin{quote}
Better understand the Ricci curvature tensor.
\end{quote}
In fact, it is not really clear if the object defined by Theorem \ref{thm:ricci} is truly a tensor nor which are the minimal regularity on the vector fields $X,Y$ in order for $\ric(X,Y)$ to be well defined. Some comments on this problem are collected at the end of Section \ref{se:ricci}.

\section{The machinery of $L^p(\mm)$-normed modules}
\subsection{Assumptions and notation}\label{se:assnot} 
We say that   $(\X,\mathcal A,\mm)$ is a $\sigma$-finite  measured space provided  $\X$ is a set, $\mathcal A$ a $\sigma$-algebra on it and $\mm$ a non-negative measure defined on $\mathcal A$ for which there is a countable family $(E_i)\subset \mathcal A$ such that $\X=\cup_iE_i$ and $\mm(E_i)<\infty$ for every $i\in\N$.

We declare two sets $A,B\in\mathcal A$ to be equivalent if $\mm((A\setminus B)\cup(B\setminus A))=0$ and we shall denote by $\BB$ the set of all the equivalence classes. In what follows, abusing a bit the notation, we shall refer to  elements of $\BB$ as sets, thus identifying a measurable set with its equivalence class. Notice that the operations of taking the complement and of countable intersections/unions are well defined for sets in $\BB$. In this sense we shall say, for instance, that a set $A\in\BB$ is empty to mean that any representative of $A$ has measure 0, or similarly that $A\subset B$ to intend that $A\setminus B$ is empty.

Similarly, by `function' on $\X$ with values on some measurable space we mean the equivalence class w.r.t.\ $\mm$-a.e.\ equality of a measurable function on $\X$.

In those occasions where the value of a function on a $\mm$-negligible set  matters, we shall use the barred notation $\bar f$. Similarly, whenever we will have to deal with a precise representative of  a set, we shall use the barred notation $\bar A\in\mathcal A$.

Notice that the $\sigma$-finiteness ensures that 
\begin{equation}
\label{eq:maxset}
\text{a collection $\mathcal C\subset\BB$ stable by countable unions admits a unique maximal set,}
\end{equation}
where maximal is intended w.r.t.\ inclusion. It is indeed sufficient to consider a maximizing sequence $(C_n)\subset\mathcal C$ for
\[
\sup\Big\{\sum_{i}\frac{\mm(B\cap E_i)}{\mm(E_i)2^i}\ :\  B\in\mathcal C\Big\},
\]
where $(E_n)\subset \BB$ is a countable partition of $\X$ into measurable sets of finite and positive measure, and to define $C:=\cup_nC_n\in\mathcal C$: it is easy to see that $C$ is the unique maximal set.

Given a set $A\in\BB$ we denote by $\nchi_A\in L^\infty(\X,\mm)$ the characteristic function of $A$, $\mm$-a.e.\ defined by $\nchi_A(x)=1$ for  $x\in A$ and $\nchi_A(x)=0$ for $x\in A^c$.

For brevity, we will indicate the spaces $L^p(\X,\mm)$ simply by $L^p(\mm)$. We shall also make use of the space of simple functions ${\rm Sf}(\mm)\subset L^\infty(\mm)$, i.e.\ the space of functions attaining only a finite number of values.

We remind that if $(f_i)_{i\in I}$ is any collection, not necessarily countable, of functions on $\X$ with values in $\R\cup\{\pm\infty\}$, the essential supremum is defined as the minimal function $f:\X\to\R\cup\{\pm\infty\}$ such that $f\geq f_i$ $\mm$-a.e.\ for every $i\in I$. It exists and is unique.

\subsection{Basic definitions and properties}
Here we introduce the basic concept of $L^\infty(\mm)$-module and $L^p(\mm)$-normed module. Throughout all this section, $(\X,\mathcal A,\mm)$ is a given fixed $\sigma$-finite measured space.

\vspace{1cm}

We start with the following general definition:
\begin{definition}[$L^\infty(\mm)$-modules]  
An $L^\infty(\mm)$-premodule is a Banach space $(\MM,\|\cdot\|_\MM)$ endowed with a bilinear map 
\[
\begin{split}
L^\infty(\mm)\times \MM& \quad \to\quad \MM,\\
(f,v)&\quad \to\quad f\cdot v,
\end{split}
\]
satisfying 
\[
\begin{split}
(fg)\cdot v&=f\cdot(g\cdot v),\\
{\mathbf 1}\cdot v&= v,\\
\|f\cdot v\|_\MM&\leq \|f\|_{L^\infty(\mm)}\|v\|_\MM,\\
\end{split}
\]
for every $v\in \MM$ and $f,g\in L^\infty(\mm)$, where ${\mathbf 1}\in L^\infty(\mm)$ is the function identically equal to 1.

An $L^\infty(\mm)$-module  is an $L^\infty(\mm)$-premodule which further has the following two properties:
\begin{itemize}
\item[-]\underline{Locality}. For every $v\in \MM$ and $A_n\in\BB$, $n\in\N$, we have
\begin{equation}
\label{eq:locality}
\nchi_{A_n}\cdot v=0,\qquad\forall  n\in\N\qquad\Rightarrow\qquad \nchi_{\cup_{n\in\N}A_n}\cdot v=0.
\end{equation}
\item[-]\underline{Gluing}. For every sequence $(v_n)\subset \MM$ and sequence $(A_n)\subset\BB$  such that 
\begin{equation}
\label{eq:gluingip}
\nchi_{A_i\cap A_j}\cdot v_i=\nchi_{A_i\cap A_j}\cdot v_j,\quad\forall i,j\in\N,\qquad\text{ and }\qquad \lims_{n\to\infty}\big\|\sum_{i=1}^n\nchi_{A_i}\cdot v_i\big\|_\MM<\infty,
\end{equation}
there exists $v\in \MM$ such that 
\[
\nchi_{A_i}\cdot v=\nchi_{A_i}\cdot v_i,\quad\forall i\in\N\qquad\text{ and }\qquad \|v\|_\MM\leq \limi_{n\to\infty}\big\|\sum_{i=1}^n\nchi_{A_i}\cdot v_i\big\|_\MM.
\]
\end{itemize}
Given two $L^\infty(\mm)$-modules $\MM_1,\MM_2$, a map $T:\MM_1\to \MM_2$ is a module morphism provided it is a bounded linear map from $\MM_1$ to $\MM_2$ viewed as Banach spaces and further satisfies the locality condition
\[
T(f\cdot v)=f\cdot T(v),\qquad\forall v\in \MM_1,\ f\in L^\infty(\mm).
\]
The set of all module morphisms from $\MM_1$ to $\MM_2$ will be denoted by $\Hom(\MM_1,\MM_2)$.
\end{definition}
In the following we shall often omit the dot $\cdot$ when indicating the multiplication with functions, keeping it only when we believe it clarifies the expression.

Notice that we stated the gluing property the way we did to point out the analogy with the definition of sheaf, but in fact in our situation our `base sets' are not open sets of a topology but rather elements of a $\sigma$-
algebra.  In particular, we can freely consider their difference and thus up to replace the $A_n$'s with $A_n\setminus\cup_{i<n}A_i$, when checking the gluing property we can restrict the attention to sequences 
$(A_n)\subset \BB$ made of disjoint sets: in this way the compatibility condition $\nchi_{A_i\cap A_j}\cdot v_i=\nchi_{A_i\cap A_j}\cdot v_j$ is automatically trivially satisfied.

\bigskip

The basic examples of $L^\infty(\mm)$-modules  are $L^p$ spaces and $L^p$ vector fields on a Riemannian manifold:
\begin{example}[$L^p$ spaces as $L^\infty$-modules]{\rm
The Banach space $L^p(\mm)$ has a natural structure of $L^\infty(\mm)$-module, the multiplication with a function in $L^\infty(\mm)$ being just the pointwise one.
}\fr\end{example}
\begin{example}[$L^p$ vector fields as $L^\infty$ modules]{\rm
Let $M$ be a smooth Riemannian manifold, $\vol$ its volume measure and $p\in[1,\infty]$. Then the space of $L^p(\vol)$ vector fields has a natural structure of $L^\infty(\vol)$ module, the multiplication with an $L^\infty(\vol)$ function being again the pointwise one.
}\fr\end{example}

The following two simple examples show that both the locality and the gluing property can fail in $L^\infty$-premodules:
\begin{example}[Lack of locality]\label{ex:nolocal}{\rm Let $(\X,\mathcal A,\mm)$ be the interval $[0,1]$ equipped with the Borel $\sigma$-algebra and the Lebesgue measure. Use the Hahn-Banach theorem to find $\ell:L^\infty(\mm)\to\R$ linear,  such that $|\ell(f)|\leq \|f\|_{L^\infty(\mm)}$ for every $f\in L^\infty(\mm)$ and $\ell(f)=1$  on functions $f\in L^\infty(\mm)$ such that $\exists \eps>0$ for which $f(x)=1$ for a.e. $x\in[0,\eps]$.

Now let $B\neq\{0\}$ be an arbitrary Banach space and define the bilinear map $L^\infty(\mm)\times B\to B$  as
\[
L^\infty(\mm)\times B\ni (f,v)\qquad\mapsto \qquad f\cdot v:=\ell(f)v\in B.
\]
It is clear that this structure is a $L^\infty(\mm)$-premodule. On the other hand, locality does not hold. Indeed, let  $A_n:=[\tfrac1n,1]$ and observe that $\ell(\nchi_{A_n})=0$, so that $ \nchi_{A_n}\cdot v=0$, for every $n\in\N$. Yet, $\ell(\nchi_{\cup_n}A_n)=\ell({\mathbf 1})=1$ and therefore  $ \nchi_{\cup_nA_n}\cdot v=v$ for every $v\in B$.
}\fr\end{example}
\begin{example}[Lack of gluing]{\rm
Let $c_0$ be the classical Banach space of sequences of real numbers converging to 0 equipped with the $\sup$ norm and consider $\X:=\N$ with the discrete $\sigma$-algebra and the counting measure $\mm$. Then $L^\infty(\mm)$ consists of the space of bounded sequences of real numbers and multiplication component by component endows $c_0$ with the structure of a $L^\infty(\mm)$-premodule.  Yet, gluing fails. Indeed, let $A_n:=\{n\}$ and $v_n\in c_0$ be with the first $n$ terms equal to 1 and the rest equal to 0. It is clear that the `glued' $v$  should  be the sequence identically 1, which however is not in $c_0$.
}\fr\end{example}
One of the main effects of the locality condition in the definition of $L^\infty$-module is that it allows to {\bf define the set} $\mathbf{\{v=0\}}\in\BB$ for a generic element $v$ of an $L^\infty(\mm)$-module $\MM$. To see in which sense, for $v\in \MM$ and $A\in\BB$ 
\[
\text{we say that $v=0$ $\mm$-a.e.\ on $A$ provided $\nchi_Av=0$}
\] 
and in this case we also say that $v$ is concentrated on $A^c$. It is then easy to see that on an arbitrary $L^\infty(\mm)$-premodule if $v=0$ $\mm$-a.e.\ on $A_i$, $i=1,\ldots, n$, then it is 0 on $\cup_{i=1}^nA_i$: just notice that $\nchi_{\cup_{n=1}^N}A_n=f\sum_{n=1}^N\nchi_{A_n}$ for some $f\in L^\infty(\mm)$  to deduce that
\[
\|\nchi_{\cup_{n=1}^NA_n} v\|_\MM=\Big\|\big(f\sum_{n=1}^N\nchi_{A_n}\big)  v\Big\|_\MM\leq\|f\|_{L^\infty(\mm)}\Big\|\sum_{n=1}^N\nchi_{A_n} v\Big\|_\MM\leq\|f\|_{L^\infty(\mm)}\sum_{n=1}^N\|\nchi_{A_n} v\|_\MM=0.
\]
The role of the locality condition is to extend this property to countable unions (which, as Example \ref{ex:nolocal} shows, may be not true on $L^\infty(\mm)$-premodules), indeed we can restate \eqref{eq:locality} as:
\[
v=0,\quad\mm\ae\text{ on}\ A_n,\ \forall n\in\N\qquad\Rightarrow\qquad v=0,\quad\mm\ae\text{ on}\ \bigcup_{n\in\N}A_n.
\]
Thus the simple property \eqref{eq:maxset} ensures that  there exists a unique maximal (w.r.t.\ inclusion) set in $ \BB$  on which $v$ is 0: we shall denote this maximal set  by  $\{v= 0\}$ and its complement by $\{v\neq0\}$.

For $v,w\in \MM$ we say that $v=w$ $\mm$-a.e.\ on $A$ provided $v-w=0$  $\mm$-a.e.\ on $A$. Then the sets $\{v=w\},\{v\neq w\}\in\BB$ are well defined as $\{v-w=0\}$ and $\{v-w\neq 0\}$ respectively.

\bigskip

A closed subspace  $\NN\subset\MM$ which is stable  w.r.t.\  multiplication with $L^\infty(\mm)$-functions is easily seen to be an $L^\infty(\mm)$-premodule having the locality property. If it is also closed w.r.t.\ the gluing operation, then it is an $L^\infty(\mm)$-module and we shall say that it is a {\bf submodule} of $\MM$. It is readily verified that the kernel of a module morphism is always a submodule.

Given $E\in\BB$, a simple example of submodule of $\MM$ is the space $\MM\restr{E}$ of elements which are $0$ $\mm$-a.e.\ on $E^c$.

Given a submodule $\NN$ of the module $\MM$, the {\bf quotient} space $\MM/\NN$ has  a natural structure of $L^\infty(\mm)$-premodule with the locality property. Indeed, the norm  
\[
\|[v]\|_{\MM/\NN}:=\inf_{w\in \NN}\|v+w\|_\MM
\]
is well defined and complete, and the fact that for any $v\in \MM$, $w\in \NN$ and $f\in L^\infty(\mm)$  we have  $f(v+w)-fv\in \NN$   shows that the definition $f[v]:=[fv]$ is well posed. It is then clear that $\MM/\NN$ is an $L^\infty(\mm)$-premodule. To see that locality holds, notice that $\nchi_A[v]=0$ if and only if $\nchi_Av\in \NN$ and consider $v\in \MM$ and a sequence $(A_n)\subset \BB$ such that $\nchi_{A_n}[v]=0$ for every $n\in\N$. Then $v_n:=\nchi_{A_n}v\in \NN$ for every $n\in\N$ and using the gluing property in $\NN$ for the sequences $(v_n)\subset \NN$ and $(A_n)\subset \BB$ we deduce the existence of $\bar v\in \NN$ such that $\nchi_{A_n}\bar v=v_n$ for every $n\in\N$, which is the same as to say that $\nchi_{A_n}(\bar v-v)=0$ for every $n\in\N$. By the locality property in $\MM$ we conclude that $\nchi_{\cup_nA_n}(\bar v-v)=0$ which yields $\nchi_{\cup_nA_n}v\in \NN$, i.e.\ $\nchi_{\cup_nA_n}[v]=0$, as desired. In this generality it is not clear to us whether $\MM/\NN$ has the gluing property, for a positive result under additional assumptions see Proposition \ref{prop:normevarie}.

Notice that the submodule $\MM\restr E$ is canonically identifiable with the quotient module $\MM/\MM\restr{E^c}$.

\bigskip

We claim that for  two given $L^\infty(\mm)$-modules $\MM,\NN$, the set $\Hom(\MM,\NN)$ has a canonical structure of $L^\infty(\mm)$-module. Indeed, given that $\MM,\NN$ are Banach spaces, $\Hom(\MM,\NN)$ is a Banach space when endowed with the operator norm $\|T\|:=\sup_{v:\|v\|_\MM\leq 1}\|T(v)\|_\NN$. Then the fact that $L^\infty(\mm)$ is a commutative ring ensures that for $T\in \Hom(\MM,\NN)$ and $f\in L^\infty(\mm)$ the operator $fT:\MM\to \NN$ defined by
\[
(fT)(v):=f(T(v)),\qquad\forall v\in \MM,
\]
is still a module morphism and is then clear that with this multiplication $\Hom(\MM,\NN)$ is an $L^\infty(\mm)$-premodule. To see that locality holds, let $T\in\Hom(\MM,\NN)$ and $(A_n)\subset \BB$ such that $\nchi_{A_n}T=0$ for every $n\in\N$. Then for given $v\in  \MM$ we have $0=(\nchi_{A_n}T)(v)=\nchi_{A_n}(T(v))$ for every $n\in\N$ which, by the  locality property in $\NN$, yields $\nchi_{\cup_nA_n}T(v)=0$ and thus, by the arbitrariness of $v$, that $\nchi_{\cup_nA_n}T=0$. For the gluing the argument is analogous: let $(T_n)\subset\Hom(\MM,\NN)$ and $(A_n)\subset \BB$ be such that $\nchi_{A_i\cap A_j}T_i=\nchi_{A_i\cap A_j}T_j$ for every $i,j\in \N$ and $\|\sum_{i=1}^n\nchi_{A_i}T_i\|\leq C$ for every $n\in\N$ and some $C>0$. Pick $v\in \MM$ and use the gluing property in $\NN$ for the sequences $(T_n(v))$ and $(A_n)$: since 
\[
\|\sum_{i=1}^n\nchi_{A_i}T_i(v)\|_\NN\leq \|\sum_{i=1}^n\nchi_{A_i}T_i\|\|v\|_\MM\leq C\|v\|_\MM
\] 
for every $n\in\N$, we  get the existence of $w\in \NN$ with $\|w\|_\NN\leq C\|v\|_\MM$ and such that $\nchi_{A_i}w=\nchi_{A_i}T_i(v)$ for every $i\in\N$. Then define  $T(v):=\nchi_{\cup_nA_n}w$, notice that $\|T(v)\|_\NN\leq \|w\|_\NN\leq C\|v\|_\MM$ and that the locality property in $\NN$ grants that $T(v)$ is well defined, i.e.\ it does not depend on the particular choice of  $w$. It is then clear that the map $v\mapsto T(v)$ is a module morphism with norm bounded above by $C$ such that $\nchi_{A_n}T=\nchi_{A_n}T_n$ for every $n\in\N$.

\bigskip

Recalling that  $L^1(\mm)$ has a natural structure of $L^\infty(\mm)$-module, we propose the following definition:
\begin{definition}[Dual module]
Let $\MM$ be an  $L^\infty(\mm)$-module. The dual module $\MM^*$ is defined as $\Hom(\MM,L^1(\mm))$.
\end{definition}
From a purely algebraic perspective it might be bizarre to define the dual of an $L^\infty(\mm)$-module as the set of morphism from the module to $L^1(\mm)$, rather than to $L^\infty(\mm)$. The choice is motivated by the following two simple facts:
\begin{example}[The dual of $L^p(\mm)$ is $L^q(\mm)$]{\rm For $p\in[1,\infty]$ the dual module of $L^p(\mm)$ can naturally be identified with $L^q(\mm)$, where $\frac1p+\frac1q=1$, in the sense that for every $g\in L^q(\mm)$ the map 
\[
L^p(\mm)\ni f\qquad\mapsto\qquad T(f):=fg\in L^1(\mm),
\]
is a module morphism (this is trivial) and conversely for every $T\in \Hom(L^p(\mm),L^1(\mm))$ there exists $g\in L^q(\mm)$ such that $T(f)=fg$ for every $f\in L^p(\mm)$. 

To see the latter, let for the moment  $p<\infty$ and $T\in \Hom(L^p(\mm),L^1(\mm))$.  Then the map $L^p(\mm)\ni f\mapsto \int_\X T(f)\,\d\mm\in \R$ is bounded and linear and thus by the classical $L^p-L^q$ duality we know that there exists a unique $g\in L^q(\mm)$ such that $\int_\X T(f)\,\d\mm=\int_\X fg\,\d\mm$.  We claim that $T(f)=fg$ $\mm$-a.e. Being both $T(f)$ and $fg$ two functions in $L^1(\mm)$, to prove that they are the same it is sufficient to prove that for any   $A\in\BB$ it holds $\int_AT(f)\,\d\mm=\int_Afg\,\d\mm$, which  follows from 
\[
\int_AT(f)\,\d\mm=\int_\X\nchi_AT(f)\,\d\mm=\int_\X T(\nchi_Af)\,\d\mm=\int_\X\nchi_Afg\,\d\mm=\int_Afg\,\d\mm.
\]
In the case $p=\infty$ just put  $g:=T({\bf 1})\in L^1(\mm)$ and observe that
\[
T(f)=T(f{\bf 1})=fT({\bf 1})=fg,\qquad\forall f\in L^\infty(\mm).
\]
In particular, we see that the dual of  $L^\infty(\mm)$ as $L^\infty(\mm)$-module, is precisely $L^1(\mm)$, in contrast with the difficult characterization of the dual of $L^\infty(\mm)$ as Banach space. This is not surprising, because while $L^\infty(\mm)$ viewed as Banach space is  in general infinite dimensional and non-separable, the same space viewed as $L^\infty(\mm)$ module has dimension 1, being every function a multiple, in the sense of modules, of the constant function ${\bf 1}$.
}\fr\end{example}

\begin{example}[$\Hom(L^p(\mm),L^\infty(\mm))=\{0\}$ for $p<\infty$]{\rm Let $p<\infty$ and assume there is $T:L^p(\mm)\to L^\infty(\mm)$ not zero. Then we can find $E\in\BB$ with $\mm(E)>0$, $v\in L^p(\mm)$ and numbers $ a,b> 0$ such that $|v|\leq a$ and $|T(v)|\geq b$ $\mm$-a.e.\ on $E$. Now let $f\in L^p\setminus L^\infty(\mm)$ be non-negative and identically 0 on $\X\setminus E$. For $n\in\N$ put $f_n:=\min\{f,n\}$ and notice that 
\[
\|f_nv\|_{L^p(\mm)}=\int_E f^p_n|v|^p\,\d\mm\leq a^p\|f\|_{L^p(\mm)}^p,\qquad\forall n\in\N,
\]
and on the other side 
\[
\|T(f_nv)\|_{L^\infty(\mm)}=\|f_nT(v)\|_{L^\infty(\mm)}\geq \|f_n\|_{L^\infty(\mm)}b= nb,\qquad\forall n\in\N.
\]
These two inequalities show that $T$ is not a bounded operator from $L^p(\mm)$ to $L^\infty(\mm)$, contradiction the assumption.
}\fr\end{example}

Given an $L^\infty(\mm)$-module $\MM$, consider it as a Banach space and denote by $\MM'$ its dual Banach space. Integration provides a canonical map $\intmap_{\MM^*}:\MM^*\to \MM'$
\[
\text{given $T\in \MM^*$ we define $\intmap_{\MM^*}(T)\in \MM'$ as }\qquad\intmap_{\MM^*}(T)(v):=\int_\X T(v)\,\d\mm,\qquad\forall v\in \MM.
\]
It is obvious that $\intmap_{\MM^*}(T)$ is indeed an element of $\MM'$ satisfying  $\|\intmap_{\MM^*}(T)\|_{\MM'}\leq \|T\|_{\MM^*}$. Actually, it is easily seen that $\intmap_{\MM^*}$ is norm preserving:  for given $T\in\MM^*$ and $v\in\MM$ with $\|v\|_\MM\leq 1$, let $f:={\rm sign}\, T(v)\in L^\infty(\mm)$, $\tilde v:=fv$ and notice that from $\|f\|_{L^\infty(\mm)}\leq 1$ we get   $\|\tilde v\|_\MM\leq \|v\|_\MM\leq 1$ and thus
\[
\|T(v)\|_{L^1(\mm)}= \int |T(v)|\,\d\mm= \int fT(v)\,\d\mm=\int T(\tilde v)\,\d\mm=\intmap_{\MM^*}(T)(\tilde v)\leq \|\intmap_{\MM^*}(T)\|_{\MM'}.
\]
The claim then follows recalling that  $\|T\|_{\MM^*}$ is defined as $\sup_{v:\|v\|_\mm\leq 1}\|T(v)\|_{L^1(\mm)}$.

Hence $\MM^*$ is isometrically embedded in $\MM'$. As the example of $L^\infty(\mm)$ viewed as $L^\infty(\mm)$-module shows, in general, such embedding is not surjective.
\begin{definition}[Modules with full-dual]
We say that the $L^\infty(\mm)$-module $\MM$ has full-dual provided the map $\intmap_{\MM^*}$ described above is surjective.
\end{definition}
It is worth to underline that for a generic $L^\infty(\mm)$-module we are not able to ensure the existence of a non-zero element of the dual, the problem being the lack of an analogous of the Hahn-Banach theorem for modules. It is possible that this is due to a somehow incorrect choice of the axiomatization of such generic structures. We chose to start our discussion with the definition of $L^\infty(\mm)$-module only to provide a unifying framework, but in fact we shall mostly work with $L^p(\mm)$-normed modules, defined as follows.
\begin{definition}[$L^p(\mm)$-normed modules]
Let $\MM$ be an $L^\infty(\mm)$-premodule and $p\in[1,\infty]$. We say that $\MM$  is an $L^p(\mm)$-normed premodule provided there exists a map $|\cdot |:\MM\to L^p(\mm)$ with non-negative values such that
\begin{equation}
\label{eq:pointnorm}
\begin{split}
\||v|\|_{L^p(\mm)}&=\|v\|_\MM,\\
|f v|&=|f| |v|,\quad\mm\ae,
\end{split}
\end{equation}
for every $v\in\MM$ and $f\in L^\infty(\mm)$. In this case we shall call $|\cdot|$ the pointwise $L^p(\mm)$-norm, or simply pointwise norm. $L^p(\mm)$-normed premodules which are also $L^\infty(\mm)$-modules will be called $L^p(\mm)$-normed modules.
\end{definition}
Notice that the simple inequality
\[
\||v|-|w|\|_{L^p(\mm)}\leq\| |v-w|\|_{L^p(\mm)}=\|v-w\|_\MM,\qquad\forall v,w\in\MM,
\]
shows that the pointwise norm is a continuous map from $\MM$ to $L^p(\mm)$.
\begin{example}\label{ex:lpmod}{\rm
Forgetting about measurability issues and identification $\mm$-a.e., we can think to an $L^p(\mm)$-normed module as follows. To each $x\in \X$ is  associated a Banach space $(B_x,\|\cdot\|_x)$ and we consider   maps $v$ associating to each $x\in \X$ a vector $v(x)\in B_x$ in such a way that the $L^p(\mm)$-norm of $x\mapsto \|v(x)\|_x$ is finite.
}\fr\end{example}
In fact, the main idea behind the definition of $L^p(\mm)$-normed modules is to make an abstract version of the construction in this example which does not rely on the a priori given spaces $(B_x,\|\cdot\|_x)$. Still, it is  interesting to note that this informal procedure can always somehow be reversed, so that appropriate Banach spaces $(B_x,\|\cdot\|_x)$ really exist behind/can be built from a given $L^p(\mm)$-normed module: for results in this direction see \cite{HLR91}, we shall detail the argument only for the important class of Hilbert modules - see Definition \ref{def:hilmod}, Proposition \ref{prop:hill2} and Theorem \ref{thm:structhil}.

It is worth to point out a conceptual difference between the case of abstract $L^p(\mm)$-normed modules and Example \ref{ex:lpmod}: while in Example \ref{ex:lpmod} elements of the module are defined as equivalence classes w.r.t.\ $\mm$-a.e.\ equality of some objects which are truly defined at every point $x\in \X$ (much like elements of $L^p(\mm)$  are  equivalence classes of functions), this is not the case for a generic element of an $L^p(\mm)$-normed module, which should rather be thought of as something intrinsically defined only up to $\mm$-a.e.\ equality.

Moreover, even in the case where the Banach spaces $(B_x,\|\cdot\|_x)$ exist, there is in general no relation at all between $B_x$ and $B_{x'}$ for $x\neq x'$. In particular, there is no metric/topology on $\sqcup_xB_x$ which means that it makes no sense to speak about, say, Lebesgue points of an element of an $L^p(\mm)$-normed module.

With that said, we turn to the study of $L^p(\mm)$-normed (pre)modules: their very basic properties are collected in the following proposition.
\begin{proposition}[Basic properties of $L^p(\mm)$-normed premodules]\label{prop:baselp}
Let $\MM$ be an $L^p(\mm)$-normed premodule, $p\in[1,\infty]$. Then the following holds.
\begin{itemize}
\item[i)] For $v\in\MM$ and $E\in\BB$ we have $v=0$ $\mm$-a.e.\ on $E$ if and only if $|v|=0$ $\mm$-a.e.\ on $E$.
\item[ii)] $\MM$ has the locality property and the pointwise $L^p(\mm)$-norm is unique and fulfills the pointwise triangle inequality
\begin{equation}
\label{eq:pointtriangle}
\begin{split}
|v+w|&\leq|v|+|w|,\qquad\mm\ae,\qquad\forall v,w\in \MM.
\end{split}
\end{equation}
\item[iii)] If there exists $v\in \MM$ and $E\in \BB$ such that $0\neq \nchi_Ev\neq v$ then $\MM$ is not an $L^{p'}(\mm)$-normed premodule for $p'\neq p$. 
\item[iv)] If $p<\infty$ then $\MM$ has also the gluing property and is therefore an $L^p(\mm)$-normed module.
\item[v)] Let $\MM_1,\MM_2$ be $L^{p_1}(\mm)$ and $L^{p_2}$-normed modules respectively, $p_1\geq p_2\in[1,\infty]$, and $T:\MM_1\to\MM_2$ a linear map. Then $T\in\Hom(\MM_1,\MM_2)$ if and only if there is $g\in L^q(\mm)$with $\frac1{p_2}=\frac1{p_1}+\frac1q$ such that
\begin{equation}
\label{eq:normhomlp}
|T(v)|\leq g|v|,\qquad\mm\ae,\qquad\forall v\in\MM_1,
\end{equation}
and in this case the operator norm $\|T\|$ is the least of $\|g\|_{L^q(\mm)}$ among all the $g$'s satisfying the above.
\end{itemize}
\end{proposition}
\begin{proof}$\ $\\
\noindent{$\mathbf{(i)}$}  Direct consequence of the fact that, by definition, $v=0$ $\mm$-a.e.\ on $E$ if and only if $\|\nchi_Ev\|_\MM=0$ and the defining properties of the pointwise norm.

\noindent{$\mathbf{(ii)}$} For the locality just notice that the $L^p(\mm)$-function $|\nchi_{\cup_iA_i}v|$ is, by the second in \eqref{eq:pointnorm}, the $\mm$-a.e.\ limit of $|\nchi_{\cup_{i=1}^nA_i}v|$ as $n\to\infty$ and thus the claim follows from the first in \eqref{eq:pointnorm}.

For uniqueness, assume $|\cdot|'$ is another pointwise $L^p(\mm)$-norm, then for every $v\in \MM$ we have $|v|,|v|'\in L^p(\mm)$ and for every $A\in\BB$ it holds
\[
\|\nchi_A|v|\|_{L^p(\mm)}=\||\nchi_Av|\|_{L^p(\mm)}=\|\nchi_Av\|_\MM=\||\nchi_Av|'\|_{L^p(\mm)}=\|\nchi_A|v|'\|_{L^p(\mm)},
\]
which forces $|v|=|v|'$ $\mm$-a.e..

To prove  \eqref{eq:pointtriangle} we argue by contradiction. Then  we can find  $v,w\in \MM$, $A\in \BB$  with $\mm(A)\in(0,\infty)$ and numbers $a,b,\eps>0$ such that $|v|\leq a$, $|w|\leq b$ and $|v+w|\geq a+b+\eps$ $\mm$-a.e.\  on $A$.  But in this case we would have
\[
\begin{split}
\|\nchi_A(v+w)\|_\MM&=\|\nchi_A|v+w|\|_{L^p(\mm)}\geq \mm(A)^{\frac1p}(a+b+\eps)\\
&>\mm(A)^{\frac1p}(a+b)\geq\|\nchi_A|w|\|_{L^p(\mm)}+\|\nchi_A|w|\|_{L^p(\mm)}=\|\nchi_Av\|_\MM+\|\nchi_Aw\|_\MM,
\end{split}
\]
contradicting the triangle inequality in $\MM$ (here $\mm(A)^{\frac1p}$ is taken equal to 1 if $p=\infty$).

\noindent{$\mathbf{(iii)}$}  Say for instance that $p<\infty$ and notice that in the identity
\[
\|a\nchi_Av+b\nchi_{A^c}v\|_\MM=\sqrt[p]{|a|^p\|\nchi_Av\|^p_\MM+|b|^p\|\nchi_{A^c}v\|^p_\MM},
\]
the left hand side does not depend on $p$ and therefore the right hand side identifies, by the arbitrariness of $a,b\in\R$, the exponent $p$. Similarly for $p=\infty$.

\noindent{$\mathbf{(iv)}$} Let $(v_n)\subset M$ and  $(A_n)\subset \BB$ as in \eqref{eq:gluingip}, put $\tilde A_n:=A_n\setminus\cup_{i<n}A_i$, $\tilde v_n:=\sum_{i=1}^n\nchi_{\tilde A_i}v_i$ and notice that the sequence of maps $(|\tilde v_n|)$ is non-decreasing and, by assumption, bounded in $L^p(\mm)$. Being $p<\infty$, this is sufficient to deduce that $(|\tilde v_n|)$ is a Cauchy sequence in $L^p(\mm)$ and since by construction for every $m,n\in\N$ we have $|\tilde v_n-\tilde v_m|=\big||\tilde v_n|-|\tilde v_m|\big|$ $\mm$-a.e., we see that $(\tilde v_n)$ is  a Cauchy sequence in $\MM$. It is then easy to see that the limit of such sequence is a gluing of the $v_n$'s.

\noindent{$\mathbf{(v)}$} Assume that $T$ is linear and that \eqref{eq:normhomlp} holds for some $g\in L^q(\mm)$. Then the fact that  $T$ is continuous and  $\|T\|\leq \|g\|_{L^q(\mm)}$ is obvious. To check that it is a module morphism, pick $v\in\MM_1$, $E\in\BB$ and notice that \eqref{eq:normhomlp} gives $|T(\nchi_Ev)|\leq g|\nchi_Ev|=0$ $\mm$-a.e.\ on $E^c$. Considering also the same bound with $E^c$ in place of $E$ we get
\[
T(\nchi_Ev)=0,\quad\mm\ae\text{ on }E^c\qquad\text{and}\qquad T(\nchi_{E^c}v)=0,\quad\mm\ae\text{ on }E.
\]
Since by linearity we have $T(\nchi_Ev)+T(\nchi_{E^c}v)=T(v)$, the above is sufficient to ensure $T(\nchi_Ev)=\nchi_Ev$. Then by linearity we can check the identity $T(fv)=fT(v)$ for simple functions $f$ and conclude that the same holds for general $f\in L^\infty(\mm)$ by approximation taking into account the continuity of $T$.

Conversely, define $g:\X\to[0,\infty]$ as
\[
g:=\esssup_{v\in\MM\ :\ |v|\leq 1\ \mm\ae}|T(v)|.
\]
We claim that 
\[
|T(v)|\leq g|v| \quad\mm\ae,\qquad\forall v\in \MM.
\]
Indeed, fix $v$, put $f_n:=\min\{\frac1{|v|},n\}\in L^\infty(\mm)$ and notice that $|f_nv|\leq 1$ $\mm$-a.e.\ for every $n\in\N$. Hence
\[
|T(v)|=\frac1{f_n}|f_nT(v)|=\frac1{f_n}|T(f_nv)|\leq\frac{g}{f_n}=\left\{\begin{array}{ll}
|v|g&\mm\ae \ \textrm{on }|v|\geq n^{-1},\\
\frac{g }n&\mm\ae \ \textrm{on }|v|<n^{-1},
\end{array}\right.
\]
and the claim follows letting $n\to\infty$. To conclude it is sufficient to prove that $\|g\|_{L^q(\mm)}\leq\|T\|$ and  to this aim start recalling that, by the properties of the essential-supremum, there is a sequence $(v_n)\subset \MM$ such that $|v_n|\leq 1$ $\mm$-a.e.\ for every $n$ and such that $g=\sup_n|T(v_n)|$ $\mm$-a.e..  Recursively define a sequence $(w_n)\subset \MM$ by putting $w_1:=v_1$
and then for every $n\geq 1$
\[
w_{n+1}:=\nchi_{A_n}w_n+\nchi_{A_n^c}v_{n+1},
\]
where $A_n:=\{|T(w_n)|\geq |T(v_{n+1})|\}$. In this way we have that $|w_n|\leq 1$ $\mm$-a.e.  and $|T(w_n)|=\max_{i\leq n}|T(v_n)|$ $\mm$-a.e. for every $n\in\N$. In particular, we have 
\begin{equation}
\label{eq:wnperT}
|T(w_n)|\leq |T(w_{n+1})|\qquad \textrm{ and }\qquad \lim_{n\to\infty}|T(w_n)|=g\qquad \mm\ae.
\end{equation}
Now assume $p_2<p_1<\infty$, so that also $q<\infty$, let $f\in L^q\cap L^\infty(\mm)$ be non-negative and such that $f\leq g $ $\mm$-a.e. and put $\tilde w_n:=f^{\frac q{p_1}}w_n\in\MM_1$. Letting $n\to\infty$ in the inequality
\[
\sqrt[p_2]{\int|f|^{\frac{qp_2}{p_1}}|T( \omega_n)|^{p_2}\,\d\mm}=\|T(\tilde\omega_n)\|_{\MM_2}\leq\|T\|\|\tilde\omega_n\|_{\MM_1}\leq\|T\|\,\|f\|_{L^q(\mm)}^{\frac{q}{p_1}}
\]
and using the monotone convergence theorem (recall \eqref{eq:wnperT}) we obtain
\[
\sqrt[p_2]{\int|f|^{\frac{qp_2}{p_1}}g^{p_2}\,\d\mm}\leq\|T\|\,\|f\|_{L^q(\mm)}^{\frac{q}{p_1}}
\]
and therefore, since $f\leq g$, that $\|f\|_{L^q(\mm)}^{\frac{q}{p_2}}\leq \|T\|\|f\|_{L^q(\mm)}^{\frac{q}{p_1}}$, i.e. $\|f\|_{L^q(\mm)}\leq \|T\| $. Being this true for any $f\in L^q\cap L^\infty(\mm)$ with $f\leq g $, we conclude that $\|g\|_{L^q(\mm)}\leq \|T\| $, as desired. 

If $q=p_2<p_1=\infty$ the thesis follows letting $n\to\infty$ in the inequality
\[
\||T(w_n)|\|_{L^{p_2}(\mm)}=\|T(w_n)\|_{\MM_2}\leq \|T\|\|w_n\|_{\MM_1}\leq \|T\|.
\]
If $p:=p_2=p_1$ we have $q=\infty$ and for every $E\in\BB$ with $\mm(E)\in(0,\infty)$ the inequality $\|T(\nchi_E\mm(E)^{-\frac1p}w_n)\|_{\MM_2}\leq \|T\|\|\nchi_E\mm(E)^{-\frac1p} w_n\|_{\MM_1}$ reads as
\[
\frac1{\mm(E)^{\frac1p}}\|\nchi_E|T(w_n)|\|_{L^p(\mm)}\leq \|T\|,
\]
which gives the conclusion letting $n\to\infty$ by the arbitrariness of $E$..
\end{proof}
An element $v$ of an $L^p(\mm)$-normed module is said {\bf bounded} provided $|v|\in L^\infty(\mm)$. It is easy to see that 
\begin{equation}
\label{eq:seplplim}
\text{a separable $L^p(\mm)$-normed module admits  a countable dense subset of bounded elements.}
\end{equation}
Indeed, in the case $p=\infty$ there is nothing to prove, for $p<\infty$ just notice that for every element $v$ of the module, the sequence $n\mapsto {\rm Tr}_n(v):=\nchi_{\{|v|<n\}}v$ converges to $v$, so that if $D$ is countable and dense, the set $\{{\rm Tr}_n(v):v\in D,\ n\in\N\}$ is countable and dense as well.

A somehow similar property of  $L^p(\mm)$-normed modules, valid for $p<\infty$, is that
\begin{equation}
\label{eq:dalp}
\text{$\forall v\in \MM$ and $(A_i)\subset\BB$, $i\in\N$, disjoint we have $\lim_{n\to\infty }\|\nchi_{\cup_{i=1}^nA_i}v- \nchi_{\cup_{i=1}^\infty A_i}v\|_\MM=0$,}
\end{equation}
which follows noticing that
\[
\|\nchi_{\cup_{i=1}^\infty A_i}v-\nchi_{\cup_{i=1}^n A_i}v\|^p_\MM=\|\nchi_{\cup_{i=n+1}^\infty A_i}v\|^p_\MM=\int_{\cup_{i=n+1}^\infty A_i}|v|^p\,\d\mm,
\]
and that  the rightmost side goes to 0 as $n\to\infty$ by the dominated convergence theorem.

From this fact it is easy to see that $L^p(\mm)$-normed modules, $p<\infty$,  always have  full-dual:
\begin{proposition}\label{prop:fulllp}
Let $\MM$ be an $L^p(\mm)$-normed module, $p<\infty$. Then $\MM$ has full-dual.
\end{proposition}
\begin{proof}
Let $l\in \MM'$ and for $v\in \MM$ consider the map $\mathcal A\ni \bar A\mapsto l(\nchi_{ A} v)$, where $A\in \BB$ is the equivalence class of $\bar A$. It is clear that this map is real valued and additive, while \eqref{eq:dalp} grants that it is $\sigma$-additive. In other words it is a measure and since by construction it is absolutely continuous w.r.t.\ $\mm$,  the Radon-Nikodym theorem ensures that there is a unique $L(v)\in L^1(\mm)$ such that $l(\nchi_Av)=\int_AL(v)\,\d\mm$ for every $A\in\BB$. By construction, the map $\MM\ni v\mapsto L(v)\in L^1(\mm)$ is linear, satisfies 
\begin{equation}
\label{eq:prefull}
L({\nchi_Av})=\nchi_AL(v),\qquad \forall v\in \MM,\  A\in\BB,
\end{equation}
and such that $\int_\X L(v)\,\d\mm=l(v)\leq \|l\|_{\MM'}\|v\|_{\MM}$ for every $v\in \MM$. For given $v\in \MM$ we put $f:={\rm sign}\, L(v)\in L^\infty(\mm)$, $\tilde v:=fv$, notice that $\|\tilde v\|_\MM\leq\|f\|_{L^\infty(\mm)}\|v\|_\MM\leq \|v\|_\MM$ and that the fact that $f$ takes its values in $\{-1,0,1\}$ and \eqref{eq:prefull} grant that $|L(v)|=L(\tilde v)$ so that we have
\[
\|L(v)\|_{L^1(\mm)}=\int_\X |L(v)|\,\d\mm=\int L(\tilde v)\,\d\mm \leq \|l\|_{\MM'}\|\tilde v\|_{\MM} \leq \|l\|_{\MM'}\|v\|_{\MM},
\] 
i.e.\ $\MM\ni v\mapsto L(v)\in L^1(\mm)$ is continuous.

Finally, approximating a generic function in $L^\infty(\mm)$ with simple functions    and using again \eqref{eq:prefull}, linearity and continuity we  see that $ L(fv)=fL(v)$ for every $v\in \MM$ and $f\in L^\infty(\mm)$, so the proof is complete.
\end{proof}
Some useful constructions based on $L^p(\mm)$-normed modules are the following:
\begin{proposition}\label{prop:normevarie}
Let $\MM$ be an $L^p(\mm)$-normed module, $p\in[1,\infty]$. Then the following hold.
\begin{itemize}
\item[i)] Let $q\in[1,\infty]$ be such that $\frac1p+\frac1q=1$. Then $\MM^*$ is an $L^q(\mm)$-normed module, the pointwise norm being defined by 
\begin{equation}
\label{eq:normdual}
|L|_*:=\esssup_{v\in \MM\ :\ |v|\leq 1\ \mm\ae}\ |L(v)|.
\end{equation}
\item[ii)] Let $\NN\subset\MM$ be a submodule. Then $\MM/\NN$ is an $L^p(\mm)$-normed module, the pointwise norm being defined by
\[
|[v]|:=\essinf_{w\in \NN}|v+w|.
\]
\end{itemize}
\end{proposition}
\begin{proof} $\ $\\
\noindent{$\mathbf{(i)}$} By point $(v)$ of Proposition \ref{prop:baselp} with $p_2=1$ we see that $\|L\|=\||L|_*\|_{L^q(\mm)}$ so that to conclude it is sufficient to notice  that, directly by the definition, we have $|fL|_*=|f||L|_*$ $\mm$-a.e.\ for every $f\in L^\infty(\mm)$ and $L\in\MM^*$.

\noindent{$\mathbf{(ii)}$} We have already seen in the beginning of the section that $\MM/\NN$ is an $L^\infty(\mm)$-premodule  and it is clear that the function $|[v]|$ is in $L^p(\mm)$, non-negative and depends only on $[v]$. Now let $v\in \MM$ and  $f\in L^\infty(\mm)$ and notice that
\[
|[fv]|=\essinf_{w\in \NN}|fv+w|\leq\essinf_{w\in \NN}|fv+fw|=|f|\essinf_{w\in \NN}|v+w|,\qquad\mm\ae.
\]
On the other hand on the set $\{f=0\}$ we certainly have $|[fv]|=0$ (just pick $w=0$), and defining $A_n:=\{|f|\geq n^{-1}\}$ we have $\frac{\nchi_{A_n}}f\in L^{\infty}(\mm)$ and thus
\[
\nchi_{A_n}|[fv]|=\nchi_{A_n}\essinf_{w\in \NN}|fv+w|=\nchi_{A_n}\essinf_{w\in \NN}|f(v+\tfrac{\nchi_{A_n}}fw)|\geq\nchi_{A_n}|f|\essinf_{w\in \NN}|v+w|,\quad\mm\ae,
\]
which, by the arbitrariness of $n\in\N$, is sufficient to conclude that $|[fv]|=|f||[v]|$ $\mm$-a.e..

Now we prove that $\||[v]|\|_{L^p(\mm)}=\|[v]\|_{\MM/\NN}$. Start noticing that   for every $w\in \NN$ we have $|[v]|\leq|v+w|$ $\mm$-a.e.\ and thus $\||[v]|\|_{L^p(\mm)}\leq \||v+w|\|_{L^p(\mm)}$ so that the arbitrariness of $w\in \NN$ yields $\||[v]|\|_{L^p(\mm)}\leq\|[v]\|_{\MM/\NN}$. For the other inequality, let $(w_n)\subset \NN$ be a sequence such that $|[v]|=\inf_n|v+w_n|$ and define inductively $\tilde w_n\in \NN$ by putting $\tilde w_1:=w_1$ and having defined $\tilde w_n$ put
\[
\tilde w_{n+1}:=\nchi_{A_n}w_{n+1}+\nchi_{A_n^c}\tilde w_n,\qquad\text{ where }\qquad  A_n:=\{|v+w_{n+1}|\leq |v+\tilde w_n|\}.
\]
Then we have  $|[v]|=\inf_n|v+\tilde w_n|$ and $|v+\tilde w_{n+1}|\leq |v+\tilde w_n|$ $\mm$-a.e.\ for every $n\in\N$. 

If $p<\infty$ this is sufficient to deduce that 
\[
\||[v]|\|_{L^p(\mm)}=\inf_n\||v+\tilde w_n|\|_{L^p(\mm)}\geq\inf_{w\in \NN}\||v+w|\|_{L^p(\mm)}=\|[v]\|_{\MM/\NN}.
\]
Thus we proved that $\MM/\NN$ is an $L^p(\mm)$-normed premodule, which by  Proposition \ref{prop:baselp} gives the thesis. 

In the case $p=\infty$, notice that from $|v+\tilde w_n|\leq |v+\tilde w_1|$ $\mm$-a.e.\ we deduce 
\begin{equation}
\label{eq:cenalevico}
\|\tilde w_n\|_\MM=\||\tilde w_n|\|_{L^\infty(\mm)}\leq  \||v+\tilde w_1|\|_{L^\infty(\mm)}+\||v|\|_{L^\infty(\mm)}=\|v+\tilde w_1\|_\MM+\|v\|_\MM,\qquad\forall n\in\N.
\end{equation}
Then  fix $\eps>0$ let $A_n:=\{|v+\tilde w_n|\leq \||[v]|\|_{L^\infty(\mm)}+\eps\}$ and $\tilde A_n:=A_n\setminus\cup_{i<n}A_i$, notice that \eqref{eq:cenalevico} ensures that $\|\sum_{i=1}^n\nchi_{\tilde A_i}w_i\|_\MM\leq \|v+\tilde w_1\|_\MM+\|v\|_\MM$ for every $n\in\N$ and  use the gluing property in $\NN$ for the sequences $(\tilde w_n)$ and $(\tilde A_n)$ to find $\bar w\in \NN$ such that $\nchi_{\tilde A_i}\bar w=\nchi_{\tilde A_i}\tilde w_n$ for every $n\in\N$. Since $\cup_n\tilde A_n=\X$ we see that 
\[
\|[v]\|_{\MM/\NN}\leq \|v+\bar w\|_\MM=\||v+\bar w|\|_{L^\infty(\mm)}=\sup_n\||v+\tilde w_n|\|_{L^\infty(\mm)}\leq \||[v]|\|_{L^\infty(\mm)}+\eps,
\]
and letting $\eps\downarrow 0$ we conclude that $\MM/\NN$ is an $L^\infty(\mm)$-normed premodule, as desired.

By  Proposition \ref{prop:baselp} we know that $\MM/\NN$ has the locality property and the very same construction just made also shows that $\MM/\NN$ has the gluing property, thus concluding the proof.
\end{proof}

Given any $L^\infty(\mm)$-module $\MM$ there is a canonical map $\mathcal I_\MM:\MM\to \MM^{**}$ which associates to every $v\in \MM$ the morphism $\mathcal I_\MM(v):\MM^*\to L^1(\mm)$ defined by
\[
\mathcal I_\MM(v)(L):=L(v),\qquad\forall L\in \MM^*.
\]
The trivial inequality $\int|\mathcal I_\MM(v)(L)|\,\d\mm=\int|L(v)|\,\d\mm\leq \|L\|_{\MM^*}\|v\|_\MM$ and the identity $\mathcal I_\MM(v)(f L)=f L(v)$ show that $\mathcal I_\MM(v)\in \MM^{**} $ for every $v\in \MM$ with $\|\mathcal I_\MM(v)\|_{\MM^{**}}\leq \|v\|_\MM$. It is not clear to us whether $\mathcal I_\MM$ is an isometry for arbitrary modules, the problem being that, as already pointed out, for a generic module $\MM$ we are not able to exhibit a non-zero element in  $\MM^*$. The situation is instead simple for modules with full-dual:
\begin{proposition}\label{prop:dualnonvuoto}
Let $\MM$ be a module with full-dual. Then the map $\mathcal I_\MM:\MM\to \MM^{**}$  is an isometry.
\end{proposition}
\begin{proof} Pick $v\in \MM$ and use the Hahn-Banach extension theorem to find a functional $l\in \MM'$ such that $\|l\|_{\MM'}=1$ and $l(v)=\|v\|_\MM$. Since $\MM$ has full-dual, there exists $L\in \MM^*$ such that  $l(w)=\int L(w)\,\d\mm$ for every $w\in \MM$. Then $\|L\|_{\MM^*}=\|l\|_{\MM'}=1$ and we have
\[
\|v\|_\MM=l(v)=\int L(v)\,\d\mm=\int \mathcal I_\MM(v)(L)\,\d\mm\leq \|\mathcal I_\MM(v)\|_{\MM^{**}}\|L\|_{\MM^*}= \|\mathcal I_\MM(v)\|_{\MM^{**}},
\]
which is sufficient to conclude.
\end{proof}
For future reference, we point out in the next corollary that  in the case of $L^p(\mm)$-normed modules, $p<\infty$, we can associate to any $v\in \MM$ a functional $L\in \MM^*$ in a way similar to that of the proof of the previous proposition, but somehow more symmetric in $v,L$:
\begin{corollary}\label{cor:perduali}
Let $\MM$ be an $L^p(\mm)$-normed module, $p<\infty$, and $v\in \MM$. Then there exists $L\in \MM^*$ such that 
\begin{equation}
\label{eq:lduale}
|L|_*^q=|v|^p=L(v),\qquad\mm\ae,
\end{equation}
where $\frac1p+\frac1q=1$ and the first equality should be intended as $|L|_*=\nchi_{\{v\neq 0\}}$ in the case $p=1$. 
\end{corollary}
\begin{proof} In the proof of Proposition \ref{prop:dualnonvuoto} above we built $L\in \MM^*$ with norm 1 and such that $\int L(v)=\|v\|_\MM$. Define $\tilde L:=\|v\|_\MM^{p-1}L$ and notice that
\[
\|v\|_\MM^p=\int \tilde L(v)\,\d\mm \leq \int |\tilde L|_*|v|\,\d\mm\leq \||v|\|_{L^p(\mm)}\||\tilde L|_*\|_{L^q(\mm)}=\|v\|_\MM\|\tilde L\|_{\MM^*}=\|v\|_\MM^p.
\]
Then the inequalities are equalities and if $p>1$ the equality case in the Holder inequality gives the thesis. If $p=1$ replace if necessary $\tilde L$ with $\nchi_{\{v\neq 0\}}\tilde L$ to conclude.
\end{proof}
Notice that in particular for any $L^p(\mm)$-normed module, $p\in[1,\infty]$, we have the following duality property
\[
|v|=\esssup_{L\in\MM^*\ : |L|_*\leq 1\ \mm\ae}|L(v)|.
\]
Indeed, the inequality $\leq$ follows from the definition of $|\cdot|_*$ and for $p<\infty$ the equality comes from Corollary \ref{cor:perduali} above. The case $p=\infty$ can be directly handled by approximation.

\bigskip

Given a $L^\infty(\mm)$-module $\MM$ with full-dual, we shall say that it is {\bf reflexive} provided $\mathcal I_\MM:\MM\to \MM^{**}$ is surjective. We see from the example of $L^1(\mm)$ viewed as $L^\infty(\mm)$-module that a module can be reflexive as module but not as Banach space. The other implication is instead always true, as shown by the following simple proposition:
\begin{proposition}\label{prop:fullrefl}
Let $\MM$ be a module with full dual which is reflexive as Banach space. Then it is reflexive as module.
\end{proposition}
\begin{proof}
By assumption the map $\intmap_{\MM^*}:\MM^*\to \MM'$ is an isomorphism of Banach spaces and thus its transpose $\intmap_{\MM^*}^{\rm t}$ is an isomorphism from the dual $\MM''$ of $\MM'$ and the Banach dual $(\MM^*)'$ of $\MM^*$, and in particular it is surjective. Then denoting by $\mathcal J_\MM:\MM\to \MM''$ the canonical isometry from $\MM$ to  its bidual as Banach space, it is immediate to verify that the diagram
\begin{center}
\begin{tikzpicture}[node distance=2.5cm, auto]
  \node (A) {$\MM$};
  \node (B) [right of=A] {$\MM^{**}$};
  \node (C) [below  of=A] {$\MM''$};
  \node (D) [below  of=B] {$(\MM^*)'$};
  \draw[->] (A) to node {$\mathcal I_\MM$} (B);
  \draw[->] (A) to node [swap] {$\mathcal J_\MM$} (C);
  \draw[->] (B) to node  {$\intmap_{\MM^{**}}$} (D);
  \draw[->] (C) to node  {$\intmap_{\MM^*}^{\rm t}$} (D);
\end{tikzpicture}
\end{center}
commutes. We know that $\intmap_{\MM^{**}}$ is injective and, by assumption and what said at the beginning, that $\mathcal J_\MM$ and $\intmap_{\MM^*}^{\rm t}$ are surjective. Thus $\mathcal I_\MM$ is surjective as well.
\end{proof}
An easy consequence of this proposition is:
\begin{corollary}[Reflexivity of $L^p(\mm)$-normed modules, $p\in(1,\infty)$]\label{cor:refllp} Let $\MM$ be an $L^p(\mm)$-normed module, $p\in(1,\infty)$. Then it is reflexive as module if and only if it is reflexive as Banach space.
\end{corollary}
\begin{proof} The `if' follows from Proposition \ref{prop:fullrefl} above and the fact that $\MM$ has full dual (Proposition \ref{prop:fulllp}). For the `only if', consider the same commutative diagram in the proof of Proposition \ref{prop:fullrefl} above: by assumption and Proposition \ref{prop:fulllp} we know that both $\MM$ and $\MM^*$ have full dual, and thus $\intmap_{\MM^{**}}$ is surjective and $\intmap_{\MM^*}^{\rm t}$ injective. Hence if $\mathcal I_\MM$ is surjective, $\mathcal J_\MM$ must be surjective as well.
\end{proof}

Although in fact trivial, it is worth to underline that to be an $L^p(\mm)$-normed module for some $p\in(1,\infty)$ is not sufficient to ensure reflexivity:
\begin{example}\label{ex:delta}{\rm Let $\X$ be consisting of just one point and let $\mm$ be the Dirac delta on such point. Then $L^\infty(\mm)\sim \R$ and thus  every Banach space $B$ has a natural structure of $L^\infty(\mm)$-module and the Banach-dual $B'$ and the module-dual $B^*$ can be  canonically identified. Given that  $B$ can also be seen as $L^p(\mm)$-normed module for every $p\in[1,\infty]$ in a trivial way, to conclude the example it is sufficient to pick $B$ not reflexive in the Banach sense, to obtain a non-reflexive $L^p(\mm)$-normed module.
}\fr\end{example}
This example shows that in order to ensure the reflexivity of a module $\MM$, some sort of reflexivity should be imposed on the Banach spaces $(B_x,\|\cdot\|_x)$ mentioned in Example \ref{ex:lpmod}. We shall not investigate in this direction, but rather focus on the important subclass of Hilbert modules:
\begin{definition}[Hilbert modules]\label{def:hilmod}
We shall say that an $L^\infty(\mm)$-module $\MM$ is an Hilbert module  provided $\MM$, viewed as Banach space, is in fact an Hilbert space.
\end{definition}
It is not entirely trivial that Hilbert modules, whose Hilbert structure is only imposed globally, in a sense, are always $L^2(\mm)$-normed with a norm satisfying a pointwise parallelogram identity:
\begin{proposition}[Hilbert modules are $L^2(\mm)$-normed]\label{prop:hill2}
Let $\HH$ be an $L^\infty(\mm)$-premodule with the locality property which, when seen as Banach space, is an  Hilbert space. 

Then it is an $L^2(\mm)$-normed module. Moreover, the pointwise norm (which is unique by Proposition \ref{prop:baselp}) satisfies 
\begin{equation}
\label{eq:pointquad}
|v+w|^2+|v-w|^2=2|v|^2+2|w|^2,\qquad\mm\ae,
\end{equation}
for every $v,w\in \HH$.
\end{proposition}
\begin{proof} We shall denote by $\la\cdot,\cdot\ra_\HH$ the scalar product in the Hilbert space $\HH$. We begin by proving that 
\begin{equation}
\label{eq:ortv}
\la \nchi_A \,v,\nchi_B \,v\ra_\HH=0,\qquad\forall v\in \HH, \ \forall A,B\in\BB\ \text{such that }A\cap  B =\emptyset.
\end{equation}
Thus fix $v,A,B$ as above, notice that
\[
\|\nchi_A \,v+\eps\nchi_B \, v\|^2_\HH=\|\nchi_A  \,v\|_\HH^2+2\eps\la \nchi_A \, v,\nchi_B \, v\ra_\HH+\eps^2\|\nchi_B \,  v\|_\HH^2,\qquad\forall\eps\in\R,
\]
and that from the fact that $\|\nchi_A\|_{L^\infty(\mm)}\leq 1$ and $\nchi_A\nchi_B=\nchi_{A\cap B}=0$ we deduce that
\[
\|\nchi_A\, v\|^2_\HH=\|\nchi_A(\nchi_A v+\eps\nchi_B v)\|^2\leq\|\nchi_A v+\eps\nchi_B v\|^2,\qquad\forall\eps\in\R.
\]
These two formulas give that $2\eps\la \nchi_A \, v,\nchi_B \, v\ra_\HH+\eps^2\|\nchi_B\, v\|_\HH^2\geq 0$ for every $\eps\in\R$ which trivially yields the claim \eqref{eq:ortv}. Next, we claim that
\begin{equation}
\label{eq:passla}
\la\nchi_A\, v,w\ra_\HH=\la v,\nchi_A\,  w\ra_\HH,\qquad\forall v,w\in \HH, \ A\in\BB.
\end{equation}
Indeed, use \eqref{eq:ortv} with  $\nchi_A\, v+\nchi_{A^c}\,  w$ in place of $v$ and $ A^c$ in place of $B$ to deduce that
$\la\nchi_A\, v,\nchi_{ A^c}\, w\ra_M=0$. Similarly we obtain $\la\nchi_{ A^c}\,  v,\nchi_{A}\,  w\ra_\HH=0$ and using these two information we get
\[
\la \nchi_A\,  v,w\ra_\HH=\la \nchi_A\,  v,\nchi_A\,  w\ra_\HH+\la \nchi_A\, v,\nchi_{ A^c}\,  w\ra_\HH=\la \nchi_A\,  v,\nchi_A\,  w\ra_\HH=\ldots=\la  v,\nchi_A\,  w\ra_\HH,
\]
as desired.

Now fix $v\in \HH$ and consider the map 
\[
\mathcal A\ni \bar A\quad\mapsto\quad\mu_v(\bar A):=\|\nchi_A\, v\|_\HH^2,\qquad\text{where $A\in \BB$ is the equivalence class of $\bar A\in\mathcal A$}.
\]
We claim that $\mu_v$ is a non-negative measure, the non-negativity being obvious. For additivity, let $\bar A,\bar B\in\mathcal A$ with $\bar A\cap\bar B=\emptyset$, expand the square in the definition of  $\mu_v(\bar A\cup \bar B)$ and use \eqref{eq:ortv} to conclude.
 
To prove $\sigma$-additivity it is enough to  show that given a decreasing sequence $(\bar A_n)\subset \mathcal A$ such that $\cap_n\bar A_n=\emptyset$ we have $\mu_v(\bar A_n)\to 0$ as $n\to\infty$. Pick such a sequence, put $w_n:=\nchi_{A_n}\, v$, $A_n\in\BB$ being the equivalence class of $\bar A_n\in\mathcal A$, and $v_n:=v-w_n$. The bound $\|w_n\|_\HH\leq \|v\|_\HH$  grants that up to pass to a subsequence, not relabeled, the sequence $(w_n)$ converges to some $w\in \HH$ in the  weak topology of the Hilbert space  $\HH$. We claim that $w=0$ and thanks to the locality property of the Hilbert module $\HH$ to this aim it is sufficient to prove that $\nchi_{A_k^c}\, w=0$ for every $k\in\N$. Thus fix   $k\in \N$ and notice that
\[
\begin{split}
\|\nchi_{A_k^c}\, w\|_\HH^2&=\la \nchi_{A_k^c}\, w,\nchi_{A_k^c}\, w\ra_\HH\stackrel{\eqref{eq:passla}}=\la \nchi_{A_k^c}\,  w, w\ra_\HH\\
&=\lim_{n\to\infty}\la \nchi_{A_k^c}\, w, w_n\ra_\HH\stackrel{\eqref{eq:passla}}=\lim_{n\to\infty}\la w, \nchi_{A_k^c}\, w_n\ra_\HH=0.
\end{split}
\]
Hence $(w_n)$ weakly converges to 0, and thus  $(v_n)$ weakly converges to $v$. Since $v_n=\nchi_{\X\setminus A_n}v$, we have $\|v_n\|_\HH\leq \|v\|_\HH$ for every $n\in\N$, which together with the weak convergence just proved implies strong convergence. Therefore  $(w_n)$ strongly converges to 0, which yields the desired $\sigma$-additivity of $\mu_v$.

Finally, observe that directly by definition we have $\mu_v\ll\mm$. Hence we can define
\[
|v|:=\sqrt{\frac{\d\mu_v}{\d\mm}}\in L^2(\mm),\qquad\forall v\in\HH.
\]
By construction we have $\int_\X |v|^2\,\d\mm=\mu_v(\X)=\|v\|^2_\HH$, thus to conclude the proof we need only to show that $|f v|=|f||v|$ for every $f\in L^\infty(\mm)$. To this aim notice that for $f$ constant this is trivial, and that the construction ensures that 
\begin{equation}
\label{eq:localmuv1}
|\nchi_A\,v|=\nchi_A|v|,\quad\mm\ae\qquad\forall v\in \HH,\  A\in\BB
\end{equation}
so that by localization we get that for every $v\in \HH$ it holds
\begin{equation}
\label{eq:localmuv}
|fv|=|f|\,|v|,\quad\mm\ae,\quad\text{for every $f\in L^\infty(\mm)$ attaining only a countable number of values}.
\end{equation}
Now notice that for $f_1,f_2\in L^\infty(\mm)$ with $f_1\geq f_2\geq 0$ $\mm$-a.e., we have $f:=\nchi_{\{f_1>0\}}\frac{f_2}{f_1}\in L^\infty(\mm)$ with $\|f\|_{L^\infty(\mm)}\leq 1$ and thus for every $A\in\BB$ we have
\[
\begin{split}
\int_A|f_2v|^2\,\d\mm&=\int\nchi_A|f_2v|^2\,\d\mm\stackrel{\eqref{eq:localmuv1}}=\int |\nchi_Af_2v|^2\,\d\mm\\
&=\|f\nchi_Af_1v\|_\HH^2\leq\|\nchi_Af_1v\|_\HH^2=\ldots=\int_A|f_1v|^2\,\d\mm,
\end{split}
\]
which, by the arbitrariness of $A\in\BB$, forces $|f_2v|\leq |f_1v|$ $\mm$-a.e.. This monotonicity property together with \eqref{eq:localmuv} and a simple approximation argument give that $|fv|=|f||v|$ for $f\in L^\infty(\mm)$ non-negative. Then the case of $f\leq 0$ follows using the fact that $\|-v\|_\HH=\|v\|_\HH$ and the general one by writing   $v=\nchi_A\,v+\nchi_{A^c}\,v$ for $A:=\{f>0\}$ and using the local property \eqref{eq:localmuv1}.

Finally, for identity \eqref{eq:pointquad} we argue by contradiction:  if it fails for some $v,w\in\HH$ we could find a set $A\in\BB $ of positive measure where the inequality $>$, say, holds $\mm$-a.e.. The contradiction follows integrating such inequality over $A$ to obtain
\[
\|\nchi_A\,v+\nchi_A\,w\|_\HH^2+\|\nchi_A\,v-\nchi_A\,w\|_\HH^2>2\big(\|\nchi_A\,v\|_\HH^2+\|\nchi_A\,w\|_\HH^2\big),
\]
which is impossible because the norm of $\HH$ satisfies the parallelogram rule.
\end{proof}
\begin{corollary}\label{cor:hilrefl} Let $\HH$ be an Hilbert module. Then it is  reflexive.
\end{corollary}
\begin{proof}
Just apply Proposition  \ref{prop:hill2} and Corollary \ref{cor:refllp}.
\end{proof}

\begin{remark}{\rm
It is worth to point out that in the proof of Proposition \ref{prop:hill2} we used, for the first time, the fact that the continuity of multiplication with functions in $L^\infty(\mm)$ has been asked as
\begin{equation}
\label{eq:mult1}
\|fv\|_\MM\leq \|f\|_{L^\infty(\mm)}\|v\|_\MM,\qquad\forall v\in\MM,\ f\in L^\infty(\mm),
\end{equation}
in the definition of $L^\infty(\mm)$-module. Notice indeed that if one only asks for continuity of such multiplication, then in fact he is asking for the existence of a constant $C\geq 0$ such that
\begin{equation}
\label{eq:multC}
\|f v\|_\MM\leq C\|f\|_{L^\infty(\mm)}\|v\|_\MM,\qquad\forall v\in\MM,\ f\in L^\infty(\mm),
\end{equation}
and that the choice $f={\mathbf 1}$ shows that it must hold  $C\geq 1$. Thus \eqref{eq:mult1} is the strongest form of continuity we can ask for. 

We chose to force $C=1$ in the definition of $L^\infty(\mm)$-module because \eqref{eq:mult1} holds in all the examples we will work with, see below, and because it has highly desirable effects like Proposition \ref{prop:hill2} above and the inequality $\|\nchi_Av\|_\MM\leq \|v\|_\MM$, which corresponds to the intuitive idea that `setting an element to be 0 outside a certain set should decrease its norm'.

In this direction, it might be worth to point out that there are examples of Hilbert spaces endowed with a continuous multiplication with $L^\infty$ functions for which \eqref{eq:multC} holds for some $C>1$ but not for $C=1$: consider for instance $\X:=[0,1]$ equipped with the Borel $\sigma$-algebra and the Lebesgue measure and the space  $L^2(0,1)$ equipped with the norm 
\[
\|v\|^2:=\int_0^1|v|^2(x)\,\d x+(C-1)\left(\int_0^1v(x)\,\d x\right)^2.
\]
It is immediate to see that inequality \eqref{eq:multC} is satisfied by this space, the multiplication with $L^\infty$ functions being just the pointwise one, and that the optimal constant is precisely $C$.  Notice that for this example Proposition \ref{prop:hill2} does not hold and that there is $v\in L^2(0,1)$ and a Borel set  $A\subset[0,1]$ such that  $\|\nchi_Av\|>\|v\|$.

On the other hand, the bound \eqref{eq:mult1} is automatic on Banach spaces $\MM$ equipped with an $L^\infty(\mm)$ module structure in the purely algebraic sense and admitting a pointwise $L^p(\mm)$ norm, i.e.\ a map $|\cdot|:\MM\to L^p(\mm)$ with non-negative values satisfying \eqref{eq:pointnorm}.  Indeed, in this case we directly have
\[
\|fv\|_\MM=\||fv|\|_{L^p(\mm)}=\||f|\,|v|\|_{L^p(\mm)}\leq\|f\|_{L^\infty(\mm)}\||v|\|_{L^p(\mm)}=\|f\|_{L^\infty(\mm)}\|v\|_{\MM}.
\]
}\fr\end{remark}

On a given Hilbert module $\HH$  we define the {\bf pointwise scalar product} $\HH\times\HH\ni (v,w)\mapsto  \la v,w\ra\in L^1(\mm)$ as
\[
\la v,w\ra:=\frac12\big(|v+w|^2-|v|^2-|w|^2\big),
\]
and notice that the standard polarization argument grants that such map satisfies
\begin{equation}
\label{eq:proprietascalarepoint}
\begin{split}
\la f_1v_1+f_2v_2,w\ra&=f_1\la v_1,w\ra+f_2\la v_2,w\ra,\\
|\la v,w\ra|&\leq |v||w|,\\
\la v,w\ra&=\la w,v\ra,\\
\la v,v\ra&=|v|^2,
\end{split}
\end{equation}
$\mm$-a.e., for every $v_1,v_2,v,w\in \HH$ and $f_1,f_2\in L^\infty(\mm)$. Indeed, the last two are a direct consequence of the definition, which also gives
\begin{equation}
\label{eq:robaccia2}
\nchi_A\la v,w\ra= \la\nchi_A\, v,w\ra= \la v,\nchi_A \,w\ra,\qquad\mm\ae\qquad\forall v,w\in\HH,\ A\in\BB,
\end{equation}
while few applications of the  parallelogram rule  \eqref{eq:pointquad} give
\[
\la v_1+v_2,w\ra=2\la v_1,w/2\ra+2\la v_2,w/2\ra,\qquad\mm\ae\qquad\forall v_1,v_2,w\in\HH.
\]
Considering this last identity together with its special case $v_2=0$ yields additivity
\begin{equation}
\label{eq:robaccia1}
\la v_1+v_2,w\ra=\la v_1,w\ra+\la v_2,w\ra,\qquad\mm\ae\qquad\forall v_1,v_2,w\in\HH,
\end{equation}
which by induction gives  $\la nv,w\ra= n\la v,w\ra$ for every $n\in\N$ and then that $\la rv,w\ra=r\la v,w\ra$ for every $r\in \Q$ $\mm$-a.e.\ for every $v,w\in\HH$. Coupling this information with \eqref{eq:robaccia2} we deduce that
\begin{equation}
\label{eq:robaccia3}
\la fv,w\ra= f\la v,w\ra,\qquad\mm\ae\qquad\forall v,w\in \HH,
\end{equation}
for every function $f\in L^\infty(\mm)$ attaining only a finite number of values all of them being rational. Then the continuity of the pointwise scalar product as map from $\HH\times\HH$ to $L^1(\mm)$ give \eqref{eq:robaccia3} for generic $f\in L^\infty(\mm)$, which together with \eqref{eq:robaccia1} yields the first in \eqref{eq:proprietascalarepoint}. Finally,  the inequality $|\la v,w\ra|\leq |v||w|$ $\mm$-a.e.\ follows for bounded $v,w$ by expanding the square  in $\big||w|v\pm |v|w\big|^2\geq 0$ $\mm$-a.e.\ and the general case comes by approximation.

\bigskip

Now fix $v\in \HH$ and consider the map $L_v:\HH\to L^1(\mm)$ defined by
\[
L_v(w):=\la v,w\ra.
\]
It is clear from properties \eqref{eq:proprietascalarepoint} that this map is a module morphism, i.e.\  $L_v\in\HH^*$ and that
\[
|L_v|_*=|v|,\qquad\mm\ae.
\]
We then have the following result:
\begin{theorem}[Riesz theorem for Hilbert modules]\label{thm:rhil} Let $\HH$ be an Hilbert module. Then the map $\HH\ni v\mapsto L_v\in \HH^*$ is a morphism of modules, bijective and an isometry. In particular, for every $l\in \HH'$ there exists a unique $v\in \HH$ such that $l=\intmap_{\HH^*}L_v$.
\end{theorem}
\begin{proof} Linearity and norm-preservation, even in the pointwise sense, of $v\mapsto L_v$ have already been checked. The fact that $v\mapsto L_v$ is a module morphism is obvious, as we have
\[
L_{fv}(w)=\la fv,w\ra= f\la v,w\ra=(fL_v)(w),\qquad\mm\ae,\qquad\forall v,w\in\HH,\ f\in L^\infty(\mm).
\]
To check surjectivity, let $L\in \HH^*$, consider the linear functional $\intmap_{\HH^*} L\in \HH'$ and apply the standard Riesz theorem for Hilbert spaces to deduce that there exists $v\in \HH$ such that 
\[
\int L(w)\,\d\mm=\intmap_{\HH^*} L(w)=\la v,w\ra_\HH,\qquad\forall w\in \HH.
\]
Now notice that by polarization of the first identity in \eqref{eq:pointnorm} we get $\int\la v,w\ra\,\d\mm=\la v,w\ra_\HH$ for every $v,w\in\HH$ and thus the chain of equalities
\[
\int_AL(w)\,\d\mm=\int \nchi_AL(w)\,\d\mm=\int L(\nchi_Aw)\,\d\mm=\la v,\nchi_Aw\ra_\HH=\int \la v,\nchi_Aw\ra\,\d\mm=\int_A\la v,w\ra\,\d\mm
\]
holds for every $A\in\BB$, which forces $L(w)=\la v,w\ra$ $\mm$-a.e, i.e.\ $L=L_v$.

The same argument also gives the last statement.
\end{proof}

\subsection{Alteration of the integrability}\label{se:altint}
By definition, an $L^p(\mm)$-normed module $\MM$ is a space whose elements have norm in $L^p(\mm)$. In applications, such integrability requirement can be too tight because it might be necessary  to handle objects with a different integrability, or perhaps to have at disposal elements $v$ of some bigger space possessing a pointwise norm $|v|$ and the possibility to say that $v\in\MM$ if and only if $|v|\in L^p(\mm)$, much like one does with ordinary $L^p(\mm)$ spaces. 

In this section we show how such bigger space can be built.

\vspace{1cm}

Recall that the space $L^0(\mm)$ is, as a set, the collection  of all (equivalence classes w.r.t. equality $\mm$-a.e.\ and measurable) real valued functions on $\X$. $L^0(\mm)$ comes with a canonical topology: consider the distance $\sfd_{L^0}$ on it  given by
\[
\sfd_{L^0}(f,g):=\sum_{i\in\N}\frac1{2^i\mm(E_i)}\int_{E_i}\min\{|f-g|,1\}\,\d\mm,
\]
where $(E_i)\subset \BB$ is a partition of $\X$ into sets of finite and positive measure. It is clear that $\sfd_{L^0}$ is a translation invariant, complete and separable metric: the topology on $L^0(\mm)$ is the one induced by $\sfd_{L^0}$ (in general it is not locally convex). It can be checked directly from the definition that   $(f_n)\subset L^0(\mm)$ is a $\sfd_{L^0}$-Cauchy sequence   if and only if
\[
\begin{split}
\forall\eps>0,\ &E\in\BB\text{ with }\mm(E)<\infty\quad \exists n_{\eps,E}\text{ such that :} \\
&\forall n,m\in\N\text{ with }n,m>n_{\eps,E}\quad\text{ we have }\qquad \mm\big(E\cap\{|f_n-f_m|>\eps\}\big)<\eps.
\end{split}
\]
This shows that  although the choice of the partition $(E_i)$ affects the distance $\sfd_{L^0}$, it does not affect Cauchy sequences, i.e.\   the topology of  $L^0(\mm)$ is intrinsic and  independent on the particular partition chosen.

\bigskip

Now let $p\in[1,\infty]$, $\MM$  an $L^p(\mm)$-normed module and consider the distance $\sfd_{\MM^0}$ on $\MM$ defined as
\[
\sfd_{\MM^0}(v,w):=\sum_{i\in\N}\frac1{2^i\mm(E_i)}\int_{E_i}\min\{|v-w|,1\}\,\d\mm.
\]
We then give the following definition:
\begin{definition}[The space $\MM^0$]\label{def:m0}
The topological space $\MM^0$ is defined as the completion of $(\MM,\sfd_{\MM^0})$ equipped with the induced topology.
\end{definition}
As for the case of $L^0(\mm)$, the choice of the partition $(E_i)$ affects the distance $\sfd_{\MM^0}$ but not Cauchy sequences $(v_n)\subset \MM$, which are characterized by 
\[
\begin{split}
\forall\eps>0,\ &E\in\BB\text{ with }\mm(E)<\infty\quad \exists n_{\eps,E}\text{ such that :} \\
&\forall n,m\in\N\text{ with }n,m>n_{\eps,E}\quad\text{ we have }\qquad \mm\big(E\cap\{|v_n-v_m|>\eps\}\big)<\eps.
\end{split}
\]
In other words, the topology of $\MM^0$ is is intrinsic and  independent on the particular partition chosen.

Directly from the definition of the distance $\sfd_{\MM^0}$ we see that
\[
\begin{split}
\sfd_{\MM^0}(v_1+w_1,v_2+w_2)&\leq\sfd_{\MM^0}(v_1,v_2)+\sfd_{\MM^0}(w_1,w_2),\\
\sfd_{\MM^0}(\lambda v_1,\lambda w_1)&\leq \max\{1,|\lambda|\}\sfd_{\MM^0}(v_1,w_1),
\end{split}
\]
for every $v_1,v_2,w_1,w_2\in \MM$ and $\lambda\in\R$. These show that the operations of addition and multiplication with a scalar extends continuously, in a unique way, to $\MM^0$, which therefore is a topological vector space. 

Similarly, the characterization of Cauchy sequences grants that  for any $\sfd_{L^0}$-Cauchy sequence $(f_n)\subset L^\infty(\mm)$ and any $\sfd_{\MM^0}$-Cauchy sequence $(v_n)\subset \MM$, the sequence $(f_nv_n)\subset \MM$ is $\sfd_{\MM^0}$-Cauchy. Hence the operation of multiplication with functions in $L^\infty(\mm)$  on $\MM$ can, and will, be extended in a unique way to a bilinear continuous map from $L^0(\mm)\times \MM^0$ to $\MM^0$, which clearly satisfies
\[
f(gv)=(fg)v,\qquad\text{and}\qquad{\mathbf 1}v=v,
\]
for any $v\in \MM^0$ and $f,g\in L^0(\mm)$. 

Finally,  the inequality
\[
\sfd_{L^0}(|v_1|,|v_2|)\leq\sfd_{\MM^0}(v_1,v_2),
\]
valid for every $v_1,v_2\in \MM$, shows that the pointwise norm $|\cdot|:\MM\to L^p(\mm)$ can, and will, be continuously extended in a unique way to a map from $\MM^0$ to $L^0(\mm)$, still denoted by $v\mapsto|v|$. It is then clear that
\[
\begin{split}
|v+w|&\leq |v|+|w|,\\
|fv|&=|f|\,|v|,
\end{split}
\]
$\mm$-a.e.\ for every $v,w\in\MM^0$ and $f\in L^0(\mm)$.

We shall refer to spaces $\MM^0$ built this way as ${\boldsymbol L^0(\mm)}${\bf -modules}.

The same continuity arguments show that if $\MM$ was an Hilbert module, then the pointwise scalar product extends uniquely to a bilinear continuous map from $\MM^0\times\MM^0$ to $L^0(\mm)$ satisfying \eqref{eq:proprietascalarepoint} for vectors in $\MM^0$ and functions in $L^0(\mm)$.

\bigskip

We remark that by construction $\MM$ is a subspace of $\MM^0$, the embedding being continuous and preserving the operation of   $L^p(\mm)$-normed module. From this perspective  we can think at $\MM^0$ as the extension of $\MM$ which `includes elements without any restriction on their norm' and notice that  for $v\in \MM^0$ we have $v\in \MM$ if and only if $|v|\in L^p(\mm)$.

\bigskip

Now consider an $L^p(\mm)$-normed module $\MM$ and its dual $\MM^*$, which by point $(i)$ of Proposition \ref{prop:normevarie} is  $L^q(\mm)$-normed, where  $\frac1p+\frac1q=1$. Notice that there is a natural duality pairing 
\begin{equation}
\label{eq:dual0}
\MM^0\times(\MM^*)^0\ni (v,L)\qquad\mapsto\qquad L(v)\in L^0(\mm),
\end{equation}
obtained as the unique continuous extension of the duality pairing of $\MM$ and $\MM^*$. By continuity, we can check that the inequality 
\begin{equation}
\label{eq:stupida}
|L(v)|\leq |v|\,|L|_*,\qquad\mm\ae,\ \forall v\in \MM^0,\ L\in (\MM^*)^0,
\end{equation}
holds also in this setting. The following simple result will be useful in what follows:
\begin{proposition}[The dual of $\MM^0$]\label{prop:dualm0} Let $\MM$ be an $L^p(\mm)$-normed module, $p\in[1,\infty]$ and   $T:\MM^0\to L^0(\mm)$ a linear map such that for some $f\in L^0(\mm)$ we have
\begin{equation}
\label{eq:dualm0}
|T(v)|\leq f|v|,\quad\mm\ae,\qquad\forall v\in\MM.
\end{equation}
Then there exists a unique $L\in (\MM^*)^0$ such that $T(v)=L(v)$ for every $v\in\MM$ and this $L$ satisfies
\[
|L|_*\leq f,\qquad\mm\ae.
\]
\end{proposition}
\begin{proof} Uniqueness is obvious. For existence, let $(E_n)\subset\BB$ be a partition of  $\X$  such that $\nchi_{E_n}f\in L^q(\mm)$ for every $n\in\N$, where $\frac1p+\frac1q=1$. Then the bound \eqref{eq:dualm0} ensures that the linear map $\MM\ni v\mapsto L_n(v):=\nchi_{E_n}T(v)$ takes values in $L^1(\mm)$ and satisfies
\[
|L_n(v)|\leq (\nchi_{E_n}f)|v|,\qquad\mm\ae.
\]
By point $(v)$ of Proposition \ref{prop:baselp} we deduce that $L_n\in \MM^*$ with $|L_n|_*\leq \nchi_{E_n}f$. It is then clear that the series $\sum_n\nchi_{E_n}L_n$ converges in $(\MM^*)^0$ and that its   sum satisfies the thesis.
\end{proof}

Having the space $\MM^0$ at disposal allows us to build other $L^p$-normed modules, possibly related to a different $\sigma$-finite measure $\mm'$ on $\X$. A construction that we shall sometime use later on is the following: given $\bar p\in[1,\infty]$ and a non-negative $\sigma$-finite measure $\mm'\ll\mm$, the space $\MM_{\bar p,\mm'}\subset \MM^0$ is the space of all $v$'s in $\MM^0$ such that $\||v|\|_{L^{\bar p}(\mm')}<\infty$ equipped with the norm $\|v\|_{\bar p,\mm'}:=\||v|\|_{L^{\bar p}(\mm')}$. The operation of multiplication with functions in $L^0(\mm)$ induces a multiplication with functions in $L^\infty(\mm')$ thus endowing $\MM_{\bar p,\mm'}$ with a natural structure of $L^{\bar p}(\mm')$-normed   module.

Notice that 
\begin{equation}
\label{eq:sepmmp}
\text{if $\MM$ is separable and $\bar p<\infty$, then  $\MM_{\bar p,\mm'}$ is separable as well.}
\end{equation}
which can be checked by picking $D\subset \MM$ countable, dense and made of bounded elements (recall \eqref{eq:seplplim}), an increasing sequence $(E_n)\subset \BB$ of sets of finite $\mm'$-measure such that $\mm'(\X\setminus\cup_nE_n)=0$ and observing that the set $\{\nchi_{E_n\cap\{\rho\leq m\}}v:v\in D,\ n,m\in\N\}$ is dense in $\MM_{\bar p,\mm'}$, where $\rho:=\frac{\d\mm'}{\d\mm}$.

Now let  $\bar q\in[1,\infty]$ such that  $\frac1{\bar p}+\frac1{\bar q}=1$ and consider the module   $\MM^*_{\bar q,\mm'}\subset(\MM^*)^0$. Taking into account \eqref{eq:stupida}, the duality pairing \eqref{eq:dual0} restricts to a bilinear continuous map
\[
\MM_{\bar p,\mm'}\times\MM^*_{\bar q,\mm'}\ni (v,L)\qquad\mapsto\qquad L(v)\in L^1(\mm'),
\]
which shows that $\MM^*_{\bar q,\mm'}$ embeds into $(\MM_{\bar p,\mm'})^*$ the embedding being a module morphism and, as can be directly checked from the definition of pointwise dual norm, an isometry. In other words, for any $L\in \MM^*_{\bar q,\mm'}$ the map
\[
\MM_{\bar p,\mm'}\ni v\qquad\mapsto\qquad L(v)\in L^1(\mm'),
\]
is a module morphism whose pointwise norm $\mm'$-a.e.\ coincides with $|L|_*$. From Proposition \ref{prop:dualm0} it is easy to see  that in fact every element of $(\MM_{\bar p,\mm'})^*$ is of this form, i.e.
\begin{equation}
\label{eq:dualmmp}
\text{for $\bar p\in[1,\infty]$ the dual $(\MM_{\bar p,\mm'})^*$ of $\MM_{\bar p,\mm'}$ can be canonically identified with $\MM^*_{\bar q,\mm'}$.}
\end{equation}
Indeed, let $L\in (\MM_{\bar p,\mm'})^*$ so that
\begin{equation}
\label{eq:perdmmp}
|L|_*\in L^{\bar q}(\mm')\qquad\text{ and }\qquad |L(v)|\leq |L|_*|v|\quad\mm\ae
\end{equation}
for every $v\in  \MM_{\bar p,\mm'}$.  Let $E:=\{\frac{\d\mm'}{\d\mm}>0\}\in\BB$ and notice that the space $(\MM_{\bar p,\mm'})^0$ can, and will, be identified with the subspace $\MM^0\restr{E}$ of $\MM^0$ made of elements which are 0 $\mm$-a.e.\ on $E^c$, so that $L$   can be uniquely extended to a continuous map, still denoted by $L$, from $(\MM_{\bar p,\mm'})^0\sim \MM^0\restr{E}$ to $L^0(\mm)$. We further extend $L$ to the whole $\MM^0$ by putting $L(v):=L(\nchi_Ev)$ for arbitrary  $v\in\MM^0$ and notice that such extension still satisfies inequality \eqref{eq:perdmmp}, so by Proposition \ref{prop:dualm0} we see that this functional is induced by an element of $(\MM^*)^0$. The conclusion now comes from the fact that $|L|_*\in L^{\bar q}(\mm')$ and point $(v)$ of Proposition \ref{prop:baselp}.

\subsection{Local dimension}\label{se:locdim}
In this, perhaps not really exciting, section we show that for $L^\infty(\mm)$-modules there is a natural notion of local dimension. The main results here are the fact that finitely generated $L^p(\mm)$-normed modules, $p<\infty$, are reflexive (Theorem \ref{thm:fingenrefl}), and the structural characterization of Hilbert modules (Theorem \ref{thm:structhil}) which among other things provides a link between the concept of Hilbert module and that of direct integral of Hilbert spaces (Remark \ref{re:dirintH}).

As we learned at a late stage of development of this project, similar decomposition results, in some case in a more general formulation, were already obtained in \cite{HLR91}: what we call $L^p(\mm)$-normed module is a particular case of what in \cite{HLR91} is called a `randomly normed space'.

\vspace{1cm}

We start with few definitions.
\begin{definition}[Local independence]\label{def:locind} Let $\MM$ be an $L^\infty(\mm)$-module and $A\in\BB$ be with $\mm(A)>0$.

We say that a finite family $v_1,\ldots,v_n\in \MM$ is independent on $A$ provided the identity
\[
\sum_{i=1}^n f_iv_i=0,\qquad\mm\ae \ \textrm{\rm on}\ A
\]
holds only if $f_i=0$ $\mm$-a.e.\ on $A$ for every $i=1,\ldots,n$.
\end{definition}
\begin{definition}[Local span and generators]\label{def:genmod} Let $\MM$ be an $L^\infty(\mm)$-module, $V\subset \MM$ a subset and $A\in\BB$. 

The span of $V$ on $A$, denoted by ${\rm Span}_A(V)$, is the subset of  $\MM$ made of vectors $v$ concentrated on $A$ with the following property:  there are $(A_n)\subset \BB$, $n\in\N$, disjoint such that $A=\cup_iA_i$ and for every $n$ elements $v_{1,n},\ldots,v_{m_n,n}\in \MM$ and functions $f_{1,n},\ldots,f_{m_n,n}\in L^\infty(\mm)$ such that 
\[
\nchi_{A_n}v=\sum_{i=1}^{m_n}f_{i,n}v_{i,n}.
\]
We refer to ${\rm Span}_A(V)$ as the space spanned by $V$ on $A$, or simply spanned by $V$ if $A=\X$. 

Similarly, we refer to the closure $\overline{{\rm Span}_A(V)}$ of ${\rm Span}_A(V)$ as the space generated by $V$ on $A$, or simply as the space generated by $V$ if $A=\X$.
\end{definition}
Notice that the definition is given in such a way that the module $L^p(\mm)$, $p\in[1,\infty]$ is spanned by one element:  any function which is non-zero $\mm$-a.e.\ does the job.

We say that $\MM$ is {\bf finitely generated} if there is a finite family $v_1,\ldots,v_n$ spanning $\MM$ on the whole $\X$ and {\bf locally finitely generated} if there is a partition $(E_i)$ of $\X$ such that $\MM\restr{E_i}$ is finitely generated for every $i\in\N$.

It is obvious from the definitions that if $v_1,\ldots,v_n$ are independent on $A$ (resp.\ span/generate $\MM$ on $A$), then they are independent on every $B\subset A$ (resp.\  span/generate $\MM$ on every $B\subset A$) and that if the are independent on $A_n$ for every $n\in\N$ (resp.\ span/generate $\MM$ on $A_n$ for every $n\in\N$), then they are independent on $\cup_nA_n$ (resp.\ span/generate $\MM$ on $\cup_nA_n$).

These definitions are also invariant by {\bf isomorphism}: we say that two $L^\infty(\mm)$-modules $\MM_1,\MM_2$ are isomorphic provided there exists $T\in\Hom(\MM_1,\MM_2)$ and $S\in\Hom(\MM_2,\MM_1)$ such that $T\circ S={\rm Id}_{\MM_2}$ and $S\circ T={\rm Id}_{\MM_1}$ and in this case both $T$ and $S$ are called isomorphisms. We shall say that $\MM_1,\MM_2$ are {\bf isometric} provided they are isomorphic and there are norm-preserving isomorphisms. 

Then it is clear that  if $v_1,\ldots,v_n\in \MM_1$ are independent on $A$ and $T:\MM_1\to \MM_2$ is an isomorphism of modules, then $T(v_1),\ldots,T(v_n)\in \MM_2$ are independent on $A$ as well. Similarly for local generators.

A basic fact which we do \emph{not} know at this level of generality is whether the span of a finite number of elements on a given set $A$ is closed or not. We are only able to prove this under the additional assumption that the space is $L^p(\mm)$-normed for some $p\in[1,\infty]$, see Proposition \ref{prop:spanclosed}; in this sense the discussion prior to that point is quite incomplete. Yet, from the definition we know at least that ${\rm Span}_A(V)$ is always closed w.r.t.\ the gluing operation in $\MM$, a fact which we shall use later on.

\begin{definition}[Local basis and dimension]
We say that a finite family $v_1,\ldots,v_n$ is a basis on $A\in\BB$ provided it is independent on $A$ and ${\rm Span}_A\{v_1,\ldots,v_n\}=\MM\restr A$.

If $\MM$ admits a basis of cardinality $n$ on $A$, we say that it has dimension $n$ on $A$, or that the local dimension of $\MM$ on $A$ is $n$. If $\MM$ has not dimension $n$ for each $n\in\N$ we say that it has infinite dimension.
\end{definition}
We need to check that the dimension is well defined, which  follows along the same lines used in basic linear algebra. 
\begin{proposition}[Well posedness of definition of dimension]\label{prop:localdimension}
Let $\MM$ be an $L^\infty(\mm)$-module and $A\in\BB$. Assume that $v_1,\ldots,v_{n}$ generates $\MM$ on $A$  and $w_1,\ldots,w_m$ are independent on $A$. Then $n\geq m$. In particular, if both $v_1,\ldots,v_n$ and $w_1,\ldots,w_m$ are basis of $\MM$ on $A$, then $n=m$.
\end{proposition}
\begin{proof} Since the $v_1,\ldots,v_n$ generate $\MM$ on $A$, we know that there are sets $A_i\in\BB$, $i\in\N$, such that $A=\cup_iA_i$ and functions $f^i_j\in L^\infty(\mm)$ such that
\begin{equation}
\label{eq:linind}
\nchi_{A_i}w_1=\sum_{j=1}^nf^i_jv_j.
\end{equation}
Pick $i\in\N$ such that $\mm(A_i)>0$. Since $w_1,\ldots,w_m$ is independent on $A$, we must have $w_1\neq 0$ $\mm$-a.e.\ on $A$ (or otherwise we easily get a contradiction), thus the previous identity ensures that for some $j\in\{1,\ldots,n\}$ and some $\tilde A_i\subset A_i$ with $\mm(\tilde A_i)>0$ we have $f^i_j\neq 0$ $\mm$-a.e.\ on $\tilde A_i$. Up to permuting the $v_i$'s we can assume that $j=1$. Hence for some $B_1\subset \tilde A_i$ with $\mm(B_1)>0$ and some $c>0$ we have $|f^i_1|\geq c$ $\mm$-a.e.\ on $B_1$, so that $g_1:=\nchi_{B_1}\frac1{f^i_1}\in L^\infty(\mm)$.

From \eqref{eq:linind} we deduce that
\[
\nchi_{B_1}v_1=(\nchi_{B_1}g_1)w_1-\sum_{j=2}^n(\nchi_{B_1}g_1f^i_j)v_j,
\]
and from this identity and the fact that $v_1,\ldots,v_n$ generate, we easily obtain that $w_1,v_2,\ldots,,v_n$ also generates $\MM$ on $B_1$.

We now proceed by induction. Let $k<m$ and suppose we already proved that there exists $B_k\in\BB$ with $\mm(B_k)>0$  such that $w_1,\ldots,w_k,v_{k+1},\ldots,v_{n}$ generates $\MM$ on $B_k$. Then arguing as before we find a set $B_k'\subset B_k$ with $\mm(B_k')>0$ and functions $f_1,\ldots,f_n\in L^\infty(\mm)$ such that
\[
\nchi_{B_k'}w_{k+1}=\sum_{j=1}^kf_jw_j+\sum_{j=k+1}^nf_jv_j.
\]
If $f_j=0$ $\mm$-a.e.\ on $B_k'$ for every $j=k+1,\ldots,n$ we would obtain that $\nchi_{B_k'}w_{k+1}=\sum_{j=1}^kf_jw_j$ $\mm$-a.e.\ on $B_k'$, contradicting the independence of the $w_i$'s on $A\supset B_k'$. In particular, $k<n$ and  there must exist $B_{k+1}\subset B_k'$ with $\mm(B_{k+1})>0$ and $c>0$ such that for some $j\in\{k+1,\ldots,n\}$ we have $|f_j|\geq c$ $\mm$-a.e.\ on $B_{k+1}$. Up to relabeling the indexes we can assume that $j=k+1$ and arguing as before we deduce that $w_1,\ldots,w_{k+1},v_{k+2},\ldots,v_{n}$ generates $\MM$ on $B_{k+1}$. 

Iterating the procedure up to $k=m$ we conclude.
\end{proof}
It is then  easy to see that given an $L^\infty(\mm)$-module $\MM$, we can partition $\X$ into sets where $\MM$ has given dimension:
\begin{proposition}[Dimensional decomposition]\label{prop:dimdec}
Let $\MM$ be an $L^\infty(\mm)$-module. Then there is a unique partition $\{E_i\}_{i\in\N\cup\{\infty\}}$ of $\X$  such that the following holds:
\begin{itemize}
\item[i)] for every $i\in\N$ such that $\mm(E_i)>0$, $\MM$ has dimension $i$ on $E_i$,
\item[ii)] for every $E\subset E_\infty$ with $\mm(E)>0$, $\MM$ has infinite dimension  on $E$.
\end{itemize}
\end{proposition}
\begin{proof}
Uniqueness is obvious by Proposition \ref{prop:localdimension} so we turn to existence. 

Pick $n\in\N$, let $\mathcal B_n\subset\BB$ be the set of $E$'s such that $\MM$ has dimension $n$ on $E$. We claim that if $E_i\in\mathcal B_n$ for every $i\in\N$, then $\cup_iE_i\in\mathcal B_n$ as well. We can certainly assume $\mm(E_i)>0$ for every $i\in\N$ and up to replacing $E_i$ with $E_i\setminus(\cup_{j<i}E_j)$, it is not restrictive to assume that $E_i\cap E_j=\emptyset$ for  $i\neq j$. For every $i\in\N$, let $v^i_1,\ldots,v^i_n$ be a basis of $\MM$ on $E_i$. Up to rescaling, we can assume that $\|v^i_k\|_\MM\leq 2^{-i}$ for every $i\in\N$, $k\in\{1,\ldots,n\}$ so that we can define
\[
v_k:=\sum_{i\in\N}\nchi_{E_i}v^i_k,\qquad\forall k=1,\ldots,n,
\]
because the series at the right hand side converges in $\MM$. Directly by the definition we see that $v_1,\ldots,v_n$ is a basis of $\MM$ on $E$, thus proving our claim. 

Hence the class $\mathcal B_n$ is stable by countable unions and we deduce that  there exists a unique maximal on it. Call it  $E_n$ and put  $E_\infty:=\X\setminus\cup_{n\in\N}E_n$: by construction these fulfill  the thesis.
\end{proof}
It is worth underlying that this last proposition tells little about the structure of $\MM$ unless one knows that the span of a finite number of elements is closed, which as said we don't have in this generality.  Things become clearer if the module is $L^p(\mm)$-normed.

Start observing that  having a basis at disposal allows one to locally write an element of a module in {\bf coordinates}. Indeed, let  $v_1,\ldots,v_n$ be a  local basis of $\MM$ on $A$, $v\in \MM$ and  $A_i,\tilde A_i\in \BB$, $i\in\N$, such that $A=\cup_{i}A_i=\cup_i\tilde A_i$ and $f^i_j,\tilde f^i_j\in L^\infty(\mm)$ such that for every $i\in\N$ we have
\[
\nchi_{A_i}v=\sum_{j=1}^nf^i_jv_j,\qquad\qquad\nchi_{\tilde A_i}v=\sum_{j=1}^n\tilde f^i_jv_j.
\]
Then we have
\[
f^i_j=f^{i'}_j,\qquad\mm\ae \ \text{on}\ A_i\cap \tilde A_{i'}\qquad\forall i,i'\in\N,
\]
as it can be easily seen by noticing that  $\sum_{j=1}^n(f^i_j-\tilde f^{i'}_j)v_j=0$ $\mm$-a.e.\   on $A_i\cap \tilde A_{i'}$ and using the definition of local independence.

This means that the functions $f_j:\X\to\R$, $j=1,\ldots,n$, defined by
\[
f_j:=f^i_j,\quad\mm\ae \text{ on}\ A_i,\qquad\forall i\in\N,
\]
and set 0 outside $A$, are well defined in the sense that they depend only on the local basis $v_1,\ldots,v_n$ and the vector $v$. We shall refer to them as the coordinates of $v$ on $A$ w.r.t.\ the local basis $v_1,\ldots,v_n$.

Evidently, in general coordinates are not in $L^\infty(\mm)$, but if  the module is $L^p(\mm)$-normed for some $p\in[1,\infty]$ it still makes sense to write
\[
\nchi_Av=\sum_{i=1}^nf_iv_i,
\]
indeed, recalling the discussion of the previous section, we can identify $\MM$  with the subspace of $\MM^0$ made of elements with finite $L^p(\mm)$-norm and interpret the last identity in $\MM^0$. This procedure can also be reversed, meaning that if $v_1,\ldots,v_n$ is a local basis of $\MM$ on $A$ and $f_i\in L^0(\mm)$, $i=1,\ldots,n$, are such that $\nchi_A|\sum_{i=1}^nf_iv_i|\in L^p(\mm)$, then the element of $\MM^0$ given by $\sum_{i=1}^nf_iv_i$ belongs in fact to $\MM$ and its coordinates w.r.t.\ the gives basis are precisely the $f_i$'s.

We then have the following closure  result:

\begin{proposition}\label{prop:spanclosed}
Let $\MM$ be an $L^p(\mm)$-normed module, $v_1,\ldots,v_n\in \MM$ and $A\in\BB$. Then ${\rm Span}_A\{v_1,\ldots,v_n\}$  is closed. In particular it is a submodule and  coincides with the intersection of all the submodules of $\MM$ containing  $\nchi_Av_1,\ldots,\nchi_A v_n$.
\end{proposition}
\begin{proof}
We identify $\MM$ with the subset of $\MM^0$ made of vectors of finite $L^p(\mm)$-norm and we shall proceed by induction. Taking into account that ${\rm Span}_A\{v_1,\ldots,v_n\}$ is, by definition, closed under the gluing operation, it is easy to see that we can  assume that $\mm(A)<\infty$.

We proceed by induction. 
Let  $n=1$, notice that $\{v_1\}$ is a local basis of $\MM$ on $A\cap \{v_1\neq 0\}$, let $(w_n)\subset {\rm Span}_A\{v_1\}$ be a Cauchy sequence and write $w_n$ in coordinates, i.e.\  write $w_n=f_n v_1$ for some $f_n\in L^0(\mm)$ which is 0 outside $A\cap \{v_1\neq 0\}$. Up to pass to a subsequence we can assume that $\sum_{n}\|w_n-w_{n+1}\|_\MM<\infty$, which is the same as to require that $\sum_n\||v||f_n-f_{n+1}|\|_{L^p(\mm)}<\infty$. It is then clear that the sequence $(\nchi_{A\cap \{v_1\neq 0\}} f_n)\subset L^0(\mm)$ converges in $L^0(\mm)$ to some limit function $ f$ such that $\||v|f\|_{L^p(\mm)}<\infty$. Therefore $w:=fv$ is an element of $\MM$. The definition of $w$ ensures that it belongs to $ {\rm Span}_A\{v_1\}$  and the construction that it is the limit in $\MM$ of the sequence $(w_n)$ thus giving the thesis in this case.

Now we  assume the claim to be valid for $n$ and want to prove it for $n+1$.  

Let $v_1,\ldots,v_{n+1}\in \MM$ and, up to replace $v_i$ by $\min\{1,|v_i|^{-1}\}v_i$,  assume that $|v_i|=1$ $\mm$-a.e.\ on $\{v_i\neq 0\}\cap A$ for every $i=1,\ldots,n+1$ (the assumption that $\mm(A)<\infty$ ensures that this normalization does not destroy the fact that $v_i\in \MM$). Let  $B\in\BB$ the maximal set such that $\nchi_{B}v_{n+1}\in {\rm Span}_A\{v_1,\ldots,v_n\}$ (it exists because spans are closed under the gluing operation) and put $C:=A\setminus B$. 
We can assume that $C\neq \emptyset$, or otherwise there is nothing to prove. 

Let $V_1\subset {\rm Span}_A\{v_1,\ldots,v_n\}$ be the set of $w$'s such that $|w|\leq1$ $\mm$-a.e., and, for given $\eps>0$, define $C_\eps\in \BB$ as
\[
C_\eps:=\Big\{\essinf_{w\in V_1}|v_{n+1}-w|\geq \eps\Big\}.
\]
We claim that $C=\cup_{\eps>0}C_\eps$. Indeed, obviously $C_\eps\subset C$ for every $\eps$ and the inclusion $C_\eps\subset C_{\eps'}$ valid for $\eps\leq \eps'$ shows that the union $\cup_{\eps>0}C_\eps$ is well defined in $\BB$. Now assume by contradiction that $D:=C\setminus\cup_{\eps>0}C_\eps\neq \emptyset$. By definition of $C_\eps$ and using the gluing property, for every $\eps$ we can find $w_\eps\in V_1$ such that $|v_{n+1}-w_\eps|\leq\eps$ $\mm$-a.e.\ on $D$, which implies $\|\nchi_D(v_{n+1}-w_\eps)\|_\MM=\|\nchi_D|v_{n+1}-w_\eps|\|_{L^p(\mm)}\to 0$ as $\eps\downarrow0$. By inductive assumption, ${\rm Span}_A\{v_1,\ldots,v_n\}$ is closed, hence also $V_1$ is closed and the last limit implies that $\nchi_Dv_{n+1}\in V_1$, contradicting the definition of $C$, thus indeed  $C=\cup_{\eps>0}C_\eps$, as claimed.

Notice that for arbitrary $f,g\in L^0(\mm)$, $w\in V_1$ and $\mm$-a.e.\ on $C_\eps$ we have
\[
\begin{split}
|f v_{n+1}+gw|&\geq\max\big\{|f ||v_{n+1}\pm w|-|g\mp f||w|\big\}\geq \eps |f|-\min\{|g+f|,|g-f|\},\\
|fv_{n+1}+gw|&\geq \big||f |-|g|\big|,
\end{split}
\]
thus noticing that for any $a,b\in\R$ we have $\min\{|a-b|,|a+b|\}=||a|-|b||$ we can conclude that
\begin{equation}
\label{eq:lunedi}
\eps |f ||v_{n+1}|\leq 2|f v_{n+1}+gw_m|,\qquad\mm\ae\text{ on }C_\eps\qquad \forall w\in V_1,\ f,g\in L^0(\mm).
\end{equation}

Now fix a Cauchy sequence $(w_m)\subset {\rm Span}_A\{v_1,\ldots,v_{n+1}\}$ converging to some $\bar w\in \MM$ and write the $w_m$'s in coordinates, i.e.\ $w_m=\nchi_A\sum_{i=1}^{n+1}f_{i,m}v_i$. Our aim is to prove that $\bar w\in{\rm Span}_A\{v_1,\ldots,v_{n+1}\}$ and it is easy to see by definition of Span and that of $C$ that to this aim it is sufficient to prove that $\nchi_{C}\bar w\in{\rm Span}_A\{v_1,\ldots,v_{n+1}\}$.

Fix $\eps>0$ and  define $g_m:=\nchi_{C_\eps}|\sum_{i=1}^{n}f_{i,m}v_i|\in L^0(\mm)$ and $\tilde w_m\in \MM^0$ and  by requiring that $\tilde w_m$ is concentrated on $A_\eps$ and that $\nchi_{A_\eps}\sum_{i=1}^{n}f_{i,m}v_i=g_m\tilde w_m$. In particular $|\tilde w_m|=1$ $\mm$-a.e.\ on $A_\eps$ so that $\tilde w_m\in\MM$.

The assumption that  $(w_m)$ is a Cauchy sequence in $\MM$ implies that $(\nchi_{A_\eps}(f_{n+1,m}v_{n+1}+g_m\tilde w_m))$ is Cauchy in $\MM$ as well, thus picking  $f:=f_{n+1,m_1}-f_{n+1,m_2}$ and $g:=g_{m_1}-g_{m_2}$ in \eqref{eq:lunedi} and letting $m_1,m_2\to\infty$ we deduce that $m\mapsto \nchi_{A_\eps}f_{n+1,m}v_{n+1}$ is also a Cauchy sequence in $\MM$ and we call $v^\eps$ its limit. It follows that the sequence $m\mapsto \nchi_{A_\eps }g_m\tilde w_m $ is also Cauchy and denoting its limit by $\tilde w^{\eps }$ we certainly have that $v^{\eps }+\tilde w^{\eps}=\nchi_{A_{\eps}}\bar w$. 

By inductive assumption and the case $n=1$, we know that $v^{\eps}\in{\rm Span}_{A}\{v_{n+1}\}$ and  $w^{\eps}\in{\rm Span}_{A}\{v_1,\ldots,v_n\}$; hence  $\nchi_{C_{\eps}}\bar w\in {\rm Span}_{A}\{v_1,\ldots,v_{n+1}\}$.  Now   let $\eps\downarrow0$ and use the gluing property in $ {\rm Span}_{A}\{v_1,\ldots,v_{n+1}\}$ (use that $\|\nchi_{A_{\eps}}\bar w\|_\MM\leq\|\bar w\|_\MM$ for every $\eps$) to get the existence of $\bar v\in {\rm Span}_{A}\{v_1,\ldots,v_{n+1}\}$ such that $\nchi_{C_{\eps} }\bar v=\nchi_{C_{\eps} }\bar w$ for every $\eps>0$. Since $\cup_{\eps>0}C_{\eps}=C$, by the locality property we conclude that $\nchi_C\bar w=\bar v\in  {\rm Span}_{A}\{v_1,\ldots,v_{n+1}\}$, which is the thesis.
\end{proof}
As a direct consequence we deduce the following reflexivity result:
\begin{theorem}\label{thm:fingenrefl}
Let $\MM$ be an $L^p(\mm)$-normed module, $p<\infty$, $A\in \BB$ and assume that the local dimension of $\MM$ on $A$ is $n\in\N$. Then the local dimension of the dual module $\MM^*$ on $A$ is also $n$.

In particular, a locally finitely generated $L^p(\mm)$-normed module with $p<\infty$ is reflexive.
\end{theorem}
\begin{proof} Let $v_1,\ldots,v_n$ be a local basis of $\MM$ on $A$, put $\MM_i:={\rm Span}_A\{v_1,\ldots,v_{i-1},v_{i+1},\ldots,v_n\}$, $i=1,\ldots,n$, and notice that by Proposition \ref{prop:spanclosed} above we know that the $\MM_i$'s are submodules of $\MM$. Then by point $(ii)$ of Proposition \ref{prop:normevarie} we know that $\MM/\MM_i$ is an $L^p(\mm)$-normed module as well and denoting by $\pi_i:\MM\to \MM/\MM_i$ the natural projection, the fact that the $v_i$'s form a base on $A$ ensures that $\pi_i(v_i)\neq 0$ $\mm$-a.e. on $A$. Apply Corollary \ref{cor:perduali} to the $L^p(\mm)$-normed module $\MM/\MM_i$ and its element $\pi_i(v_i)$ to obtain the existence of $\tilde L_i\in (\MM/\MM_i)^*$ such that $\tilde L_i(\pi_i(v_i))\neq 0$ $\mm$-a.e. on $A$. Define $L_i\in \MM^*$ as $L_i:=\tilde L_i\circ \pi_i$ so that $L_i(v_j)=0$ $\mm$-a.e.\ on $A$ for $i\neq j$ and $L_i(v_i)\neq 0$ $\mm$-a.e.\ on $A$.

We claim that $L_1,\ldots,L_n$ is a base of $\MM^*$ on $A$. We start proving that they span $\MM^*$ on $A$: pick $L\in M^*$ and define $f_i:=L(v_i)\in L^1(\mm)$. Then writing a generic $v\in M$ concentrated on $A$ using its coordinates w.r.t.\ the $v_i$'s,  we see that the identity  $\nchi_AL=\nchi_A\sum_if_iL_i$, to be understood in $(\MM^*)^0$, is valid. This proves that indeed  ${\rm Span}_A\{L_i\}=\MM^*\restr A$. For linear independence, assume that $\sum_if_iL_i=0$ $\mm$-a.e.\ on $A$ for some $f_i\in L^\infty(\mm)$ and  evaluate this identity in $v_j$ to obtain $f_jL_j(v_j)=0$ $\mm$-a.e.\ on $A$. Recalling that $L_j(v_j)\neq 0$ $\mm$-a.e.\ on $A$, we deduce  that $f_j=0$ $\mm$-a.e.\ on $A$, as desired.

The second part is then a simple consequence of the first and of the fact that $\mathcal I_\MM:\MM\to \MM^{**}$ is an isomorphism of $\MM$ with its image.
\end{proof}

On $L^p(\mm)$-normed modules we also have the following simple and useful criterion for recognizing elements in the dual module via their action on a generating space:
\begin{proposition}\label{prop:extension}
Let $\MM$ be an $L^p(\mm)$-normed module, $p<\infty$, $V\subset \MM$ a linear subspace which generates $\MM$ and  $L:V\to L^1(\mm)$ linear and such that
\begin{equation}
\label{eq:extension}
|L(v)|\leq \ell |v|,\quad\mm\ae\qquad\forall v\in V,
\end{equation}
for some function $\ell\in L^q(\mm)$, where $\frac1p+\frac1q=1$.

Then $L$ can be uniquely extended to a module morphism $\tilde L$ from $\MM$ to $L^1(\mm)$, i.e.\ to an element of $\MM^*$, and such $\tilde L$ satisfies
\[
|\tilde L|\leq \ell,\qquad\mm\ae.
\]
\end{proposition}
\begin{proof}
Let $\tilde V\subset \MM$ be the set of elements of the form $\sum_{i=1}^n\nchi_{A_i}v_i$ for some $n\in\N$, $A_i\in\BB$ and $v_i\in V$, $i=1,\ldots, n$. Notice that $\tilde V$ is a vector space and define $\tilde L:\tilde V\to L^1(\mm)$ as
\[
\tilde L\Big(\sum_{i=1}^n\nchi_{A_i}v_i\Big):=\sum_{i=1}^n\nchi_{A_i}L(v_i).
\] 
The bound \eqref{eq:extension} grants that this is a good definition, i.e. the right hand side depends only on $v:=\sum_{i=1}^n\nchi_{A_i}v_i$ and not on the particular way of writing $v$. Then clearly $\tilde L$ is linear. Moreover, any $v\in \tilde V$ can be written as $v=\sum_{i=1}^n\nchi_{A_i}v_i$ with the $A_i$'s disjoint and therefore we have
\[
\|\tilde L(v)\|_{L^1(\mm)}=\sum_{i=1}^n\int_{A_i}|L(v_i)|\,\d\mm\stackrel{\eqref{eq:extension}}\leq \sum_{i=1}^n\int_{A_i}\ell|v_i|\,\d\mm=\int\ell|v|\,\d\mm\leq \|\ell\|_{L^q(\mm)}\|v\|_\MM,
\]
which shows the continuity of $\tilde L$. Hence $\tilde L$ can be uniquely extended to a continuous linear map, still denoted by $\tilde L$, from the closure of $\tilde V$ to $L^1(\mm)$. From the definition of ${\rm Span}_\X(V)$ and property \eqref{eq:dalp} it is immediate to see that the closure of $\tilde V$ contains ${\rm Span}_\X(V)$, and hence its closure, which by assumption is the whole $\MM$. 

The construction ensures that $|\tilde L(v)|\leq\ell|v|$ $\mm$-a.e.\ for every $v\in\MM$, thus point $(v)$ of Proposition \ref{prop:baselp} ensures that $\tilde L$ is a module morphism.

It is now clear that $\tilde L$ is unique, so the proof is complete.
\end{proof}
Another property of generating subspaces is the following:
\begin{proposition}\label{prop:normgen}
Let $\MM$ be an $L^p(\mm)$-normed module, $p\in(1,\infty)$ and $V\subset\MM$ a linear subspace which generates $\MM$. Then for every $L\in\MM^*$ we have
\[
\frac1q|L|_*^q=\esssup_{v\in V}L(v)-\frac1p|v|^p,\qquad\mm\ae.
\]
\end{proposition}
\begin{proof}
Inequality $\geq$ is trivial. To prove $\leq$ notice that if the $\esssup$ was taken among all $v\in\MM$ then the claim would be trivial, then conclude noticing that the assumption that $V$ generates $\MM$ grants that the space of elements of the form
\[
\sum_{i=1}^n\nchi_{E_i}v_i,
\]
with $n\in\N$, $E_i\in\BB$ and $v_i\in V$ for every $i=1,\ldots,n$ is dense in $\MM$.
\end{proof}

We shall also make use of the following simple result:
\begin{proposition}\label{prop:finsep}
Let $\MM$ be an $L^p(\mm)$-normed module, $p<\infty$, and $V\subset\MM$ a generating set. Assume that $V$, when endowed with the induced topology, is separable.

Then $\MM$ is separable as well.
\end{proposition}
\begin{proof}
Notice that the set $\tilde V:=\{\nchi_{|v|\leq n}v\ :\ n\in\N,\ v\in V\}$ is still separable and generating and, by definition, made of bounded elements.

Then letting $D_{\tilde V}\subset \tilde V$ and $D_{L^p}\subset L^p(\mm)$  countable dense sets, it is easy to check directly from the definitions that the space of elements of the form
\[
\sum_{i=1}^nf_iv_i,\qquad v_i\in D_{\tilde V},\ f_i\in D_{L^p},\ i=1,\ldots,n,
\]
is dense in $\MM$, thus giving the result.
\end{proof}

In the case of Hilbert modules, the discussions made provide a quite complete structural characterization. Notice that to  a given  Hilbert space $H$ we can associate the Hilbert module $L^2(\X,H)$  of  $L^2$ maps from $\X$ to $H$, i.e.\ of maps $v:\X\to H$ such that $\|v\|_{L^2(\X,H)}^2:=\int|v|^2(x)\,\d\mm(x)<\infty$. It is clear that this is indeed an $L^\infty(\mm)$-module, the multiplication with a function in $L^\infty(\mm)$ being simply the pointwise one, and given that $(L^2(\X,H),\|\cdot\|_{L^2(\X,H)})$ is an Hilbert space, we have indeed an Hilbert module.

Now fix an infinite dimensional separable Hilbert space $H$ together with a sequence of subspaces $V_i\subset H$, $i\in\N$, such that $\dim V_i=i$ for every $i\in\N$. Then to each partition $\{E_i\}_{i\in\N\cup\{\infty\}}$ of $\X$ we associate the Hilbert module $\mathcal H(\{E_i\};H,\{V_i\})$ made of elements $v$ of $L^2(\X,H)$ such that
\[
v(x)\in V_i,\qquad \mm\ae\text{ on }  E_i,\qquad\qquad \forall i\in\N. 
\]
It is readily verified that $\mathcal H(\{E_i\};H,\{V_i\})$ is a submodule of $L^2(\X,H)$ and thus an Hilbert module.

We then have the following structural result:
\begin{theorem}[Structural characterization of separable Hilbert modules]\label{thm:structhil}
Let $\HH$ be a separable Hilbert module. Then there exists a unique  partition $\{E_i\}_{i\in\N\cup\{\infty\}}$ of $\X$ such that $\HH$ is isomorphic to $\mathcal H(\{E_i\};H,\{V_i\})$.
\end{theorem}
\begin{proof} $\ $\\
\noindent{\bf Uniqueness} It is evident from the construction  that the local dimension of  $\mathcal H(\{E_i\};H,\{V_i\})$ on $E_i$ is precisely $i$ for any $i\in\N$. In particular, due to Proposition \ref{prop:localdimension}, given another partition $\{\tilde E_i\}_{i\in\N\cup\{\infty\}}$, the modules $\mathcal H(\{E_i\};H,\{V_i\})$  and $\mathcal H(\{\tilde E_i\};H,\{V_i\})$ are isomorphic if and only if $E_i=\tilde E_i$ for every $i\in\N\cup\{\infty\}$. The claim follows.

\noindent{\bf Existence} For any $i\in\N$ pick an orthonormal basis $e^i_1,\ldots,e^i_i$ of $V_i$,  let $E_i$, $i\in\N\cup\{\infty\}$, be given by Proposition \ref{prop:dimdec}, $I\subset\N\cup\{\infty\}$ be the set of indexes $i$ such that $\mm(E_i)>0$ and for each $i\in I\setminus\{0,\infty\}$, let $v^i_1,\ldots,v^i_i$ be  a basis of $\HH$ on $E_i$. We apply the Gram-Schmidt orthogonalization process:  for each $i\in  I\setminus\{0,\infty\}$ define a new base $\tilde v^{i}_1,\ldots,\tilde v^i_i$ of $\HH$ on $E_{i}$ by putting $\tilde v^{i}_1:=v^{i}_1$ and for $j=2,\ldots,i$ defining
\[
\tilde v^{i}_j:=\nchi_{E_i}\Big(v^{i}_j-\sum_{k=1}^{j-1}\frac{\la v^{i}_j, \tilde v^{i}_k\ra}{|\tilde v^{i}_k|^2}\tilde v^{i}_k\Big),
\]
where the right hand side is a priori understood in $\HH^0$. Then by direct computation of the norm we see that in fact $\tilde v^i_j\in \HH$ and it is readily verified that $\tilde v^{i}_1,\ldots,\tilde v^i_i$ is still a basis of $\HH$ on $E_{i}$. Now define $T_i:\HH\restr{E_i}\to \mathcal H(\{E_i\};H,\{V_i\})$ by declaring that
\[
T_i(\tilde v^i_j):=\nchi_{E_i}|\tilde v^i_j|e^i_j,
\]
and extending it to the whole $\HH\restr{E_i}$ by $L^\infty(\mm)$-linearity: the fact that $|T_i(\tilde v^i_j)|=|\tilde v^i_j|$ and $\la T_i(\tilde v^i_j),T_i(\tilde v^i_{j'})\ra=0=\la v^i_j,v^i_{j'}\ra$ $\mm$-a.e.\  for $j\neq j'$ and that $\tilde v^{i}_1,\ldots,\tilde v^i_i$ is a basis grant that this extension exists, is unique and defines a norm-preserving modulo morphism.

Now define $T_\N:\HH\restr{\cup_{i\in\N}E_i}\to \mathcal H(\{E_i\};H,\{V_i\})$ by
\[
T_\N(v):=\sum_{i\in I\setminus\{0,\infty\}}T_i(\nchi_{E_i}v),\qquad\forall v\in \HH\restr{\cup_{i\in\N}E_i}.
\]
Given that the $T_i$'s are norm-preserving, the series at the right hand side of this last expression converges in $ \mathcal H(\{E_i\};H,\{V_i\})$, so the definition is well-posed. It is then clear that $T_\N$ is a norm preserving modulo morphism whose image is precisely $\mathcal H(\{E_i\};H,\{V_i\})\restr{\cup_{i\in\N}E_i}$. If $\mm(E_\infty)=0$, then $\HH\restr{\cup_{i\in\N}E_i}=\HH$ and $T_\N$ is surjective, so that the proof is complete.

Otherwise assume that $\mm(E_\infty)>0$ and let $\{v_i\}_{i\in\N}$ be a countable dense subset of $\HH\restr{E_\infty}$, which exists because $\HH$ is separable. Put $\HH_0:=\{0\}\subset \HH$ and for $i\in\N$, $i>0$, we shall inductively define:
\begin{itemize}
\item[i)]  an increasing sequence $(\HH_i)$ of submodules such that the dimension of $\HH_i$ on $E_\infty$ is exactly $i$, and $v_j\in \HH_i$ for every $j\leq i$,
\item[ii)] a sequence $(w_i)\subset \HH$ such that $w_1,\ldots,w_i$ is a local basis of $\HH_i$ on $E_\infty$ and $\la w_i,w_j\ra=0 $ $\mm$-a.e.\ on $E_\infty$ for any $i\neq j$.
\end{itemize}
The construction goes as follows. Pick $i\in\N$ and assume that  both the submodule $\HH_{i}$ and the local basis $w_1,\ldots,w_{i}$ have been defined (the latter being the empty condition if $i=0$). Let $n_1$ be the minimum $n\in\N$ such that $v_n\notin \HH_i$. Such integer exists because $(v_n)$ is dense in $\HH\restr{E_\infty}$,  $\HH_i$ has  finite dimension on $E_\infty$, while $\HH\restr{E_\infty}$ has not. Then notice that the family of subsets $B$ of $E_\infty$ such that $\nchi_Bv_{n_1}\in \HH_i$ is stable by countable union and thus there exists a maximal set $B_1$ with this property. Put $A_1:=E_\infty\setminus B_1$ and  let $\HH_{i,1}$ be the module generated by $\HH_i\cup\{\nchi_{A_1}v_{n_1}\}$. Notice that by induction we have $n_1>i$ and $v_n\in \HH_{i,1}$ for every $n\leq n_1$.

Continue by iteration: if $A_{k-1}$ has already been defined and $\mm(E_\infty\setminus\cup_{j<k}A_j)>0$, find the  minimal integer $n_k$ so that $v_{n_k}\notin \HH_{i,k-1}$, the maximal subset $B_k$ of  $ E_\infty\setminus\cup_{j<k}A_j$ such that $\nchi_{B_k}v_{n_k}\in \HH_{i,k-1}$, put $A_k:=(E_\infty\setminus\cup_{j<k}A_j)\setminus B_k$ and let $\HH_{i,k}$ be the module generated by $\HH_{i,k-1}\cup\{\nchi_{E_\infty\setminus\cup_{j<k}A_j}v_{n_k}\}$.

If for some $k\in\N$ we have $\mm(E_\infty\setminus\cup_{j< k}A_j)=0$, put 
\[
\HH_{i+1}:=\HH_{i,k-1}\qquad\text{ and }\qquad  \tilde v_{i+1}:=\sum_{l=1}^{k-1} \nchi_{A_l}v_{n_l},
\]
otherwise put
\[
\HH_{i+1}:=\text{the module generated by }\bigcup_{k\in\N}M_{i,k},\qquad \text{and}\qquad \tilde v_{i+1}:=\sum_{l\in\N} \frac{\nchi_{A_l}}{2^l\|v_{n_l}\|_\HH}v_{n_l}.
\]
We claim that in either case it holds:
\begin{equation}
\label{eq:perbase}
\forall B\in\BB\ \text{with }\mm(B)>0,\text{ we have }\nchi_B\tilde v_{i+1}\notin \HH_i.
\end{equation}
In the first case this is obvious by the fact that the procedure stopped, in the second we argue by contradiction: if there is $B\in\BB$ with positive measure such that $\nchi_B\tilde v_{i+1}\in \HH_i$ by the way we built the $v_{n_k}$'s this means that for every $n\in\N$ we have $\nchi_Bv_n\in \HH_i$. But since $\{v_n\}_{n\in\N}$ is dense in $\HH\restr{E_\infty}$, this would imply that $(\HH\restr{E_\infty})\restr B=\HH_i\restr B$, so that $\HH\restr{E_\infty}$ would have finite dimension on $B$, contradicting the definition of $E_\infty$.

Now put
\[
w_{i+1}:=\tilde v_{i+1}-\sum_{j=1}^i\frac{\la \tilde v_{i+1}, w_j\ra}{|w_j|^2} w_j,
\]
and notice that by construction $(i)$ and $(ii)$ above are fulfilled, the role of \eqref{eq:perbase} being to ensure that the $w_j$'s are independent on $E_\infty$. 

We remark that by $(i)$ we have that $\cup_i\HH_i$ is dense in $\HH\restr{E_\infty}$. Then let $(e_i)_{i\in\N}$ be an Hilbert basis of $H$ and define $T_\infty:\HH\restr{E_\infty}\to  \mathcal H(\{E_i\};H,\{V_i\})$ by declaring that
\[
T_\infty(w_i):=\nchi_{E_\infty}|w_i|e_i,\qquad\forall i\in\N,
\]
and extending it by $L^\infty(\mm)$-linearity and continuity: the fact that $\mm$-a.e.\ on $E_\infty$ we have $|T_\infty(w_i)|=|w_i|$ for every $i\in\N$ and $\la T_\infty(w_i),T_\infty(w_j)\ra $ for $i\neq j$, and properties $(i)$, $(ii)$ grant that such extension exists, is unique and that $T_\infty$ is a norm-preserving module morphism whose image is precisely $\mathcal H(\{E_i\};H,\{V_i\})\restr{E_\infty}$.

To conclude, define $T:\HH\to \mathcal H(\{E_i\};H,\{V_i\})$ as
\[
T(v):=T_\N(\nchi_{\cup_{i\in\N E_i}}v)+T_\infty(\nchi_{E_\infty}v),\qquad\forall v\in \HH.
\]
The construction grants that $T$ is the desired isomorphism.
\end{proof}
It is worth to underline that the construction of the isomorphism in this last theorem is somehow similar to the choice of a basis in a given vector space, in the sense that it is possible but not intrinsic. Thus although in practical situations it is useful to have this structural result at disposal, it is reductive and dangerous, and in fact wrong, to think that every separable Hilbert module is of the form $\mathcal H(\{E_i\};H,\{V_i\})$. Consider for instance a Riemannian manifold $M$ and the Hilbert module of $L^2$ vector fields on it. Such module is certainly isometric to $L^2(M,\R^{\dim M})$, but the choice of the isomorphism is equivalent to a Borel choice of an orthonormal basis in a.e.\ point of $M$, which is definitively possible but highly not intrinsic in general.
\begin{remark}[Direct integral of Hilbert spaces]\label{re:dirintH}{\rm A byproduct of Theorem \ref{thm:structhil} is that it shows that a posteriori the theory of separable Hilbert modules and that of direct integral of Hilbert spaces (a concept generalizing that of direct sum to a `continuous family of indexes' - we refer to \cite{Takesaki79}  for an overview of this topic and detailed bibliography) are tightly linked. Indeed, it is clear that a direct integral of Hilbert spaces is a separable Hilbert module, while Theorem \ref{thm:structhil} shows that the viceversa is also true.

Both approaches describe in some sense the notion of measurable bundle of Hilbert spaces, the  difference being in the way the relevant object is built: with the direct integral one starts the construction from the notion of `fiber', while with Hilbert modules one rather picks the `section' point of view.

In this analogy, the space $\HH^0$ would correspond, in the direct integral language, to a measurable field of Hilbert spaces, while a module morphism of Hilbert modules corresponds to a decomposable operator.
}\fr\end{remark}

\subsection{Tensor and exterior products of Hilbert modules}\label{se:tensorproduct}

In this section we discuss the definition of tensor products of Hilbert modules and related constructions like the exterior product. Shortly said, we want a construction that does the following job: if we consider the Hilbert module $\HH$ of $L^2$ vector fields on $\R^n$, then the tensor product module $\HH\otimes\HH$ should be the module of $L^2$ matrix fields $x\mapsto A_x\in \R^{n\times n}$, the pointwise norm being the Hilbert-Schmidt norm. Notice that considering the tensor product of $\HH$ with itself in the sense of Hilbert spaces would produce a different Hilbert structure, see Remark \ref{rem:tensdiversi}.

Theoretical discussions apart, we have a very practical reason for looking for such a construction: on a smooth Riemannian manifold  we have the Bochner identity
\[
\Delta\frac{|\nabla f|^2}2= |\H f|_\HS^2+\la \nabla f,\nabla \Delta f\ra+{\rm Ric}(\nabla f,\nabla f),
\] 
the norm of the Hessian appearing here being the Hilbert-Schmidt one. Given that one of the aims of this paper is to build a second order calculus and to provide an analogous of the above identity on non-smooth spaces with Ricci curvature bounded from below,  we must have at disposal the notion of pointwise Hilbert-Schmidt norm.

A similar comment concerns exterior product, which is necessary in order to define $L^2$ differential forms.

\vspace{1cm}

It is worth to recall that for two given Hilbert spaces $H_1,H_2$, their tensor product $H_1\otimes H_2$ is defined as follows. One starts considering the algebraic tensor product $H_1\otimes^{\rm Alg}H_2$, i.e.\ the tensor product of the two spaces seen as vector spaces, and defines a bilinear form $\la\cdot,\cdot\ra$ on it by declaring that
\[
\la v_1\otimes v_2,w_1\otimes w_2\ra:= \la v_1,w_1\ra_{H_1}\la v_2,w_2\ra_{H_2}\qquad\forall v_1,w_1\in H_1,\ v_2,w_2\in H_2,
\]
and extending it by linearity, $\la\cdot,\cdot\ra_{H_i}$ being the scalar product on $H_i$, $i=1,2$. Such bilinear form is evidently symmetric and one can verify, for instance using Hilbert bases, that it is positive definite, i.e.
\begin{equation}
\label{eq:tenshil}
\begin{split}
\la A,A\ra&\geq 0\\
\la A,A\ra&= 0\qquad\Leftrightarrow \qquad  A=0,
\end{split}
\end{equation}
for every $A\in H_1\otimes^{\rm Alg}H_2$. The tensor product $H_1\otimes H_2$ is then defined as the completion of $H_1\otimes^{\rm Alg}H_2$ w.r.t.\ the norm induced by $\la\cdot,\cdot\ra$.

\bigskip

We turn to the construction in the case of modules. Let  $\HH_1,\HH_2$ be Hilbert modules on a given fixed $\sigma$-finite measured space $(\X,\mathcal A,\mm)$. Consider the topological vector spaces  $\HH^0_1,\HH^0_2$ defined in Section \ref{se:altint}, recall that they are endowed with a natural structure of $L^0(\mm)$-module (here we refer just to the algebraic point of view) and consider their algebraic tensor product $\HH^0_1\otimes^{\rm Alg}_{L^0}\HH^0_2$. This means that  we take the vector space of the formal finite sums of objects of the kind $v\tilde\otimes w$, with  $v\in \HH^0_1$ and $w\in \HH_2^0$, and quotient it w.r.t.\ the subspace generated by the elements of the form
\[
\begin{split}
&(\alpha_1v_1+\alpha_2v_2)\tilde\otimes w-\alpha_1(v_1\tilde\otimes w)-\alpha_2(v_2\tilde\otimes w),\\
&v\tilde\otimes (\beta_1w_1+\beta_2w_2)-\beta_1(v\tilde\otimes w_1)-\beta_2(v\tilde\otimes w_2),\\
&(fv)\tilde\otimes w-v\tilde\otimes(fw), 
\end{split}
\]
for generic $v,v_1,v_2\in\HH^0_1$, $w,w_1,w_2\in \HH^0_2$, $\alpha_1,\alpha_2,\beta_1,\beta_2\in\R$ and $f\in L^0(\mm)$.

The resulting quotient vector space is, by definition, $\HH^0_1\otimes^{\rm Alg}_{L^0}\HH_2^0$ and as customary we shall denote by $v\otimes w$ the equivalence class of $v\tilde\otimes w$.  Given that $L^0(\mm)$ is an abelian ring (w.r.t.\ pointwise $\mm$-a.e.\ multiplication),  $\HH_1^0\otimes^{\rm Alg}_{L^0}\HH_2^0$ is naturally equipped with a multiplication with $L^0(\mm)$-functions  by declaring that:
\[
f(v\otimes w):=(fv)\otimes w=v\otimes(fw),\qquad\forall f\in L^0(\mm),\ v\in \HH_1^0,\ w\in \HH_2^0,
\]
and extending it by linearity, the role of the commutativity of the product in $L^0(\mm)$ being to ensure that this is a good definition.

Denoting by $\la\cdot,\cdot\ra$ the pointwise scalar product on both $\HH_1^0$ and $\HH_2^0$ (recall the discussion after Definition \ref{def:m0}), we define the bilinear map $ :$ from $[\HH_1^0\otimes^{\rm Alg}_{L^0}\HH_2^0]^2$ to $L^0(\mm)$  by declaring that 
\[
\HH_1^0\otimes^{\rm Alg}_{L^0}\HH_2^0\ni v_1\otimes w_1, \ v_2\otimes w_2\qquad\mapsto\qquad  (v_1\otimes w_1):( v_2\otimes w_2):=\la v_1,v_2\ra\la w_1,w_2\ra,
\]
and extending it by bilinearity. It is clear that this map is symmetric, i.e. $A:B = B: A$ $\mm$-a.e.\ for every $A,B\in  \HH_1^0\otimes^{\rm Alg}_{L^0}\HH_2^0$, and local, i.e.
\[
f( A:B)=( f A):B  = A:(fB)\qquad\mm\ae \qquad\forall f\in L^0(\mm),\ A,B\in \HH_1^0\otimes^{\rm Alg}_{L^0}\HH_2^0.
\]
We further claim that it is positive definite in the sense  that
\begin{equation}
\label{eq:esame}
\begin{split}
A:A&\geq 0\qquad\mm\ae\\
A:A&= 0\qquad\mm\ae\text{ on }E \qquad\Leftrightarrow\qquad \nchi_EA=0,
\end{split}
\end{equation}
for every $ A\in \HH_1^0\otimes^{\rm Alg}_{L^0}\HH_2^0$ and $E\in \BB$. To verify this property, assume for the moment that   $\HH_1,\HH_2$ are separable, use the structural characterization provided by Theorem \ref{thm:structhil} and then conclude using the analogous property \eqref{eq:tenshil} valid for Hilbert spaces. To remove the assumption of separability, pick $A\in \HH_1^0\otimes^{\rm Alg}_{L^0}\HH_2^0$ and write it as $A=\sum_{i=1}^nv_i\otimes w_i$ for some  $v_i\in \HH_1$, $w_i\in \HH_2$, $i=1,\ldots,n$. Then notice that $\tilde \HH_1:={\rm Span}_\X\{v_1,\ldots,v_n\}\subset \HH_1$ and $\tilde \HH_2:={\rm Span}_\X\{w_1,\ldots,w_n\}\subset \HH_2$ are separable because they are  finitely generated  (Proposition \ref{prop:finsep}), observe that   by construction $A\in  \tilde \HH_1^0\otimes^{\rm Alg}_{L^0}\tilde \HH_2^0$ and conclude applying  the previous argument.

Then define $|\cdot|_\HS: \HH_1^0\otimes^{\rm Alg}_{L^0}\HH_2^0\to L^0(\mm)$ as
\begin{equation}
\label{eq:hspunt}
|A|_\HS:=\sqrt{  A:A }.
\end{equation}
By \eqref{eq:esame} the definition is well posed and taking into account bilinearity and locality it is not difficult to see that
\begin{equation}
\label{eq:cdg}
\begin{array}{rll}
|A|_\HS&\!\!\!=0\qquad\qquad\qquad&\mm\ae\text{ on }E\qquad\Leftrightarrow\qquad A=0\quad \text{ on }E,\\
|A+B|_\HS&\!\!\!\leq |A|_\HS+|B|_\HS\quad&\mm\ae,
\\
|fA|_\HS&\!\!\!=|f||A|_\HS\ \qquad&\mm\ae, 
\end{array}
\end{equation}
for every $A,B\in \HH_1^0\otimes^{\rm Alg}_{L^0}\HH_2^0$, $f\in L^0(\mm)$ and $E\in\BB$.

The topological vector space $ \HH_1^0\otimes \HH_2^0$ is defined as the completion of $ \HH_1^0\otimes^{\rm Alg}_{L^0}\HH_2^0$ w.r.t.\ the distance
\[
\sfd_{\otimes}(A,B):=\sum_i\frac{1}{2^i\mm(E_i)}\int \min\{1,|A-B|_\HS\}\,\d\mm,
\]
where $(E_i)\subset\BB$ is a partition of $\X$ in sets of positive and finite measure. It is readily verified that the topology of $ \HH_1^0\otimes \HH_2^0$ does not depend on the particular partition chosen and that the operations of addition, multiplication with a function in $L^0(\mm)$ and of pointwise norm all can, and will, be uniquely extended by continuity to the whole  $\HH_1^0\otimes \HH_2^0$ still satisfying \eqref{eq:hspunt} and \eqref{eq:cdg}.
\begin{definition}[The tensor product $\HH_1\otimes \HH_2$]
The tensor product $\HH_1\otimes \HH_2$ is defined as the subset of  $ \HH_1^0\otimes \HH_2^0$ made of tensors $A$ such that
\begin{equation}
\label{eq:hstutto}
\|A\|^2:={\int|A|_\HS^2\,\d\mm}<\infty.
\end{equation}
\end{definition}
We endow  $\HH_1\otimes \HH_2$ with the norm defined in \eqref{eq:hstutto}. It is clear from the definition \eqref{eq:hspunt} that such norm is induced by a scalar product and the very definition of  $ \HH_1\otimes \HH_2$ and  $ \HH_1^0\otimes \HH_2^0$ grant that  $ (\HH_1\otimes \HH_2,\|\cdot\|)$ is complete, and thus an Hilbert space. Moreover,  the last in \eqref{eq:cdg} shows that the multiplication with functions in $L^0(\mm)$ defined on  $ \HH_1^0\otimes \HH_2^0$ induces by restriction a multiplication with $L^\infty(\mm)$ functions on $\HH_1\otimes\HH_2$ which takes values in  $\HH_1\otimes\HH_2$, giving  $\HH_1\otimes\HH_2$ the structure of an $L^\infty(\mm)$-premodule. Given that it is also, by \eqref{eq:hstutto} and the last in \eqref{eq:cdg}, $L^2(\mm)$-normed, we see from Proposition \ref{prop:baselp} that  $\HH_1\otimes\HH_2$ is in fact an Hilbert module. 

\begin{remark}[Hilbert modules and Hilbert spaces]\label{rem:tensdiversi}{\rm Denote by  $\HH_1\otimes_{\rm Hilb}\HH_2$ the tensor product of $\HH_1$ and $\HH_2$ in the sense of Hilbert spaces. Then it is important to underline that the Hilbert spaces  $\HH_1\otimes_{\rm Hilb}\HH_2$ and  $\HH_1\otimes\HH_2$ are in general different.

Indeed, for $v_1\in \HH_1$ and $v_2\in\HH_2$, the norm of $v_1\otimes v_2$ as element of  $\HH_1\otimes_{\rm Hilb}\HH_2$ is $\|v_1\|_{\HH_1}\|v_2\|_{\HH_2}$ while its norm in  $\HH_1\otimes\HH_2$  is given by $\sqrt{\int |v_1|^2|v_2|^2\,\d\mm}$. In fact, this shows that in general $v_1\otimes v_2$ might be not an element of  $\HH_1\otimes\HH_2$  at all. 
}\fr\end{remark}

\begin{remark}[Lack of universal property]{\rm
It is debatable whether we are entitled to call $\HH_1\otimes\HH_2$  `tensor product' or not, the problem being that the lack of the expected universal property in the category of Hilbert modules.

In this direction, it is worth recalling that  the same issue occurs with Hilbert spaces. Consider for instance an infinite dimensional and separable Hilbert space $H$, the bilinear continuous map $\otimes:H\times H\to H\otimes H$ defined by $\otimes(v,w):=v\otimes w$ and the bilinear continuous map $B:H\times H\to\R$ given by $ B(v,w):=\la v,w\ra_H$. Then there is no linear continuous map $L:H\otimes H\to \R$ such that $L\circ \otimes=B$. This can be seen considering an Hilbert base $(e_i)\subset H$ and noticing that
\[
\sup_n\Big\|\sum_{i=1}^n\frac1ie_i\otimes e_i\Big\|^2_{H\otimes H}=\sup_n\sum_{i=1}^n\frac1{i^2}<\infty \qquad\text{while}\qquad B\big(\sum_{i=1}^n\frac1ie_i\otimes e_i\big)=\sum_{i=1}^n\frac1i\sim \log n.
\]

Given that if the reference measure $\mm$ is a Dirac delta the category of Hilbert $L^\infty(\mm)$-modules is trivially isomorphic to that of Hilbert spaces (see also Example \ref{ex:delta}), the same issue occurs for modules.

Actually, for Hilbert modules over arbitrary measured spaces the situation is much worse: as we have seen in the previous remark, in general we don't even have the bilinear continuous map $\otimes:\HH\times\HH\to\HH\otimes\HH$.

Still, the object $\HH\otimes\HH$ arises naturally in this framework, will be crucial in our discussions,  and given we won't propose any other `tensor-like product', we reserve for this one the name of tensor product.
}\fr
\end{remark}
Notice that
\begin{equation}
\label{eq:perdensotp}
\begin{split}
&\text{if $D_1\subset\HH_1,D_2\subset\HH_2$ are dense subsets made of bounded elements, then }\\
&\text{the linear span of $\{v_1\otimes v_2\ :\  v_1\in D_1,\ v_2\in D_2\}$ is dense in $\HH_1\otimes\HH_2$,}
\end{split}
\end{equation}
as can be directly checked with a truncation argument. It also follows that
\begin{equation}
\label{eq:baseseptenssep}
\text{ if $\HH_1$ and $\HH_2$ are separable then so is  $ \HH_1\otimes \HH_2$}.
\end{equation}
Indeed, if $D_1,D_2$ are countable dense subsets of $\HH_1,\HH_2$ made of bounded elements (recall property \eqref{eq:seplplim}), then the $\Q$-vector space generated by $\{v_1\otimes v_2\ :\  v_1\in D_1,\ v_2\in D_2\}$ is dense in the $\R$-vector space generated by the same set, and thus, by \eqref{eq:perdensotp} above, dense in  $ \HH_1\otimes \HH_2$.

In the special case $\HH_1=\HH_2$ we denote the tensor product $\HH\otimes \HH$ as $\HH^{\otimes 2}$. In this case following standard ideas one can  define what are symmetric and antisymmetric tensors. In detail, for $A\in \HH^0\otimes^{\rm Alg}_{L^0}\HH^0$ we define the \emph{transpose} $A^t\in \HH^0\otimes^{\rm Alg}_{L^0}\HH^0$ by declaring that 
\[
(v\otimes w)^t:=w\otimes v,\quad\mm\ae\qquad\forall v,w\in \HH^0,
\]
and extending such map by linearity. It is readily verified that this is a good definition, and that $A\mapsto A^t$ is a linear involution preserving the pointwise norm of $\HH^0\otimes^{\rm Alg}_{L^0}\HH^0$ which is local in the sense that  $(fA)^t=fA^t$ $\mm$-a.e.\ for every $f\in L^0(\mm)$ and $A\in \HH^0\otimes^{\rm Alg}_{L^0}\HH^0$. In particular it can be extended by continuity to a linear local involution of $\HH^0\otimes \HH^0$ preserving the pointwise norm which therefore restricts to a morphism, still denoted as $A\mapsto A^t$, of $\HH^{\otimes 2}$ into itself which is also an involution and an isometry.
In particular, we have
\begin{equation}
\label{eq:scaltrbase}
A:B=A^t:B^t,\quad \mm\ae\qquad\forall A,B\in \HH^{\otimes 2}.
\end{equation}
Then we define $\HH^{\otimes 2}_{\sf Sym}\subset \HH^{\otimes 2}$ as the space of $A$'s such that $A^t=A$ and $\HH^{\otimes 2}_{\sf Asym}\subset \HH^{\otimes 2}$ as the space of $A$'s such that $A^t=-A$. The fact that transposition is a module morphism  grants that $\HH^{\otimes 2}_{\sf Sym}, \HH^{\otimes 2}_{\sf Asym}$ are submodules, while identity \eqref{eq:scaltrbase} gives that 
\begin{equation}
\label{eq:scaltr2base}
A:B=0,\quad\mm\ae,\qquad\forall A\in \HH^{\otimes 2}_{\sf Sym},\ B\in \HH^{\otimes 2}_{\sf Asym}.
\end{equation}
By integration, we see that in particular $\HH^{\otimes 2}_{\sf Sym}, \HH^{\otimes 2}_{\sf Asym}$ are orthogonal subspaces, in the sense of Hilbert spaces, of $\HH^{\otimes 2}$. For a generic $A\in \HH^{\otimes 2}$ define 
\begin{equation}
\label{eq:asim}
A_{\sf Sym}:=\frac{A+A^t}2,\qquad A_{\sf Asym}:=\frac{A-A^t}2,
\end{equation}
and notice that $A=A_{\sf Sym}+A_{\sf Asym}$ with $A_{\sf Sym}\in \HH^{\otimes 2}_{\sf Sym}$ and $A_{\sf Asym}\in \HH^{\otimes 2}_{\sf Asym}$. It follows  that the direct sum of $\HH^{\otimes 2}_{\sf Sym}$ and $\HH^{\otimes 2}_{\sf Asym}$ is the whole $\HH^{\otimes 2}$ and the maps $A\mapsto A_{\sf Sym}$ and $A\mapsto A_{\sf Asym}$, which are module morphisms, are the orthogonal projections onto $\HH^{\otimes 2}_{\sf Sym}$, $\HH^{\otimes 2}_{\sf Asym}$. Identity \eqref{eq:scaltr2base} also gives the identity
\begin{equation}
\label{eq:decompsym}
|A|^2_\HS=|A_{\sf Sym}|_\HS^2+|A_{\sf Asym}|_\HS^2\qquad\mm\ae\qquad\forall A\in \HH^{\otimes 2}.
\end{equation}
Now we claim that if $D\subset \HH$ is a set of bounded elements generating $\HH$ in the sense of modules, then the set $\{v\otimes v:v\in D\}$ generates $\HH^{\otimes 2}_{\sf Sym}$ in the sense of modules. Indeed, being the elements of $D$ bounded, we have $v\otimes v\in \HH^{\otimes 2}$ for every $v\in D$ and since  $v\otimes v$ is certainly symmetric, we get that $\{v\otimes v:v\in D\}\subset \HH^{\otimes 2}_{\sf Sym}$. On the other hand, using the projection $A\mapsto A_{\sf Sym}$ we see that $\HH^{\otimes 2}_{\sf Sym}$   is generated by elements of the kind $v\otimes  w+w\otimes v$ and thus we conclude noticing that
\[
 v \otimes w+w\otimes v=(v+w)\otimes(w+v)-v\otimes v-w\otimes w,\qquad\mm\ae\qquad\forall v,w\in \HH^0.
\]
In particular, recalling Proposition \ref{prop:normgen} and observing that  $A:(v\otimes v)=A_{\sf Sym}:(v\otimes v)$ $\mm$-a.e.\  for any $A\in M^{\otimes 2}$, we deduce the formula
\begin{equation}
\label{eq:normsym}
|A_{\sf Sym}|_\HS^2=\esssup \Big(2A:\sum_{i=1}^n v_i\otimes v_i-\Big|\sum_{i=1}^nv_i\otimes v_i\Big|_\HS^2\Big)\qquad\mm\ae\qquad\forall A\in \HH^{\otimes 2},
\end{equation}
where the $\esssup$ is taken among all $n\in\N$ and $v_1,\ldots,v_n\in D$.


\bigskip

We pass to the {\bf exterior product}. Recall that for $k\in\N$ the $k$-th exterior power $\Lambda^kH$ of an Hilbert space is defined as the quotient of $H^{\otimes n}$ w.r.t.\ the closed subspace generated by elements of the form
\[
v_1\otimes\ldots\otimes v_k\qquad\text{ with $v_i=v_j$ for at least two different indexes $i,j$}.
\]
The equivalence class of $v_1\otimes\ldots\otimes v_k$ is denoted by $v_1\wedge\ldots\wedge v_k$ and one can verify that, up to a factor $k!$, the scalar product on $\Lambda^kH$ induced by the quotient is characterized by
\[
\la v_1\wedge\cdots\wedge v_k,w_1\wedge\cdots\wedge w_k\ra:=\det\big( v_i\cdot  w_j \big).
\]
For Hilbert modules the construction is similar. Let us fix an Hilbert module $\HH$ and $k\in\N$. If $k=0$ we put $\Lambda^0\HH:=L^2(\mm)$. For $k>0$ we start taking the quotient of the tensor product $(\HH^0)^{\otimes k}$  w.r.t.\ the closure of the subspace generated by elements of the form
\[
v_1\otimes\ldots\otimes v_k\qquad\text{ with $v_i=v_j$ for at least two different indexes $i,j$},
\]
denote such quotient as $\Lambda^k\HH^0$ and the equivalence class of $v_1\otimes\ldots\otimes v_k$ by $v_1\wedge\ldots\wedge v_k$. The multiplication with functions in $L^0(\mm)$ passes to the quotient and defines a multiplication with $L^0(\mm)$ functions satisfying
\[
f(v_1\wedge\ldots\wedge v_k)=(fv_1)\wedge\ldots\wedge v_k=\cdots=v_1\wedge\ldots\wedge (fv_k).
\]
Starting from the result on Hilbert spaces and using the structural characterization of Hilbert modules (Theorem \ref{thm:structhil}) as done before we see that, up to a multiplication by $k!$, the scalar product on $\Lambda^k\HH^0$ is characterized by
\[
\la v_1\wedge\cdots\wedge v_k,w_1\wedge\cdots\wedge w_k\ra=\det\big(\la v_i,  w_j\ra\big).
\]
It is then clear that such scalar product is positively definite, i.e. 
\[
\begin{split}
\la \omega,\omega\ra&\geq 0\qquad\mm\ae\\
\la \omega,\omega\ra &= 0\qquad\mm\ae\text{ on }E \qquad\Leftrightarrow\qquad \nchi_E\omega=0,
\end{split}
\]
for every $\omega\in  \Lambda^k\HH^0$. Then we define the pointwise norm $|\omega|:=\sqrt{\la\omega,\omega\ra}\in L^0(\mm)$ and finally the $k$-th exterior power of an Hilbert module as:
\begin{definition}[Exterior power]
The $k$-th exterior power $\Lambda^k\HH$ of $\HH$ is the subspace of $\Lambda^k\HH^0$ made of elements $\omega$ such that
\[
\|\omega\|^2:=\int|\omega|^2\,\d\mm<\infty.
\]
\end{definition}
The same simple arguments used for the tensor product ensure that $\Lambda^k\HH$ is an Hilbert module and that 
\begin{equation}
\label{eq:extsep}
\text{if $\HH$ is separable then so is $\Lambda^k\HH$ for any $k\in\N$.}
\end{equation}
We conclude noticing that the map
\[
(v_1\wedge\ldots\wedge v_k,w_1\wedge\ldots\wedge w_{k'})\qquad\mapsto \qquad v_1\wedge\ldots\wedge v_k\wedge w_1\wedge\ldots\wedge w_{k'},
\]
can be extended by bilinearity and continuity to a map, called wedge product, from\linebreak  $\Lambda^k\HH^0\times  \Lambda^{k'}\HH^0$ to  $\Lambda^{k+k'}\HH^0$.

\subsection{Pullback}\label{se:pullback}
A basic construction in differential geometry is that of pullback bundle and in this section we show that this has an analogous in the context of $L^p(\mm)$-normed modules. The concept that such definition grasps is the following. Let $(\X_i,\mathcal A_i,\mm_i)$, $i=1,2$, be measured spaces, pretend that $\MM$ is an $L^p(\mm_1)$-normed module  given by a family of Banach spaces $(B_x,\|\cdot\|_x)_{x\in \X_1}$ as informally described in Example \ref{ex:lpmod} and let  $\varphi:\X_2\to \X_1$ measurable with $\varphi_*\mm_2\leq C\mm_1$ for some $C>0$. Then the pullback module $\varphi^*\MM$ would be the $L^p(\mm_2)$-normed module made of maps which assign to $\mm_2$-a.e.\ point $x_2\in \X_2$ an element $v(x_2)\in B_{T(x_2)}$ in such a way that $\| |v|_{T(x_2)}\|_{L^p(\mm_2)}<\infty$.

We shall concentrate on the case $p<\infty$.

\vspace{1cm}

The rigorous definition comes via explicit construction.  We start with the following definition:
\begin{definition}[Map of bounded compression]\label{def:bcompr}
Let $(\X_1,\mathcal A_1,\mm_1)$ and  $(\X_2,\mathcal A_2,\mm_2)$ be two $\sigma$-finite measured spaces as discussed in Section \ref{se:assnot}. A map of bounded compression $\varphi:\X_2\to \X_1$ is (the equivalence class w.r.t.\ equality $\mm_2$-a.e.\ of) a measurable map $ \varphi:\X_2\to \X_1$ such that 
\[
 \varphi_*\mm_2\leq C\mm_1,
\]
for some $C\geq 0$.
\end{definition}
Now fix  $\sigma$-finite measured spaces $(\X_1,\mathcal A_1,\mm_1)$ and  $(\X_2,\mathcal A_2,\mm_2)$, a map $\varphi:\X_2\to \X_1$ of bounded compression and an $L^p(\mm_1)$-normed module $\MM$, $p\in[1,\infty)$.

Define the `pre-pullback' set ${\rm Ppb}$ as
\[
\begin{split}
{\rm Ppb}:=\Big\{\{(v_i,A_i)\}_{i\in\N}\ :&\  (A_i)_{i\in\N}\text{ is a partition of $\X_2$,} \\
&\ v_i\in \MM \ \forall i\in\N, \text{ and }\Big\|\sum_{i\in\N}\nchi_{A_i}|v_i|\circ \varphi\Big\|_{L^p(\mm_2)}<\infty \Big\},
\end{split}
\]
and an equivalence relation on ${\rm Ppb}$ by declaring $\{(v_i,A_i)\}_{i\in\N}\sim \{(w_j,B_j)\}_{j\in\N}$ provided
\[
|v_i-w_j|\circ\varphi=0,\quad\mm_2\ae \text{ on } \{A_i\cap B_j\},\qquad\forall i,j\in\N.
\]
It is readily verified that this is indeed an equivalence relation and we shall denote by $[(v_i,A_i)_i]$ the equivalence class of $\{(v_i,A_i)\}_{i\in\N}$ (with respect to the informal and non-rigorous presentation given before, we shall think of $[(v_i,A_i)_i]$ as the element of the pullback module which  takes the value $v_i(\varphi(x))$ on $x\in A_i$). 

Define the sum of elements of ${\rm Ppb}/\sim$    as
\[
[(v_i,A_i)_i]+[(w_j,B_j)_{j}]:=[(v_i+w_j,A_i\cap B_j)_{i,j}],
\]
and the multiplication  with a scalar $\lambda\in\R$ as
\[
\lambda[(v_i,A_i)\}_i]:=[(\lambda v_i,A_i)_i].
\]
It is clear that these operations are well defined and endow ${\rm Ppb}/\sim$ with a vector space structure.

Recalling that   ${\rm Sf}(\mm_2)\subset L^\infty(\mm_2)$ is the space of simple functions, i.e.\ those attaining only a finite number of values, for  $[(v_i,A_i)_i]\in {\rm Ppb}/\sim$ and $g=\sum_{j}a_j\nchi_{B_j}\in {\rm Sf}(\mm_2)$ with $(B_j)$ finite partition of $\X_2$, we define the product $g [(v_i,A_i)_i]\in {\rm Ppb}/\sim$ as
\[
g[(v_i,A_i)_i]:=[(a_jv_i,A_i\cap B_j)_{i,j}].
\]
It is readily verified that this definition is well posed, that the resulting multiplication is a bilinear map from ${\rm Sf}(\mm_2)\times {\rm Ppb}/\sim$ into ${\rm Ppb}/\sim$ and  that ${\bf 1}[(f_i,A_i)_i]=[(f_i,A_i)_i]$. Finally, we consider the map $|\cdot|: {\rm Ppb}/\sim\,\to L^p(\mm_2)$ given by
\[
\big|[(v_i,A_i)_i]\big|:=\sum_{i\in\N}\nchi_{A_i}|v_i|\circ \varphi.
\]
As direct consequences of the definitions we get the (in)equalities 
\begin{equation}
\label{eq:basepullback}
\begin{split}
\big|[(v_i+w_j,A_i\cap B_j)_{i,j}]\big| &\leq \big|[((v_i,A_i)_i]\big|+\big|[(w_j,B_j)_j]\big|,\\
\big|\lambda[(v_i,A_i)_i]\big|&=|\lambda|\,\big|[(v_i,A_i)_i]\big|,\\
\big|g[(v_i,A_i)_i]\big|&=|g|\,\big|[(v_i,A_i)_i]\big|,
\end{split}
\end{equation}
valid $\mm_2$-a.e.\ for every $[((v_i,A_i)_i], [(w_j,B_j)_j]\in  {\rm Ppb}/\sim$, $\lambda\in\R$ and $g\in {\rm Sf}(\mm_2)$,  so that in particular the map $\|\cdot\|:{\rm Ppb}/\sim\to[0,\infty)$ defined by
\begin{equation}
\label{eq:normpull}
\big\|[(v_i,A_i)_i]\big\|:=\big\|\big|[(v_i,A_i)_i]\big|\big\|_{L^p(\mm_2)}=\Big\|\sum_{i\in\N}\nchi_{A_i}|v_i|\circ \varphi\Big\|_{L^p(\mm_2)}
\end{equation}
is a norm on ${\rm Ppb}/\sim$.

\begin{definition}[The pullback module $\varphi^*\MM$]
With the above notation and assumptions, the space  $\varphi^*\MM$ is defined as the completion of $({\rm Ppb}/\sim,\|\cdot\|)$.
\end{definition}
With this definition, a priori $\varphi^*\MM$ is only a Banach space, but in fact it comes with a canonical structure of $L^p(\mm_2)$-normed module.  Indeed, from the last in \eqref{eq:basepullback} and the definition \eqref{eq:normpull} we see that for any $g\in {\rm Sf}(\mm_2)$ and $[(v_i,A_i)_i]\in {\rm Ppb}/\sim$ we have
\[
\|g[(v_i,A_i)_i]\|\leq \|g\|_{L^\infty}\|[(v_i,A_i)_i]\|,
\]
and thus, by the density of ${\rm Sf}(\mm_2)$ in $L^\infty(\mm_2)$, such multiplication can be uniquely extended to a bilinear continuous map from $L^\infty(\mm_2)\times \varphi^*\MM$ to $\varphi^*\MM$ which gives the  structure   of $L^\infty(\mm_2)$ premodule. Moreover, from the first in \eqref{eq:basepullback} we also see that for a Cauchy sequence $n\mapsto [(v_{n,i},A_{n,i})]\in  {\rm Ppb}/\sim$ the sequence $n\mapsto |[(v_{n,i},A_{n,i})]|$ is $L^p(\mm_2)$-Cauchy. The limit of such sequence defines a map $|\cdot|:\varphi^*\MM\to L^p(\mm_2)$ which, passing to the limit in \eqref{eq:basepullback} and \eqref{eq:normpull}, is shown to be a pointwise norm.

With this structure and recalling Proposition \ref{prop:baselp} we see that $\varphi^*\MM$ is an $L^p(\mm_2)$-normed module for $p<\infty$, as claimed. 

\bigskip

The pullback map $\varphi^*:\MM\to \varphi^*\MM$ is defined as 
\begin{equation}
\label{eq:pullbackv}
\varphi^*v:=[(v,\X_2)],\qquad\forall v\in \MM,
\end{equation}
where $(v,\X_2)\in {\rm Ppb}$ is a shorthand for $(v_i,A_i)_{i\in\N}$ with $v_0=v$, $A_0=\X_2$ and $v_i=0$, $A_i=\emptyset$ for $i>0$. 
Notice that
\begin{equation}
\label{eq:pullbackv2}
\begin{split}
\varphi^*(fv)&=f\circ \varphi\,\varphi^*v,\\
|\varphi^*v|&=|v|\circ\varphi,
\end{split}
\end{equation}
$\mm_2$-a.e.\ for every $v\in\MM$ and $f\in L^\infty(\mm_1)$. Indeed, the second is a direct consequence of the definition of pointwise norm on $\varphi^*\MM$, while the first  can be verified directly from the definition \eqref{eq:pullbackv} and the one of equivalence relation on ${\rm Ppb}$ for $f=\nchi_A$, $A\in \BBB{\X_1}$, then by linearity it follows for simple functions $f$ and finally the general case is achieved by approximation.

We remark that directly from property \eqref{eq:pointquad} and the definition of pointwise norm on the pullback, we have that
\[
\text{the pullback of an Hilbert module is an Hilbert module.}
\]

The simplest case of the pullback construction is when the $L^p(\mm_1)$-normed module $\MM$ is precisely the space of functions $L^p(\mm_1)$: in this case we can identify $\varphi^*\MM$ with $L^p(\mm_2)$ via the injection $[(f_i,A_i)_i]\mapsto\sum_i \nchi_{A_i}f_i\circ \varphi$ and formulas \eqref{eq:pullbackv}, \eqref{eq:pullbackv2} read as
\[
\varphi^*f=f\circ \varphi,\qquad\forall f\in L^p(\mm_1).
\]
Notice that directly from the definition of $\varphi^*\MM$ (recall also the property \eqref{eq:dalp}) we get that 
\begin{equation}
\label{eq:genpullback}
\varphi^*\MM\text{ is generated, in the sense of modules, by the vector space $\{\varphi^*v:v\in\MM\}$}.
\end{equation}

It is worth underlying that the measure $\mm_1$ on $\X_1$ does not really play a key role in the definition of pullback module besides the compatibility requirement $\varphi_*\mm_2\leq C\mm_1$. To see why, consider another measure $\mm_1'$  on $\X_1$ such that  $\varphi_*\mm_2\leq C'\mm_1'\ll\mm_1$ for some $C'\in\R$ and recall the construction of the $L^p(\mm_1')$-normed module $\MM_{p,\mm_1'}$ given at the end of Section \ref{se:altint}. Then directly by the definitions it is easy to establish that
\begin{equation}
\label{eq:invarianzamisurapb}
\varphi^*\MM=\varphi^*\MM_{p,\mm_1'}.
\end{equation}
In particular, the choice $\mm_1':=\varphi_*\mm_2$ is always possible. 

Another direct consequence of the definitions is that if $(\X_3,\mathcal A_3,\mm_3)$ is another measured space and $\psi:\X_3\to \X_2$ is of bounded compression, then 
\begin{equation}
\label{eq:tstarfunt}
\psi^*(\varphi^*\MM)\sim(\psi\circ \varphi)^*\MM,
\end{equation}
the identification being given by $[([(v_{i,j},A_{i,j})_i], B_j)_j]\to [(v_{i,j},\psi^{-1}(A_{i,j})\cap B_j)_{i,j}]$ for $v_{i,j}\in \MM$, $(A_{i,j})_i$ partition of $\X_2$ for every $j$ and $(B_j)$ partition of $\X_3$. Having done this identification we also have
\[
\psi^*\varphi^*v=(\psi\circ \varphi)^*v\qquad\forall v\in \MM.
\]
The pullback module has the following universal property:
\begin{proposition}[Universal property of the pullback]\label{prop:univpullback}
Let $(\X_i,\mathcal A_i,\mm_i)$, $i=1,2,3$ be three measured spaces, $p\in[1,\infty)$, $\MM$ an $L^p(\mm_1)$-normed module and $\varphi:\X_2\to \X_1$  and $\psi:\X_3\to \X_2$  of bounded compression.

Also, let $\NN$ be an $L^p(\mm_3)$-normed module  and $T:\MM\to \NN$ be a linear map such that
\begin{equation}
\label{eq:ipperuniv}
\begin{array}{rll}
T(fv)&\!\!\!=f\circ \varphi\circ \psi\, T(v)\qquad&\forall f\in L^\infty(\mm_1),\ v\in \MM,\\
|T(v)|&\!\!\!\leq C|v|\circ \varphi\circ \psi,\qquad\mm_3\ae\qquad&\forall v\in\MM,
\end{array}
\end{equation}
for some constant $C\geq 0$, so that in particular $T$ is continuous.

Then there exists a unique linear  map $S:\varphi^*\MM\to \NN$ such that
\begin{equation}
\label{eq:univpullback}
\begin{array}{rll}
S(gw)&\!\!\!=g\circ \psi\, S(w)\qquad\quad\,\ &\forall g\in L^\infty(\mm_2),\ w\in \varphi^*\MM,\\
S(\varphi^*v)&\!\!\!=T(v) \qquad \qquad \qquad&\forall v\in \MM,\\
|S(w)|&\!\!\!\leq C'|w|\circ\psi,\quad\mm_3\ae\qquad &\forall w\in \varphi^*\MM,
\end{array}
\end{equation}
for some $C'\geq 0$. In this case, the best constant $C'$ coincides with the best constant $C$ in \eqref{eq:ipperuniv}.
\end{proposition}
\begin{proof} Uniqueness comes by the requirements \eqref{eq:univpullback} recalling property \eqref{eq:genpullback}.

For existence, we denote by $V\subset\varphi^*\MM$ the vector space generated by  elements of the form $\nchi_{B} \varphi^*v$ for  $v\in\MM$ and $B\in \BBB{\X_2}$.  Notice that each $w\in V$ can be written as   $w=\sum_{i=1}^n\nchi_{B_i} \varphi^*v_i$ with the $B_i$'s disjoint and for such way of expressing $w$ declare
\[
S(w):=\sum_i\nchi_{\psi^{-1}(B_i)} \,T(v_i).
\]
Notice that
\[
|S(w)|=\sum_i\nchi_{\psi^{-1}(B_i)} \,|T(v_i)|\leq C\Big(\sum_i\nchi_{B_i}|v|\circ\varphi\Big)\circ\psi =C|w|\circ\psi,\qquad\mm_3\ae,
\]
which shows in particular that the definition of $S$ is well posed, i.e.\ $S(w)$ depends only on $w$ and not on the particular way of writing it. It is now clear that $S$ is linear and continuos and thus can be extended to a continuous map form the closure of $V$, which by \eqref{eq:genpullback} and \eqref{eq:dalp} is easily seen to be the whole $\varphi^*\MM$, to $\NN$. The second and third properties in \eqref{eq:univpullback} then follow by construction. In turn, the first is obtained for simple functions $g$ by the very definition of $S$ and then the general case by approximation. The last claim is now obvious.
\end{proof}
Notice that if one replaces  the last requirements in \eqref{eq:ipperuniv} and \eqref{eq:univpullback} by asking only for the continuity of the maps involved, the resulting statement is false.

\begin{remark}[The construction from the point  of view of category theory]\label{rem:catpullback}{\rm
Not surprisingly, the above universal property fits in the framework of category theory, but given that two different structures appear in the statement, that of measured space and that of $L^p$-normed module, it is possibly worth to spend two words on how to make such interpretation.

Thus we introduce the category ${\bf Meas}$, say, whose objects are $\sigma$-finite measured spaces $(\X,\mathcal A,\mm)$ as in Section \ref{se:assnot} and whose morphisms are maps of bounded deformation. Then for $p\in[1,\infty)$ we consider also the category ${\bf Mod}_{p-L^\infty}$ whose objects are couples $(\X,\MM)$ where $\X=(\X,\mathcal A,\mm)$ is an object of ${\bf Meas}$ and $\MM$ is an $L^p(\mm)$-normed module, and where  a morphism from $(\X_2,\MM_2)$ to $(\X_1,\MM_1)$ is a couple $(\varphi,\Phi)$ with $\varphi:\X_2\to \X_1$ of bounded  compression and $\Phi:\MM_1\to \MM_2$ a linear  map satisfying
\[
\begin{array}{rll}
\Phi(fv)&\!\!\!=f\circ\varphi\,\Phi(v)\qquad&\forall v\in \MM_1,\ f\in L^\infty(\mm_1),\\
|\Phi(v)|&\!\!\!\leq C\,|v|\circ\varphi\qquad\mm_2\ae\qquad&\forall v\in\MM_1,
\end{array}
\]
for some $C\geq 0$.

The map associating to the measured space $(\X,\mathcal A,\mm)$ the couple $(\X,0)$ and to the map of bounded compression $\varphi$ the couple $(\varphi,0)$ is then a fully faithful functor injective on objects (see e.g.\ \cite{MacLane98}). In other words, we can realize ${\bf Meas}$ as full subcategory of ${\bf Mod}_{p-L^\infty}$. Moreover, given an object $(\X,\MM)$ in ${\bf Mod}_{p-L^\infty}$ we have  a canonical morphism from $(\X,\MM)$ to $(\X,0)$, namely $(\Id,0)$.

After this identification, Proposition \ref{prop:univpullback} tells exactly that the pullback module is the pullback in the category ${\bf Mod}_{p-L^\infty}$. Notice indeed that property \eqref{eq:pullbackv2} tells that   $(\varphi,\varphi^*)$ is a morphism from $(\X_2,\varphi^*\MM)$ to $(\X_1,\MM)$ and it is clear that the diagram
\begin{center}
\begin{tikzpicture}[node distance=2.5cm, auto]
  \node (A) {$(\X_2,\varphi^*\MM)$};
  \node (B) [right of=A] {$(\X_2,0)$};
  \node (C) [below  of=A] {$(\X_1,\MM)$};
  \node (D) [below  of=B] {$(\X_1,0)$};
  \draw[->] (A) to node {$(\Id,0)$} (B);
  \draw[->] (A) to node [swap] {$(\varphi,\varphi^*)$} (C);
  \draw[->] (B) to node  {$(\varphi,0)$} (D);
  \draw[->] (C) to node  {$(\Id,0)$} (D);
\end{tikzpicture}
\end{center}
commutes. Then notice that if  $(\X_3,\NN)$ is an object in ${\bf Mod}_{p-L^\infty}$ and $(\psi_1,\Psi_1):(\X_3,\NN)\to (\X_2,0)$ and $(\psi_2,\Psi_2):(\X_3,\NN)\to (\X_1,\MM)$ are morphisms such that $(\psi_1,\Psi_1)\circ(\varphi,0)=(\psi_2,\Psi_2)\circ (\Id,0)$, we must have $\psi_2=\psi_1\circ\varphi$. In this case,  Proposition \ref{prop:univpullback} tells that there is a unique morphism $(\psi_3,\Psi_3):(\X_3,\NN)\to (\X_1,\varphi^*\MM)$ such that the diagram
\begin{center}
\begin{tikzpicture}[node distance=2.5cm, auto]
  \node (A) {$(\X_2,\varphi^*\MM)$};
  \node (B) [right of=A] {$(\X_2,0)$};
  \node (C) [below of=A] {$(\X_1,\MM)$};
  \node (D) [below of=B] {$(\X_1,0)$};
  \node (A1) [node distance=2cm, left of=A, above of=A] {$(\X_3,\NN)$};
  \draw[->] (A) to node {$(\Id,0)$} (B);
  \draw[->] (A) to node [swap] {$(\varphi,\varphi^*)$} (C);
  \draw[->] (C) to node [swap] {$(\Id,0)$} (D);
  \draw[->] (B) to node {$(\varphi,0)$} (D);
  \draw[->, bend right] (A1) to node [swap] {$(\psi_1\circ\varphi,\Psi_2)$} (C);
  \draw[->, bend left] (A1) to node {$(\psi_1.\Psi_1)$} (B);
  \draw[->, dashed] (A1) to node {$(\psi_3,\Psi_3)$} (A);
\end{tikzpicture}
\end{center}
commutes.
}\fr\end{remark}

Consider now the question of understanding who is the dual of the pullback module $\varphi^*\MM$, which we shall discuss only for $p\in(1,\infty)$. 

We claim that such dual always contains a copy of the pullback $\varphi^*\MM^*$ of the dual module $\MM^*$ and to this aim we define an isometric morphism  ${\sf I}_{\varphi}:\varphi^*(\MM^*)\to (\varphi^*\MM)^*$ by declaring that for $v\in \MM$ and $L\in \MM^*$ we have
\begin{equation}
\label{eq:perdefsci}
{\sf I}_{\varphi}(\varphi^*L)(\varphi^*v):=L(v)\circ \varphi,
\end{equation}
and extending the map by $L^\infty(\mm_2)$-linearity and continuity. More in detail, for given $L\in\MM^*$ we can use the trivial inequality
\[
|L(v)\circ \varphi|\leq |L|_*\circ\varphi\,|v|\circ\varphi=|\varphi^*L|_{\varphi^*\MM^*}\,|\varphi^*v|_{\varphi^*\MM},\qquad\mm_2\ae\qquad\forall v\in\MM,
\]
property \eqref{eq:genpullback} and Proposition \ref{prop:extension} applied to the $L^p(\mm_2)$-normed module $\varphi^*\MM$ to deduce that there exists a unique element of $(\varphi^*\MM)^*$, which we shall call ${\sf I}_{\varphi}(\varphi^*L)$, for which \eqref{eq:perdefsci} holds for every $v\in\MM$ and such ${\sf I}_{\varphi}(\varphi^*L)$ also fulfills
\begin{equation}
\label{eq:candydiv}
|{\sf I}_{\varphi}(\varphi^*L)|_{(\varphi^*\MM)^*}\leq |\varphi^*L|_{\varphi^*\MM^*},\qquad\mm_2\ae.
\end{equation}
We thus built a map ${\sf I}_{\varphi}:\{\varphi^*L:L\in\MM^*\}\to (\varphi^*\MM)^*$ and using property \eqref{eq:genpullback} for the $L^q(\mm_2)$-normed module $\varphi^*\MM^*$, where $q\in(1,\infty)$ is such that $\frac1p+\frac1q=1$, and arguments similar to that used for Proposition \ref{prop:extension}, it is easy to see that the bound \eqref{eq:candydiv} grants that ${\sf I}_{\varphi}$ can be uniquely extended to a module morphism, still denoted by ${\sf I}_{\varphi}$, from $\varphi^*\MM^*$ to $(\varphi^*\MM)^*$ which satisfies
\[
|{\sf I}_{\varphi}(\tilde L)|_{(\varphi^*\MM)^*}\leq |\tilde L|_{\varphi^*\MM^*},\qquad\mm_2\ae,\qquad\forall \tilde L\in\varphi^*\MM^*.
\]
Moreover, it is easy to see that equality holds in this last inequality. Indeed, by considering $\tilde L$ of the form $[(L_i,B_i)]$, choosing for every $i$ a maximizing sequence $(v_{i,n})\subset M$  in the definition \eqref{eq:normdual} of the dual norm $|L_i|_{\MM*}$ and then testing ${\sf I}_{\varphi}( \tilde L)$ on the elements $[(v_{i,n},B_i)_i]$,   we obtain the equality 
\[
|{\sf I}_{\varphi}(\tilde L)|_{(\varphi^*\MM)^*}=|\tilde L|_{\varphi^*\MM^*},\qquad\mm_2\ae.
\]
Then the same identity for general elements of $\varphi^*\MM^*$ follows by density. 

Therefore, as claimed, \eqref{eq:perdefsci} defines an isometric morphism  ${\sf I}_{\varphi}:\varphi^*(\MM^*)\to (\varphi^*\MM)^*$ .

\bigskip

The question is whether ${\sf I}_{\varphi}$ is surjective or not. Shortly said,  the general answer is negative and a sufficient condition for it to be positive is that $\MM^*$ is separable and the measurable structures $(\X_i,\mathcal A_i)$ involved are Borel structures of complete separable metric spaces. 

In order to understand the situation we remark  that there is a particular case where the construction of the pullback module reduces, from the perspective of Banach spaces, to a well known one: if $\X_2=\X_1\times[0,1]$ is endowed with the product of the measures  $\mm_1$ and $\mathcal L^1\restr{[0,1]}$, and the map $\pi:\X_2\to \X_1$ is the canonical projection, then the pullback $\pi^*\MM$ of an $L^p(\mm_1)$-normed module $\MM$ can be canonically identified, when seen as Banach space, with the space $L^p([0,1],\MM)$. To see why, recall that for any given Banach space $B$, the space $L^p([0,1],B)$ consists of all (equivalence classes w.r.t.\ $\mathcal L^1$-a.e.\ equality of Borel) functions $f:[0,1]\to B$ such that $\|f\|^p_{L^p([0,1],B)}:=\int_0^1\|f(t)\|_B^p\,\d t <\infty$ and that when endowed with the norm $ \|\cdot\|_{L^p([0,1],B)}$ this is a Banach space itself. 

Now observe that an argument based on subsequent subdivisions shows that the set $D\subset \pi^*\MM$ made of elements of the form $\sum_i \nchi_{\X_1\times A_i}\pi^*v_i$ for some $v_i\in \MM$, $i\in\N$, and partition $(A_i)$ of $[0,1]$ is dense in $\pi^*\MM$ and that the map $\iota_\MM:D\to L^p([0,1],\MM)$ defined as
\[
\iota_\MM\Big(\sum_{i\in\N}\nchi_{\X_2\times A_i}\pi^*v_i\Big):=\sum_{i\in\N}\nchi_{A_i}v_i,
\]
is a well-defined linear isometry whose image is dense in $L^p([0,1],\MM)$. Thus it can, and will, be extended to a bijective isometry from $\pi^*\MM$ to $L^p([0,1],\MM)$, still denoted by $\iota_\MM$.

Therefore, recalling that from Proposition \ref{prop:fulllp} we know that $L^p$-normed modules have full dual, we see that in this case asking for   the image of ${\sf I}_{\pi}$ to be the whole $(\pi^*\MM)^*$ is equivalent to ask  the  dual of $L^p([0,1],\MM)$ to be (identifiable with) the space $L^q([0,1],\MM^*)$, which in turn is well known to be equivalent to the Radon-Nikodym property of $\MM^*$. We refer to \cite{DiestelUhl77} for the definition of the Radon-Nikodym property of a Banach space and all the relevant related results which we will use.

Having clarified this, we are going to proceed as follows:
\begin{itemize}
\item[i)] In Proposition \ref{prop:casoprodotto} we shall precisely state in which sense the Banach dual of $\pi^*\MM$ can be identified with $\pi^*\MM^*$ under the assumption that $\MM^*$ is separable, which is sufficient to grant the Radon-Nikodym property and general enough for our purposes. The proof is based just on keeping track of the various identifications we built so far. 
\item[ii)] In the technical Lemma \ref{le:tecnicomisure} we shall present a measure-theoretic construction that will allow to reduce the study of general pullbacks to the base case $\pi^*\MM$. Notice that in order to be able to construct the relevant object, we will need some `universal' constructions in measure theory like  the isomorphism theorem for measure rings, the disintegration theorem and the Kuratowski-Ryll Nardzewski Borel selection theorem. In turn, all of these require some additional structure, of topological nature, to that of general measured spaces. We will work with complete separable metric spaces, as these are the structures we shall deal with from the next chapter on.
\item[iii)] Finally, in Theorem \ref{thm:dualpullback} we will show how such  technical lemma allows to prove,  under the stated topological assumptions, that the pullback module $\varphi^*\MM$ can be seen as a sort of quotient of $\pi^*\MM$. This fact coupled with the analysis of $\pi^*\MM$ will allow us to conclude that if $\MM^*$ is separable, then indeed $\varphi^*\MM^*\sim (\varphi^*\MM)^*$ the identification being given by the map ${\sf I}_{\varphi}$.
\end{itemize}
We thus start with the following statement:
\begin{proposition}[The Banach dual of $\pi^*\MM$]\label{prop:casoprodotto}
Let $(\X,\mathcal A,\mm)$ be a $\sigma$-finite measured space as in Section \ref{se:assnot}  and  $\MM$ an $L^p(\mm)$-normed  module, $p\in(1,\infty)$. Consider the pullback module $\pi^*\MM$, where $\pi:\X\times[0,1]\to \X$ is the canonical projection and $\X\times [0,1]$ is endowed with the product measure $\bar\mm:=\mm\times\mathcal L^1\restr{[0,1]}$.

Assume that $\MM^*$ is separable. Then for every linear continuous functional $l:\pi^*\MM\to\R$ there exists a unique $L\in \pi^*\MM^*$ such that 
\[
l(v)=\int {\sf I}_{\pi} (L)(v)\,\d\bar\mm,\qquad\forall v\in\pi^*\MM.
\]
\end{proposition}
\begin{proof} Uniqueness is obvious, so we focus on existence. We start observing that for arbitrary $\j\in L^q([0,1],\MM^*)$ and $v\in \pi^*\MM$ the identity
\begin{equation}
\label{eq:brutta}
{\sf I}_{\pi}(\iota_{\MM^*}^{-1}\j)(v)(x,t)=\j_t(\iota_\MM(v)_t)(x)
\end{equation}
holds for $\bar\mm$-a.e.\ $(x,t)$. This can be checked directly from the definitions for  $\j$ step function (i.e.\ attaining only a finite number of values)  and  $v=[(v_i,\X\times A_i)_i]$, $(A_i)$ being a partition of $[0,1]$,  then the general case follows by density.

Now notice that the map $l\circ\iota_\MM^{-1}:L^p([0,1],\MM)\to \R$ is linear and continuous, i.e.\ an element of the dual $(L^p([0,1],\MM))'$ of $L^p([0,1],\MM)$. By Proposition \ref{prop:fulllp} we know that the Banach dual $\MM'$ of $\MM$ can be canonically identified with the module dual $\MM^*$ in the sense that $\intmap_{\MM^*}:\MM^*\to\MM'$ is a surjective isometry. In particular, being $\MM^*$ separable, so is $\MM'$ and therefore by the Dunford-Pettis theorem (see Theorem 1 in Section 3 of Chapter 3 in \cite{DiestelUhl77})   $\MM'$ has the Radon-Nikodym property. It follows   that $(L^p([0,1],\MM))'\sim L^q([0,1],\MM')$, where $\frac1p+\frac1q=1$,  (see Theorem 1 in Section 1 of Chapter 4 in \cite{DiestelUhl77} and notice that the hypothesis that the space has finite measure can be easily replaced with `the space is $\sigma$-finite') and thus there exists $\j\in L^q([0,1],\MM')$ such that
\[
\int_0^1\j_t(w_t)\,\d t=l\circ\iota_\MM^{-1}(w),\qquad\forall w\in L^p([0,1],\MM).
\]
Using again the fact that $\intmap_{\MM^*}:{\MM^*}\to\MM'$ is surjective we see that there is $\tilde \j\in L^q([0,1],\MM^*)$ such that
\[
\int_0^1 \int\tilde\j_t (w_t)\,\d\mm\,\d t =l\circ\iota_M^{-1}(w),\qquad\forall w\in L^p([0,1],\MM).
\]
It is then clear from \eqref{eq:brutta} that $L:=\iota_{\MM^*}^{-1}(\tilde\j)\in \pi^*\MM^*$  fulfills the requirements.
\end{proof}
We turn to the technical lemma:
\begin{lemma}\label{le:tecnicomisure}
Let $(\X_1,\sfd_1)$, $(\X_2,\sfd_2)$ be  complete separable metric spaces, $\mm_2$ a Borel measure on $\X_2$ and $\varphi:\X_2\to \X_1$ a Borel map such that $\varphi_*\mm_2$ is $\sigma$-finite. 

Then there exists a Borel map $\psi:\X_1\times[0,1]\to \X_2$ such that $\psi_*(\varphi_*\mm_2\times\mathcal L^1\restr{[0,1]})=\mm_2$ and the identity $\varphi\circ \psi=\pi$ holds $\varphi_*\mm_2\times\mathcal L^1\restr{[0,1]}$-a.e., where $\pi:\X_1\times[0,1]\to \X_1$ is the canonical projection.
\end{lemma}
\begin{proof} Without loss of generality and up to replace $\sfd_2$ with $\sfd_2\wedge 1$ we can, and will, assume that $\sfd_2$ is a bounded distance so that the distance $W_2$ metrizes the narrow convergence on the space $\prob{\X_2}=\probt{\X_2}$ (see for instance Chapter 2 of \cite{AmbrosioGigli11} or Chapter 6 of \cite{Villani09} for the relevant definitions and properties). 

Let $Y$ be the space of (equivalence classes w.r.t.\ $\mathcal L^1$-a.e.\ equality of Borel) maps from $[0,1]$ to $\X_2$ endowed with the distance $\sfd_Y(f,g):=\sqrt{\int_{[0,1]}\sfd^2_2(f(t),g(t))\,\d\mathcal L^1(t)}$. Notice that $(Y,\sfd_Y)$ is complete and separable and consider the push-forward map ${\sf Pf}:Y\to \prob{\X_2}$ given by $f\mapsto {\sf Pf}(f):=f_*\mathcal L^1\restr{[0,1]}$. We claim that
\begin{itemize}
\item[i)] For every $\mu\in\prob{\X_1}$ we have ${\sf Pf}^{-1}(\mu)\neq\emptyset$,
\item[ii)] For every $\mu,\nu\in\prob{\X_1}$, $f\in{\sf Pf}^{-1}(\mu)$ and $\eps>0$ there exists $g\in{\sf Pf}^{-1}(\nu)$ such that $\sfd_Y(f,g)\leq W_2(\mu,\nu)+\eps$.
\end{itemize}
Indeed, $(i)$ follows by the classical isomorphism theorem for measure rings, see for instance  Theorem 9 in Chapter 15 of the book \cite{Royden88}. For $(ii)$ assume for a moment that $\mu$ has no atoms and recall that by a result of Pratelli (\cite{Pratelli07}) in this case for every $\eps>0$ there is a Borel map $T:\X_2\to \X_2$ such that $T_*\mu=\nu$ and  $\sqrt{\int\sfd_1(x, T(x))^2\,\d\mu(x)}\leq W_2(\mu,\nu)+\eps$. Then define $g:= T\circ f$ to conclude. If $\mu$ has atoms, replace $\X_2$, and $[0,1]$ by $\X_2\times[0,1]$ and $[0,1]^2$, the measures  $\mu,\nu$ by $\mu\times\mathcal L^1\restr{[0,1]},\nu\times\mathcal L^1\restr{[0,1]}$ and $f$ by $\tilde f(t,s):=(f(t),s)$:  apply the previous argument  and then project everything neglecting the additional factor $[0,1]$ to conclude.

Thus for any open set $\Omega\subset Y$ the set $\{\mu\ :\ {\sf Pf}^{-1}(\mu)\in\Omega\}$ is non-empty by $(i)$ and open - and a fortiori Borel - in $(\prob{\X_2},W_2)$ by $(ii)$. By the Kuratowski-Ryll Nardzewski Borel selection theorem (see for instance Theorem 6.9.3 in \cite{Bogachev07}) we deduce the existence of a Borel map $S:\prob{\X_2}\to Y$ such that 
\begin{equation}
\label{eq:pfr}
{\sf Pf}(S(\mu))=\mu,\qquad\text{ for every }\mu\in\prob{\X_1}.
\end{equation}
Now we apply the disintegration theorem (see for instance Section III-70 of \cite{DellacherieMeyer78} and notice that the assumption that $\mm_2$ is a probability measure can be easily removed using a patching argument based on the fact that $\varphi_*\mm_2$ is $\sigma$-finite) to deduce the existence of a map $\X_1\ni x_1\mapsto \mm_{2,x_1}\in\prob {\X_2}$ (which is unique up to $\varphi_*\mm_2$-a.e.\ equality) such that 
\begin{itemize}
\item[a)] $\mm_{2,x_1}(\X_2\setminus \varphi^{-1}(\{x_1\}))=0$ for $\varphi_*\mm_2$-a.e.\ $x_1\in \X_1$,
\item[b)] the map $x_1\mapsto \mm_{2,x_1}(E)$ is Borel for every $E\subset \X_2$ Borel,
\item[c)] for every $\bar f\in C_b(\X_2)$ we have $\int \bar f\,\d\mm_2=\int \big(\int \bar f\,\d\mm_{2,x_1}\big)\,\d \varphi_*\mm_2(x_1)$.
\end{itemize}
From $(b)$ it follows that the map $x_1\mapsto \mm_{2,x_1}$ is Borel when we endow the target space with the distance $W_2$ and therefore the map $\psi:\X_1\times[0,1]\to \X_2$ given by
\[
\psi(x_1,t):=S(\mm_{2,x_1})(t),\qquad\forall x_1\in \X_1,\ t\in[0,1],
\]
is Borel. Properties $(c)$ and \eqref{eq:pfr} ensure that $\psi_*(\varphi_*\mm_2\times\mathcal L^1\restr{[0,1]})=\mm_2$, while $(a)$ grants that $\varphi_*\mm_2\times\mathcal L^1\restr{[0,1]}$-a.e.\ it holds $\varphi\circ \psi=\pi$, thus the proof is achieved.
\end{proof}
We now have all the tools we need to prove our main result concerning the dual of $\varphi^*\MM$:
\begin{theorem}[The module dual of $\varphi^*\MM$]\label{thm:dualpullback}
Let $(\X_1,\sfd_1)$, $(\X_2,\sfd_2)$ be two complete and separable metric spaces, $\mm_1$, $\mm_2$ $\sigma$-finite Borel measures an $\X_1$ and $\X_2$ respectively and $\varphi:\X_2\to \X_1$ a Borel  map such that $\varphi_*\mm_2\leq C\mm_1$ for some $C\geq 0$. 

Furthermore, let  $\MM$ be an $L^p(\mm_1)$-normed module, $p\in(1,\infty)$, and assume that $\MM^*$ is separable. 

Then the isometric morphism ${\sf I}_{\varphi} :\varphi^*\MM^*\to (\varphi^*\MM)^*$ defined by \eqref{eq:perdefsci} is surjective.
\end{theorem}
\begin{proof}
By the discussion made at the end of Section \ref{se:altint}, we can consider the module $\MM_{p,\varphi_*\mm_2}$: by \eqref{eq:dualmmp}  we know that its dual can be canonically identified with $\MM^*_{q,\varphi_*\mm_2}$, where $q\in(1,\infty)$ is such that $\frac1p+\frac1q=1$. Then the assumption that $\MM^*$ is separable yields (recall \eqref{eq:sepmmp}) that $\MM^*_{q,\varphi_*\mm_2}$ is separable as well and thus up to replace  $\mm_1$ with $\varphi_*\mm_2$ and $\MM$ with $\MM_{p,\varphi_*\mm_2}$ and recalling property \eqref{eq:invarianzamisurapb} we can, and will, assume that $\mm_1=\varphi_*\mm_2$. We also  put for brevity $\bar \mm:=\mm_1\times\mathcal L^1\restr{[0,1]}$.

Let $\psi:\X_1\times[0,1]\to \X_2$ be the Borel map given by Lemma \ref{le:tecnicomisure} above and notice that by the identification \eqref{eq:tstarfunt} we have $\psi^*(\varphi^*\MM)\sim\pi^*\MM$ and $\psi^*(\varphi^*\MM^*)\sim\pi^*\MM^*$. We define a left inverse ${\sf Pr}_\psi:\pi^*\MM\to \varphi^*\MM$ of the map $\psi^*:\varphi^*\MM\to \pi^*\MM$ as follows. On  elements $\tilde v\in \pi^*\MM$ of the kind $\tilde v=\sum_i\nchi_{A_i}\psi^*v_i$ with $v_i\in \varphi^*M$ and $(A_i)$ partition of $\X_1\times[0,1]$, which form a dense subset of $\pi^*\MM$,  it is defined as
\begin{equation}
\label{eq:prs}
{\sf Pr}_\psi(\tilde v):=\sum_{i\in\N}\frac{\d(\psi_*(\nchi_{A_i}\bar \mm))}{\d \psi_*\bar \mm}v_i.
\end{equation}
We want to check that the right hand side depends only on $\tilde v$, and not on the particular representation as  $\sum_i\nchi_{A_i}\psi^*v_i$, and that the series converges in  $\varphi^*\MM$. To this aim let $\X_2\ni x_2\mapsto \bar\mm_{x_2}\in\prob{\X_1\times[0,1]}$ be the disintegration of $\bar\mm$ w.r.t.\ $\psi$ (see Section III-70 of \cite{DellacherieMeyer78} or the argument used in the previous lemma),  notice that $\frac{\d \psi_*(f\bar\mm)}{\d \psi_*\bar\mm}(x_2)=\int f\,\d\bar\mm_{x_2}$ for $\mm_2$-a.e.\ $x_2$ and every $f\in L^\infty(\bar \mm)$ and thus for $\mm_2$-a.e.\ $x_2$ we have
\[
\begin{split}
\bigg|\sum_{i\in\N}\frac{\d(\psi_*(\nchi_{A_i}\bar\mm ))}{\d \psi_*\bar\mm}v_i\bigg|(x_2)& \leq \sum_{i\in\N}\frac{\d(\psi_*(\nchi_{A_i}\bar\mm ))}{\d \psi_*\bar\mm}(x_2)|v_i|(x_2) = \sum_{i\in\N}\Big(\int\nchi_{A_i}\,\d\bar\mm_{x_2}|v_i|(x_2) \Big) \\
&=\int\sum_{i\in\N} \nchi_{A_i}(\bar x)|v_i|(x_2)\,\d\bar\mm_{x_2 }(\bar x)=\int |\tilde v|(\bar x)\,\d\bar \mm_{x_2}.
\end{split}
\]
Raising to the exponent $p$ and integrating w.r.t.\ $\mm_2$ we deduce that
\[
\int|{\sf Pr}_\psi(\tilde v)|^p\,\d\mm_2\leq \int \left|\int |\tilde v|(\bar x)\,\d\bar \mm_{x_2}\right|^p\,\d\mm_2(x_2)\leq \int|\tilde v|^p\,\d\bar \mm,
\]
which addresses both our claims by also showing that ${\sf Pr}_\psi$ can be extended in a unique way to a linear and continuous map from $\pi^*\MM$ to $\varphi^*\MM$. It is clear from the definition \eqref{eq:prs} that 
\begin{equation}
\label{eq:prsl}
{\sf Pr}_\psi(\psi^*v)=v,\qquad\forall v\in \varphi^*\MM.
\end{equation}

The same construction produces a linear and continuous map, still denoted by ${\sf Pr}_\psi$ from $\pi^*\MM^*$ to $\varphi^*\MM^*$. The fact that ${\sf Pr}_\psi$ acts as a sort of left inverse of $\psi^*$ is also implicit in the identity
\begin{equation}
\label{eq:moltobrutta}
\int \big({\sf I}_{\varphi}({\sf Pr}_\psi L)(w)\big)\circ \psi\,\d\bar\mm=\int {\sf I}_{\pi}(L)(\psi^*w)\,\d\bar\mm,\qquad \forall  L\in \pi^*\MM^*,\  w\in \varphi^*\MM,
\end{equation}
which can be proved as follows. First of all, by a density argument we can reduce to the case $w=\sum_i\nchi_{A_i}\varphi^*w_i$ and  $L=\sum_j\nchi_{B_j}\psi^*\varphi^*L_j$, with $w_i\in \MM$ and $L_j\in \MM^*$, where $(A_i)$ and $(B_j)$ are partitions of $\X_2$ and $\X_1\times [0,1]$ respectively. Then notice that $\psi^*w=\sum_i\nchi_{\psi^{-1}(A_i)}\psi^*\varphi^*w_i$ so that  by definition of ${\sf I}_\pi:\psi^*\varphi^*\MM^*\to (\psi^*\varphi^*\MM)^*$ we have
\[
{\sf I}_\pi(L)(\psi^*w)=\sum_{i,j}\nchi_{\psi^{-1}(A_i)\cap B_j}L_j(w_i)\circ \varphi\circ \psi.
\]
On the other hand, by the definitions \eqref{eq:prs} and \eqref{eq:perdefsci} of ${\sf Pr}_\psi$ and ${\sf I}_\varphi$ respectively and recalling that the latter is a module morphism we have 
\[
\big({\sf  I}_\varphi({\sf Pr}_\psi L)(w)\big)\circ \psi=\sum_{i,j}\Big(\frac{\d \psi_*(\nchi_{B_j}\bar\mm)}{\d \psi_*\bar\mm}\nchi_{A_i}L_j(w_i)\circ \varphi\Big)\circ \psi.
\]
Integrating and noticing that $\int \nchi_{B_j}\,\d\bar\mm_{x_2}=\frac{\d \psi_*(\nchi_{B_j}\bar\mm)}{\d \psi_*\bar\mm}(x_2) $ for every $i,j\in\N$ and $\mm_2$-a.e.\ $x_2$,   we get the claim \eqref{eq:moltobrutta}.

We are now ready to conclude: pick an arbitrary  $L_1\in (\varphi^*\MM)^*$ and notice that the map 
\[
\psi^*(\varphi^*\MM)\sim \pi^*\MM \ni v\qquad\mapsto\qquad l(v):=\int L_1({\sf Pr}_\psi(v))\,\d\mm_2\in\R,
\]
being linear and continuous, belongs to the Banach dual of $\pi^*\MM$. By Proposition \ref{prop:casoprodotto} we deduce that there exists $L_2\in \pi^*\MM^*$ such that
\[
l(v)=\int {\sf I}_\pi(L_2)(v)\,\d\bar \mm ,\qquad\forall v\in \pi^*\MM.
\]
Define $L_3\in \varphi^*\MM^*$ as $L_3:={\sf Pr}_\psi L_2 $ and notice that for any $w\in \varphi^*\MM$ we have
\[
\int{\sf  I}_\varphi ( L_3)(w)\,\d\mm_2\!\!\stackrel{\eqref{eq:moltobrutta}}=\!\!\int {\sf I}_{\pi}( L_2)(\psi^*w)\,\d\bar \mm=l(\psi^*w)= \int L_1({\sf Pr}_\psi(\psi^*w) )\,\d\mm_2\!\!\stackrel{\eqref{eq:prsl}}=\!\!\int L_1(w)\,\d\mm_2.
\]
This is sufficient to deduce that ${\sf I}_\varphi( L_3)=L_1$, which is the thesis.
\end{proof}
\begin{remark}{\rm
The need of the assumption that the measured spaces come from a Polish topology is only due to the structure of the proof we proposed. Nothing excludes that the same conclusion holds under the only assumption that $\MM^*$ has the Radon-Nikodym property. For instance, if $\MM$ is an Hilbert module, using the Riesz theorem for Hilbert modules (Theorem \ref{thm:rhil}) it is easy to see that the dual of $\varphi^*\MM$ can be identified with $\varphi^*\MM^*$ without further assumptions.
}\fr\end{remark}

\section{First order differential structure of general metric measure spaces}
\subsection{Preliminaries: Sobolev functions on metric measure spaces}\label{se:presob}
Without exceptions, all the   metric measure spaces $(\X,\sfd,\mm)$ we shall consider are such that:
\begin{itemize}
\item[-] $(\X,\sfd)$ is complete and separable,
\item[-] the measure $\mm$ is Radon and non-negative.
\end{itemize}
We endow $\X$ with the Borel $\sigma$-algebra $\mathcal A$ and denote by $\Bo(\X)$ the set of equivalence classes of Borel sets w.r.t.\ the equivalence relation given by $A\sim B$ provided $\mm((A\setminus B)\cup(B\setminus A))=0$.

By $C([0,1],\X)$ we shall denote the space of continuous curves with value in $\X$ endowed with the $\sup$ norm. Notice that due to the fact that $(\X,\sfd)$ is complete and separable, $C([0,1],\X)$ is complete and separable as well. For $t\in[0,1]$ we consider the evaluation map  $\e_t:C([0,1],\X)\to \X$  defined by
\[
\e_t(\gamma):=\gamma_t,\qquad\forall \gamma\in C([0,1],\X).
\]

\begin{definition}[Absolutely continuous curves and metric speed]\label{def:speedcurve}
A curve $\gamma:[0,1]\to \X$ is said absolutely continuous provided there exists a function $f\in L^1(0,1)$ such that
\begin{equation}
\label{eq:accurve}
\sfd(\gamma_t,\gamma_s)\leq \int_t^sf(r)\,\d r,\qquad\forall t,s\in[0,1],\ t<s.
\end{equation}
The   metric speed $t\mapsto |\dot\gamma_t|\in L^1(0,1)$ of an absolutely continuous curve $\gamma$ is defined as the essential-infimum of all the functions $f\in L^1(0,1)$ for which \eqref{eq:accurve} holds.
\end{definition}
It is worth to recall that the metric speed can be equivalently defined as limit of incremental ratios:
\begin{theorem}[Metric speed as incremental ratio]\label{thm:112}
Let $\gamma:[0,1]\to \X$ be an absolutely continuous curve. Then  for a.e. $t\in[0,1]$ the limit of $\frac{\sfd(\gamma_{t+h},\gamma_t)}{|h|}$ as $h\to 0$ exists and is equal to $|\dot\gamma_t|$.
\end{theorem}
See for instance Theorem 1.1.2 of \cite{AmbrosioGigliSavare08} for a proof.

In what follows we shall often write $\int|\dot\gamma_t|^2\,\d t$ for a generic curve $\gamma\in C([0,1],\X)$: in case $\gamma$ is not absolutely continuous the value of the integral is taken $+\infty$ by definition.

\bigskip

There are several equivalent definitions of Sobolev functions on a metric measure space, we shall follow an approach proposed in \cite{AmbrosioGigliSavare11}. We start with the following definition:
\begin{definition}[Test Plans]
Let $\ppi\in\prob{C([0,1], \X)}$. We say that $\ppi$ is a test plan provided there is a constant $\cf(\ppi)$ such that
\[
(\e_t)_\sharp\ppi\leq \cf(\ppi)\mm,\qquad\forall t\in[0,1],
\]
and moreover
\[
\iint_0^1|\dot\gamma_t|^2\,\d t\ppi(\gamma)<\infty.
\]
\end{definition}
Notice that  according to the convention $\int_0^1|\dot\gamma_t|^2\,\d t=+\infty$ if $\gamma$ is not absolutely continuous, any test plan must be concentrated on absolutely continuous curves.
\begin{definition}[The Sobolev class $\s^2(\X,\sfd, \mm)$]
The Sobolev class $\s^2(\X,\sfd,\mm)$, or simply $\s^2(\X)$ is the space of all functions $f\in L^0(\mm)$ such that there exists a non-negative $G\in L^2( \mm)$  for which it holds
\begin{equation}
\label{eq:defsob}
\int|f(\gamma_1)-f(\gamma_0)|\,\d\ppi(\gamma)\leq \iint_0^1G(\gamma_t)|\dot\gamma_t|\,\d t\,\d\ppi(\gamma),\qquad\forall \ppi\textrm{ test plan}.
\end{equation}
\end{definition}
Notice that due to the fact that $(\e_0)_\sharp\ppi,(\e_1)_\sharp\ppi\ll\mm$, the integral in the left hand side of \eqref{eq:defsob} is well defined, i.e.\ depends only on the equivalence class of the function $f\in L^0(\mm)$ and not on its chosen representative. Similarly for the right hand side.

It turns out, see \cite{AmbrosioGigliSavare11}, that for $f\in \s^2(\X)$ there exists a minimal $G$ in the $\mm$-a.e. sense for which \eqref{eq:defsob} holds: we will denote it by $\weakgrad f$ and, in line with the literature on the topic, call it minimal weak upper gradient. Notice that  the terminology is misleading, because being this object defined in duality with speed of curves, it is closer to the norm of a cotangent vector rather to a tangent one, whence the choice of denoting it by $\weakgrad f$.

We recall below the basic properties of Sobolev functions and minimal weak upper gradients.
\noindent\underline{Lower semicontinuity of minimal weak upper gradients}. Let $(f_n)\subset \s^2(\X)$ and $f\in L^0(\mm)$ be such that $f_n\to f$ as $n\to\infty$ in $L^0(\mm)$ (i.e.\ $\mm$-a.e.). Assume that $(\weakgrad {f_n})$ converges to some $G\in L^2(\mm)$ weakly in $L^2(\mm)$.

Then 
\begin{equation}
\label{eq:lscwug}
f\in \s^2(\X )\qquad\text{ and }\qquad\weakgrad f\leq  G,\quad \mm\ae.
\end{equation}

\noindent\underline{Weak gradients and local Lipschitz constant - 1}. For any $\bar f:\X\to\R$ locally Lipschitz with ${\rm lip}(\bar f)\in L^2(\mm)$ we   have $f\in \s^2(\X)$ with
\begin{equation}
\label{eq:lipweak}
\weakgrad f\leq  {\rm lip}(\bar f),\qquad\mm\ae,
\end{equation}
where the function  ${\rm lip}(\bar f):\X\to\R^+$ is the local Lipschitz constant  defined by
\[
{\rm lip}(\bar f)(x):=\lims_{y\to x}\frac{|\bar f(y)-\bar f(x)|}{\sfd(y,x)},
\]
 at points $x\in \X$ which are not isolated, 0 otherwise.
 
\noindent\underline{Weak gradients and local Lipschitz constant - 2}. Suppose that $\mm$ gives finite mass to bounded sets. Then for any $f\in L^2\cap\s^2(\X)$ there exists a sequence $(f_n)\subset  L^2\cap\s^2(\X)$ of functions having Lipschitz representatives $\bar f_n$ converging to $f$ in $L^2(\X)$ such that
\begin{equation}
\label{eq:aplip2}
\int\weakgrad f^2\,\d\mm=\lim_{n\to\infty}\int{\rm lip}^2(\bar f_n)\,\d\mm.
\end{equation}
\noindent\underline{Vector space structure}. $\s^2(\X)$ is a vector space and 
\begin{equation}
\label{eq:vectorstru}
\weakgrad{(\alpha f+\beta g)}\leq |\alpha|\weakgrad f+|\beta|\weakgrad g,\qquad\textrm{for any $f,g\in\s^2(\X)$, $\alpha,\beta\in\R$.}
\end{equation}
\noindent\underline{Algebra structure}.  $\s^2\cap L^\infty(\X)$ is an algebra and
\begin{equation}
\label{eq:leibbase}
\weakgrad{(fg)}\leq |f|\weakgrad g+|g|\weakgrad f,\qquad\textrm{for any $f,g\in \s^2\cap L^\infty(\X)$.}
\end{equation}
\noindent\underline{Locality}. The minimal weak upper gradient is local in the following sense:
\begin{equation}
\label{eq:localgrad0}
\weakgrad f=\weakgrad g,\qquad\mm\ae\textrm{  on }\{f=g\},\qquad \forall f,g\in\s^2(\X).
\end{equation}
\noindent\underline{Chain rule}. For every $f\in\s^2(\X)$ we have
\begin{equation}
\label{eq:nullgrad}
\weakgrad f=0,\qquad\textrm{on }f^{-1}(\mathcal N),\qquad\forall \mathcal N\subset \R\textrm{, Borel with }\mathcal L^1(\mathcal N)=0,
\end{equation}
moreover, for $f\in \s^2(\X)$, $I\subset \R$ open such that  $\mm(f^{-1}(\R\setminus I))=0$ and $\varphi:I\to\R$ Lipschitz, we have $\varphi\circ f\in\s^2(\X)$ and
\begin{equation}
\label{eq:chainweakgrad}
\weakgrad{(\varphi\circ f)}=|\varphi'|\circ f\weakgrad f,
\end{equation}
where $|\varphi'|\circ f$ is defined arbitrarily at points where $\varphi$ is not differentiable. (observe that the identity \eqref{eq:nullgrad} ensures that on $f^{-1}(\mathcal N)$ both $\weakgrad{(\varphi\circ f)}$ and $\weakgrad f$ are 0 $ \mm$-a.e., $\mathcal N$ being the negligible set of points of non-differentiability of $\varphi$).

\noindent\underline{Compatibility with the smooth case}. If $(\X,\sfd,\mm)$ is a smooth Finsler manifold, then 
\begin{equation}
\label{eq:compsmweakgrad}
\begin{split}
&\text{$f\in L^1_{\rm loc}(\mm)$ belongs to $\s^2(\X)$ if and only if its distributional differential $Df$ is}\\
&\text{an $L^2$-covector field and in this case  the norm of $D f$ coincides with $\weakgrad f$ $\mm$-a.e..}
\end{split}
\end{equation}

\bigskip

The Sobolev space $W^{1,2}(\X,\sfd,\mm)$, or simply $W^{1,2}(\X)$, is defined as $W^{1,2}(\X):=L^2\cap \s^2(\X)$ and endowed with the norm
\[
\|f\|_{W^{1,2}(\X)}^2:=\|f\|_{L^2(\mm)}^2+\|\weakgrad f\|^2_{L^2(\mm)}.
\]
$W^{1,2}(\X)$ is always a Banach space, but in general not an Hilbert space. We recall the following result, proved in \cite{ACM14}, about reflexivity and separability of $W^{1,2}(\X)$:
\begin{theorem}\label{thm:ACM}
Let $(\X,\sfd,\mm)$ be a complete separable metric space endowed with a non-negative Radon measure $\mm$. Then the following hold.
\begin{itemize}
\item[i)] Assume that $W^{1,2}(\X)$ is reflexive. Then it is separable.
\item[ii)] A sufficient condition for $W^{1,2}(\X)$ to be reflexive is that $(\X,\sfd)$ is metrically doubling, i.e.\ there a constant $c\in\N$ such that for every $r>0$, every ball of radius $r$ can be covered with $c$ balls of radius $r/2$.
\end{itemize}
Moreover:
\begin{itemize}
\item[iii)] There exists a compact space $(\X,\sfd,\mm)$ such that $W^{1,2}(\X)$ is non-separable, and thus also non-reflexive.
\end{itemize}
\end{theorem}
We now turn to the definition of Laplacian. The 2-energy $\E:L^2(\mm)\to[0,\infty]$, called Cheeger energy in \cite{AmbrosioGigliSavare11}, is defined as
\[
\E(f):=\left\{\begin{array}{ll}
\displaystyle{\frac12\int\weakgrad f^2\,\d\mm}&\qquad\text{ if }f\in W^{1,2}(\X),\\
+\infty&\qquad\text{ otherwise.}
\end{array}
\right.
\]
From the properties of the minimal weak upper gradient we see that $\E$ is convex, lower semicontinuous and with dense domain (i.e. $\{f:\E(f)<\infty\}$ is dense in $L^2(\mm)$). Recall that the subdifferential $\partial^-\E(f)\subset L^2(\mm)$ at a function $f\in L^2(\mm)$ is defined to be the empty set if $\E(f)=+\infty$ and otherwise as the set, possibly empty, of functions $v\in L^2(\mm)$ such that
\[
\E(f)+\int vg\,\d\mm\leq \E(f+g),\qquad\forall g\in L^2(\mm).
\]
The set $\partial^-\E(f)$ is always closed and convex, possibly empty, and the class of $f$'s such that $\partial^-\E(f)\neq\emptyset$ is dense in $L^2(\mm)$. We then define the domain $D(\Delta)$ of the Laplacian as $\{f:\partial^-\E(f)\neq\emptyset \}\subset L^2(\X)$ and for $f\in D(\Delta)$ the Laplacian $\Delta f\in L^2(\mm)$ as $\Delta f:=-v$, where $v$ is the element of minimal norm in $\partial^-\E(f)$. Notice that  in general $D(\Delta)$ is not a vector space and that the Laplacian is not linear, these being true if and only if $W^{1,2}(\X)$ is Hilbert.

The classical theory of gradient flows of convex functions on Hilbert spaces (see e.g.\ \cite{AmbrosioGigliSavare08} and the references therein) ensures existence and uniqueness of a 1-parameter semigroup $(\h_t)_{t\geq 0}$ of continuous operators from $L^2(\mm)$ to itself such that for every $f\in L^2(\mm)$ the curve  $t\mapsto \h_t(f)\in L^2(\mm)$ is continuous on $[0,\infty)$, absolutely continuous on $(0,\infty)$ and fulfills
\[
\frac{\d}{\d t}\h_t(f)=\Delta f\qquad{\rm a.e. }t>0,
\]
where it is part of the statement the fact that $\h_t(f)\in D(\Delta)$ for every $f\in L^2(\mm)$ and $t>0$. Notice that since in general the Laplacian is not linear, without the assumption that $W^{1,2}(\X)$ is Hilbert we cannot expect  the operators $\h_t$ to be linear.

\subsection{Cotangent module}
We fix once and for all a metric measure space $(\X,\sfd,\mm)$ which is complete, separable and equipped with a non-negative Radon measure.

\subsubsection{The construction}
Here we construct a key object of our analysis: the cotangent module $L^2(T^*\X)$ of the given metric measure space $(\X,\sfd,\mm)$, which by definition will be an $L^2(\mm)$-normed module. Notice that although we start with the definition of cotangent module and then introduce the tangent one by duality in the next section, to keep consistency with the notation used in the smooth setting, we shall denote by $|\cdot|_*$ the pointwise norm on the cotangent module

Technically speaking, the construction is strongly reminiscent of that of pullback module given in Section \ref{se:pullback}, but from the conceptual point of view there is a non-entirely negligible difference: now we don't have any module to start with.


\vspace{1cm}

We introduce the set `Pre-cotangent module' ${\rm Pcm}$ as
\[
\begin{split}
{\rm Pcm}:=\Big\{\{(f_i,A_i)\}_{i\in\N}\ :&\  (A_i)_{i\in\N}\subset\Bo(\X)\text{ is a   partition of $\X$,} \\
&\ f_i\in\s^2(\X)\ \forall i\in\N, \text{ and }\sum_{i\in\N}\int_{A_i}\weakgrad f^2\,\d\mm<\infty \Big\}
\end{split}
\]
and  define an equivalence relation on ${\rm Pcm}$ by declaring $\{(f_i,A_i)\}_{i\in\N}\sim \{(g_j,B_j)\}_{j\in\N}$ provided
\[
\weakgrad{(f_i-g_j)}=0,\quad\mm\ae\text{ on } A_i\cap B_j,\qquad\forall i,j\in\N.
\]
It is readily verified that this is indeed an equivalence relation and we shall denote by $[(f_i,A_i)_i]$ the equivalence class of $\{(f_i,A_i)\}_{i\in\N}$.  In a sense that will be made rigorous by Definition \ref{def:diff}, the object $[(f_i,A_i)_i]$ should be though of as the 1-form which is equal to the differential $\d f_i$ of $f_i$ on $A_i$.

We endow ${\rm Pcm}/\sim$ with a vector space structure by defining the sum  and multiplication   with a scalar  as
\[
\begin{split}
[(f_i,A_i)_i]+[(g_j,B_j)_{j}]&:=[(f_i+g_j,A_i\cap B_j)_{i,j}],\\
\lambda[(f_i,A_i)\}_i]&:=[(\lambda f_i,A_i)_i].
\end{split}
\]
It is clear that these operations are well defined and endow ${\rm Pcm}/\sim$ with a vector space structure.

Now recall that by  ${\rm Sf}(\mm)\subset L^\infty(\mm)$ we intend the space of simple functions (i.e. attaining only a finite number of values) and  for $[(f_i,A_i)_i]\in {\rm Pcm}/\sim$ and $h=\sum_{j}a_j\nchi_{B_j}\in {\rm Sf}(\mm)$ with $(B_j)$ partition of $\X$  define the product $h [(f_i,A_i)_i]\in {\rm Pcm}/\sim$ as
\[
h[(f_i,A_i)_i]:=[(a_jf_i,A_i\cap B_j)_{i,j}].
\]
It is readily verified that this definition is well posed and that the resulting multiplication is a bilinear map from ${\rm Sf}(\mm)\times {\rm Pcm}/\sim$ into ${\rm Pcm}/\sim$ such that  ${\bf 1}[(f_i,A_i)_i]=[(f_i,A_i)_i]$. 

Finally, we consider the map $|\cdot|_*:{\rm Pcm}/\sim\,\to L^2(\mm)$ given by
\[
\big|[(f_i,A_i)_i]\big|_*:=\weakgrad{f_i},\qquad\mm\ae\text{  on } A_i,\ \forall i\in\N,
\]
which is well defined by the very definition of the equivalence relation $\sim$.

Then recalling the vector space structure of $\s^2(\X)$ encoded in inequality \eqref{eq:vectorstru}  we see that that  the (in)equalities 
\begin{equation}
\label{eq:basecotan}
\begin{split}
\big|[(f_i+g_j,A_i\cap B_j)_{i,j}]\big|_* &\leq \big|[((f_i,A_i)_i]\big|_*+\big|[(g_j,B_j)_j]\big|_*,\\
\big|\lambda[(f_i,A_i)_i]\big|_*&=|\lambda|\,\big|[(f_i,A_i)_i]\big|_*,\\
\big|h[(f_i,A_i)_i]\big|_*&=|h|\,\big|[(f_i,A_i)_i]\big|_*,
\end{split}
\end{equation}
are valid $\mm$-a.e.\ for every $[((f_i,A_i)_i],[((g_j,B_j)_j]\in {\rm Pcm}/\sim$, $h\in {\rm Sf}(\X)$ and $\lambda\in\R$.

In particular the map $\|\cdot\|_{L^2(T^*\X)}:{\rm Pcm}/\sim\to[0,\infty)$ defined by
\[
\|[(f_i,A_i)_i]\|^2_{L^2(T^*\X)}:={\int\big|[(f_i,A_i)_i]\big|^2\,\d\mm}={\sum_{i\in\N}\int_{A_i}\weakgrad{f_i}^2\,\d\mm},
\]
is a norm on ${\rm Pcm}/\sim$.
\begin{definition}[Cotangent module] The cotangent module $(L^2(T^*\X),\|\cdot\|_{L^2(T^*\X)})$ is defined as the completion of $({\rm Pcm}/\sim,\|\cdot\|)$. Elements of $L^2(T^*\X)$ will be called cotangent vector fields or 1-forms.
\end{definition}
By construction and due to the first two inequalities in \eqref{eq:basecotan} the space $(L^2(T^*\X),\|\cdot\|_{L^2(T^*\X)})$  is a Banach space, in general being not an Hilbert in line with the case of Finsler manifolds (see also Remark \ref{rem:compcot}). 

Moreover, $L^2(T^*\X)$ comes with the structure of $L^2(\mm)$-normed module. Indeed,  the third in \eqref{eq:basecotan} ensures that the bilinear map $(h,[(f_i,A_i)_i])\mapsto h[(f_i,A_i)_i]$ from ${\sf Sf}(\mm)\times {\rm Pcm}/\sim$ to ${\rm Pcm}/\sim$ can, and will, be uniquely extended to a bilinear map $(g,\omega)\mapsto g\omega$ from $L^\infty(\mm)\times L^2(T^*\X)$ to $L^2(T^*\X)$  satisfying $|h\omega|_*= |h| |\omega|_*$ $\mm$-a.e.\ for every $h\in L^\infty(\mm)$ and $\omega\in L^2(T^*\X)$. It is then clear that $L^2(T^*\X)$ is an $L^2(\mm)$-normed premodule and thus by Proposition \ref{prop:baselp} an $L^2(\mm)$-normed module.

Notice also that the notation $L^2(T^*\X)$ is purely formal, in the sense that we didn't, nor we will, define a cotangent bundle $T^*\X$. Such notation has been chosen because in the smooth setting the cotangent module can be canonically identified with the space of $L^2$ sections of the cotangent bundle, see Remark \ref{rem:compcot}.

\subsubsection{Differential of a Sobolev function}
With the definition of cotangent module it comes naturally that of differential of a Sobolev function. In this section we investigate the basic properties of such object, both from differential calculus and the functional analytic perspectives. 

\vspace{1cm}

\begin{definition}[Differential of a Sobolev function]\label{def:diff}
Let $f\in \s^2(\X)$. Its differential $\d f\in L^2(T^*\X)$ is defined as 
\[
\d f:=[(f,\X)]\quad\in {\rm Pcm}/\!\sim\ \subset L^2(T^*\X),
\]
where $(f,\X)$ is a shorthand for $(f_i,A_i)_{i\in\N}$ with $f_0=f$, $A_0=\X$ and $f_i=0$, $A_i=\emptyset$ for $i>0$. 
\end{definition}
It is clear that $\d f$ linearly depends on $f$. From the very definition of the pointwise norm on $L^2(T^*\X)$ we also get that
\begin{equation}
\label{eq:diffuguali}
|\d f|_*=\weakgrad f,\qquad\mm\ae\quad\forall f\in\s^2(\X).
\end{equation}
Other expected properties are the locality and the chain and Leibniz rules. We start from the locality and see later how to deduce the other calculus rules out of it.
\begin{theorem}[Locality of the differential]\label{thm:diffloc} Let $f,g\in\s^2(\X)$. Then
\[
\d f=\d g,\qquad\mm\ae\ \text{\rm on } \{f=g\}.
\]
\end{theorem}
\begin{proof} Taking into account the linearity of the differential, the thesis is equivalent to $\d(f-g)=0$ $\mm$-a.e.\ on $\{f-g=0\}$ which by point $(i)$ of Proposition \ref{prop:baselp}  is equivalent to  $|\d(f-g)|_*=0$ $\mm$-a.e.\ on $\{f-g=0\}$. Then the thesis follows recalling \eqref{eq:diffuguali} and using the locality property of the minimal weak upper gradient \eqref{eq:localgrad0} with $f-g$ in place of $f$ and $0$ in place of $g$.
\end{proof}
\begin{remark}[Compatibility with the smooth case]\label{rem:compcot}{\rm If $(\X,\sfd,\mm)$ is a smooth Finsler manifold endowed with the distance induced by the Finsler norm and a smooth reference measure (i.e.\ a measure which, when read in charts, is absolutely continuous w.r.t.\ Lebesgue with smooth density), then the cotangent module can be canonically identified with the space of $L^2$ sections of the cotangent bundle. Indeed, the map taking the differential of  a Sobolev function $f$ in the sense of Definition \ref{def:diff} to its usual $\mm$-a.e.\ defined differential preserves, by the compatibility property \eqref{eq:compsmweakgrad}, the pointwise norm and thus can be uniquely extended to linear isometry preserving the module structure. The fact that the image of such morphism is the whole space of $L^2$ sections can then be easily obtained by approximation recalling that for every $x\in\X$ the cotangent space $T_x^*\X$ coincides with the space of differentials at $x$ of smooth functions.

A curious byproduct of this fact is that with the terminology of $L^2$-normed modules and the constructions presented here we produced a notion of differential of a Sobolev function fully compatible with the classical one, whose definition never requires, not even implicitly, the notion of limit of incremental ratios.
}\fr\end{remark}
Another direct consequence of the definitions is:
\begin{proposition}[$L^2(T^*\X)$ is generated by differentials]\label{prop:gencotan}
The cotangent module $L^2(T^*\X)$ is generated, in the sense of modules (Definition \ref{def:genmod}), by the space $\{\d f:f\in W^{1,2}(\X)\}$. 

In particular, if $W^{1,2}(\X)$ is separable  then so is $L^2(T^*\X)$.
\end{proposition}
\begin{proof} We first claim that $L^2(T^*\X)$ is generated by $\{\d f:f\in \s^2(\X)\}$. This follows by just keeping track of the various definitions, indeed, we have $\d f=[(f,\X)]$ by definition of differential, and thus by the definition of operations on ${\rm Pcb}/\sim$ also that $\sum_{i=1}^n\nchi_{A_i}\d f_i=[(f_i,A_i)_i]$ for any finite partition $(A_i)$ of $\X$ and $f_i\in\s^2(\X)$. Taking into account the definition of norm in ${\rm Pcb}/\sim$ we can pass to the limit and extend this property to generic elements $[(f_i,A_i)_i]$ of ${\rm Pcb}/\sim$ and given that ${\rm Pcb}/\sim$ is by construction dense in $L^2(T^*\X)$, the claim follows.

Then pick $f\in \s^2(\X)$ and observe that for $f_n:=\min\{\max\{f,-n\},n\}$, the chain rule \eqref{eq:chainweakgrad} yields $\weakgrad{(f-f_n)}=0$ $\mm$-a.e.\ on $\{|f|\leq n\}$, which in particular implies   that $\d f_n\to \d f$ in $L^2(T^*\X)$ as $n\to\infty$.  Thus by approximation we see that $L^2(T^*\X)$ is generated by $\{\d f:f\in L^\infty\cap \s^2(\X)\}$. 

Now let $f\in L^\infty\cap \s^2(\X)$, pick $x\in\X$, find $r_x>0$ so that $\mm(B_{2r_x}(x))<\infty$ and put $\eta_x(y):=\max\{1-r_x^{-1}\sfd(y,B_{r_x}(x)),0\}$. Then $\eta_{x}$ is Lipschitz, bounded and in $W^{1,2}(\X)$ and thus taking into account \eqref{eq:leibbase} we get that $\eta_x f\in W^{1,2}(\X)$ while the construction ensures that $f=\eta_xf$ $\mm$-a.e.\ on $B_{r_x}(x)$. Thus   Theorem \ref{thm:diffloc} above grants that
\begin{equation}
\label{eq:pergenw12}
\nchi_{B_{r_x}(x)}\d f=\nchi_{B_{r_x}(x)}\d (\eta_xf).
\end{equation}
By the Lindel\"of property of $\X$ we know that there exists a countable set $\{x_n\}_{n\in\N}\subset \X$ such that $\cup_{n}B_{r_{x_n}}(x_n)=\X$, thus using \eqref{eq:pergenw12} we deduce that for every $n\in\N$ the 1-form $\omega_n:=\sum_{i=1}^n\nchi_{B_{r_{x_n}}(x_n)}\d f$ is in the submodule generated by  $\{\d f:f\in W^{1,2}(\X)\}$. Since $|\d f-\omega_n|_*\leq |\d f|_*\in L^2(\X)$ and $|\d f-\omega_n|_*\to 0$ $\mm$-a.e., the dominated convergence theorem gives that  $\omega_n\to\d f$ in $L^2(T^*\X)$ as $n\to\infty$, thus showing that $\d f$ belongs to the submodule generated by  $\{\d f:f\in W^{1,2}(\X)\}$. By what we already proved, this submodule is thus the whole $L^2(T^*\X)$.

For the last claim just notice that the inequality $\|\d f\|_{L^2(T^*\X)}\leq \|f\|_{W^{1.2}(\X)}$ grants that $\{\d f: f\in W^{1,2}(\X)\}$ is a separable subset of $L^2(T^*\X)$ and apply Proposition \ref{prop:finsep}.
\end{proof}
In order to continue our investigation we now prove that there is a deep link between the Leibniz and chain rules and the locality of a module-valued map which is  continuous w.r.t.\ the Sobolev norm. Notice that the result does not (explicitly) mention the notion of cotangent module.
\begin{theorem}\label{thm:der}
Let $\MM$ be an $L^\infty(\mm)$-module and $L:\s^2(\X) \to \MM$ a linear map continuous w.r.t. the Sobolev norm, i.e.\ satisfying 
\begin{equation}
\label{eq:contprop}
\|L(f)\|_\MM\leq C\sqrt{\int\weakgrad f^2\,\d\mm},\qquad\forall f\in\s^2(\X),
\end{equation}
for some constant $C>0$. Then the following are equivalent.
\begin{itemize}
\item[i)] \underline{\rm Leibniz rule.} For any $f,g\in L^\infty\cap \s^2(\X)$ it holds
\begin{equation}
\label{eq:leibprop}
L(fg)=f\,L(g)+g\,L(f),\qquad\mm\ae.
\end{equation}
\item[ii)]\underline{\rm Chain rule.} For any $f\in\s^2(\X)$ and any $\mathcal N\subset \R$ Borel and $\mathcal L^1$-negligible it holds 
\begin{equation}
\label{eq:neglprop}
L(f)=0,\qquad\mm\ae\ {\rm on}\  f^{-1}(\mathcal N),
\end{equation}
and for $I\subset \R$ open such that $\mm(f^{-1}(\R\setminus I))=0$ and  $\varphi: I\to\R$ Lipschitz it holds
\begin{equation}
\label{eq:chainprop}
L(\varphi\circ f)=\varphi'\circ f\,L(f),\qquad\mm\ae,
\end{equation}
where $\varphi'\circ f$ is defined arbitrarily on $f^{-1}(\{\text{non differentiability points of $\varphi$}\})$.
\item[iii)]\underline{\rm Locality.} For any $f,g\in\s^2(\X)$ it holds
\begin{equation}
\label{eq:localprop}
L(f)=L(g),\qquad\mm\ae\ {\rm on}\   \{f=g\}.
\end{equation}
\end{itemize}
\end{theorem}
\begin{proof}$\ $\\
\noindent{$\mathbf{(ii)\Rightarrow(i)}.$} Assume at first that $f,g \geq c$ $\mm$-a.e. for some $c>0$. Then apply the chain rule \eqref{eq:chainprop} with $\varphi(z):=\log z$ to get
\[
\begin{split}
L(fg)&=fg\, L\big(\log(fg)\big)=fg\big(L(\log(f))+L(\log(g))\big)=fg\Big(\frac{L(f)}{f}+\frac{L(g)}{g}\Big)=g\,L(f)+f\,L(g),
\end{split}
\]
$\mm$-a.e..
For the general case, observe that  applying \eqref{eq:chainprop} with $\varphi(z):=z+\alpha$, $\alpha\in\R$, we deduce that $L(f+\alpha)=L(f)$ $\mm$-a.e.. From this identity, the validity of the Leibniz rule for positive functions and the linearity of $L$ we easily conclude. 

\noindent{$\mathbf{(i)\Rightarrow(iii)}.$} By linearity it is sufficient to show that for any $f\in\s^2(\X)$ we have $L(f)=0$ $\mm$-a.e. on $\{f=0\}$. For $\eps>0$ define $\varphi_\eps,\psi_\eps:\R\to\R$ as
\[
\varphi_\eps(z):=\left\{
\begin{array}{ll}
\eps^{-1}-\eps,&\quad\textrm{if } z\in[\eps^{-1},+\infty)\\
z-\eps,&\quad\textrm{if }z\in(\eps,\eps^{-1}),\\
0,&\quad\textrm{if }z\in[-\eps,\eps],\\
z+\eps,&\quad\textrm{if }z\in(-\eps^{-1},-\eps),\\
-\eps^{-1}+\eps,&\quad\textrm{if } z\in(-\infty,-\eps^{-1}],\\
\end{array}
\right.
\qquad\quad
\psi_\eps(z):=\left\{
\begin{array}{ll}
|z|-\eps,&\quad\textrm{if }z\in[-\eps,\eps],\\
0,&\quad\textrm{if }z\in\R\setminus[-\eps,\eps].
\end{array}
\right.
\]
Notice that $\varphi_\eps,\psi_\eps$ are 1-Lipschitz functions and that $\varphi_\eps(z)\psi_\eps(z)=0$ for every $z\in\R$. It follows that $\varphi_\eps\circ f,\psi_\eps\circ f\in \s^2\cap L^\infty(\X)$ with $\varphi_\eps\circ f\,\psi_\eps\circ f=0$ $\mm$-a.e.\ so the Leibniz rule \eqref{eq:leibprop} gives
\begin{equation}
\label{eq:perloc}
0=L(\varphi_\eps\circ f\,\psi_\eps\circ f)=\varphi_\eps\circ f\,L(\psi_\eps\circ f)+\psi_\eps\circ f\,L(\varphi_\eps\circ f)\qquad\mm\ae.
\end{equation}
By construction, on the set $\{f=0\}$ we have $\varphi_\eps\circ f=0$ and $\psi_\eps\circ f\neq 0$ $\mm$-a.e., therefore \eqref{eq:perloc} ensures that  $L(\varphi_\eps\circ f)=0$ $\mm$-a.e.\ on $\{f=0\}$, which gives
\begin{equation}
\label{eq:rigirata}
\begin{split}
\|\nchi_{\{f=0\}}L(f)\|_\MM&\leq\|\nchi_{\{f=0\}}L(f-\varphi_\eps\circ f)\|_\MM+\|\nchi_{\{f=0\}}L(\varphi_\eps\circ f)\|_\MM\\
&\leq \|L(f-\varphi_\eps\circ f)\|_\MM\leq C \sqrt{\int \weakgrad{(f-\varphi_\eps\circ f)}^2\,\d\mm}.
\end{split}
\end{equation}
Now notice that by definition we have $\varphi_\eps'=1$ on $(-\eps^{-1},-\eps)\cup (\eps,\eps^{-1})$ and thus the chain rule \eqref{eq:chainweakgrad}  yields
\[
\weakgrad{(f-\varphi_\eps\circ f)}=|1-\varphi_\eps'|\circ f \,\weakgrad f=0,\qquad\mm\ae\text{   on } f^{-1}\big((-\eps^{-1},-\eps)\cup (\eps,\eps^{-1})\big),
\]
and therefore  $\lim_{\eps\downarrow0}\int\weakgrad{(f-\varphi_\eps\circ f)}^2\,\d\mm=0$. Hence letting $\eps\downarrow0$ in \eqref{eq:rigirata}  we conclude that $\|\nchi_{\{f=0\}}L(f)\|_\MM=0$, which is the thesis.

\noindent{$\mathbf{(iii)\Rightarrow(ii)}.$} Fix $f\in\s^2(\X)$ and recall that by the chain rule \eqref{eq:chainweakgrad} we know that $\varphi\circ f\in\s^2(\X)$. 

The claim is obvious if $\varphi$ is linear, because $L$ is linear itself. Hence, noticing that the continuity assumption \eqref{eq:contprop} ensures that $L(f)\equiv 0$ for any constant function $f$, using the linearity of $L$ once again we get the chain rule \eqref{eq:chainprop} for affine $\varphi$. 

We claim that
\begin{equation}
\label{eq:countnull}
L(f)=0,\qquad \mm\ae\text{ on }  f^{-1}(\mathcal D),\text{ for $\mathcal D\subset \R$ at most countable},
\end{equation}
and notice that by the locality property \eqref{eq:locality} of the module $\MM$ this is equivalent to the fact that $L(f)=0$ $\mm$-a.e.\ on $\{f=z_0\}$ for any $z_0\in\R$. Since $\weakgrad{(f-(f-z_0))}=0$, the continuity assumption \eqref{eq:contprop} grants that $L(f)=L(f-z_0)$, so it is sufficient to prove that $L(f)=0$ $\mm$-a.e.\ on $\{f=0\}$, which  follows noticing that by linearity we have $L(f)=-L(-f)$ and that by the locality \eqref{eq:localprop} we have $L(f)=L(-f)$ $\mm$-a.e.\ on $\{f=-f\}=\{f=0\}$.

Now let $\varphi:\R\to\R$ be Lipschitz and countably piecewise affine, i.e. such that there are closed intervals $I_n\subset \R$, $n\in\N$, covering $\R$ such that $\varphi\restr{I_n}$ is affine for every $n\in\N$. By the locality property \eqref{eq:localprop} and the fact that we proved the chain rule for affine functions, we deduce that formula \eqref{eq:chainprop} holds $\mm$-a.e.\ on $f^{-1}(I_n)$ for every $n\in\N$, and hence $\mm$-a.e.\ on $\X$.

We now claim that the chain rule \eqref{eq:chainprop} holds for $\varphi\in C^1(\R)$ with bounded derivative. Thus fix such $\varphi$ and find a sequence $(\varphi_n)$ of uniformly Lipschitz and piecewise affine functions such that $\varphi_n'\to\varphi'$ as $n\to\infty$ uniformly on  $\R\setminus \mathcal D$, where $\mathcal D$ is the countable set of $z$'s such that $\varphi_n$ is non-differentiable at $z$ for some $n\in\N$. By construction and from \eqref{eq:chainweakgrad}  we get that $\int \weakgrad{(\varphi\circ f-\varphi_n\circ f)}^2\,\d\mm \to 0$, so that the continuity assumption \eqref{eq:contprop} grants that $L(\varphi_n\circ f)\to L(\varphi\circ f)$ in $\MM$ as $n\to\infty$.  Given that we know that the chain rule \eqref{eq:chainprop} holds for $\varphi_n$, to get it for $\varphi$ it is sufficient to prove that
\[
\|\varphi_n'\circ fL(f)- \varphi'\circ fL(f)\|_\MM\to0,\qquad \text{ as}\ n\to\infty.
\]
To check this recall that by \eqref{eq:countnull} we have that $\nchi_{f^{-1}(\mathcal D)}L(f)=0$, and thus  it is sufficient to prove that $\nchi_{\X\setminus f^{-1}(\mathcal D)}(\varphi_n'\circ f- \varphi'\circ f)L(f)\to 0$ in $\MM$ as $n\to\infty$. But this is obvious, because
\[
\begin{split}
\|\nchi_{\X\setminus f^{-1}(\mathcal D)}(\varphi_n'\circ f- \varphi'\circ f)L(f)\|_\MM\leq\|\nchi_{\X\setminus f^{-1}(\mathcal D)}(\varphi_n'\circ f- \varphi'\circ f)\|_{L^\infty(\X)}\|L(f)\|_\MM,
\end{split}
\]
and by construction we have that $\|\nchi_{\X\setminus f^{-1}(\mathcal D)}(\varphi_n'\circ f- \varphi'\circ f)\|_{L^\infty(\X)}\to 0$ as $n\to\infty$. This proves the chain rule for  $\varphi\in C^1(\R)$ with bounded derivative. 

We are now ready to prove \eqref{eq:neglprop}. Assume at first that  $\mathcal N$ is compact and find a decreasing sequence of open sets $\Omega_n\subset \R$ containing $\mathcal N$ such that $\mathcal L^1(\Omega_n\setminus\mathcal N)\downarrow 0$. For each $n\in\N$ let $\psi_n:\R\to[0,1]$ be a continuous function identically 0 on $\mathcal N$ and identically 1 on $\Omega_n$ and define $\varphi_n:\R\to\R$ so that
\[
\varphi_n(0)=0,\qquad\qquad\text{and}\qquad\qquad\varphi_n'(z)=\psi_n(z),\qquad\forall z\in\R.
\]
Since $|\varphi_n'|(z)\leq 1$ for every $z\in\R$ and $n\in\N$ and $\varphi_n'(z)\to 1$ for every $z\in\R\setminus\mathcal N$, by \eqref{eq:nullgrad} and \eqref{eq:chainweakgrad} we see that   $\int\weakgrad{(f-\varphi_n\circ f)}^2\,\d\mm\to 0$ as $n\to\infty$ and therefore by the continuity \eqref{eq:contprop} we get $L(\varphi_n\circ f)\to L(f)$ in $\MM$. Hence using the chain rule for the $C^1$ functions with bounded derivative $\varphi_n$ and noticing that  $\varphi'_n=0$ on $\mathcal N$ for every $n\in\N$ we obtain
\[
\nchi_{f^{-1}(\mathcal N)}L(f)=\lim_{n\to\infty}\nchi_{f^{-1}(\mathcal N)}L(\varphi_n\circ f)=\lim_{n\to\infty}\nchi_{f^{-1}(\mathcal N)}\varphi'_n\circ f L( f)=0,
\]
which establishes the claim \eqref{eq:neglprop} for the case of $\mathcal N$ compact. For the general case, use the fact that $\mm$ is $\sigma$-finite to find a Borel probability measure $\mm'$ on $\X$ such that $\mm\ll\mm'\ll\mm$ and consider the Borel probability measure $\mu:=f_*\mm'$ on $\R$. Then $\mu$ is Radon and by internal regularity we can find a sequence of compact sets $\mathcal N_n\subset \mathcal N$ such that $\mu(\mathcal N\setminus\cup_n\mathcal N_n)=0$. Since $\mathcal N$ is $\mathcal L^1$-negligible, so are the $\mathcal N_n$'s and thus applying what we just proved to each of the $\mathcal N_n$'s and using the locality property \eqref{eq:locality} of the module $\MM$ we get \eqref{eq:neglprop}.

It remains to prove the chain rule \eqref{eq:chainprop} for arbitrary $\varphi: I\to\R$ Lipschitz. Up to extending $\varphi$ to the whole $\R$ without altering the Lipschitz constant, we can assume it to be defined on $\R$.  Find a Borel $\mathcal L^1$-negligible set $\mathcal N\subset \R$ such that $\varphi$ is differentiable outside $\mathcal N$ and a sequence $(\varphi_n)$ of $C^1$ and uniformly Lipschitz functions on $\R$ such that $\varphi_n'(z)\to\varphi'(z)$ as $n\to\infty$ for every $z\in \R\setminus\mathcal N$. Notice that  the chain rule \eqref{eq:chainweakgrad}  yields that $\int \weakgrad{(\varphi\circ f-\varphi_n\circ f)}^2\,\d\mm\to 0$ as $n\to\infty$ and therefore the continuity assumption \eqref{eq:contprop} gives
\begin{equation}
\label{eq:2014-2}
L(\varphi_n\circ f)\to L(\varphi\circ f),\qquad \text{ in }  \MM\ as\ n\to\infty.
\end{equation}
Now fix $C\in\BBB\X$  with $\mm(C)<\infty$ and $\eps>0$   and consider the functions $g_n,g\in L^\infty(\mm)$ defined as
\[
g_n:=\nchi_{C\setminus f^{-1}(\mathcal N)} \varphi_n'\circ f,\qquad\qquad g:=\nchi_{C\setminus f^{-1}(\mathcal N)} \varphi'\circ f.
\]
By construction we have $g_n\to g$ $\mm$-a.e.\ and $g=g_n=0$ $\mm$-a.e.\ outside $C$. Hence by the Severini-Egoroff theorem we obtain the existence of $A_\eps\in\BBB\X $ with $\mm(A_\eps)<\eps$ such that 
\begin{equation}
\label{eq:2014-3}
\|\nchi_{A_\eps^c}(g_n-g)\|_{L^\infty(\mm)}\to 0 \qquad \text{as } n\to\infty.
\end{equation}
Therefore using the chain rule for the $C^1$ functions with bounded derivative $\varphi_n$  we get
\[
\begin{split}
\nchi_{A_\eps^c}\nchi_C L(\varphi\circ f)&\stackrel{\eqref{eq:2014-2}}=\lim_{n\to\infty}\nchi_{A_\eps^c}\nchi_C L(\varphi_n\circ f)\stackrel{\eqref{eq:chainprop}}=\lim_{n\to\infty}\nchi_{A_\eps^c}\nchi_C\,\varphi'_n\circ f L( f)\\
&\stackrel{\eqref{eq:neglprop}}=\lim_{n\to\infty}\nchi_{A_\eps^c}\nchi_{C\setminus f^{-1}(\mathcal N)}\,\varphi'_n\circ f L( f)\stackrel{\eqref{eq:2014-3}}=\nchi_{A_\eps^c}\nchi_{C\setminus f^{-1}(\mathcal N)}\,\varphi'\circ f L( f)\\
&\stackrel{\eqref{eq:neglprop}}=\nchi_{A_\eps^c}\nchi_{C}\,\varphi'\circ f L( f),
\end{split}
\]
which is the same as to say that
\[
L(\varphi\circ f)=\varphi'\circ f L( f),\qquad\mm\ae\text{ on }  C\setminus A_\eps.
\]
Letting $\eps\downarrow0$ and using the locality property \eqref{eq:locality} of the module $\MM$ we deduce that  $L(\varphi\circ f)=\varphi'\circ f L( f)$ $\mm$-a.e.\ on $C$. Then the  $\sigma$-finiteness of $\mm$, the arbitrariness of $C$ with $\mm(C)<\infty$ and again the locality property of $\MM$ give the thesis.
\end{proof}
\begin{remark}{\rm
The existence of a relation between the three properties in Theorem \ref{thm:der} is certainly not a new insight. For instance, it is a very well known fact in basic differential geometry that the Leibniz rule implies locality and  in the metric setting it has been shown in \cite{AmbrosioKirchheim00} how to get  the  Leibniz and chain rules out of a local condition. In both cases the  axiomatization has  technical differences with ours, but the idea behind the proof is the same. 
}\fr\end{remark}
It is interesting to read Theorem \ref{thm:der} with the choice  $\MM=L^1(\mm)$: in this case there is no a priori explicit reference to any module structure in the statement. Still, the theorem  tells that an $L^\infty$-module structure naturally arises when considering differentiation properties of Sobolev functions, because the natural calculus rules are linked - in fact equivalent - to the locality property \eqref{eq:localprop} which is, in a sense, the defining property of an $L^\infty$-module. See also Theorem \ref{thm:dervf}.

\medskip

Specializing Theorem \ref{thm:der} to the case of the cotangent module yields:
\begin{corollary}[Chain and Leibniz rules for the differential]\label{cor:calcdiff}
We have
\[
\begin{array}{rlll}
\d(fg)&\!\!\!=g\,\d f+f\,\d g,\quad&\mm\ae,\\
\d f&\!\!\!=0,\quad&\mm\ae\ \text{\rm on }f^{-1}(\mathcal N),\\
\d(\varphi\circ f)&\!\!\!=\varphi'\circ f\,\d f,\quad&\mm\ae,\\
\end{array}
\]
for every $f,g\in\s^2(\X)$, $\mathcal N\subset \R$ Borel and $\mathcal L^1$-negligible, and $\varphi:\R\to\R$ Lipschitz, where $\varphi'\circ f$ is defined arbitrarily on $f^{-1}(\{\text{non differentiability points of $\varphi$}\})$.
\end{corollary}
\begin{proof}
Choose  $\MM:=L^2(T^*\X)$ and $L:=\d$ in Theorem  \ref{thm:der} and recall the identity \eqref{eq:diffuguali} to check the continuity of $\d:\s^2(\X)\to \MM$ and the locality property of the differential expressed in Theorem \ref{thm:diffloc} to conclude.
\end{proof}
We conclude the section discussing  the closure properties of the differential.
\begin{theorem}[Closure of $\d$]\label{thm:closd} Let $(f_n)\subset \s^2(\X)$ be converging $\mm$-a.e.\ to some $f\in L^0(\mm)$ and such that $(\d f_n)$ converges to some $\omega\in L^2(T^*\X)$ in the $L^2(T^*\X)$-norm.

Then $f\in \s^2(\X)$ and $\d f=\omega$.

In particular, if  $(f_n)\subset W^{1,2}(\X)$ is such that $f_n\weakto f$ and $\d f_n\weakto \omega$ for some $f\in L^2(\mm)$ and $\omega\in L^2(T^*\X)$ in the weak topologies of $L^2(\mm)$ and $L^2(T^*\X)$ respectively, then $f\in W^{1,2}(\X)$ and $\d f=\omega$.
\end{theorem}
\begin{proof} 
The assumptions grant that $|\d f_n|\to |\omega|$ in $L^2(\mm)$ and thus by \eqref{eq:diffuguali} and the stability property \eqref{eq:lscwug} we deduce that $f\in\s^2(\X)$. Taking into account that $\|\d f_n-\d f_m\|^2_{L^2(T^*\X)}=\int\weakgrad{(f_n-f_m)}^2\,\d\mm$, the same stability property applied to the sequence $n\mapsto f_n-f_m$ for given $m\in\N$ gives
\[
\|\d f-\d f_m\|_{L^2(T^*\X)}\leq \limi_{n\to\infty}\|\d f_m-\d f_n\|_{L^2(T^*\X)}.
\]
Taking the $\lims$ as $m\to\infty$ and using the fact that $(\d f_n)$ is $L^2(T^*\X)$-Cauchy we deduce that $\d f_n\to \d f$ in  $L^2(T^*\X)$, thus forcing $\d f=\omega$.

The second part of the statement now follows by a standard application of Mazur's lemma.
\end{proof}
The next statement is of genuine functional analytic nature:
\begin{proposition}[Weak compactness of $\d$  and reflexivity of $W^{1,2}(\X)$]\label{prop:compd}
The following are equivalent:
\begin{itemize}
\item[i)] $W^{1,2}(\X)$ is reflexive
\item[ii)] For every $W^{1,2}(\X)$-bounded sequence $(f_n)$ there exists a subsequence $(f_{n_k})$ and $f_\infty\in W^{1,2}(\X)$ such that $f_{n_k}\weakto f$ and $\d f_{n_k}\weakto \d f_\infty$ in the weak topologies of $L^2(\mm)$ and $L^2(T^*\X)$ respectively.
\end{itemize}
\end{proposition}
\begin{proof} By  the theorems of Eberlein-\v Smulian and  Kakutani (see e.g.\ \cite{Diestel84}), we know that  $(i)$ is equivalent to relative weak sequential compactness of $W^{1,2}$-bounded sets. Then the thesis easily follows observing that the map
\[
W^{1,2}(\X)\ni f\qquad\mapsto \qquad (f,\d f)\in L^2(\mm)\times L^2(T^*\X)
\]
is an isometry of $W^{1,2}(\X)$ with its image, the target space being endowed with the norm $\|(f,\omega)\|:=\sqrt{\|f\|_{L^2(\mm)}^2+\|\omega\|_{L^2(T^*\X)}^2}$.
\end{proof}
\begin{remark}{\rm
We don't know if  the reflexivity of $W^{1,2}(\X)$ implies that of $L^2(T^*\X)$.
}\fr\end{remark}
\begin{remark}\label{re:norefl}{\rm As we recalled in Theorem \ref{thm:ACM},  in \cite{ACM14}  the authors built  examples of a compact metric measure space $(\X,\sfd,\mm)$ such that $W^{1,2}(\X)$ is not reflexive.
}\fr\end{remark}

\subsection{Tangent module}

As in the previous section, we fix once and for all a metric measure space $(\X,\sfd,\mm)$ which is complete, separable and equipped with a non-negative Radon measure.

\subsubsection{Tangent vector fields and derivations}
Having both a notion of Sobolev function and of cotangent module we have at disposal two ways to introduce the tangent module: either in terms of derivations of  the former or by duality with the latter. 

Here we prove that the two approaches, when properly interpreted, are equivalent.  Then we introduce the concept of gradient of a Sobolev function and discuss its basic properties.

\vspace{1cm}

\begin{definition}[Tangent module]
The tangent module $L^2(T\X)$ is defined as the  dual of $L^2(T^*\X)$. Elements of the tangent module will be called vector fields.
\end{definition}
By point $(i)$ of Proposition \ref{prop:normevarie} we know that $L^2(T\X)$ is an $L^2(\mm)$-normed module. Despite the fact that we introduced it by duality, to keep consistency with the notation used in the smooth world we shall denote the pointwise norm in $L^2(T\X)$ by $|\cdot|$. Similarly, the duality between $\omega\in L^2(T^*\X)$ and $X\in L^2(T\X)$ will be denoted by $\omega(X)\in L^1(\mm)$.

We turn to the definition of derivation. Informally, a  derivation should be a linear operator satisfying the Leibniz rule taking functions in  $\s^2(\X)$ and returning $\mm$-a.e. defined functions.  Given that for $f\in \s^2(\X)$ we have $\weakgrad f\in L^2(\mm)$, it is quite natural by duality to think about $L^2$-derivations, i.e.\ satisfying the inequality
\begin{equation}
\label{eq:defder}
|L(f)|\leq \ell\weakgrad f,\qquad\mm\ae,\qquad\forall f\in\s^2(\X).
\end{equation}
for some $\ell\in L^2(\mm)$.

Interestingly, from this inequality alone the other expected properties of derivations follow, so we give the following definition:
\begin{definition}[Derivations] A derivation is a linear map $L:\s^2(\X)\to L^1(\mm)$ such that for some $\ell\in L^2(\mm)$ the bound \eqref{eq:defder} holds.
\end{definition}
To see that for such a defined object the Leibniz and chain rules hold,  notice that given a derivation $L$ we have
\[
|L(f-g)|\leq \ell \weakgrad{(f-g)}=0,\qquad\mm\ae\text{ on } \{f=g\},
\]
and thus
\begin{equation}
\label{eq:localder}
L(f)=L(g),\qquad\mm\ae\text{ on}\ \{f=g\}.
\end{equation}
Since clearly we also have 
\[
\| L(f)\|_{L^1(\mm)}\leq \|\ell\|_{L^2(\mm)}\|\weakgrad f\|_{L^2(\mm)},\qquad\forall f\in \s^2(\X),
\]
the locality property \eqref{eq:localder} and Theorem \ref{thm:der} applied to the module $\MM=L^1(\mm)$ ensure that indeed the Leibniz and chain rules \eqref{eq:leibprop}, \eqref{eq:chainprop} hold.

It is now easy to see that `vector fields' and `derivations' are in fact two different points of view of the same concept:
\begin{theorem}[Derivations and vector fields]\label{thm:dervf}
 For any vector field  $X\in L^2(T\X)$ the map $X\circ \d:\s^2(\X)\to \X$ is a derivation.  
 
Conversely, given a  derivation $L$ there exists a unique vector field  $X\in L^2(T\X)$ such that  the diagram
\begin{center}
\begin{tikzpicture}[node distance=2.5cm, auto]
  \node (S) {$\s^2(\X)$};
  \node (C) [right of=S] {$L^2(T^*\X)$};
  \node (L) [below  of=C] {$L^1(\mm)$};
  \draw[->] (S) to node {$\d$} (C);
  \draw[->] (C) to node {$X$} (L);
  \draw[->] (S) to node [swap] {$L$} (L);
\end{tikzpicture}
\end{center}
commutes.
\end{theorem}
\begin{proof} The map $X\circ\d$ is linear and satisfies
\[
|(X\circ \d)(f)|=|\d f(X)|\leq |X|\,|\d f|_*=|X|\,\weakgrad f,\qquad\mm\ae\qquad \forall f\in\s^2(\X).
\]
Since $|X|\in L^2(\mm)$,  the first claim is addressed.

For the second, let $L$ be a derivation satisfying \eqref{eq:defder} and consider the linear map from $\{\d f:f\in\s^2(\X)\}$ to $L^1(\X)$ defined by
\[
\d f\qquad\mapsto \qquad\tilde L(\d f):= L(f).
\]
Notice that inequality \eqref{eq:defder} (and the identity $|\d f|_*=\weakgrad f$) ensures that this map is well defined, i.e.\ its value depends only on the differential and not the function itself, and satisfies
\[
|\tilde L(\d f)|\leq \ell |\d f|_*.
\]
The conclusion then follows from Propositions \ref{prop:extension} and \ref{prop:gencotan}.
\end{proof}
The definition of gradient of a Sobolev function can be given via duality with that of differential:
\begin{definition}[Gradient]
Let $f\in\s^2(\X)$. We say that $X\in L^2(T\X)$ is a gradient of $f$ provided
\[
\d f(X)=|X|^2=|\d f|_*^2,\qquad\mm\ae.
\] 
The set of all gradients of $f$ will be denoted by ${\rm Grad}(f)$.
\end{definition}
This notion of gradient is in line with the one used in Finsler geometry, see for instance \cite{BCS00}.

Noticing that for arbitrary $X\in L^2(T\X)$ and $f\in\s^2(\X)$ we have
\begin{equation}
\label{eq:peraltragrad}
 \d f(X)\leq |\d f|_*|X|\leq \frac12|\d f|_*^2+\frac12|X|^2,\qquad\mm\ae,
\end{equation}
we see that
\begin{equation}
\label{eq:altragrad}
X\in{\rm Grad}(f)\qquad\Leftrightarrow\qquad \int \d f(X)\,\d\mm\geq \frac12\int |\d f|^2_*+|X|^2\,\d\mm.
\end{equation}

Corollary \ref{cor:perduali} ensures that ${\rm Grad}(f)$ is not empty  for every  $f\in \s^2(\X)$.  Uniqueness in general fails, the basic example being $\R^2$ equipped with the Lebesgue measure and the $L^\infty$-distance: in this case the linear function $(x_1,x_2)\mapsto x_1$ has as gradients, at every point, all the vectors of the form $(1,\alpha)$ with $|\alpha|\leq 1$.

\begin{proposition}[Gradients as generators of the tangent module]\label{prop:gentan}
The  subset of $L^2(T\X)$ of vector fields of the form
\[
\text{ $\sum_{i=1}^n\nchi_{A_i}X_i$ \qquad for some $n\in\N$, $(A_i)\subset \BBB\X$ and $X_i\in{\rm Grad}(f_i)$ with $f_i\in W^{1,2}(\X)$}
\]
is weakly$^*$-dense in $L^2(T\X)$.

In particular, if  $L^2(T\X)$ is reflexive, it  is generated, in the sense of modules, by $\cup_{f\in W^{1,2}(\X)}{\rm Grad}(f)$.
\end{proposition}
\begin{proof} Call $V$ the subset of $L^2(T\X)$ given by the statement and notice that since  the sets $A_i$ are not required to be disjoint, $V$ is a vector space. Also, let   $W\subset L^2(T^*\X)$ the  subset of $L^2(T^*\X)$ made of elements of the form
\[
\text{ $\sum_{i=1}^n\nchi_{A_i}\d f_i$ \qquad for some $n\in\N$, $(A_i)\subset \BB$ disjoint and  $f_i\in W^{1,2}(\X)$}.
\]
From Proposition \ref{prop:gencotan} and recalling the simple property \eqref{eq:dalp} we see that the $W$ is strongly dense in $L^2(T^*\X)$.

For given $\omega=\sum_{i=1}^n\nchi_{A_i}\d f_i\in W$ consider $X=\sum_{i=1}^n\nchi_{A_i}X_i\in V$ where $X_i\in{\rm Grad}(f_i)$ for every $i$ and observe that by definition of gradients  we have $\int \omega(X)\,\d\mm=\|\omega\|^2_{L^2(T^*\X)}$. This fact and the strong density of $W$ is sufficient to grant that if $\omega\in L^2(T^*\X)$ is such that $\int \omega(X)\,\d\mm=0$ for every $X\in V$, then $\omega=0$, which is precisely the weak$^*$-density of $V$. 

For the second claim, observe that if $L^2(T\X)$ is reflexive then its weak$^*$ topology coincides with the weak one and conclude applying Mazur's lemma.
\end{proof}
We now show how to deduce the locality and chain rules for gradients out of the same properties of differentials: the argument is based on the fact that the duality (multivalued) map from the cotangent to the tangent module given by Corollary \ref{cor:perduali} is 1-homogeneous. Notice that in general such map is non-linear, which means that we cannot expect the Leibniz rule for gradients to hold in general (see also Proposition \ref{prop:infhil})
\begin{proposition}[Basic calculus properties for gradients]
For  $f,g\in \s^2(\X)$ and $X\in{\rm Grad}(f)$ we have
\begin{equation}
\label{eq:localgrad00}
\nchi_{\{f=g\}}X=\nchi_{\{f=g\}}Y,\qquad\mm\ae,\qquad\text{\rm for some $Y\in{\rm Grad}(g)$}.
\end{equation}
Moreover, for $\mathcal N\subset \R$ Borel and negligible and $\varphi:\R\to\R$ Lipschitz we have
\begin{equation}
\label{eq:chaingrad00}
\begin{split}
X&=0,\qquad\mm\ae\ {\rm on }\ f^{-1}(\mathcal N),\\
 \varphi'\circ fX&\in{\rm Grad}(\varphi\circ f),
\end{split}
\end{equation}
where $\varphi'\circ f$ is defined arbitrarily on $f^{-1}(\{\text{non differentiability points of $\varphi$}\})$.
\end{proposition}
\begin{proof}
For \eqref{eq:localgrad00} let $Y\in{\rm Grad}(g)$ be arbitrary, recall the locality of the differential and the definition of gradient to obtain that $\nchi_{\{f=g\}}X+\nchi_{\{f\neq g\}}Y\in{\rm Grad}(g)$, which gives the claim.

The first in \eqref{eq:chaingrad00} follows directly from the analogous property of the differential, for the second put $\tilde X:=\varphi'\circ f$ and recall the chain rule for differentials to get
\[
\begin{split}
|\tilde X|=|\varphi'\circ f|\,|X|=|\varphi'\circ f|\,|\d f|_*=|\d(\varphi\circ f)|_*
\end{split}
\]
and
\[
\d(\varphi\circ f)(\tilde X)=\varphi'\circ f\d f(\tilde X)=|\varphi'\circ f|^2\d f(X)=|\varphi'\circ f|^2|\d f|_*^2=|\d(\varphi\circ f)|_*^2.
\]
\end{proof}

\subsubsection{On the duality between differentials and gradients}
In this section we briefly review the duality relation between differentials and gradients comparing the machinery built here with the approach used in \cite{Gigli12}. As a result we will see that they are fully consistent.

\vspace{1cm}

Notice that  for given $f,g\in\s^2(\X)$, inequality \eqref{eq:vectorstru} yields that the map $\eps\mapsto \frac12\weakgrad{(g+\eps f)}^2$ is $\mm$-a.e.\ convex, in the sense that for every $\eps_0,\eps_1\in\R$ the inequality
\[
\frac12\weakgrad{(g+((1-\lambda)\eps_0+\lambda\eps_1)f )}^2\leq(1-\lambda)\frac12\weakgrad{(g+\eps_0f)}^2+\lambda\frac12\weakgrad{(g+\eps_1f )}^2,\qquad\mm\ae,
\]
holds. It follows that for $\eps_0\leq \eps_1$, $\eps_0,\eps_1\neq 0$ we have
\begin{equation}
\label{eq:monotoneratio}
\frac{\weakgrad{(g+\eps_0f)}^2-\weakgrad g^2}{2\eps_0}\leq\frac{\weakgrad{(g+\eps_1f)}^2-\weakgrad g^2}{2\eps_1},\qquad\mm\ae,
\end{equation}
and in particular
\begin{equation}
\label{eq:sinistradestra}
\esssup_{\eps<0}\frac{\weakgrad{(g+\eps f)}^2-\weakgrad g^2}{2\eps}\leq \essinf_{\eps>0}\frac{\weakgrad{(g+\eps f)}^2-\weakgrad g^2}{2\eps},\qquad\mm\ae,
\end{equation}
where  the $\esssup$ and the $\essinf$ could be replaced by $\lim_{\eps\uparrow0}$ and $\lim_{\eps\downarrow0}$ respectively.

We then have the following result:
\begin{proposition}\label{prop:compatibili}
Let $f,g\in\s^2(\X)$. Then for every $X\in{\rm Grad}(g)$ we have
\begin{equation}
\label{eq:dpfng}
\d f(X)\leq\essinf_{\eps>0}\frac{\weakgrad{(g+\eps f)}^2-\weakgrad g^2}{2\eps},\qquad\mm\ae,
\end{equation}
and there exists $X_{f,+}\in{\rm Grad}(g)$ for which equality holds. Similarly, for every $X\in{\rm Grad}(g)$ we have
\begin{equation}
\label{eq:dmfng}
\d f(X)\geq\esssup_{\eps<0}\frac{\weakgrad{(g+\eps f)}^2-\weakgrad g^2}{2\eps},\qquad\mm\ae,
\end{equation}
and there exists $X_{f,-}\in{\rm Grad}(g)$ for which equality holds. 
\end{proposition}
\begin{proof}
Fix $\eps\in\R$  and recall the simple inequality \eqref{eq:peraltragrad} to get
\[
\d(g+\eps f)(X)\leq \frac12|\d(g+\eps f)|_*^2+\frac12|X|^2,\qquad\mm\ae.
\]
On the other hand, from the definition of ${\rm Grad}(g)$ we see that for $X\in{\rm Grad}(g)$ we have
\[
\d g(X)=\frac12|\d g|^2_*+\frac12|X|^2,\qquad\mm\ae.
\]
Subtracting this identity from the inequality above we deduce that
\[
\eps\,\d f(X)\leq \frac12\big(|\d(g+\eps f)|_*^2-|\d g|_*^2\big),\qquad\mm\ae,
\]
so that dividing by $\eps>0$ (resp. $\eps<0$) we obtain \eqref{eq:dpfng} (resp. \eqref{eq:dmfng}).

For the equality case in \eqref{eq:dpfng} we pick, for $\eps>0$, a vector field $X_\eps\in{\rm Grad}(g+\eps f)$ and notice that $\sup_{\eps\in(0,1)}\|X_\eps\|_{L^2(T\X)}<\infty$. Thus by the Banach-Alaoglu theorem the set $\{X_\eps\}_{\eps\in(0,1)}$ is weakly$^*$ relatively compact and hence there exists 
\[
X_{f,+}\in\bigcap_{\eps\in(0,1)}{\rm cl}^*\big(\{X_{\eps'}:\eps'\in(0,\eps)\}\big),
\]
where ${\rm cl}^*(A)$ denotes the weak$^*$-closure of $A\subset L^2(T\X)$.

By definition of $X_\eps$ and inequalities  \eqref{eq:peraltragrad} and \eqref{eq:altragrad}, for any $\eps>0$ we have
\begin{equation}
\label{eq:xeps}
\begin{split}
\int \d(g+\eps f)(X_\eps)\,\d\mm&\geq\frac12\|\d(g+\eps f)\|^2_{L^2(T^*\X)}+\frac12\|X_\eps\|^2_{L^2(T\X)},\\
\int g(X_\eps)\,\d\mm&\leq\frac12\|\d g\|^2_{L^2(T^*\X)}+\frac12\|X_\eps\|^2_{L^2(T\X)}.
\end{split}
\end{equation}
In particular from the first we obtain, after minor manipulation based on the trivial bound $\big||\d(g+\eps f)|_*-|\d g|_*\big|\leq \eps|\d f|_*$, that
\[
\frac12\|\d g\|^2_{L^2(T^*\X)}+\frac12\|X_{\eps'}\|^2_{L^2(T\X)}-\int \d g(X_{\eps'})\,\d\mm\leq  O(\eps),\qquad\forall \eps'\in(0,\eps),
\]
thus noticing that the left hand side of this expression is weakly$^*$ lower semicontinuous as a function of $X_{\eps'}$, by definition of $X_{f,+}$ we deduce that 
\[
\frac12\|\d g\|^2_{L^2(T^*\X)}+\frac12\|X_{f,+}\|^2_{L^2(T\X)}-\int \d g(X_{f,+})\,\d\mm\leq 0,
\]
which by \eqref{eq:altragrad}  is sufficient to deduce that $X_{f,+}\in {\rm Grad}(g)$.

Now subtracting the second from the first in \eqref{eq:xeps}   we obtain
\[
\int \d f(X_\eps)\,\d\mm\geq \int \essinf_{\tilde\eps>0}\frac{\weakgrad{(g+\tilde\eps f)}^2-\weakgrad g^2}{2\tilde\eps}\,\d\mm,\qquad\forall \eps>0,
\]
and again since the left hand side is weakly$^*$ continuous as a function of $X_\eps$, we deduce that 
\[
\int \d f(X_{f,+})\,\d\mm\geq \int \essinf_{\tilde\eps>0}\frac{\weakgrad{(g+\tilde\eps f)}^2-\weakgrad g^2}{2\tilde\eps}\,\d\mm.
\]
Given that $X_{f,+}\in{\rm Grad}(g)$, by \eqref{eq:dpfng} this inequality is sufficient to deduce that $\mm$-a.e.\ equality of the integrands.

The equality case in \eqref{eq:dmfng} is handled analogously.
\end{proof}
This proposition can be interpreted by saying that to know the duality relation between differentials and gradients of Sobolev functions it is not really necessary to know who differentials and gradients are, but is instead sufficient to have at disposal the notion of modulus of the differentials. This was precisely the approach taken in \cite{Gigli12} where the right hand sides of \eqref{eq:dpfng} and \eqref{eq:dmfng} have been taken by definition as basis of an abstract differential calculus on metric measure  spaces. With this proposition we thus showed that the approach of \cite{Gigli12} and the current one are fully compatible.

This kind of computation is also useful to recognize spaces where gradients are unique:
\begin{proposition}\label{prop:infstrconv}
The following are equivalent:
\begin{itemize}
\item[i)] for every $g\in\s^2(\X)$ the set ${\rm Grad}(g)$ contains exactly one element,
\item[ii)] for every $f,g\in\s^2(\X)$ equality holds $\mm$-a.e.\ in \eqref{eq:sinistradestra}.
\end{itemize}
\end{proposition}
\begin{proof}$\ $\\
\noindent{$\mathbf{(i)\Rightarrow (ii)}$} Direct consequence of Proposition \ref{prop:compatibili}: the two gradient vector fields $X_{f,+}$ and $X_{f,-}$ for which equality hold in \eqref{eq:dpfng} and \eqref{eq:dmfng} must coincide.

\noindent{$\mathbf{(ii)\Rightarrow (i)}$} Fix $g\in\s^2(\X)$, let $X_1,X_2\in{\rm Grad}(g)$ and notice that by assumption and Proposition \ref{prop:compatibili} we deduce that for every $f\in\s^2(\X)$ the equality $\d f(X_1)=\d f(X_2)$ holds $\mm$-a.e.. Thus for $\omega\in L^2(T^*\X)$ of the form $\omega=\sum_{i}\nchi_{A_i}\d f_i$ for some $f_i\in\s^2(\X)$ and $A_i\in\BB$, $i\in\N$, we also have $\omega(X_1)=\omega(X_2)$ $\mm$-a.e.. Since the set of such $\omega$'s is dense in $L^2(T^*\X)$, this is sufficient to conclude that $X_1=X_2$.
\end{proof}
In line with \cite{Gigli12} we thus give the following definition:
\begin{definition}[Infinitesimally strictly convex spaces]\label{def:infstrconv}
A complete separable metric space equipped with a non-negative Radon measure $(\X,\sfd,\mm)$ is said infinitesimally strictly convex provided the  two equivalent properties in Proposition \ref{prop:infstrconv} above hold. 

On infinitesimally strictly convex spaces, for $g\in\s^2(\X)$ the only element of ${\rm Grad}(g)$ will be denoted by $\nabla g$.
\end{definition}
\begin{remark}{\rm
The terminology comes from the fact that $\R^d$ equipped with a norm and the Lebesgue measure is infinitesimally strictly convex if and only if the norm is strictly convex, so that in a sense infinitesimal strict convexity speaks about the $\mm$-a.e.\  strict convexity of the norm in the tangent spaces. 

Still, we point out that in a general Finsler manifold we do not really know if being infinitesimally strictly convex in the sense of the above definition is the same as requiring the norm in the tangent module to be a.e.\ strictly convex: the point is that Definition \ref{def:infstrconv} only speaks about the strict convexity of the norm when looked on gradient vector fields.
}\fr\end{remark}

\subsubsection{Divergence}
In the smooth setting the (opposite of) the adjoint of the differential is the divergence. We have thus the possibility of introducing and studying the divergence in metric measure spaces, which  is the aim of this section.

\vspace{1cm}

We start with the definition:
\begin{definition}[Divergence]
The space $D(\div)\subset  L^2(T\X)$ is the set of all vector fields $X$ for which  there exists $f\in L^2(\mm)$ such that
\begin{equation}
\label{eq:defdiv}
\int fg\,\d\mm=-\int \d g(X)\,\d\mm,\qquad\forall g\in W^{1,2}(\X).
\end{equation}
In this case we call $f$ (which is unique by the density of $W^{1,2}(\X)$ in $L^2(\mm)$) the divergence of $X$ and denote it by $\div (X)$. 

We also introduce the functional $\mathcal E_{\rm div}:L^2(T\X)\to[0,\infty]$ as
\[
\mathcal E_{\rm div}(X):=\left\{\begin{array}{ll}
\displaystyle{\frac12\int |\div(X)|^2\,\d\mm},&\qquad\text{ if }X\in D(\div),\\
+\infty,&\qquad\text{ otherwise}.
\end{array}
\right.
\]
\end{definition}
Notice that the linearity of the differential grants that  $D(\div)$ is a vector space and that the divergence is a linear operator. In particular, $\mathcal E_\div$ is a quadratic form.

The Leibniz rule for differentials immediately gives the Leibniz rule for the divergence:
\begin{equation}
\label{eq:leibdiv}
\begin{split}
&\text{for  $X\in D(\div)$ and $f\in L^\infty \cap \s^2(\X)$ with $|\d f|_*\in L^\infty(\X)$}\\
&\text{we have $fX\in D(\div)$ and: }\qquad \div(fX)=\d f(X)+f\div X.
\end{split}
\end{equation}
Indeed with these assumptions we have $\d f(X)+f\div X\in L^2(\mm)$ and for every $g\in W^{1,2}(\X)$ we have $fg\in W^{1,2}(\X)$ and thus
\[
-\int gf\div (X)\,\d\mm=\int \d(fg)(X)\,\d\mm=\int g\d f(X)+\d g(f X)\,\d\mm,
\]
which is the claim.

Another direct consequence of the definition is the following duality formula:
\begin{proposition}[Dual representation of $\mathcal E_{\rm div}$]
$\mathcal E_{\rm div}$ is $L^2(T\X)$-lower semicontinuous and the duality formula
\begin{equation}
\label{eq:dualdiv}
\mathcal E_{\rm div}(X)=\sup_{g\in W^{1,2}(\X)}\int -\d g(X)\,\d\mm-\frac12\int |g|^2\,\d\mm.
\end{equation}
holds for every $X\in L^2(T\X)$.
\end{proposition}
\begin{proof}  For $L^2(T\X)$ lower semicontinuity notice that if $(X_n)\subset D(\div)$ is $L^2(T\X)$-converging to some $X$  with $(\div(X_n))$ uniformly bounded  in $L^2(\mm)$, then up to pass to a subsequence, not relabeled, the sequence $(\div(X_n))$ has a weak limit $f\in L^2(\mm)$. Conclude noticing   that $\|f\|_{L^2(\X)}\leq\limi_n\|\div(X_n)\|_{L^2(\X)}$ and that  the condition \eqref{eq:defdiv} passes to the limit, so  that $f=\div(X)$.

For what concerns \eqref{eq:dualdiv},  inequality $\geq$ in \eqref{eq:dualdiv} is obvious. To prove the other we can assume that  the $\sup$ in the right hand side of \eqref{eq:dualdiv} is finite. Call it $S$. Then  we have
\[
\int -\d g(X)\,\d\mm\leq S+\frac12\int|g|^2\,\d\mm,\qquad\forall g\in W^{1,2}(\X),
\]
and writing this inequality for $\lambda g$ in place of $g$ and optimizing in $\lambda\in\R$ we deduce that
\[
\int -\d g(X)\,\d\mm\leq \|g\|_{L^2(\mm)}\sqrt{2S},\qquad\forall g\in W^{1,2}(\X).
\]
In other words the map $W^{1,2}(\X)\ni g\mapsto \int -\d g(X)\,\d\mm\in\R$ is continuous w.r.t.\ the $L^2$-norm, and by the density of $W^{1,2}(\X)$ in $L^2(\mm)$ this map can be uniquely extended to a continuous linear functional on $L^2(\mm)$. By the Riesz theorem we deduce that there exists $f\in L^2(\mm)$ such that $ \int -\d g(X)\,\d\mm=\int fg\,\d\mm$, i.e., by definition, $X\in D(\div)$ and $\div(X)=f$.
\end{proof}
It is worth to underline that without further assumptions on the space we are not able to produce any non-zero vector field in $D(\div)$. In fact, the next proposition in conjunction with Remark \ref{re:norefl} shows that to find a large class of vector fields with a divergence might be impossible:
\begin{proposition}[Domain of the divergence and reflexivity of $W^{1,2}(\X)$]\label{thm:dchiuso} Assume that $D(\div)$ is dense in $L^2(T\X)$ w.r.t.\ the strong topology. Then $W^{1,2}(\X)$ is reflexive.
\end{proposition}
\begin{proof} We shall use Proposition \ref{prop:compd}. Thus let $(f_n)\subset W^{1,2}(\X)$ be a $W^{1,2}(\X)$-bounded sequence and find  a subsequence, not relabeled, such  that $f_n\weakto f$ as $n\to\infty$ in the weak topology of $L^2(\mm)$ for some $f\in L^2(\mm)$.

The $L^2(\mm)$-weak lower semicontinuity of $\E$ grants that
\[
\E(f)\leq\limi_{n\to\infty}\E(f_n)=\limi_{n\to\infty}\frac12\int|\d f_n|_*^2\,\d\mm\leq \sup_n\|f_n\|_{W^{1,2}(\X)}^2<\infty,
\]
and thus  $f\in W^{1,2}(\X)$. We now want to prove the convergence of $(\d f_n)$ to $\d f$ in the weak topology of $L^2(T^*\X)$. To this aim   let $X\in L^2(T\X)$, $X'\in D(\div)$, put  $S:=\sup_n\|\d f_n\|_{L^2(T^*\X)}\geq \|\d f\|_{L^2(T^*\X)}$ and notice that
\[
\begin{split}
\left|\int \d f_n(X)-\d f(X)\,\d\mm\right|&\leq\left|\int \d f_n(X')-\d f(X')\,\d\mm\right|+\int |\d f_n(X-X')|+|\d f(X-X')|\,\d\mm\\
&\leq\left|\int (f_n-f)\div (X')\,\d\mm\right|+2S\|X-X'\|_{L^2(T\X)}.
\end{split}
\]
Letting first $n\to\infty$ and then, using the assumption on the density of $D(\div)$ in $L^2(T\X)$, $X'\to X$ in $L^2(T\X)$ we deduce that $\int \d f_n(X)\,\d\mm\to\int \d f(X)\,\d\mm$. Since by Proposition \ref{prop:fulllp} every element of the Banach dual of $L^2(T^*\X)$ is of the form $\omega\mapsto\int \omega(X)\,\d\mm$ for some $X\in L^2(T\X)$, we just proved that $(\d f_n)$ weakly converges to $\d f$. The conclusion comes from  Proposition \ref{prop:compd}.
\end{proof}

Recalling the definition of Laplacian and that of its domain given at the end of Section \ref{se:presob}, we might ask whether in this setting the divergence of the gradient is the Laplacian. The answer is given by the next proposition:
\begin{proposition}\label{prop:divlapl} Let $f\in W^{1,2}(\X)$ and assume that there is $X\in{\rm Grad}(f)\cap D(\div)$. Then $\div(X)\in\partial^-E(f)$ and in particular $f\in D(\Delta)$. 

Conversely, if the space is infinitesimally strictly convex and $f\in D(\Delta)$, then $\nabla f\in D(\div)$ and $\div(\nabla f)=\Delta f$.
\end{proposition} 
\begin{proof} For given $f,X$ as  in the statement, let $g\in W^{1,2}(\X)$ be arbitrary and use the convexity of $\E$ to get
\[
\begin{split}
\E(f+g)-\E(f)&\stackrel{\phantom{\eqref{eq:dpfng}}}\geq \lim_{\eps\downarrow0}\int \frac{\weakgrad{(f+\eps g)}^2-\weakgrad f^2}{2\eps}\,\d\mm\stackrel{\eqref{eq:dpfng}}\geq \int \d g (X)\,\d\mm=-\int g\,\div(X)\,\d\mm,
\end{split}
\]
which is the first claim.

For the second, notice that by definition of $\Delta f$ for every $g\in W^{1,2}(\X)$ we have
\[
\E(f+\eps g)-\E(f)\geq -\int \eps g\Delta f\,\d\mm,\qquad\forall \eps\in\R.
\]
Then dividing by $\eps$ and letting $\eps\downarrow0$ and $\eps\uparrow 0$ and recalling  Proposition \ref{prop:infstrconv}, we see that
\[
\lim_{\eps\to 0}\frac{\E(f+\eps g)-\E(f)}{\eps}= -\int g\Delta f\,\d\mm.
\]
On the other hand, by Proposition \ref{prop:compatibili} we know that the limit at the left hand side is equal to $\int \d g(\nabla f)\,\d\mm$ and the thesis follows.
\end{proof}

Still in the realm of infinitesimally strictly convex spaces, we can prove the following natural duality formula for the energy $\E$. Notice that if we knew that   the Banach space $L^2(T^*\X)$ was reflexive, then such duality would  follow along the same lines used to prove \eqref{eq:dualdiv}, so that the interest of the statement is in the fact that no reflexivity assumption is made.
\begin{proposition}[Duality formula for $\E$]
Assume that $(\X,\sfd,\mm)$ is infinitesimally strictly convex. Then for every $f\in L^2(\mm)$ we have
\[
\E(f)=\sup_{X\in D(\div)}-\int f\div(X)\,\d\mm-\frac12\int |X|^2\,\d\mm. 
\]
\end{proposition}
\begin{proof} Inequality $\geq$ is obvious. For the converse we fix $f\in L^2(\mm)$ and let $f_t:=\h_t(f)$, where $(\h_t)$ is  the gradient flow   of $\E$ in $L^2(\mm)$ introduced at the end of Section \ref{se:presob}. Then $[0,\infty)\ni t\mapsto f_t\in L^2(\mm)$ is continuous and locally absolutely continuous on $(0,\infty)$, with $f_t\in D(\Delta)\subset \{\E<\infty\}=W^{1,2}(\X)$ for every $t>0$ and such that
\begin{equation}
\label{eq:caloremah}
\frac{\d}{\d t}f_t=\Delta f_t,\qquad {\rm a.e.}\ t>0,
\end{equation}
where the derivative in the left hand side is the strong $L^2$-limit of the incremental ratios.

Put $X_t:=\nabla f_t$ and $g(t):=\int-f_0\div(X_t)-\frac12|X_t|^2\,\d\mm$. The thesis will be achieved if we prove that $\lims_{t\downarrow0}g(t)\geq \E(f_0)$. Consider the function $[0,1]\ni t\mapsto d(t):=\frac12\|f_t-f_0\|_{L^2(\mm)}^2$. Being $(f_t)$ an $L^2$-absolutely continuous curve on $[0,1]$, $d$ is absolutely continuous. Taking into account \eqref{eq:caloremah}, for a.e.\ $t>0$ we have
\[
\begin{split}
\frac{\d}{\d t}d(t)&=\int(f_t-f_0)\Delta f_t\,\d\mm=\int-f_0\Delta f_t\,\d\mm+\int f_t\Delta f_t\,\d\mm=\int-f_0\Delta f_t\,\d\mm-\int |\d f_t|^2\,\d\mm.
\end{split}
\]
Since $\E(f_t)=\frac12\int|\d f_t|^2\,\d\mm=\frac12\int|X_t|^2\,\d\mm$, we deduce that 
\begin{equation}
\label{eq:parevada}
g(t)=\E(f_t)+\frac{\d}{\d t}d(t),\qquad {\rm a.e.}\ t>0.
\end{equation}
Now notice that $d(t)\geq 0$ for every $t\in[0,1]$ and that $d(0)=0$, thus  for the set $A:=\{t:\frac{\d}{\d t}d(t)\text{ exists and is non-negative}\}$ we have $\mathcal L^1(A\cap[0,\eps])>0$ for every $\eps>0$. Hence there exists a decreasing sequence $(t_n)$ converging to 0 such that $t_n\in A$ and \eqref{eq:parevada} holds for $t=t_n$ for every $n\in\N$. Therefore
\[
\lims_{t\downarrow0}g(t)\geq\lims_{n\to\infty} g(t_n)\geq \lims_{n\to\infty}\E(f_{t_n})\geq \E(f_0),
\]
having used the $L^2$-lower semicontinuity of $\E$ in the last step.
\end{proof}

\subsubsection{Infinitesimally Hilbertian spaces}
In this section we discuss how first order calculus works on spaces which, from the Sobolev calculus point of view, resemble Riemannian manifolds rather than the more general Finsler ones. The concept of infinitesimally Hilbertian space was introduced in \cite{Gigli12} to capture some key aspects of the analysis done in \cite{AmbrosioGigliSavare11}  and the content of Proposition \ref{prop:infhil} below is, more or less, a review of analogous results contained in \cite{Gigli12} in terms of the language we developed here.

\vspace{1cm}

We recall the following definition:
\begin{definition}[Infinitesimally Hilbertian spaces] 
A complete separable metric space equipped with a non-negative Radon measure $(\X,\sfd,\mm)$ is said  infinitesimally Hilbertian provided $W^{1,2}(\X)$ is an Hilbert space.
\end{definition}
This definition captures a `global' Hilbert property of the space, but the following result, and in particular point $(iii)$, shows that in fact it forces a `pointwise' Hilbertianity of the space, thus justifying the word `infinitesimal' in the terminology.
\begin{proposition}[Equivalent characterizations of infinitesimal Hilbertianity]\label{prop:infhil}
The following are equivalent:
\begin{itemize}
\item[i)] $(\X,\sfd,\mm)$ is infinitesimally Hilbertian.
\item[ii)] $(\X,\sfd,\mm)$ is infinitesimally strictly convex and 
\begin{equation}
\label{eq:simmdfgg}
\d f(\nabla g)=\d g(\nabla f),\qquad\mm\ae,\qquad\forall f,g\in\s^2(\X).
\end{equation}
\item[iii)]  $L^2(T^*\X)$ and $L^2(T\X)$ are Hilbert modules.
\item[iv)] $(\X,\sfd,\mm)$ is infinitesimally strictly convex and 
\begin{equation}
\label{eq:lingrad}
\nabla(f+g)=\nabla f+\nabla g,\qquad\mm\ae,\qquad\forall f,g\in\s^2(\X).
\end{equation}
\item[v)] $(\X,\sfd,\mm)$ is infinitesimally strictly convex and
\begin{equation}
\label{eq:leibgrad}
\nabla(fg)=f\nabla g+g\nabla f,\qquad\mm\ae,\qquad\forall f,g\in \s^2\cap L^\infty(\X).
\end{equation}
\end{itemize}
\end{proposition}
\begin{proof}$\ $\\
\noindent{$\mathbf{(i)\Rightarrow(ii)}$}  The assumption and the definition of $W^{1,2}$-norm yield that the map $W^{1,2}(\X)\ni f\mapsto \E(f):=\frac12\|\d f\|^2_{L^2(T^*\X)}$ is a quadratic form, i.e.
\begin{equation}
\label{eq:quade}
\|\d (f+g)\|^2_{L^2(T^*\X)}+\|\d (f-g)\|^2_{L^2(T^*\X)}=2\big(\|\d f\|^2_{L^2(T^*\X)}+\|\d g\|^2_{L^2(T^*\X)}\big),
\end{equation}
for every $f,g\in W^{1,2}(\X)$. 
Writing $\eps g$ in place of $g$ and noticing that $\E(\eps g)=\eps^2 \E(g)$, from the last identity we deduce that
\[
\lim_{\eps\downarrow0}\frac{\E(f+\eps g)-\E(f)}{2\eps}=\lim_{\eps\uparrow0}\frac{\E(f+\eps g)-\E(f)}{2\eps},
\]
which, together with the monotonicity conditions \eqref{eq:monotoneratio} and \eqref{eq:sinistradestra}, forces the equality to hold $\mm$-a.e.\ in \eqref{eq:sinistradestra} for any $f,g\in W^{1,2}(\X)$, i.e.\ the space is infinitesimally strictly convex. 

Moreover, \eqref{eq:quade} yields that for fixed $f,g\in W^{1,2}(\X)$ the map $\R^2\ni (t,s)\mapsto Q(t,s):= \|\d(tf+sg)\|^2_{L^2(T^*\X)}$ is a quadratic polynomial, in particular it is smooth and thus $\frac{\d}{\d t}\restr{t=0}\frac{\d }{\d s}\restr{s=0}Q(t,s)=\frac{\d}{\d s}\restr{s=0}\frac{\d }{\d t}\restr{t=0}Q(t,s)$. Expanding this equality recalling Proposition \ref{prop:compatibili} we deduce that
\begin{equation}
\label{eq:simmint}
\int \d f(\nabla g)\,\d\mm=\int \d g(\nabla f)\,\d\mm,\qquad\forall f,g\in W^{1,2}(\X)
\end{equation}
We need to pass from this integrated identity to the pointwise formulation \eqref{eq:simmdfgg}. To this aim, pick $f,h\in W^{1,2}\cap L^\infty(\X)$ and recall the Leibniz rule for the differential and the chain rule for the gradient  to get that
\begin{equation}
\label{eq:persimmpunt}
\begin{split}
\int h|\d f|^2\,\d\mm&=\int h\d f(\nabla f)\,\d\mm=\int \d(hf)(\nabla f)-f\d h(\nabla f)\,\d\mm\\
&=\int  \d(hf)(\nabla f)-\d h(\nabla (\tfrac{f^2}2))\,\d\mm\stackrel{\eqref{eq:simmint}}=\int  \d(hf)(\nabla f)-\d (\tfrac{f^2}2)(\nabla h)\,\d\mm.
\end{split}
\end{equation}
Now notice that since the map $W^{1,2}\cap L^\infty(\X)\ni f\mapsto\int \d f(\nabla h)\,\d\mm$ is linear, the map  $W^{1,2}\cap L^\infty(\X)\ni f\mapsto\int \d (\tfrac{f^2}2)(\nabla h)\,\d\mm$ is a quadratic form. Similarly, since both the maps  $W^{1,2}\cap L^\infty(\X)\ni f\mapsto\int \d (hf)(\nabla h')\,\d\mm$ and  $W^{1,2}\cap L^\infty(\X)\ni f\mapsto\int \d h'(\nabla f)\,\d\mm=\int \d f(\nabla h')\,\d\mm$ are linear for any $h,h'\in W^{1,2}\cap L^\infty(\X)$, we deduce that the map  $W^{1,2}\cap L^\infty(\X)\ni f\mapsto\int \d (hf)(\nabla f)\,\d\mm$ is a quadratic form.  Thus \eqref{eq:persimmpunt} yields that $W^{1,2}\cap L^\infty(\X)\ni f\mapsto \int h|\d f|^2\,\d\mm$ is a quadratic form and with the same arguments used to pass from \eqref{eq:quade} to \eqref{eq:simmint} we deduce that
\[
\int h\,\d f(\nabla g)\,\d\mm=\int h\,\d g(\nabla f)\,\d\mm,\qquad\forall f,g,h\in W^{1,2}\cap L^\infty(\X).
\]
Fix $f,g$ and use the weak$^*$-density of the set of $h$'s in $W^{1,2}\cap L^\infty(\X)$ in $L^\infty(\X)$ to deduce that \eqref{eq:simmdfgg} holds for $f,g\in W^{1,2}\cap L^\infty(\X)$. Conclude that the same identity holds for $f,g\in\s^2(\X)$ using the locality properties of differential and gradients and a truncation and cut-off argument.

\noindent{$\mathbf{(ii)\Rightarrow(iii)}$} Let $f,g\in\s^2(\X)$, $h\in L^\infty(\mm)$ and notice that the map $t\mapsto\frac12\int h|\d(f+tg)|^2\,\d\mm$ is Lipschitz. Computing its derivative using Proposition \ref{prop:compatibili}  we get
\[
\begin{split}
\int h(|\d (f+g)|^2-|\d f|^2)\,\d\mm&=\iint_0^1h\frac{\d}{\d t}|\d (f+t g)|^2\,\d t\,\d\mm=2\iint_0^1 h\,\d g(\nabla(f+tg))\,\d t\,\d\mm\\
\text{by \eqref{eq:simmdfgg}}\qquad&=2\iint_0^1 h\,\d(f+tg)(\nabla g)\,\d t\,\d\mm=\int h\,(2\,\d f(\nabla g)+|\d g|^2)\,\d\mm.
\end{split}
\]
Replacing $g$ with $-g$ and adding up, we see that $\s^2(\X)\ni f\mapsto \int h|\d f|^2\,\d\mm$ satisfies the parallelogram rule for any $h\in L^\infty(\mm)$ and from this arbitrariness and the linearity of the differential we further obtain that
\[
|\d f+\d g|^2+|\d f-\d g|^2=2\big(|\d f|^2+|\d g|^2\big),\qquad\mm\ae,\qquad\forall f,g\in \s^2(\X).
\]
By the very definition of $L^2(T^*\X)$-norm, we then see that the parallelogram rule is satisfied for forms of the kind $\sum_i\nchi_{A_i}\d f_i$ and since these are dense in $L^2(T^*\X)$ we deduce that $L^2(T^*\X)$, and thus its dual $L^2(T\X)$, is an Hilbert module.

\noindent{$\mathbf{(iii)\Rightarrow(iv)}$} By the Riesz theorem for Hilbert modules (Theorem \ref{thm:rhil}) we know that every vector field in $L^2(T\X)$ is of the form $L_\omega$ for some $\omega\in L^2(T^*\X)$, where $L_\omega(\omega'):=\la \omega,\omega'\ra$ for every $\omega'\in L^2(T^*\X)$. Pick $f\in\s^2(\X)$,  $X\in{\rm Grad}(f)$ and let $\omega\in L^2(T^*\X)$ such that $X=L_\omega$. Then by definition of ${\rm Grad}(f)$ we have $\la\omega,\d f\ra_{L^2(T^*\X)}=\frac12\|\d f\|^2_{L^2(T^*\X)}+\frac12\|\omega\|_{L^2(T^*\X)}^2$, which forces $\omega=\d f$. Thus the gradient is unique and linearly depends, by the linearity of the differential and of $\omega\mapsto L_\omega$, on the function, which is the thesis.

\noindent{$\mathbf{(iv)\Rightarrow(iii)}$} We know that $|\d f|^2= \d f(\nabla f)$ $\mm$-a.e.\  for any $f\in\s^2(\X)$. Hence from the linearity of the differential and using \eqref{eq:lingrad}, for any $f,g\in\s^2(\X)$ we get
\[
|\d(f+g)|^2=(\d f+\d g)(\nabla f+\nabla  g)=|\d f|^2+|\d g|^2+ \d f(\nabla g)+\d g(\nabla f),\qquad\mm\ae.
\]
Adding the analogous computation for $|\d(f-g)|^2$ we deduce  the parallelogram identity in $L^2(T^*\X)$ for forms of the kind  $\sum_i\nchi_{A_i}\d f_i$, $A_i\in\BB$, $f_i\in\s^2$ and since these are dense in $L^2(T^*\X)$ the claim is proved.

\noindent{$\mathbf{(iii)\Rightarrow(i)}$} Just recall that $\E(f)=\frac12\|\d f\|_{L^2(T^*\X)}^2$ and use the fact that $(L^2(T^*\X),\|\cdot\|_{L^2(T^*\X)})$ is an Hilbert space to conclude that $\E$ satisfies the parallelogram identity.

\noindent{$\mathbf{(iv)\Leftrightarrow(v)}$} With a truncation argument and recalling the locality property \eqref{eq:localgrad00} we see that \eqref{eq:lingrad} holds if and only if it holds for any $f,g\in\s^2\cap L^\infty(\X)$. Pick such $f,g$, let $f':=e^{f}$ and $g':=e^g$, notice that $f',g'\in\s^2\cap L^\infty(\X)$ and that on arbitrary infinitesimally strictly convex spaces, the chain rule \eqref{eq:chaingrad00} yields
\[
\begin{split}
f'g'\nabla(f+g)&=f'g'\nabla\log(f'g')=\nabla(f'g'),\\
f'g'(\nabla f+\nabla g)&=f'\nabla g'+g'\nabla f'.
\end{split}
\]
Thus if \eqref{eq:lingrad} holds the left hand sides of the above identities coincide, thus also the right hand sides coincide and \eqref{eq:leibgrad} is proved. Reversing the argument we obtain the converse implication.
\end{proof}
On infinitesimally Hilbertian spaces, the energy $\E:L^2(\mm)\to[0,\infty]$ is a Dirichlet form and the above proposition ensures that this forms admits a carr\'e du champ given precisely by $\la \nabla f,\nabla g\ra$, where $\la\cdot,\cdot\ra$ is the pointwise scalar product on the Hilbert module $L^2(T\X)$.

It is then easy to see (see also the arguments in Section \ref{se:clap}) that the Laplacian  $\Delta$ and its domain $D(\Delta)\subset W^{1,2}(\X)$ as we defined at the end of Section \ref{se:presob} can be equivalently characterized in the following more familiar way:
\begin{equation}
\label{eq:equivlapl}
f\in D(\Delta)\text{ and }h=\Delta f\qquad\Leftrightarrow\qquad \int gh\,\d\mm=-\int\la\nabla g,\nabla f\ra\,\d\mm\quad\forall g\in W^{1,2}(\X),
\end{equation}
which in particular shows that $D(\Delta)$ is a vector space and $\Delta:D(\Delta)\to L^2(\mm)$ a linear operator. 

A further consequence of infinitesimal Hilbertianity, and in particular of the Leibniz rule for gradients, is that we have the Leibniz rule for the Laplacian. A version of this formula that we shall use later on is:
\begin{equation}
\label{eq:leiblap}
\left.
\begin{array}{l}
f,g\in D(\Delta)\cap L^\infty(\mm)\\
\\
|\d f|_*,|\d g|_*\in L^\infty(\mm)
\end{array}
\right\}
\qquad\Rightarrow\qquad 
\left\{
\begin{array}{l}
fg\in D(\Delta)\quad\text{ with }\\
\\
\Delta(fg)=f\Delta g+g\Delta f+2\la\nabla f,\nabla g\ra
\end{array}
\right.
\end{equation}
To check this, notice that the assumptions on $f,g$ grant that for any $h\in W^{1,2}(\X)$ we have $fh,gh\in W^{1,2}(\X)$ and that the expression for $\Delta(fg)$ is in $L^2(\mm)$, thus using the characterization \eqref{eq:equivlapl} for the Laplacians of $f,g$ we get
\[
\begin{split}
-\int\la \nabla h,\nabla(fg)\ra\,\d\mm&=-\int g\la\nabla h,\nabla f\ra+f\la\nabla h,\nabla g\ra\,\d\mm\\
&=\int- \la\nabla( h g),\nabla f\ra-\la\nabla( h f),\nabla g\ra+2h\la\nabla f,\nabla g\ra\,\d\mm\\
&=\int h g\Delta f+h f\Delta g+2h\la\nabla f,\nabla g\ra\,\d\mm,
\end{split}
\]
which is the claim.

\subsubsection{In which sense the norm on the tangent space induces the distance}\label{se:normdistance}
In a smooth Riemannian/Finslerian manifold there is a direct link between the norm on the tangent space and the distance on the manifold, as we have, by definition in fact, that
\begin{equation}
\label{eq:normdistsmooth}
\sfd^2(x,y)=\inf\int_0^1|\gamma'_t|^2\,\d t,
\end{equation}
the $\inf$ being taken among all smooth curves $\gamma$ on $[0,1]$ connecting $x$ to $y$, $\gamma'_t$ being the usual derivative of $\gamma$ at time $t$. As a consequence of this fact/definition we also have that the metric speed of a curve equals the norm of its derivative:
\begin{equation}
\label{eq:lengthsmooth}
|\dot\gamma_t| = |\gamma_t'|,\qquad{\rm a.e. }\ t.
\end{equation}
The question then is: can these identities be interpreted in the non-smooth setting so that a link between the metric and the differential side of the story can be established in this generality?

As we try to show in this section, the answer is affirmative. We shall propose two, different but related, interpretations of the latter identity \eqref{eq:lengthsmooth}: in the first (Theorem \ref{thm:spppi}) we see in which sense  \eqref{eq:lengthsmooth} holds for $\ppi$-a.e.\ curve $\gamma$ for any given test plan $\ppi$, while in the second (Theorem \ref{thm:conteq}) rather than studying curves of points we study curves of probability measures with bounded densities.

We then discuss the relations between these two approaches and see under which conditions and in which sense we can reobtain \eqref{eq:normdistsmooth} in metric measure spaces.

\vspace{1cm}

Given a test plan $\ppi$ we would like to define for  a.e.\ $t\in[0,1]$ and $\ppi$-a.e.\ $\gamma$ the tangent vector $\gamma_t'$, i.e.\ for a.e.\ $t$ we would like to be able to assign to $\ppi$-a.e.\ $\gamma$ an element $\gamma_t'$ of the `tangent space of $\X$ at $\gamma_t$'. To a give a rigorous meaning to such object,  we need to call into play the notion of pullback of the tangent module. Thus recall that $C([0,1],\X)$ equipped with $\sup$ distance is a complete and separable metric space and that a test plan $\ppi$ is a Borel measure on it such that for every $t\in[0,1]$ the map $\e_t:(C([0,1],\X),\ppi)\to (\X,\mm)$ has bounded compression (Definition \ref{def:bcompr}). Thus it makes sense to consider the pullback of the cotangent and tangent modules via this map: we shall denote them as $L^2(T^*\X,\ppi,\e_t)$ and $L^2(T\X,\ppi,\e_t)$ respectively.

Recall that by Theorem \ref{thm:dualpullback} we know that if $L^2(T\X)$ is separable, then $L^2(T\X,\ppi,\e_t)$ is (=can be canonically identified with) the dual of $L^2(T^*\X,\ppi,\e_t)$. We then have the following result:
\begin{theorem}[The vector fields $\ppi'_t$]\label{thm:spppi} Let $(\X,\sfd,\mm)$ be such that $L^2(T\X)$ is separable with $\mm$ giving finite mass to bounded sets and $\ppi\in \prob{C([0,1],\X)}$ a test plan.

Then there exists a unique, up to  $\mathcal L^1$-a.e.\ negligible sets, family of  vector fields $\ppi'_t\in L^2(T\X,\ppi,\e_t)$ such that for  every $f\in W^{1,2}(\X)$ we have
\[
\lim_{h\to0 }\frac{f\circ \e_{t+h}-f\circ\e_{t}}{h}=(\e_t^*\d f)(\ppi'_t)\qquad\text{  for a.e.\ $t\in[0,1]$,}
\]
the limit being intended in the strong topology of $L^1(\ppi)$.

Moreover, the map $(\gamma,t)\mapsto |\ppi'_t|(\gamma)$ is (the equivalence class w.r.t.\ $\ppi\times\mathcal L^1\restr{0,1}$-a.e.\ equality of) a Borel map which satisfies
\begin{equation}
\label{eq:veluguali}
|\ppi'_t|(\gamma)=|\dot\gamma_t|,\qquad\ppi\times\mathcal L^1\restr{0,1}\ae\ (\gamma,t).
\end{equation}
\end{theorem}
\begin{proof} It will be technically convenient for the proof to consider a precise representative for the metric speed, thus for $\gamma$ absolutely continuous put
\[
|\dot{\bar\gamma}_t|:=\lim_{h\to 0}\frac{\sfd(\gamma_{t+h},\gamma_t)}{|h|},\qquad\text{if the limit exists, 0 otherwise,}
\]
and recall that Theorem \ref{thm:112} ensures that $|\dot{\bar\gamma}_t|=|\dot\gamma_t|$ for a.e.\ $t$.

Let $\bar m(t)$ be the centered maximal function of the $L^1(0,1)$ map $t\mapsto\sqrt{\int|\dot\gamma_t|^2\,\d\ppi(\gamma)}$ and notice that
\begin{equation}
\label{eq:maxf}
\sqrt{\int_t^{t+h}\int|\dot\gamma_s|^2\,\d\ppi(\gamma)\,\d s}\leq \sqrt{\int_{t-h}^{t+h}\int|\dot\gamma_s|^2\,\d\ppi(\gamma)\,\d s}\leq \sqrt{2h\bar m(t)},\qquad\forall t\in[0,1].
\end{equation}
For $f\in W^{1,2}(\X)$ and $t,h\in[0,1]$ such that $t+h\in[0,1]$ consider the incremental ratio ${\rm IR}(f;t,h):=\frac{f\circ\e_{t+h}-f\circ\e_t}{h}\in L^1(\ppi)$ and notice that from the very definition of test plan and Sobolev function we have
\begin{equation}
\label{eq:irb}
\begin{split}
\|{\rm IR}(f;t,h)\|_{L^1(\sppi)}&\leq\frac1h\iint_t^{t+h} \weakgrad f (\gamma_s)|\dot\gamma_s|\,\d\ppi(\gamma)\,\d s\\
&\leq\frac1h\sqrt{\iint_t^{t+h} \weakgrad f ^2(\gamma_s)\,\d\ppi(\gamma)\,\d s}\,\sqrt{\iint_t^{t+h} |\dot\gamma_s|^2\,\d\ppi(\gamma)\,\d s}\\
\text{using \eqref{eq:maxf}}\qquad&\leq\sqrt{2\,\cf(\ppi)\bar m(t)}\|f\|_{W^{1,2}(\X)}.
\end{split}
\end{equation}
Now observe that if $f\in W^{1,2}(\X)$ has a Lipschitz continuous representative $\bar f$, then since $\ppi$ is concentrated on absolutely continuous curves, for $\ppi$-a.e.\ $\gamma$ the map $t\mapsto \bar f(\gamma_t)$ is absolutely continuous and satisfies
\[
|\frac{\d}{\d t}\bar f(\gamma_t)|\leq {\rm lip} (\bar f)(\gamma_t)|\dot {\bar \gamma}_t|,\qquad{\rm a.e. }\ t.
\]
Thus by Fubini and using the dominate convergence theorem we get that for a.e.\ $t\in[0,1]$
\begin{equation}
\label{eq:limir}
\begin{split}
&{\rm IR}(f;t,h)\text{ has a strong limit ${\rm Der}(f,t)$ in $L^1(\ppi)$ as $h\downarrow0$}\\
&\text{and the limit is bounded in modulus by }{\rm lip} (\bar f)(\gamma_t)|\dot {\bar\gamma}_t|\text{ for $\ppi\ae$ $\gamma$}
\end{split}
\end{equation}
Now let $f\in W^{1,2}(\X)$ and $(f_n)\subset W^{1,2}(\X)$ be a sequence $W^{1,2}$-converging to $f$ made of functions with Lipschitz representatives (recall \eqref{eq:aplip2}). Then for every $t\in[0,1]$ such that  $\bar m(t)<\infty$ and  ${\rm IR}(f_n;t,h)$ has a strong limit in $L^1(\ppi)$ as $h\downarrow0$ for every $n\in\N$, from \eqref{eq:irb} we have
\[
\begin{split}
\lims_{h,h'\downarrow 0}\|{\rm IR}(f;t,h)-{\rm IR}(f;t,h')\|_{L^1(\sppi)}&\leq\lims_{h,h'\downarrow 0}\|{\rm IR}(f_n;t,h)-{\rm IR}(f_n;t,h')\|_{L^1(\sppi)}\\
&\qquad\qquad+ 2\sqrt{2C(\ppi)\bar m(t)}\|f-f_n\|_{W^{1,2}(\X)}\\
&=2\sqrt{2C(\ppi)\bar m(t)}\|f-f_n\|_{W^{1,2}(\X)},
\end{split}
\]
so that letting $n\to\infty$ we deduce that ${\rm IR}(f;t,h)$ strongly converges to some ${\rm Der}(f,t)$ in $L^1(\ppi)$ as $h\downarrow0$. The same kind of computation also shows that ${\rm Der}(f,t)=\lim_{n\to\infty}{\rm Der}(f_n,t)$, the limit being intended in $L^1(\ppi)$.

Using again \eqref{eq:aplip2} we see that the approximating sequence $(f_n)$ can be chosen so that ${\rm lip}(\bar f_n)\to \weakgrad f $ in $L^2(\mm)$ as $n\to\infty$, which implies that ${\rm lip}(\bar f_n)(\gamma_t)|\dot{\bar \gamma}_t|\to \weakgrad f(\gamma_t) |\dot{\bar\gamma}_t|$ in $L^1(\ppi)$ as $n\to\infty$ for a.e.\ $t$. Thus passing to the limit in the bound given in \eqref{eq:limir} we deduce that
\begin{equation}
\label{eq:bder}
|{\rm Der}(f,t)|(\gamma)\leq \weakgrad f (\gamma_t) |\dot{\bar\gamma}_t|,\qquad\ppi\ae\ \gamma.
\end{equation}
Repeating the argument for $h\uparrow0$ we thus proved that for every $f\in W^{1,2}(\X)$ the following holds: for a.e.\ $t\in[0,1]$ we have ${\rm IR}(f;t,h)\to{\rm Der}(f,t)$ in $L^1(\ppi)$ as $h\downarrow 0$ for some ${\rm Der}(f,t)$ for which the bound \eqref{eq:bder} holds.

Now use the separability assumption on $L^2(T\X)$ and the fact that such space is isometric to the Banach dual of $L^2(T^*\X)$ to deduce that the latter is separable and therefore that $W^{1,2}(\X)$ is separable as well (recall the isometric embedding $f\mapsto (f,\d f)$ of $W^{1,2}(\X)$ into $L^2(\mm)\times L^2(T^*\X)$). Then consider a countable dense $\Q$-vector space $D\subset W^{1,2}(\X)$ and let $A\subset[0,1]$ be a Borel set of full measure such that for every $t\in A$ we have $\bar m(t)<\infty$ and for every $f\in D$ the incremental ratios ${\rm IR}(f;t,h)$ converge to some ${\rm Der}(f,t)$ in $L^1(\ppi)$ as $h\to 0$ for which \eqref{eq:bder} holds.

The linearity of $f\mapsto{\rm IR}(f;t,h)$ ensures that for every $t\in A$ the map $D\ni f\mapsto {\rm Der}(f,t)\in L^1(\ppi)$ is $\Q$-linear, while  \eqref{eq:bder} grants that it is continuous. Hence it can be uniquely extended to a linear continuous map, still denoted by ${\rm Der}(\cdot,t)$ from $W^{1,2}(\X)$ to $L^1(\ppi)$. Using again the equicontinuity  of ${\rm IR}(\cdot;t,h)$ it is then clear that ${\rm IR}(f;t,h)\to {\rm Der}(f,t)$ in $L^1(\ppi)$ as $h\to 0$ for every $t\in A$ and that the bound \eqref{eq:bder} holds.

Now for $t\in A$ consider the vector space $V:=\{\e_t^*\d f\ :\ f\in W^{1,2}(\X)\}$, notice that by Proposition \ref{prop:gencotan} and property \eqref{eq:genpullback} it generates, in the sense of modules, the whole $L^2(T^*\X,\ppi,\e_t)$ and consider the linear map $L:V\to L^1(\ppi)$ sending $\e_t^*\d f$ to ${\rm Der}(f,t)$: the bound \eqref{eq:bder} grants that this map is well defined and we can restate such bound as
\begin{equation}
\label{eq:bder2}
|L(\e_t^*\d f)|(\gamma)\leq |\e_t^*\d f|(\gamma)\,|\dot{\bar\gamma}_t|.
\end{equation}
By Proposition \ref{prop:extension} we then deduce that $L$ can be uniquely extended to a an  element of the dual module of $L^2(T^*\X,\ppi,\e_t)$ and the assumption of separability  of $L^2(T\X)$ and Theorem \ref{thm:dualpullback} grants that such element is in fact an element $\ppi'_t$ of the pullback of the dual module $L^2(T\X,\ppi,\e_t)$. The Borel regularity of $(\gamma,t)\mapsto |\ppi'_t|(\gamma)$ follows from the duality formula
\[
\frac12|\ppi'_t|^2(\gamma)=\esssup_{f\in W^{1,2}(\X)}\e_t^*\d f(\ppi'_t)-\frac12\weakgrad f^2(\gamma_t)=\esssup_{f\in W^{1,2}(\X)}{\rm Der}(f,t)(\gamma)-\frac12\weakgrad f^2(\gamma_t),
\]
the essential supremum being intended w.r.t.\ $\ppi\times\mathcal L^1\restr{[0,1]}$. Using again  Proposition \ref{prop:extension} and the bound \eqref{eq:bder2} we get the inequality $\leq$ in \eqref{eq:veluguali}.

To prove the opposite inequality notice that by definition we have
\[
|{\rm Der}(f,t)|(\gamma)=|\e_t^*\d f(\ppi'_t)(\gamma)|\leq \weakgrad f(\gamma_t)|\ppi'_t|(\gamma),\qquad\ppi\times\mathcal L^1\restr{[0,1]}\ae\ (\gamma,t),
\]
so that for every $f\in W^{1,2}(\X)$ with a 1-Lipschitz continuous representative $\bar f$ we have
\[
\frac{\d}{\d t}\bar f(\gamma_t)\leq |\ppi'_t|(\gamma),\qquad\ppi\times\mathcal L^1\restr{[0,1]}\ae\ (\gamma,t).
\]
To conclude it is therefore sufficient to show that there exists a countable family $D'\subset W^{1,2}(\X)$ of functions having 1-Lipschitz representatives such that for every absolutely continuous curve $\gamma$ it holds
\begin{equation}
\label{eq:persar}
\sup_{\bar f\in D'}\frac{\d}{\d t}\bar f(\gamma_t)\geq |\dot\gamma_t|,\qquad{\rm a.e. }\ t.
\end{equation}

This follows by standard means, indeed assume for a moment that bounded subsets of $\X$ have finite $\mm$-measure, let $(x_n)\subset\X$ be countable and dense and define $\bar f_n:=\max\{1-\sfd(\cdot,x_n),0\}$. Then clearly the $\bar f_n$'s are 1-Lipschitz representatives of functions in $W^{1,2}(\X)$. Now fix an absolutely continuous curve $\gamma$, and notice that for every $t\in[0,1]$ the very definition of $\bar f_n$ ensures that for $s\in[t,1]$ sufficiently close to $t$ it holds $\sfd(\gamma_t,\gamma_s)= \sup_n\bar f_n(\gamma_t)-\bar f_n(\gamma_s)$ and therefore 
\[
\sfd(\gamma_t,\gamma_s)= \sup_n\bar f_n(\gamma_t)-\bar f_n(\gamma_s)\leq\sup_n\int_t^s\frac{\d}{\d r}\bar f_n(\gamma_r)\,\d r\leq\int_t^s\sup_n\frac{\d}{\d r}\bar f_n(\gamma_r)\,\d r.
\]
A simple application of the triangle inequality shows that the above inequality holds for every $t<s$ and thus, by the very definition \ref{def:speedcurve} of metric speed, the claim \eqref{eq:persar} and the conclusion.  

To drop the assumption that bounded sets have finite mass, use the Lindel\"of property of $(\X,\sfd)$ to cover $\X$ with countably many balls of finite mass and replicate the above argument in each of the balls.
\end{proof}
\begin{remark}\label{rem:sugo}{\rm
In the special case in which the evaluation maps $\e_t$ are $\ppi$-a.e.\ invertible, from the functorial property of the operation of pullback of a module we see that in this case it is not necessary to consider any pullback at all. Indeed, let $L^0(T\X)$ be the $L^0(\mm)$-module $\MM^0$ obtained as indicated in Section \ref{se:altint} from the module $\MM:=L^2(T\X)$, put $\mu_t:=(\e_t)_*\ppi$, let $L^2(T\X,\mu_t)\subset L^0(T\X)$ be the $L^2(\mu_t)$-normed module of elements with finite $L^2(\mu_t)$ norm and assume that for some Borel map $\e_t^{\rm inv}:\X\to C([0,1],\X)$ we have $\e_t(\e_t^{\rm inv})=\Id_\X$ $\mu_t$-a.e.\ and $\e_t^{\rm inv}(\e_t)=\Id_{C([0,1],\X)}$ $\ppi$-a.e. for every $t\in[0,1]$. 

Then the modules $L^2(T\X,\ppi,\e_t)$ and $L^2(T\X,\mu_t)$ can be identified via the invertible map $X\mapsto (\e_t^{\rm inv})^*X$ and  the vector fields $\ppi_t'\in L^2(T\X,\mu_t)$ can be defined via their actions on Sobolev functions $f\in W^{1,2}(\X)$ via the formula
\[
\d f(\ppi'_t)=\lim_{h\to 0}\frac{f\circ\e_{t+h}\circ\e_t^{\rm inv}-f}{h}.
\]
}\fr\end{remark}
\begin{problem}[Vector fields as derivations along curves]\label{op:vc}{\rm
One can `reverse' the content of this last theorem and ask whether given a vector field $X\in L^2(T\X)$ and a measure $\mu\in\prob{\X}$ such that $\mu\leq C\mm$ for some $C>0$ there exists a test plan $\ppi$ such that $(\e_0)_*\ppi=\mu$ and $\ppi'_0=\e_0^*X$. Notice that part of the problem here is to show that it makes sense to speak about the vector field $\ppi_t'$ at the time $t=0$.  In other words: can we see every vector field as derivation at time 0 along a certain family of curves? 

For gradient vector fields and infinitesimally strictly convex spaces, a result in \cite{Gigli12} ensures a positive answer to this question, but the general case is open. See also Remark \ref{rem:rlf}.
}\fr\end{problem}

We turn to the interpretation of \eqref{eq:lengthsmooth} in terms of curves of probability measures. We refer to \cite{Villani09} and \cite{AmbrosioGigli11}  for the definition of the quadratic transportation distance $W_2$ on the space $\probt\X$ of Borel probability measures with finite second moment.
\begin{definition}[Curves of bounded compression]\label{def:boundcompr}
We say that a curve $(\mu_t)\subset \probt{\X}$ is of bounded compression provided it is $W_2$-continuous and for some $C\geq 0$ we have $\mu_t\leq C\mm$ for every $t\in[0,1]$.
\end{definition}
From  Otto's interpretation \cite{Otto01} of the space $(\probt \X,W_2)$ as a sort of Riemannian manifold, when the base space $\X$ is a smooth Riemannian manifold, we know that we might consider as `velocity' of a curve of measures the `optimal' vector fields for which the continuity equation is satisfied, where optimality is intended in the sense of having minimal kinetic energy. 

We want to push this interpretation up to our framework, and to do so we have to start defining what are solutions of the continuity equation. To this aim, and as in Remark \ref{rem:sugo}, we introduce the space $L^0(T\X)$ as the $L^0(\mm)$-module $\MM^0$ obtained as indicated in Section \ref{se:altint} from the module $\MM:=L^2(T\X)$.
\begin{definition}[Distributional solutions of the continuity equation]\label{def:solconteq}
Let $(\mu_t)\subset \probt \X$ be a curve of bounded compression and $(X_t)\subset  L^0(T\X)$ a family of vector fields. We say that $(\mu_t,X_t)$ solves the continuity equation
\[
\frac\d{\d t}\mu_t+\nabla\cdot(X_t\mu_t)=0,
\]
provided:
\begin{itemize}
\item[i)] the map $t\mapsto\int |X_t|^2\,\d\mu_t$ is Borel and 
\[
\int_0^1\int|X_t|^2\,\d\mu_t\,\d t<\infty,
\]
\item[ii)]  for every $f\in W^{1,2}(\X)$ the map $t\mapsto\int f\,\d\mu_t$ is absolutely continuous and satisfies
\[
\frac{\d}{\d t}\int f\,\d\mu_t=\int \d f(X_t)\,\d\mu_t,\qquad \text{\rm a.e. }t\in[0,1].
\]
\end{itemize}
\end{definition}
We want to show that this genuinely differential notion of solution of the continuity equation completely characterizes $W_2$-absolutely continuous curves of bounded compression. In fact, the `hard work' to obtain this characterization has been already carried out in \cite{GigliHan13}, where the following statement has been proved:
\begin{theorem}\label{thm:bg}
Let $(\X,\sfd,\mm)$ be such that $W^{1,2}(\X)$ is separable and $(\mu_t)\subset \probt{\X}$ a curve of bounded compression. Then the following are equivalent:
\begin{itemize}
\item[i)] $(\mu_t)$ is $W_2$-absolutely continuous with $|\dot\mu_t|\in L^2(0,1)$,
\item[ii)] for every $f\in W^{1,2}(\X)$ the map $t\mapsto\int f\,\d\mu_t$ is absolutely continuous and there exists a Borel negligible set $\mathcal N\subset[0,1]$ and for each $t\in[0,1]\setminus\mathcal N$ a linear functional $L_t:\s^2(\X)\to \R $ ($\mathcal N$ and the $L_t$'s being independent on $f$) such that for $t\in[0,1]\setminus \mathcal N$ we have
\[
\frac{\d}{\d t}\int f\,\d\mu_t=L_t(f),\qquad\forall f\in W^{1,2}(\X),
\]
and moreover putting 
\[
\|L_t\|_{\mu_t}:=\sup |L_t(f)|,\text{ the $\sup$ being taken among all $f\in\s^2(\X)$ with } \int|\d f|_*^2\,\d\mu_t\leq 1,
\]
we have that $[0,1]\setminus\mathcal N\ni t\mapsto \|L_t\|_{\mu_t}$ is Borel with $\int_0^1\|L_t\|_{\mu_t}^2\,\d t<\infty$.
\end{itemize}
If these hold, we also have
\begin{equation}
\label{eq:thevo}
|\dot\mu_t|=\|L_t\|_{\mu_t}\qquad {\rm a.e.} \ t\in[0,1].
\end{equation}
\end{theorem}
It is then not hard to interpret this result in terms of the above notion of solution of the continuity equation. Notice that the next theorem is fully equivalent to the analogous statement proved for the Euclidean space in \cite{AmbrosioGigliSavare08}, the only difference being in the requirement that the curve has bounded compression, which has the effect of `averaging out the unsmoothness of the space'.
\begin{theorem}[Continuity equation and $W_2$-absolutely continuous curves]\label{thm:conteq} Let $(\X,\sfd,\mm)$ be such that $W^{1,2}(\X)$ is separable and $(\mu_t)\subset\probt\X$ a curve of bounded compression. Then the following holds.
\begin{itemize}
\item[A)] Suppose that  $(\mu_t)$ is $W_2$-absolutely continuous with $\int_0^1|\dot\mu_t|^2\,\d t<\infty$. Then there is a family  $(X_t)\subset  L^0(T\X)$ such that $(\mu_t,X_t)$ solves the continuity equation in the sense of Definition \ref{def:solconteq} and such that
\[
\int|X_t|^2\,\d\mu_t\leq |\dot\mu_t|^2,\qquad{\rm a.e. }\ t\in[0,1].
\]
\item[B)] Conversely, suppose that $(\mu_t,X_t)$ solves the continuity equation in the sense of Definition \ref{def:solconteq}. Then $(\mu_t)$ is $W_2$-absolutely continuous and
\[
 |\dot\mu_t|^2\leq \int|X_t|^2\,\d\mu_t,\qquad{\rm a.e. }\ t\in[0,1].
\]
\end{itemize}
\end{theorem}
\begin{proof}$\ $\\
\noindent{$\mathbf{(A)}$} Use the implication $(i)\Rightarrow(ii)$ in Theorem \ref{thm:bg} to find $\mathcal N$ and operators $(L_t)$ as in the statement. Fix $t\in[0,1]\setminus\mathcal N$ such that $\|L_t\|_{\mu_t}<\infty$ and notice that this forces $L_t(f)=L_t(g)$ for any $f,g\in\s^2(\X)$ with $\d f=\d g$. Therefore the map $\tilde L_t:\{\d f:f\in\s^2(\X)\}\to\R$ given by $\tilde L_t(\d f):=L_t(f)$ is well defined and the definition ensures that it is linear and satisfies 
\[
|\tilde L_t(\d f)|=|L_t(f)|\leq \|L_t\|_{\mu_t}\sqrt{\int|\d f|^2\,\d\mu_t},\qquad\forall f\in\s^2(\X).
\]
Call $L^2(\mu_t,T^*\X)$ (resp. $L^2(\mu_t,T\X)$) the module $\MM_{2,\mu_t}$ built from $\MM:=L^2(T^*\X)$ (resp. $L^2(T\X)$) as in Section \ref{se:altint}, so that the above inequality can be restated as $|\tilde L_t(\d f)|\leq \|L_t\|_{\mu_t}\|\d f\|_{L^2(\mu_t,T^*\X)}$ and use the Hahn-Banach theorem to extend $\tilde L_t$ to a linear continuous map, still denoted by $\tilde L_t$, from $L^2(\mu_t,T^*\X)$ to $\R$ satisfying
\begin{equation}
\label{eq:thevoice}
|\tilde L_t(\omega)|\leq  \|L_t\|_{\mu_t}\|\omega\|_{L^2(\mu_t,T^*\X)},\qquad\forall \omega\in L^2(\mu_t,T^*\X).
\end{equation}
Thus $\tilde L_t$ is an element of the Banach dual of $L^2(\mu_t,T^*\X)$ and being this module $L^2(\mu_t)$-normed, by Proposition \ref{prop:fulllp} we can identify its Banach dual with its module dual and the latter can be identified, recalling \eqref{eq:dualmmp}, with the module $L^2(\mu_t,T\X)$, so that in summary we have $X_t\in L^2(\mu_t,T\X)$ such that
\[
\tilde L_t(\omega)=\int \omega(X_t)\,\d\mu_t,\qquad\forall\omega\in L^2(\mu_t,T^*\X).
\]
By construction, for every $f\in W^{1,2}(\X)$ we have
\[
\frac{\d}{\d t}\int f\,\d\mu_t=L_t(f)=\int \d f(X_t)\,\d\mm,
\]
for every $t\in[0,1]\setminus\mathcal N$ such that $\|L_t\|_{\mu_t}<\infty$ and the bound \eqref{eq:thevoice} and the identity \ref{eq:thevo} grant that $\|X_t\|_{L^2(\mu_t,T\X)}\leq \|L_t\|_{\mu_t}=|\dot\mu_t|$ for a.e.\ $t\in[0,1]$, so the proof is complete.

\noindent{$\mathbf{(B)}$} Let $D\subset W^{1,2}(\X)$ be a countable dense $\Q$-vector space and $A\subset [0,1]$ the set of Lebesgue points of $t\mapsto\int|X_t|^2\,\d\mu_t$ such that $s\mapsto\int f_n\,\d\mu_s$ is differentiable at $s=t$ for any $n\in\N$. Then $A$ is Borel and $\mathcal L^1(A)=1$. For $t\in A$ the map $L_t:D\to \R$ given by $\frac{\d}{\d s}\int f\,\d\mu_s\restr{s=t}$ is well defined, $\Q$-linear and dividing by $|h|$ and letting $h\to 0$ in the bound
\begin{equation}
\label{eq:percont}
\begin{split}
\Big|\int f\,\d(\mu_{t+h}- \mu_t)\Big|&\leq \Big|\int\limits_t^{t+h}\int\d f(X_s)\,\d\mu_s\,\d s\Big|\leq \sqrt{\int\limits_t^{t+h}\int|\d f|^2_*\,\d\mu_s\,\d s\int\limits_t^{t+h}\int|X_s|^2\,\d\mu_s\,\d s}\\
&\leq \sqrt{hC\int |\d f|_*^2\,\d\mm \int\limits_t^{t+h}\int|X_s|^2\,\d\mu_s\,\d s},
\end{split}
\end{equation}
valid for any $f\in W^{1,2}(\X)$, where $C$ is such that $\mu_t\leq C\mm$ for every $t\in[0,1]$, we see that $L_t:D\to\R$ is continuous w.r.t.\ the $W^{1,2}(\X)$ norm. Hence it can be extended in a unique way to a linear continuous functional from $W^{1,2}(\X)$ to $\R$, still denoted by $L_t$. 

For generic $f\in W^{1,2}(\X)$ and $t\in D$, writing the bound \eqref{eq:percont} with  $f-f_n$ in place of $f$, where $(f_n)\subset D$ converges to $f$, diving by $h$ and letting first $h\to 0$ and then $n\to\infty$, we see that  the identity  $\frac{\d}{\d s}\int f\,\d\mu_s\restr{s=t}=L_t(f)$ is valid for any $t\in A$ and $f\in W^{1,2}(\X)$.

Now notice that being $(\mu_t)$ weakly continuous in duality with $C_b(\X)$ and with densities uniformly bounded, it is continuous in duality with $L^1(\mm)$ and thus $t\mapsto \int|\d f|_*^2\,\d\mu_t$ is continuous for every $f\in W^{1,2}(\X)$. Hence  for $t\in A$ we have
\[
\begin{split}
\Big|\lim_{h\to 0}\frac1h\int f\,\d(\mu_{t+h}-\mu_t)\Big|&\leq\lim_{h\to0} \sqrt{\frac1h\int\limits_t^{t+h}\int|\d f|^2_*\,\d\mu_s\,\d s\frac1h\int\limits_t^{t+h}\int|X_s|^2\,\d\mu_s\,\d s}\\
&=\sqrt{\int|\d f|_*^2\,\d\mu_t\int|X_t|^2\,\d\mu_t},
\end{split}
\]
which shows that $\|L_t\|_{\mu_t}\leq \|X_t\|_{L^2(\mu_t,T\X)}$ for every $t\in A$. The conclusion then follows by the implication $(ii)\Rightarrow(i)$ in Theorem  \ref{thm:bg} and the identity \eqref{eq:thevo}.
\end{proof}
A trivial consequence of the above characterization of $W_2$-absolutely continuous curves is the following version of the Benamou-Brenier formula, which in particular allows to recover \eqref{eq:normdistsmooth}, under appropriate assumptions, in the non-smooth setting:
\begin{corollary}[Benamou-Brenier formula]
Let $(\X,\sfd,\mm)$ be such that $W^{1,2}(\X)$ is separable and $(\mu_t)$ a $W_2$-absolutely continuous curve with bounded compression with $|\dot\mu_t|\in L^2(\mm)$. Then
\[
\int_0^1 |\dot\mu_t|^2\,\d t=\inf \int_0^1\int |X_t|^2\,\d\mu_t\,\d t,
\]
where the $\inf$ is taken among all Borel maps $t\mapsto X_t\in L^0(T\X)$ such that $(\mu_t,X_t)$ solves the continuity equation in the sense of Definition \ref{def:solconteq} above.

In particular, if $\mu,\nu\in\probt \X$ are such that there exists a $W_2$-geodesic connecting them with bounded compression, then we have
\[
W_2^2(\mu,\nu)=\min \int_0^1\int |X_t|^2\,\d\mu_t\,\d t,
\]
where the $\min$ is taken among all solutions $(\mu_t,X_t)$ of the continuity equation such that $\mu_0=\mu$ and $\mu_1=\nu$.
\end{corollary}
\begin{proof} For the first statement just notice that part $(B)$ of Theorem \ref{thm:conteq} grants that for any solution $(\mu_t,X_t)$ of the continuity equation we have $ \int_0^1|\dot\mu_t|^2\,\d t\leq \int_0^1\int|X_t|^2\,\d\mu_t\,\d t$, while part $(A)$ ensures that there is a choice of the $X_t$'s for which equality holds. 

Then for second statement inequality $\leq$ is obvious, while for $\geq$ it is sufficient to pick the geodesic given by the statement, parametrize it by constant speed and use  part $(A)$ of Theorem \ref{thm:conteq} again.
\end{proof}
It is worth to underline that even assuming the underlying metric space to be geodesic, we cannot expect the existence of many $W_2$-geodesics of bounded compression, the reason being that such notion requires a link between the metric, via the use of $W_2$, and the measure, via the requirement of bounded compression. Being these two unrelated on general metric measure structures, without further assumptions on the space we cannot expect many of such geodesics.

It is a remarkable result of Rajala \cite{Rajala12-2} the fact that on $\CD(K,\infty)$ spaces, $K\in\R$, for any couple of probability measures with bounded support and bounded density, a $W_2$-geodesic of bounded compression connecting them always exists.

\begin{remark}[The link between Theorems \ref{thm:spppi} and \ref{thm:conteq}]{\rm Arguing as in the proof of Theorem \ref{thm:dualpullback}, see in particular the definition \eqref{eq:prs}, we see that for any test plan $\ppi$ and every $t\in[0,1]$, we can build a linear continuous `projection' map ${\sf Pr}_{\e_t}:L^2(T\X,\ppi,\e_t)\to L^2(T\X,\mu_t)$, where $\mu_t:=(\e_t)_*\ppi$, satisfying
\[
|{\sf Pr}_{\e_t}(Y)|\circ\e_t\leq|Y|\quad\ppi\ae\qquad\text{ and }\qquad \int \e_t^*\omega(Y)\,\d\ppi=\int\omega({\sf Pr}_{\e_t}(Y))\,\d\mu_t,\quad\forall\omega\in L^2(T^*\X).
\]
Then noticing also that   $t\mapsto\mu_t:=(\e_t)_*\ppi$ is a $W_2$-absolutely continuous curve of bounded compression with speed in $L^2([0,1])$, we have
\[
\frac{\d}{\d t}\int f\,\d\mu_t=\frac{\d}{\d t}\int f\circ\e_t\,\d\ppi=\int \e_t^*\d f(\ppi'_t)\,\d\ppi=\int \d f({\sf Pr}_{\e_t}\ppi'_t)\,\d\mu_t,
\]
which shows that $X_t:={\sf Pr}_{\e_t}\ppi'_t$ is an admissible choice of vector fields in the continuity equation which also satisfies $\|X_t\|_{L^2(T\X,\mu_t)}\leq \|\ppi'_t\|_{L^2(T\X,\sppi,\e_t)}$ for a.e.\ $t$.
}\fr\end{remark}

\subsection{Maps of bounded deformation}\label{se:mbd}
In this section we introduce those mappings between metric measure spaces which, shortly said, play the role that Lipschitz mappings have in the theory of metric spaces. We then see how for these maps the notions of pullback of 1-forms and differential come out quite tautologically from the language we developed so far.

\vspace{1cm}

All metric measure spaces we will consider will be complete, separable and endowed with a non-negative Radon measure. We start with the following definition:
\begin{definition}[Maps of bounded deformation]
Let $(\X_1,\sfd_1,\mm_1)$ and $(\X_2,\sfd_2,\mm_2)$ be two metric measure spaces. A map of bounded deformation $\varphi:\X_2\to \X_1$ is (the equivalence class w.r.t.\ equality $\mm_2$-a.e.\ of) a Borel map for which there are constants $\lf(\varphi),\cf(\varphi)\geq 0$ such that
\[
\begin{split}
\sfd_1(\varphi(x),\varphi(y))&\leq \lf(\varphi)\,\sfd_2(x,y),\qquad\mm_2\times\mm_2\ae\  (x,y),\\
\varphi_*\mm_2&\leq \cf(\varphi)\mm_1
\end{split}
\]
In other words, a map of bounded deformation  is a map of bounded compression (Definition \ref{def:bcompr}) having a Lipschitz representative.
\end{definition}
\begin{remark}{\rm
The, pedantic, choice of considering equivalence classes of maps rather than declaring maps of bounded compression simply Lipschitz maps $\varphi$ satisfying $\varphi_*\mm_2\leq C\mm_1$ for some $C$, is only motivated by the desire of providing a notion invariant under modification $\mm_2$-a.e.. In particular, we want  all of what happens outside the support of the reference measures to be irrelevant.
}\fr\end{remark}
A map of bounded compression $\varphi:\X_2\to \X_1$ induces a continuous map, still denoted by $\varphi$, from $C([0,1],\supp(\mm_2))$ to $C([0,1],\supp(\mm_1))$ sending a curve $\gamma$ to $\bar \varphi\circ\gamma$, where $\bar\varphi$ is the Lipschitz representative of $\varphi$, which is uniquely determined on $\supp(\mm_2)$. By definition, the image of an absolutely continuous curve  $\gamma$ is still absolutely continuous and its metric speed is bounded by $\lf(\varphi)\,|\dot\gamma|$ for a.e.\ $t$.

Therefore noticing that a test plan $\ppi\in\prob{C([0,1],\X_2)}$ on $\X_2$ is concentrated on curves with values in $\supp(\mm_2)$, we see that for such plan the measure $\varphi_*\ppi\in\prob{C([0,1],\X_1)}$ is well defined. In fact, $\varphi_*\ppi$ is a test plan on $\X_1$ since
\[
\begin{split}
(\e_t)_*(\varphi_*\ppi)=\varphi_*((\e_t)_*\ppi)&\leq \varphi_*(\cf(\ppi)\mm_2)\leq \cf(\ppi)\cf(\varphi)\mm_2,\qquad\forall t\in[0,1],\\
\iint_0^1|\dot\gamma_t|^2\,\d t\,\d\varphi_*\ppi(\gamma)&\leq \lf(\varphi)^2\iint_0^1|\dot\gamma_t|^2\,\d t\,\d\ppi(\gamma).
\end{split}
\]
The fact that test plans are sent to test plans allows to deduce   by duality  that 
\begin{equation}
\label{eq:sobpull}
\begin{split}
 f\in\s^2(\X_1)\qquad\Rightarrow\qquad & f\circ \varphi\in \s^2(\X_2)\qquad\text{with}\text{$\qquad |\d (f\circ \varphi)|_*\leq\lf(\varphi)|\d f|_*\circ\varphi\qquad \mm_2\ae,$}
\end{split}
\end{equation}
indeed, for arbitrary $\ppi\in \prob{C([0,1],\X_2)}$ test plan and  $f\in\s^2(\X_1)$ we have
\[
\begin{split}
\int |f\circ \varphi(\gamma_1)-f\circ \varphi(\gamma_0)|\,\d\ppi(\gamma)&=\int |f(\gamma_1)-f(\gamma_0)|\,\d\varphi_*\ppi(\gamma)\\
&\leq\iint_0^1|\d f|_*(\gamma_t)|\dot\gamma_t|\,\d t\,\d \varphi_* \ppi(\gamma)\\
&\leq\lf(\varphi)\iint_0^1(|\d f|_*\circ \varphi)(\gamma_t)|\dot\gamma_t|\,\d t\,\d\ppi(\gamma),
\end{split}
\]
and  since the fact that $\varphi$ has bounded compression grants that $|\d f|_*\circ\varphi\in L^2(\X_2)$, the claim \eqref{eq:sobpull} follows.

A map of bounded deformation canonically induces a map  $\varphi^*:L^2(T^*\X_1)\to L^2(T^*\X_2)$ which we shall call {\bf pullback of 1-forms}, notice that the argument leading to the construction of $\varphi^*$ are similar to those used in Propositions \ref{prop:extension} and \ref{prop:univpullback}:
\begin{proposition}[Pullback of 1-forms]\label{prop:pull1form}
Let $\varphi:\X_2\to \X_1$ be of bounded deformation. Then there exists a unique linear and continuous map $\varphi^*:L^2(T^*\X_1)\to L^2(T^*\X_2)$ such that
\begin{equation}
\label{eq:pullback}
\begin{array}{rll}
\varphi^*(\d f)&\!\!\!\!=\d (f\circ \varphi),\qquad&\forall f\in \s^2(\X_1)\\
\varphi^*(g\omega)&\!\!\!\!=g\circ \varphi \ \varphi^*(\omega),\qquad&\forall g\in L^\infty(\mm_1),\ \omega\in L^2(T^*\X_1),
\end{array}
\end{equation} 
and such map satisfies 
\begin{equation}
\label{eq:risopull}
|\varphi^*\omega|_*\leq \lf(\varphi) \,|\omega|_*\circ\varphi,\qquad\mm_2\ae,\ \forall \omega\in L^2(T^*\X_1).
\end{equation}
\end{proposition}
\begin{proof}
Uniqueness follows from linearity, continuity and the requirements \eqref{eq:pullback} by recalling that $L^2(T^*\X_1)$ is generated by the differentials of functions in $\s^2(\X_1)$.

For existence, we declare that
\[
\varphi^*(\omega):=\sum_i\nchi_{\varphi^{-1}(E_i)} \, \d(f_i\circ\varphi),\\
\]
for $\omega=\sum_i\nchi_{E_i}\d f_i$ for some finite partition $(E_i)$ of $\X_1$ and $(f_i)\subset\s^2(\X_1)$. The bound
\[
|\varphi^*(\omega)|_*\leq \sum_i\nchi_{\varphi^{-1}(E_i)} \, |\d(f_i\circ\varphi)|_*\leq \lf(\varphi)\sum_i\nchi_{E_i}\circ\varphi \ |\d f_i|_*\circ\varphi=\lf (\varphi)\,|\omega|_*\circ\varphi,
\]
shows that such mapping is continuous and since the closure of the set of $\omega$'s considered is the whole $L^2(T^*\X_1)$, we see that there is a unique continuous extension of $\varphi^*$, still denoted $\varphi^*$, from  $L^2(T^*\X_1)$ to $L^2(T^*\X_2)$.   It is then clear that $\varphi^*$ is linear, satisfies the bound \eqref{eq:risopull} and the first in \eqref{eq:pullback}. The second in \eqref{eq:pullback} is then obtained for simple functions $g$ directly by definition and linearity, and then extended to the whole $L^\infty(\X_1)$ by approximation.
\end{proof}
\begin{remark}[The category $\mms$]{\rm We define  the  category $\mms$ of (complete, separable and equipped with a non-negative Radon measure) metric measure spaces where morphisms are maps of bounded deformation.

Then the map sending a metric measure space $(\X,\sfd,\mm)$ to its cotangent module $L^2(T^*\X)$ and a map of bounded deformation $\varphi$ into the couple $(\varphi,\varphi^*)$ is easily seen to be a functor from $\mms$ to the category ${\bf Mod}_{2-L^\infty}$ introduced in Remark \ref{rem:catpullback} (such functor being covariant, due to the choice of arrows).
}\fr\end{remark}
\begin{remark}{\rm
We point out that all the discussions in the section would work equally well if maps of bounded deformations where defined as maps of bounded compression for which \eqref{eq:sobpull} holds, thus focussing on transormations of Sobolev functions rather than of the distance, an approach which would be more in line with the language proposed here. 

Still, for what concerns $\RCD$ spaces, this distinction does not really matter, as a map of bounded compression between two $\RCD(K,\infty)$ spaces and for which \eqref{eq:sobpull} holds, in fact admits a Lipschitz continuous representative with Lipschitz constant $\leq \lf(\varphi)$, so that in this case the two approaches coincide (see \cite{Gigli13}).
}\fr\end{remark}
It is worth underlying that we have 2 different definitions of `pullback' of 1-forms. One is given by Proposition \ref{prop:pull1form} above, which takes forms in $L^2(T^*\X_1)$ and returns forms on $L^2(T^*\X_2)$, the other is the one involving the notion of pullback module as discussed in Section \ref{se:pullback}, which takes forms on $L^2(T^*\X_1)$ and returns an element in $\varphi^*(L^2(T^*\X_1))$. These two are different operations: to distinguish them, in the foregoing discussion we shall denote the second one as $\omega\mapsto[\varphi^*\omega]$ and keep the notation $\omega\mapsto\varphi^*\omega$ for the first one.

\bigskip

We come to the differential of a map of bounded deformation. Recall that in classical smooth differential geometry, given a smooth map $\varphi:M_2\to M_1$ between smooth manifolds, its differential $\d \varphi_x$ at a point $x\in M_2$ is the well defined linear map from $T_xM_2$ to $T_{\varphi (x)}M_1$ given by $\omega(\d\varphi_x(v)):=v(\varphi^*\omega)$ for any $v\in T_xM_2$ and $\omega\in T^*_{\varphi (x)}M_1$. Thus we can see the differential as taking tangent vectors and returning tangent vectors. The situation changes when looking at tangent vector fields $X$ on $M_2$, because the map $M_2\ni x_2\mapsto \d\varphi_x(X(x_2))\in T_{\varphi(x_2)}M_1$ is not a tangent vector field on $M_1$: it is rather the section of the pullback $\varphi^*TM_1$ of the tangent bundle $TM_1$ identified by the requirement
\begin{equation}
\label{eq:requdiff}
\begin{split}
[\varphi^*\omega](\d \varphi(X))&:= (\varphi^*\omega)(X).
\end{split}
\end{equation}

Coming back to metric measure spaces, we therefore see that we must expect the differential of a map of bounded deformation $\varphi:\X_2\to\X_1$ to take an element of $L^2(T\X_2)$ and return an element of the pullback module  $\varphi^*(L^2(T\X_1))$. Yet, there is a technical complication: the (analogous of) the requirement \eqref{eq:requdiff} only defines the object $\d \varphi(X)$ as element of the dual of $\varphi^*(L^2(T^*\X_1))$ and we don't know whether in general such space can be identified with the pullback $\varphi^*(L^2(T\X_1))$ of the dual module $L^2(T\X_1)$. Still,  thanks to Theorem \ref{thm:dualpullback} we know that such identification is possible at least when $L^2(T\X_1)$ is separable (which frequently happens). We thus have the following statement:
\begin{proposition}[Differential of a map of bounded deformation] Let $\varphi:\X_2\to\X_1$ be of bounded deformation and assume that $L^2(T\X_1)$ is separable.

Then there is a unique module morphism $\d\varphi:L^2(T\X_2)\to \varphi^*(L^2(T\X_1))$ such that 
\begin{equation}
\label{eq:pushvf}
[\varphi^*\omega]\big(\d\varphi(X)\big)=(\varphi^*\omega) (X),\qquad\forall \omega\in L^2(T^*\X_1),\ X\in L^2(T\X_2),
\end{equation}
and such morphism also satisfies
\begin{equation}
\label{eq:pushvf2}
|\d \varphi (X)|\leq \lf ( \varphi)|X|,\qquad\mm_2\ae,\qquad \forall  X\in L^2(T\X_2).
\end{equation}
\end{proposition}
\begin{proof} Fix $X\in L^2(T\X_2)$ and notice that the map from the space $\{[\varphi^*\omega]:\omega\in L^2(T^*\X_1)\}$ to $L^1(\mm_2)$ assigning to $[\varphi^*\omega]$ the function $(\varphi^*\omega) (X)$ is linear and satisfies 
\[
|(\varphi^*\omega) (X)|\leq |\varphi^*\omega|_*|X|\leq \lf (\varphi)\,|\omega|_*\circ\varphi\,|X|= \lf(\varphi)\,|[\varphi^*\omega]|_*\,|X|.
\]
Thus by Proposition \ref{prop:extension}, and recalling that $(\varphi^*(L^2(T^*\X_1)))^*\sim \varphi^*(L^2(T\X_1))$ by Theorem \ref{thm:dualpullback} and our separability assumption, we see that there is a unique element of $\varphi^*(L^2(T\X_1))$, which we shall call $\d\varphi(X)$, such that \eqref{eq:pushvf} holds and for such $\d\varphi(X)$ the bound \eqref{eq:pushvf2} holds as well.

To conclude the proof is thus sufficient to show that $X\mapsto\d\varphi(X)$ is a module morphism: linearity is obvious and thus the bound \eqref{eq:pushvf2} and point $(v)$  in Proposition \ref{prop:baselp} give the conclusion.
\end{proof}
In the special case in which $\varphi$ is invertible, i.e.\ that there exists a map of bounded deformation $\psi:\X_1\to\X_2$ such that
\[
\varphi\circ\psi=\Id_{\X_1}\quad\mm_1\ae\qquad\text{ and }\qquad\psi\circ\varphi=\Id_{\X_2}\quad\mm_2\ae,
\]
one can canonically think the differential of $\varphi$ as a map from $L^2(T\X_2)$ to $L^2(T\X_1)$. In other words, in this case one can  avoiding mentioning the pullback module, in analogy with the smooth situation. This is due to the functorial property of the pullback of modules, which, for $\varphi,\psi$ as above, allows to identify $L^2(T^*\X_1)$ with $\varphi^*L^2(T^*\X_1)$ via the invertible map $\omega\mapsto[\varphi^*\omega]$. In practice, one can define the differential $\d\varphi(X)$ of $\varphi$ calculated on the vector field $X$ as the operator
\[
L^2(T^*\X_1)\ni \omega\qquad\mapsto\qquad \big([\varphi^*\omega](X)\big)\circ\psi\in L^1(\mm_1).
\]
In this case the definition is well posed regardless of the separability of $L^2(T\X_1)$.

\subsection{Some comments}
We collect here few informal comments about  our construction and the relations it has with previously defined differential structures on non-smooth setting.

\vspace{1cm}

\noindent{{\bf (1)}} We have already remarked that in the smooth setting our construction of the cotangent module is canonically identifiable with the space of $L^2$ sections of the cotangent bundle. The same identification also occurs in slightly less smooth situations, like Lipschitz manifolds or Sub-Riemannian ones, in the latter case the correspondence being with horizontal sections. Still, given that our  definitions are based on the notion of Sobolev functions, everything trivializes in spaces with non-interesting Sobolev analysis. This is the case, for instance, of discrete or fractal spaces and more generally of spaces admitting no non-constant Lipschitz curves: as can be checked directly from the definitions, in this case every real valued Borel function $f$ is in $\s^2(\X)$ with $\weakgrad f=0$, so that the cotangent module reduces to the 0 module. Of course, in these spaces a non-trivial calculus can be developed, but it must be tailored to the special structure/scale of the space considered and cannot fit the general framework developed here. The same phenomenon was observed  by Weaver in \cite{Weaver01}.

On the opposite side, it is important to underline that if the metric measure space is regular enough, then its structure is completely determined by Sobolev functions (much like a metric is determined by Lipschitz functions). For instance, it has been proved in \cite{Gigli13} that if $(\X_i,\sfd_i,\mm_i)$, $i=1,2$, are $\RCD(K,\infty)$ spaces and $\varphi:\X_2\to\X_1$ is an $\mm_2$-a.e.\ injective map, then 
\[
\text{$\varphi$ is an isomorphism}\qquad\Leftrightarrow\qquad \|f\circ\varphi\|_{W^{1,2}(\X_2)}=\|f\|_{W^{1,2}(\X_1)}\qquad\forall f:\X_1\to\R\ \text{ Borel,}
\]
where by isomorphism we intend that $\varphi_*\mm_2=\mm_1$ and $\sfd_1(\varphi(x),\varphi(y))=\sfd_2(x,y)$ for $\mm_2\times\mm_2$-a.e.\ $x,y$. In particular, in studying these spaces via the study of Sobolev functions defined on them, we are not losing any bit of information.

\bigskip

\noindent{{\bf (2)}} We built the cotangent module using Sobolev functions in $\s^2(\X)$ and their `modulus of distributional differential' $\weakgrad f$ as building block. Given that for every $p\in[1,\infty]$ one can define the analogous space $\s^p(\X)$ of functions of functions which, in the smooth case, would be those having distributional differential in $L^p$, one can wonder if anything changes with a different choice of Sobolev exponent. Without further assumptions, it does. This means that a priori the choice $p=2$ really matters, the point being the following. The definition of the space $\s^p(\X)$ (see e.g.\ \cite{AmbrosioGigliSavare11-3} and references therein) comes with an association to each function $f\in\s^p(\X)$ of a function, call it $\weakgradp fp$, in $L^p(\mm)$ playing the role of  the modulus of distributional differential of $f$ and obeying calculus rules analogous to those valid for $\weakgrad f=\weakgradp f2$. The unfortunate fact here is that, unlike the smooth case where the distributional differential is a priori given  and one can then wonder if it is in $L^p$, in general metric measure spaces the function $\weakgradp fp$ depends on $p$. There are indeed explicit examples of spaces $\X$ and functions $f\in\s^p\cap \s^{p'}(\X)$ for some $p\neq p'$ with $\weakgradp fp\neq \weakgradp f{p'}$ on a set of positive measure: see \cite{DiMarinoSpeight13} and references therein.

Due to this fact, in general we cannot expect any relation between the cotangent module as we defined it and the analogous one built starting from functions in $\s^p(\X)$ with $p\neq 2$.  In this sense, the choice $p=2$ is totally arbitrary and an approach based on a different value of $p$ would have equal dignity.

Still, in the case of $\RCD(K,\infty)$ spaces, the results in \cite{GigliHan14} show that no ambiguity occurs because the identity $\weakgradp fp=\weakgradp f{p'}$ holds $\mm$-a.e.\ for every $f\in\s^p\cap \s^{p'}(\X)$. This means that, up to taking care of the different integrability via the tools described in Section \ref{se:altint}, on such spaces there is truly only one cotangent module. Given that all our constructions have as scope the study of $\RCD$ spaces, the fact that on arbitrary spaces the cotangent module   might depend on the chosen Sobolev exponent is not really relevant and the presentation given via the case $p=2$ has been made only for convenience.

Another important situation where  $\weakgradp fp=\weakgradp f{p'}$   is when the space is doubling and supports a 1-1 Poincar\'e inequality (more generally, if a 1-$\bar p$ Poincar\'e inequality holds, then  $\weakgradp fp=\weakgradp f{p'}$   for every $p,p'\geq \bar p$). This is a consequence of the analysis done by Cheeger in \cite{Cheeger00}.

\bigskip

\noindent{{\bf (3)}} Another consequence of the analysis done by Cheeger in \cite{Cheeger00} is that on doubling spaces supporting a 1-2 Poincar\'e inequality the cotangent module is finite dimensional. Recall indeed the version of Rademacher's theorem proved in \cite{Cheeger00}:
\begin{theorem}\label{thm:cheeger}
Let $(\X,\sfd,\mm)$ be with doubling measure and supporting a 1-2 weak local Poincar\'e inequality. Then there exists $N\in\N$ depending only on the doubling and Poincar\'e constants such that the following holds. 

There is a countable family of disjoint Borel sets $(A_n)$ covering $\mm$-almost all $\X$ and for each $n\in\N$ Lipschitz functions $\bar f_{n,1},\ldots,\bar f_{n,N_n}$, $N_n\leq N$, from $\X$ to $\R$ such that for  any Lipschitz function $\bar f:\X\to\R$ there are $a_{n,i}\in L^\infty(\mm)$, $n\in\N$, $i=1,\ldots,N_n$, such that
\[
\begin{split}
{\rm lip}\Big(\bar f(\cdot)-\sum_{i=1}^{N_n}a_{n,i}(x)\bar f_{n,i}(\cdot)\Big)(x)=0,\qquad\text{\rm for }\mm\ae\ x\in A_n.
\end{split}
\]
\end{theorem}
It is then easy to see that the functions $a_i$ given in the statement can be interpreted as the coordinates of $\d f$ w.r.t.\ the basis $\{\d f_{n,1},\ldots,\d f_{n,N_n}\}$ on $A_n$ in the sense of Section \ref{se:locdim}:
\begin{corollary}
With the same assumptions and notation of Theorem \ref{thm:cheeger} above, the cotangent module $L^2(T^*\X)$ is finitely generated, its dimension being bounded by $N$.
\end{corollary}
\begin{proof}
It is sufficient to show that for every $n\in\N$ the cotangent module is generated by $\d f_{n,1},\ldots,\d f_{n,N_n}$ on $A_n$. Fix such $n$, let $L:=\max_i\Lip(\bar f_{n,i})$ and consider a Lipschitz function $\bar f:\X\to\R$ and the functions $a_i\in L^\infty(\mm)$, $i=1,\ldots,N_n$, given by Theorem \ref{thm:cheeger}. We claim that $\d f=\sum_{i=1}^{N_n}a_{i}\d f_{n,i}$. To prove this, let $\eps>0$ find a Borel partition $(B_{j})$ of $A_n$ such that $b_{i,j}:=\esssup_{B_j}a_{i}\leq \essinf_{B_j}a_{i}+\eps$ $\mm$-a.e.\ on $B_j$ for every $i=1,\ldots,N_n$, $j\in\N$ and notice that
\[
\begin{split}
\Big|\d f-\sum_{i=1}^{N_n}a_{n,i}\d f_{n,i}\Big|_*(x)&\leq \Big|\d f-\sum_{i=1}^{N_n}b_j\d f_{n,i}\Big|_*(x)+\eps N L\leq {\rm lip}\big(\bar f-\sum_{i=1}^{N_n}b_i\bar f_{n,i}\big)(x) +\eps N L\\
&\leq {\rm lip}\Big(\bar f(\cdot)-\sum_{i=1}^{N_n}a_i(x)\bar f_{n,i}(\cdot)\Big)(x) +2\eps N L=2\eps N L,\qquad\mm\ae\ x\in B_j.
\end{split}
\]
Being this true for any $j$, we deduce
\[
\Big|\d f-\sum_{i=1}^{N_n}a_{n,i}\d f_{n,i}\Big|_*\leq 2\eps N L,\qquad\mm\ae\text{ on }A_n,
\]
and the arbitrariness of $\eps$ yields the claim. 

Then using the Lusin approximation of Sobolev functions (see for instance Theorem 5.1 in \cite{Bjorn-Bjorn11}) and the locality of the differential we deduce that the module span of $\d f_{n,1},\ldots,\d f_{n,N_n}$ on $A_n$ contains the differential of every function in $W^{1,2}(\X)$. The thesis follows.
\end{proof}

\bigskip

\noindent{{\bf (4)}} We already mentioned in the introduction that the idea of using $L^\infty$-modules as the technical tool to give an abstract definition of the space of sections of the (co)tangent bundle on a metric measure spaces is due to Weaver \cite{Weaver01}. Being our approach strongly inspired by his, we briefly discuss what are the main differences.

At the technical level, in \cite{Weaver01} the kind of modules considered resemble those that we called $L^\infty(\mm)$-normed modules (although it is not clear to us if the notions fully agree). In particular, no other integrability condition has been considered, which also affects the definition of dual module (morphisms with values in $L^\infty(\mm)$ in \cite{Weaver01} and with values in $L^1(\mm)$ here). 

At a more conceptual level, here we chose to tailor the definitions on the concept of Sobolev functions, while in \cite{Weaver01} Lipschitz ones have been used. On doubling spaces supporting a weak local Poincar\'e inequality, the result of Cheeger \cite{Cheeger00} granting that ${\rm lip}(\bar f)=\weakgrad f$ $\mm$-a.e.\ for Lipschitz functions $\bar f$ ensures that the difference between the two approaches is minimal and confined essentially to the different integrability requirements. In more general situations it is always true that derivations as we defined them are also derivations in the sense of Weaver, but the converse is not clear. This is partially due to the fact that the structures considered  are slightly different: our definitions only make sense on metric measure spaces, while the ones in \cite{Weaver01} work on metric spaces carrying  a notion of negligible set, so that a measure is not truly needed.

We conclude remarking that a general property one typically wants from an abstract  definition of differential in a non-smooth setting is that it is a closed operator, i.e.\ its graph should be closed w.r.t.\ convergence of differentials and some relevant 0-th order convergence of functions. 

 In the framework proposed by Weaver this is obtained by construction starting from the fact that the space of Lipschitz function is a dual Banach space (see Chapter 2 in \cite{Weaver99} and references therein) and then requiring derivations to be  weakly$^*$ continuous.

In our setting, instead, this is ensured by Theorem \ref{thm:closd} whose proof ultimately boils down to the lower semicontinuity property \ref{eq:lscwug} which in turn is - up to minor technicalities - equivalent to the $L^2(\mm)$-lower semicontinuity of $\E$.

In this direction, notice that the $L^2(\mm)$-lower semicontinuity of $\E$ is `the' crucial property behind the definition of Sobolev functions which completely characterizes them together with the information, valid at least if $\mm$ gives finite mass to bounded sets, that $\E$ is the maximal $L^2(\mm)$-lower semicontinuous functional satisfying $\E(f)\leq\frac12\int {\rm lip}^2(\bar f)\,\d\mm$ for Lipschitz functions $\bar f$ (this follows from properties \eqref{eq:lipweak}, \eqref{eq:aplip2} - see \cite{AmbrosioGigliSavare11} for the relevant proofs).

\section{Second order differential structure of $\RCD(K,\infty)$ spaces}
Without exceptions, in this chapter we will always assume that $(\X,\sfd,\mm)$ is an $\RCD(K,\infty)$ space.

\subsection{Preliminaries: $\RCD(K,\infty)$ spaces}
Let $(\X,\sfd,\mm)$ be a complete and separable metric space endowed with a non-negative Radon measure. The \emph{relative entropy} functional $\ent_\mm:\prob\X\to\R\cup\{+\infty\}$ is defined as
\[
\ent_\mm(\mu):=\left\{
\begin{array}{ll}
\displaystyle{\int\rho\log(\rho)\,\d\mm},\qquad&\text{if $\mu=\rho\mm$ and $\big(\rho\log(\rho)\big)^-\in L^1(\mm)$,}\\
+\infty,\qquad&\text{otherwise}.
\end{array}
\right.
\]

We start with the following basic definition:
\begin{definition}[$\RCD(K,\infty)$ spaces]
Let $(\X,\sfd,\mm)$ be a complete and separable metric space endowed with a non-negative Radon measure and $K\in\R$.

We say that $(\X,\sfd,\mm)$ is a $\RCD(K,\infty)$ space provided it is infinitesimally Hilbertian and for every $\mu,\nu\in\probt{\X}$ with $\ent_\mm(\mu),\ent_\mm(\nu)<\infty$ there exists a $W_2$-geodesic $(\mu_t)$ with $\mu_0=\mu$, $\mu_1=\nu$ and  such that
\[
\ent_\mm(\mu_t) \leq (1-t)\ent_\mm(\mu)+t\ent_\mm(\nu)-\frac K2t(1-t)W_2^2(\mu,\nu),\qquad\forall t\in[0,1].
\] 
\end{definition}
Notice that on $\RCD(K,\infty)$ spaces, for any $x\in\X$ we have
\[
\mm(B_r(x))\leq Ce^{Cr^2},\qquad\forall r>0,
\]
for some constant $C$, see \cite{Sturm06I} (this is true also in $\CD(K,\infty)$ spaces). As a consequence of this bound, for $\mu\in\probt{\X}$  with $\mu=\rho\mm$ we always have $(\rho\log(\rho))^-\in L^1(\mm)$ and from this fact it is easy to conclude that $\ent_\mm$ is lower semicontinuous on $(\probt\X,W_2)$.

$\RCD$ spaces are introduced by means of optimal transport, but their properties are better understood by Sobolev calculus, the link between the two points of view being provided by the understanding of the {\bf heat flow} (\cite{Gigli09}, \cite{Gigli-Kuwada-Ohta10}, \cite{AmbrosioGigliSavare11}). For what concerns this paper, we shall be only interested in the $L^2$ approach to the heat flow, thus notice that being $\RCD(K,\infty)$ spaces infinitesimally Hilbertian,  the energy functional  $\E$ is a quadratic form:  we shall call heat flow  its gradient flow $(\h_t)$ in $L^2(\mm)$ already introduced in Section \ref{se:presob} and notice that in this case it is linear and self-adjoint. We will occasionally use the following standard a priori estimates
\begin{equation}
\label{eq:altrafacile}
\begin{split}
\E(\h_tf)&\leq \frac1{4t}\|f\|_{L^2(\mm)}^2,\\
\|\Delta\h_tf\|^2_{L^2(\mm)}&\leq \frac1{2t^2}\|f\|_{L^2(\mm)}^2,
\end{split}
\end{equation}
valid for every $f\in L^2(\mm)$ and $t>0$. These can obtained by differentiating the $L^2(\mm)$ norm and the energy along the flow, see for instance the arguments in Section \ref{se:clap}. Also, it is not hard to check that for every $p\in[1,\infty]$ it holds
\begin{equation}
\label{eq:boundlpcalore}
\|\h_tf\|_{L^p(\mm)}\leq \|f\|_{L^p(\mm)},\qquad\forall t\geq 0,
\end{equation}
for every $f\in L^2\cap L^p(\mm)$. Thus by density the heat flow can, and will, be uniquely extended to a family of linear and continuous functionals $\h_t:L^p(\mm)\to L^p(\mm)$ of norm bounded by 1 for any $p<\infty$ (in fact, on $\RCD$ spaces it can be also naturally defined on $L^\infty(\mm)$ but we shall not need this fact - see \cite{AmbrosioGigliSavare11-2} and \cite{AmbrosioGigliMondinoRajala12}).

A first non-trivial consequence of the Ricci curvature lower bound is the following regularity result:
\begin{equation}
\label{eq:sobtolip}
f\in W^{1,2}(\X),\ |\d f|_*\in L^\infty(\mm)\qquad\Rightarrow\qquad\begin{array}{ll}
&\text{$f$ has a Lipschitz representative $\bar f$}\\
&\text{with $\Lip(\bar f)\leq \||\d f|_*\|_{L^\infty(\mm)}$}
\end{array}
\end{equation}
which allows to pass from a Sobolev information to a metric one (see \cite{AmbrosioGigliSavare11-2}).

A crucial property of the heat flow which is tightly linked to the lower curvature bound is the  {\bf Bakry-\'Emery contraction} estimate (see \cite{Gigli-Kuwada-Ohta10} and \cite{AmbrosioGigliSavare11-2}):
\begin{equation}
\label{eq:BE}
|\nabla\h_tf|^2\leq e^{-2Kt}\h_t(|\nabla f|^2),\qquad\mm\ae,\ \forall t\geq 0,\ \forall f\in W^{1,2}(\X).
\end{equation}

Following \cite{Savare13} we now  introduce the set  $\fsm\subset W^{1,2}(\X)$ of  `test' functions which we shall use in several instances:
\[
\fsm:=\Big\{f\in  D(\Delta) \cap L^\infty(\mm) \ :\ |\nabla f|\in L^\infty(\mm)\quad\text{and}\quad\Delta f\in W^{1,2}(\X)\Big\}.
\]
Notice that from \eqref{eq:sobtolip} we get in particular that
\[
\text{any $f\in\fsm$ has a Lipschitz representative $\bar f:\X\to\R$ with $\Lip(\bar f)\leq \||\nabla f|\|_{L^\infty(\mm)}$}
\]
and that from the $L^\infty\to \Lip$ regularization of the heat flow established in \cite{AmbrosioGigliSavare11-2} as a direct consequence of \eqref{eq:BE} we also have
\begin{equation}
\label{eq:prodtest}
f\in L^2\cap L^\infty(\mm)\qquad\Rightarrow\qquad\h_tf\in\fsm\quad\forall t>0,
\end{equation}
which in particular gives
\begin{equation}
\label{eq:testw12}
\fsm\text{ is dense in }W^{1,2}(\X).
\end{equation}
In order to further discuss the regularity properties of functions in $\fsm$, we introduce the notion of measure valued Laplacian, which comes out naturally from integration by parts (see \cite{Gigli12}).  Thus let   $\mes(\X)$ be the Banach space of finite Radon measures on $\X$ equipped with the total variation norm $\|\cdot\|_{\sf TV}$.
\begin{definition}[Measure valued Laplacian] The space $D(\bd)\subset W^{1,2}(\X)$ is the space of  $f\in W^{1,2}(\X)$ such that there is $\mu\in\mes(\X)$ satisfying
\[
\int \bar g\,\d\mu=-\int\la\nabla \bar g,\nabla f\ra\,\d\mm,\qquad \forall \bar g:\X\to\R\text{ Lipschitz with bounded support.}
\]
In this case the measure $\mu$ is unique and we shall denote it by  $\bd f$.
\end{definition}
The bilinearity of $(f,g)\mapsto \la\nabla f,\nabla g \ra\in L^1(\mm)$ ensures that $D(\bd)$ is a vector space and $\bd:D(\bd)\to \mes(\X)$ is linear.

\bigskip

In the rest of this introduction we will recall some regularity  results due to Savar\'e  \cite{Savare13} which have been obtained by proper generalization and adaptation of the arguments proposed by Bakry  \cite{Bakry83} (see also the recent survey \cite{BakryGentilLedoux14}).

Proposition \ref{prop:regtest} below is crucial in all what comes next: it provides the first concrete step towards the definition of Hessian by ensuring Sobolev regularity for the function $|\nabla f|^2$ for a given $f\in\fsm$:
\begin{proposition}\label{prop:regtest}
Let $f\in\fsm$. Then $|\nabla f|^2\in  D(\bd)\subset W^{1,2}(\X)$ and
\begin{equation}
\label{eq:regtest}
\begin{split}
\E (|\nabla f|^2)&\leq -\int K|\nabla f|^4+|\nabla f|^2\la\nabla f,\nabla\Delta f\ra \,\d\mm,\\
\frac12\bd|\nabla f|^2&\geq \big(K|\nabla f|^2+\la\nabla f, \nabla\Delta f\ra\big)\mm.
\end{split}
\end{equation}
\end{proposition}
Formally, the second in \eqref{eq:regtest} is obtained by differentiating \eqref{eq:BE} at $t=0$ and then the first from the second by multiplying both sides by $|\nabla f|^2$ and integrating. In practice, we remark that to establish that the distributional Laplacian of $|\nabla f|^2$ is a measure, is particularly difficult in the general setting of complete and separable metric spaces (in contrast with the case of locally compact spaces, where Riesz theorem ensures that `positive distributions are measures') and requires delicate technical properties of the Dirichlet form $\E$, see \cite{Savare13} for the throughout discussion.

\medskip

Notice that for $f,g\in\fsm$ it is immediate to verify that  $fg\in L^\infty\cap W^{1,2}(\X)$ with $|\nabla (fg)|\in L^\infty(\mm)$. Recalling the Leibniz rule for the Laplacian \eqref{eq:leiblap}, we see that $fg\in D(\Delta)$ with
\[
\Delta(fg)=f\Delta g+g\Delta f+2\la\nabla f,\nabla g\ra
\]
and the fact that $\Delta f,\Delta g$ are in $W^{1,2}(\X)$ and $f,g$ bounded with bounded differential  grants that both $f\Delta g$ and $g\Delta f$ are in $W^{1,2}(\X)$. Here it comes a first crucial information from Proposition \ref{prop:regtest}: by polarization we have that $\la\nabla f,\nabla g\ra\in W^{1,2}(\X)$, so that in summary we proved that
\begin{equation}
\label{eq:testalg}
\text{$\fsm$ is an algebra.}
\end{equation}
Still, we remark that in general for $f\in\fsm$ we don't have $\Delta f\in \fsm$.

\medskip

The fact that $\la\nabla f,\nabla g\ra\in W^{1,2}(\X)$ for every $f,g\in\fsm$  allows to define a sort of  {\bf ``Hessian''} $H[f]:[\fsm]^2\to L^2(\mm)$ of a function  $f\in\fsm$  as
\begin{equation}
\label{eq:defhf}
H[f]( g,  h):=\frac12\Big(\bigl<\nabla\la\nabla f,\nabla g\ra,\nabla h \bigr>+\bigl< \nabla \la\nabla f ,\nabla h\ra,\nabla g\bigr> -\bigl< \nabla f,\nabla \la\nabla g,\nabla h\ra\bigr> \Big),
\end{equation}
the terminology being justified by the fact that in the smooth case $H[f](g,h)$ is precisely the Hessian of $f$ computed along the directions $\nabla g,\nabla h$.

Still by Proposition \ref{prop:regtest} we know that $\la\nabla f,\nabla g\ra\in D(\bd)$   for every $f,g\in\fsm$ and therefore we can  define the {\bf measure valued $\Gamma_2$} operator as the map $\Ggamma_2:[\fsm]^2\to\mes(\X)$ given by
\[
\Ggamma_2(f,g):=\frac12\bd\la\nabla f,\nabla g\ra-\frac12\Big(\la\nabla f,\nabla\Delta g\ra+\la\nabla g,\nabla\Delta f\ra\Big)\mm,\qquad\forall f,g\in\fsm.
\]
In the smooth setting, $\Ggamma_2(f,g)$ is always absolutely continuous w.r.t.\ the volume measure and its density is given by $\H f:\H g+{\rm Ric}(\nabla f,\nabla g)$. Notice that $\Ggamma_2$ is bilinear and symmetric and that the second in \eqref{eq:regtest} can be restated as
\begin{equation}
\label{eq:gamma2}
\Ggamma_2(f,f)\geq K|\nabla f|^2\mm.
\end{equation}
It is then easy to verify that
\begin{equation}
\label{eq:massaggamma2}
\Ggamma_2(f,f)(\X)=\int(\Delta f)^2\,\d\mm,\qquad\qquad \|\Ggamma_2(f,f)\|_{\sf TV}\leq \int (\Delta f)^2+2K^-|\nabla f|^2\,\d\mm.
\end{equation}
A technically useful fact is contained in the following lemma, which establishes a chain rule for $\Ggamma_2$:
\begin{lemma}[Multivariate calculus for $\Ggamma_2$]\label{thm:BS}
Let $n\in\N$, $f_1,\ldots,f_n\in\fsm$ and $\Phi\in C^\infty(\R^n)$ be with $\Phi(0)=0$. Put ${\bf f}:=(f_1,\ldots,f_n)$.

Then  $\Phi({\bf f})\in\fsm$.

Moreover, denoting by $\Phi_i$ the partial derivative of $\Phi$ w.r.t. the $i$-th coordinate and   defining the measure ${\bf A}(\Phi({\bf f}))\in\mes(\X)$ and the functions $B(\Phi({\bf f})),C(\Phi({\bf f})),D(\Phi({\bf f}))\in L^1(\mm)$ as
\[
\begin{split}
{\bf A}(\Phi({\bf f}))&:=\sum_{i,j}\Phi_i({\bf f})\Phi_j({\bf f})\Ggamma_2(f_i,f_j)\\
B(\Phi({\bf f}))&:=2\sum_{i,j,k}\Phi_i({\bf f})\Phi_{j,k}({\bf f})H[f_i]( f_j,  f_k)\\
C(\Phi({\bf f}))&:=\sum_{i,j,k,h}\Phi_{i,k}({\bf f})\Phi_{j,h}({\bf f})\la\nabla f_i,\nabla f_j\ra\la\nabla f_k,\nabla f_h\ra\\
D(\Phi({\bf f}))&:=\sum_{i,j}\Phi_i({\bf f})\Phi_j({\bf f})\la\nabla f_i,\nabla f_j\ra
\end{split}
\]
we have
\[
\begin{split}
|\nabla \Phi({\bf f})|^2&=D(\Phi({\bf f}))\\
\Ggamma_2(\Phi({\bf f}))&={\bf A}(\Phi({\bf f}))+\big(B(\Phi({\bf f}))+C(\Phi({\bf f}))\big)\mm.
\end{split}
\]
\end{lemma}
We conclude recalling the following improved version of the Bakry-\'Emery contraction rate estimate:
\begin{equation}
\label{eq:BE1}
|\nabla \h_t f|\leq e^{-Kt}\h_t(|\nabla f|),\quad\mm\ae\qquad\forall f\in W^{1,2}(\X),\ t\geq 0,
\end{equation}
which has also been obtained in \cite{Savare13} by generalizing the approach in \cite{Bakry83}.

\subsection{Test objects and some notation}\label{se:test}
Here we introduce the class of test vector fields/differential forms, discuss some basic density properties and fix some notation that will be used throughout the rest of the text.

\vspace{1cm}

As in the previous chapter, $L^2(T^*\X)$ and $L^2(T\X)$ will be the cotangent and tangent module respectively. For $\MM=L^2(T^*\X)$ (resp. $\MM=L^2(T\X)$) the corresponding $L^0$-module $\MM^0$ as defined in Section \ref{se:altint} will be denoted $L^0(T^*\X)$ (resp. $L^0(T\X)$), then for $p\in[1,\infty]$ we shall denote by $L^p(T^*\X)$ (resp. $L^p(T\X)$) the corresponding subclass of elements with pointwise norm in $L^p(\mm)$. Notice that Theorem \ref{thm:ACM} and Proposition \ref{prop:gencotan} give that $L^2(T\X), L^2(T^*\X)$ are separable, so that property \eqref{eq:sepmmp} gives that
\begin{equation}
\label{eq:seplp}
\text{$L^p(T\X)$ and $L^p(T^*\X)$ are separable for every $p<\infty$.}
\end{equation}
The Hilbert modules $L^2(T^*\X)$ and $L^2(T\X)$ are canonically isomorphic via the Riesz theorem for Hilbert modules (Theorem \ref{thm:rhil}) and it is convenient to give a name to their isomorphisms. Thus we introduce the {\bf musical isomorphisms} $\flat:L^2(T\X)\to L^2(T^*\X) $ and $\sharp:L^2(T^*\X)\to L^2(T\X)$ as
\[
X^\flat(Y):=\la X, Y\ra, \qquad \qquad \qquad \la \omega^\sharp, X\ra:=\omega(X),
\]
$\mm$-a.e.\ for every $X,Y\in L^2(T\X)$ and $\omega\in L^2(T^*\X)$. The maps $\flat$ and $\sharp$ uniquely extend to continuous maps from $L^0(T\X)$ to $L^0(T^*\X)$ and viceversa and then restrict to isometric isomorphisms of the modules $L^p(T\X)$ and $L^p(T^*\X)$. At the level of notation
\begin{quote} 
we shall drop the ${}_*$ in denoting the pointwise norm in $L^0(T^*\X)$ and thus will write $|\omega|$ for $\omega\in L^0(T^*\X)$.
\end{quote}
Since $L^2(T^*\X)$ and $L^2(T\X)$ are one the dual of the other, from \eqref{eq:dualmmp} we see that
\begin{equation}
\label{eq:dualibizzarri}
\forall p\in[1,\infty]\text{ the module } L^p(T^*\X)\text{ is reflexive and its dual is }L^q(T\X),\ \text{where }\frac1p+\frac1q=1. 
\end{equation}

A useful approximation result that we shall use in the following is:
\begin{equation}
\label{eq:aplipt}
\begin{split}
&\text{for any $f:\X\to \R$ Lipschitz with bounded support there is a sequence}\\
&\text{$(f_n)\subset \fsm$ $W^{1,2}(\X)$-converging to $f$, with $|f_n|,|\d f_n|$ uniformly bounded}\\
&\text{and such that $\Delta f_n\in L^\infty(\mm)$ for every $n\in\N$.}\\
&\text{Moreover, the $f_n$'s can be taken non-negative if $f$ is non-negative.}
\end{split}
\end{equation}
Indeed, a smoothing via the heat flow easy grants (recall \eqref{eq:prodtest}) a sequence in $\fsm$ $W^{1,2}(\X)$-converging to $f$ and properties \eqref{eq:boundlpcalore} and \eqref{eq:BE} yield the uniform bound on the functions and their differentials.  To achieve also bounded Laplacian, it is better to use the (time) mollified heat flow
\begin{equation}
\label{eq:mollheat}
\tilde\h_\eps f :=\frac1\eps\int_0^\infty\h_sf\varphi(s\eps^{-1})\,\d s,\qquad\text{ for some given }\varphi\in C^{\infty}_c(0,1)\text{ with }\int_0^1\varphi(s)\,\d s=1.
\end{equation}
It is then easy to check directly from the definitions and using the a priori estimates \eqref{eq:altrafacile} that $\tilde\h_\eps f\in\fsm$ for every $\eps>0$, that it still converges to $f$ in $W^{1,2}(\X)$ as $\eps\downarrow0$ with a uniform bound on the functions and differentials, and that its Laplacian can be computed as
\[
\begin{split}
\Delta \tilde\h_\eps f =\frac1\eps\int_0^\infty \Delta \h_sf\varphi(s\eps^{-1})\,\d s=\frac1\eps\int_0^\infty \Big(\frac{\d}{\d s} \h_sf\Big)\varphi(s\eps^{-1})\,\d s=-\frac1{\eps^2}\int_0^\infty\h_sf\,\varphi'(s\eps^{-1})\,\d s.
\end{split}
\]
Since $\h_sf$ is uniformly bounded in $L^\infty(\mm)$, this expression shows that $\Delta \tilde\h_\eps f\in L^\infty(\mm)$ for every $\eps>0$, thus concluding the proof of our claim.

From this approximation result it then follows that
\begin{equation}
\label{eq:conlapinfty}
\text{the linear span of $\{h\nabla g\ :\ h,g\in\fsm\text{ and }\Delta g\in L^\infty(\mm)\}$ is weakly$^*$ dense in $L^\infty(T\X)$.}
\end{equation}
Indeed, from \eqref{eq:prodtest} and \eqref{eq:aplipt} we see that the weak$^*$ closure of the above space contains the space, call it $V$, all vector fields of the form $\sum h_i\nabla g_i$ for $h_i\in L^2\cap L^\infty(\mm)$ and $g_i\in W^{1,2}(\X)$. Taking into account Proposition \ref{prop:gencotan} it is not hard to see that for $X\in L^2\cap L^\infty(T\X)$ there is a sequence $(X_n)\subset V$ converging to $X$ in $L^2(T\X)$ and uniformly bounded in $L^\infty(T\X)$, which is sufficient to get our claim.

\bigskip

We now introduce the class $\vsm\subset L^2(T\X)$ of {\bf test vector fields} as
\[
\vsm:=\Big\{\sum_{i=1}^ng_i\nabla f_i\ :\ n\in\N,\ f_i,g_i\in \fsm\ i=1,\ldots,n\Big\}
\]
and notice that arguing as above  we get  that $\vsm$ is dense in $L^2(T\X)$. The properties of test functions also ensure that $\vsm\subset L^1\cap L^\infty(T\X)$ and that $\vsm\subset D(\div)$.

\bigskip

The tensor product of $L^2(T^*\X)$  with itself will be denoted as $L^2((T^*)^{\otimes 2}\X)$, and for $\MM=L^2((T^*)^{\otimes 2}\X)$ the corresponding $L^0(\mm)$-module $\MM^0$ as defined in Section \ref{se:altint} will be denoted by  $L^0((T^*)^{\otimes 2}\X)$. Then $L^p((T^*)^{\otimes 2}\X)\subset L^0((T^*)^{\otimes 2}\X)$ is the space of elements whose pointwise norm is in   $L^p(\mm)$. Similarly from $L^2(T\X)$ to $L^2(T^{\otimes 2}\X)$ and $L^p(T^{\otimes 2}\X)$. We shall denote by $:$ the pointwise scalar product on $L^0(T^{\otimes 2}\X)$. By \eqref{eq:baseseptenssep} we see that  $L^2((T^*)^{\otimes 2}\X)$ and  $L^2(T^{\otimes 2}\X)$ are separable, so that from \eqref{eq:sepmmp} we deduce that
\begin{equation}
\label{eq:seplp2}
\text{$L^p(T^{\otimes 2}\X)$ and $L^p((T^*)^{\otimes 2}\X)$ are separable for every $p<\infty$.}
\end{equation}

Being $L^2((T^*)^{\otimes 2}\X)$ and $L^2((T^{\otimes 2}\X)$ Hilbert modules, they are reflexive and canonically isomorphic to their dual. It is then clear that we can identify one with the dual of the other, the duality mapping being
\[
(\omega_1\otimes\omega_2)(X_1\otimes X_2)=\omega_1(X_1)\omega_2(X_2),\quad\mm\ae,
\]
for every $ \omega_1\otimes\omega_2\in L^2((T^*)^{\otimes 2}\X)$ and $X_1\otimes X_2\in L^2(T^{\otimes 2}\X)$ and extended by linearity and continuity. As for the base case of $L^2(T^*\X)$ and $L^2(T\X)$ we can use such duality map together with the pointwise scalar product on $L^2(T^{\otimes 2}\X)$ to build the musical isomorphisms $\flat:L^2(T^{\otimes 2}\X)\to L^2((T^*)^{\otimes 2}\X) $ and $\sharp:L^2((T^*)^{\otimes 2}\X)\to L^2(T^{\otimes 2}\X)$ as
\[
T^\flat(S):= T: S, \qquad \qquad \qquad  A^\sharp: T:= A(T),
\]
for every $T,S\in L^2(T^{\otimes 2}\X)$ and $A\in L^2((T^*)^{\otimes 2}\X)$.

 We shall most often think of an element $A\in L^2((T^*)^{\otimes 2}\X)$ as the bilinear continuous map from $[L^0(T\X)]^2$ to $L^0(\mm)$ defined by
 \[
A(X,Y):=A(X\otimes Y),\qquad\mm\ae,
\]
 for every $X\otimes Y\in L^2(T^{\otimes 2}\X)$ and then extended by continuity. Notice that the identity $|X\otimes Y|_\HS=|X|\,|Y|$ gives
 \[
|A(X,Y)|\leq |A|_\HS\,|X|\,|Y|,\qquad\mm\ae,
\]
which is an instance of the fact that the operator norm is bounded from above by the Hilbert-Schmidt norm. In this sense, for given  $A\in L^2((T^*)^{\otimes 2}\X)$ and $X\in L^2(T\X)$, the objects $A(X,\cdot)$ and $A(\cdot, X)$ are the well defined elements of $L^1(T^*\X)$ given by 
\[
L^\infty(T\X)\ni Y\mapsto A(X,Y)\in L^1(\mm)\qquad\text{and}\qquad   L^\infty(T\X)\ni Y\mapsto A(Y,X)\in L^1(\mm)
\]
 respectively (recall \eqref{eq:dualibizzarri}).

We also remark that from property \eqref{eq:perdensotp}  and the density of $\vsm$ in $L^2(T\X)$ it follows that 
\begin{equation}
\label{eq:span2}
\begin{split}
&\text{the linear span of }\Big\{ h_1h_2\nabla g_1\otimes\nabla g_2: g_1,g_2,h_1,_2\in\fsm\Big\}\text{ is dense in }L^2(T^{\otimes 2}\X),
\end{split}
\end{equation}
and arguments similar to those yielding \eqref{eq:conlapinfty} grant that
\begin{equation}
\label{eq:conlapinfty2}
\begin{split}
&\text{the linear span of $\{h\nabla g_1\otimes\nabla g_2\ :\ h,g_1,g_2\in\fsm\text{ and }\Delta g_1,\Delta g_2\in L^\infty(\mm)\}$}\\
&\text{is weakly$^*$ dense in $L^\infty(T^{\otimes 2}\X)$.}
\end{split}
\end{equation}

\bigskip

Finally, the $k$-th exterior power of $L^2(T^*\X)$ will be denoted by $L^2(\Lambda^kT^*\X)$. In the case $k=0$ and $k=1$ we shall stick to the simpler notations $L^2(\mm)$ and $L^2(T^*\X)$ respectively.  Then for $\MM=L^2(\Lambda^kT^*\X)$  the corresponding space $\MM^0$ as built in Section \ref{se:altint} will be denoted by $L^0(\Lambda^kT^*\X)$ and $L^p(\Lambda^kT^*\X)\subset L^0(\Lambda^kT^*\X)$ is the space of elements with pointwise norm in $L^p(\mm)$. Elements of $L^0(\Lambda^kT^*\X)$ will be called $k$-differential forms, or simply $k$-forms. 
 
Notice that from \eqref{eq:extsep} and \eqref{eq:sepmmp} we get that
\begin{equation}
\label{eq:sepext}
\text{$L^p(\Lambda^kT^*\X)$ is separable for every $p<\infty$}.
\end{equation}
The space $\ffsm k\subset L^2(\Lambda^kT^*\X)$ of {\bf test $k$-forms} is defined as
\[
\begin{split}
\ffsm k:=\Big\{\text{linear combinations of forms of the}& \text{ kind }f_0\d f_1\wedge\ldots\wedge\d f_k\\
&\text{with }\ f_i\in\fsm\ \forall i=0,\ldots,k\Big\}.
\end{split}
\]
Arguing as for \eqref{eq:span2} and recalling that $\fsm$ is an algebra and that the exterior product is obtained as a quotient of the tensor product, it is easy to see that
\begin{equation}
\label{eq:densetforms}
\ffsm k\text{ is dense in $L^2(\Lambda^kT^*\X)$ for every $k\in\N$.}
\end{equation}
Similarly, the $k$-th  exterior power of $L^2(T\X)$ will be denoted by $L^2(\Lambda^kT\X)$. Evidently, since $\vsm\subset L^2\cap L^\infty(T\X)$, for $X_1,\ldots,X_k\in\vsm$ we have $X_1\wedge\ldots\wedge X_k\in L^2(\Lambda^kT\X)$ and as for \eqref{eq:densetforms} we have
\begin{equation}
\label{eq:denskvf}
\text{the linear span of $\{X_1\wedge\ldots\wedge X_k\ :\ X_1,\ldots,X_k\in\vsm\}$ is dense in $L^2(\Lambda^kT\X)$.}
\end{equation}
We shall most often denote the application of a $k$-form $\omega$ to $X_1\wedge\ldots\wedge X_k$ by $\omega(X_1,\ldots,X_k)$, i.e.\ we shall think to a $k$-form as a $k$-multilinear and alternating map acting on vector fields.

\subsection{Hessian}
\subsubsection{The Sobolev space $W^{2,2}(\X)$}
Here we define the Sobolev space $W^{2,2}(\X)$,  the approach being based on integration by parts.  Read in a smooth setting, the norm we put on $W^{2,2}(\X)$ is 
\[
\|f\|_{W^{2,2}(\X)}^2=\int |f|^2+|\d f|_*^2+|\H f|_\HS^2\,\d\mm,
\]
where $|\H f|_\HS$ denotes the Hilbert-Schmidt norm, so that the discussion about the tensor product of Hilbert modules we did in Section \ref{se:tensorproduct} is crucial for the current purposes.

The idea behind the definition is to  recall  that in a smooth setting given a smooth function $f$ the validity of the identity
\[
2\H f(\nabla g_1,\nabla g_2)=\big< \nabla g_1,\nabla\la\nabla f,\nabla g_2\ra\big>+\big< \nabla g_2,\nabla\la\nabla f,\nabla g_1\ra\big>-\big< \nabla f,\nabla\la\nabla g_1,\nabla g_2\ra\big>
\]
for a sufficiently large class of test functions $g_1,g_2$ characterizes the Hessian $\H f$ of  $f$.  Multiplying both sides of this identity by a smooth function $h$ and then integrating we then see that 
\[
\begin{split}
2\int h \H f&(\nabla g_1,\nabla g_2)\,\d\mm\\
&=\int -\la\nabla f,\nabla g_2\ra\div(h \nabla g_1) -\la\nabla f,\nabla g_1\ra\div(h \nabla g_2) -h\big< \nabla f,\nabla\la\nabla g_1,\nabla g_2\ra\big>\,\d\mm,
\end{split}
\]
and again the validity of this identity for a sufficiently large class of test functions $g_1,g_2,h$ characterizes the Hessian of $f$. The advantage of this formulation is that at the right hand side only one derivative of $f$ appears, so that we can use this identity to define which are the functions having an Hessian.

\vspace{1cm}

Notice that for $g_1,g_2,h\in\fsm$  we have $\div(h\nabla g_1)=\la\nabla h,\nabla g\ra+h\Delta g\in L^2(\mm)$ and thus $\nabla g_2\,\div(h\nabla g_1)\in L^2(T\X)$, and that from Proposition \ref{prop:regtest} we get $h\nabla\la\nabla g_1,\nabla g_2\ra\in L^2(T\X)$ as well.

\begin{definition}[The space $W^{2,2}(\X)$]\label{def:w22} The space $W^{2,2}(\X)\subset W^{1,2}(\X)$ is the space of all functions   $f\in W^{1,2}(\X)$ for which there exists $A\in L^2((T^*)^{\otimes 2}\X)$ such that for any $g_1,g_2,h\in\fsm$ the equality
\begin{equation}
\label{eq:defhess}
\begin{split}
2\int &hA(\nabla g_1,\nabla g_2)\,\d\mm\\
&=\int -\la\nabla f,\nabla g_1\ra \div(h\nabla g_2) -\la\nabla f,\nabla g_2\ra \div(h\nabla g_1)-h\big<\nabla f,\nabla\la\nabla g_1,\nabla g_2\ra\big>\,\d\mm,
\end{split}
\end{equation}
holds. In this case the operator $A$ will be called \emph{Hessian} of $f$ and denoted as $\H f$.

We endow $W^{2,2}(\X)$ with the norm $\|\cdot\|_{W^{2,2}(\X)}$ defined by
\[
\|f\|^2_{W^{2,2}(\X)}:=\|f\|_{L^2(\X)}^2+\|\d f\|_{L^2(T^*\X)}^2+\|\H f\|_{L^2((T^*)^{\otimes 2}\X)}^2
\]
and  define the functional $\ed:W^{1,2}(\X)\to[0,\infty]$ as:
\[
\ed (f):=\left\{\begin{array}{ll}
\displaystyle{\frac12\int |\H f|_\HS^2\,\d\mm},&\qquad if\ f\in W^{2,2}(\X),\\
&\\
+\infty,&\qquad otherwise.
\end{array}\right.
\]
\end{definition}
Notice that  the density property \eqref{eq:span2} and the fact \eqref{eq:testalg} that $\fsm$ is an algebra  ensure that there is at most one $A\in L^2((T^*)^{\otimes 2}\X)$  for which \eqref{eq:defhess} holds, so that Hessian is uniquely defined. It is then clear that it  linearly depends on $f$, so that $W^{2,2}(\X)$ is a vector space and  $\|\cdot\|_{W^{2,2}(\X)}$ a norm.

The basic properties of this space are collected in the following theorem.
\begin{theorem}[Basic properties of $W^{2,2}(\X)$]\label{thm:basew22} The following holds.
\begin{itemize}
\item[i)] $W^{2,2}(\X)$ is a separable Hilbert space. 
\item[ii)] The Hessian is a closed operator, i.e.\ the set $\{(f,\H f):f\in W^{2,2}(\X)\}$ is a closed subset of $W^{1,2}(\X)\times L^2(T^{\otimes 2}\X)$.
\item[iii)] For every $f\in W^{2,2}(\X)$ the Hessian $\H f$ is symmetric (recall the discussion in Section \ref{se:tensorproduct})
\item[iv)] The energy $\ed$ is lower semicontinuous on $W^{1,2}(\X)$ and for any $f\in W^{1,2}(\X)$  the following duality formula holds:
\begin{equation}
\label{eq:claimed}
\begin{split}
2\ed(f)=\sup&\Bigg\{\sum_{i,j}\int -\la\nabla f,\nabla g_{i,j}\ra\div(h_{i,i}\tilde h_{i,j}\nabla \tilde g_{i,j})-\la\nabla f,\nabla \tilde g_{i,j}\ra\div (h_{i,j}\tilde h_{i,j}\nabla g_{i,j})\\
&-h_{i,j}\tilde h_{i,j}\big<\nabla f,\nabla\la\nabla g_{i,j},\nabla \tilde g_{i,j}\ra\big>\,\d\mm-\Big\|\sum_{i,j}h_{i,j}\tilde h_{i,j}\nabla g_{i,j}\otimes\nabla\tilde g_{i,j}\Big\|^2_{L^2(T^{\otimes 2}\X)}\Bigg\},
\end{split}
\end{equation}
where the $\sup$ runs among all finite collections of functions $g_{i,j},h_{i,j},\tilde g_{i,j},\tilde h_{i,j}\in\fsm$.
\end{itemize}
\end{theorem}
\begin{proof} For given $g_1,g_2,h\in\fsm$ the left and right hand sides of  expression \eqref{eq:defhess} are continuous w.r.t.\ weak convergence of $f$ and $A$ in $W^{1,2}(\X)$ and $L^2(T^{\otimes 2}\X)$ respectively. This addresses point $(ii)$, which in turn also gives the completeness of $W^{2,2}(\X)$. The lower semicontinuity of $\ed$ follows from the same argument taking into account also that, being $L^2(T^{\otimes 2}\X)$ an Hilbert space, bounded sets are weakly relatively compact. Endowing $W^{1,2}(\X)\times L^2(T^{\otimes 2}\X)$ with the norm $\|(f,A)\|^2:=\|f\|^2_{W^{1,2}(\X)}+\|A\|^2_{L^2(T^{\otimes 2}\X)}$ we see that such space is Hilbert and separable (recall \eqref{eq:seplp} and \eqref{eq:seplp2}) and that the map
\[
W^{2,2}(\X)\ni f\qquad\mapsto\qquad (f,\H f)\in W^{1,2}(\X)\times L^2(T^{\otimes 2}\X),
\]
is an isometry of $W^{2,2}(\X)$ with its image, which grants the separability of $W^{2,2}(\X)$ and thus completes the proof of point  $(i)$.

The symmetry of the Hessian is a direct consequence of the symmetry in $g_1,g_2$ of the defining property \eqref{eq:defhess}.

It only remains to prove the duality formula \eqref{eq:claimed}. To check inequality $\geq$ we can assume  that  $f\in W^{2,2}(\X)$, or otherwise there is nothing to prove. In this case for $X_i=\sum_jh_{i,j}\nabla g_{i,j}$ and  $\tilde X_i=\sum_j\tilde h_{i,j}\nabla \tilde g_{i,j}$ in $\vsm$ by the very definition of $\H f$ we have
\[
\begin{split}
&2\sum_{i}\int \H f(X_i,\tilde X_i)\,\d\mm=2\sum_{i,j}\int h_{i,j}\tilde h_{i,j}\H f(\nabla g_{i,j},\nabla \tilde g_{i,j})\,\d\mm\\
&=\sum_{i,j}\int -\la \nabla f,\nabla g_{i,j}\ra\div (h_{i,i}\tilde h_{i,j}\nabla \tilde g_{i,j})-\la\nabla f,\nabla \tilde g_{i,j}\ra\div (h_{i,j}\tilde h_{i,j}\nabla g_{i,j})\\
&\hspace{4cm} \qquad\qquad\qquad\qquad\qquad\qquad\qquad-h_{i,j}\tilde h_{i,j}\big<\nabla f,\nabla\la\nabla g_{i,j},\nabla \tilde g_{i,j}\ra\big>\,\d\mm,
\end{split}
\]
and the claim follows from the simple duality formula $\|v\|^2=\sup_{L\in H'}2L(v)-\|L\|^2$ valid for any Hilbert space $H$ and any $v\in H$. We pass to the inequality $\leq $ and assume  that $f\in W^{1,2}(\X)$ is such that the $\sup$ at the right hand side of \eqref{eq:claimed} is finite, or otherwise there is nothing to prove. We claim that in this case for given finite number of functions $g_{i,j},h_{i,j},\tilde g_{i,j},\tilde h_{i,j}\in\fsm$, say $i=1,\ldots,n$ and, for given $i$, $j=1,\ldots, n_i$, the  value of
\begin{equation}
\label{eq:longexpr}
\begin{split}
\sum_{i,j}\int -\la\nabla f,\nabla g_{i,j}\ra\div (h_{i,i}\tilde h_{i,j}\nabla \tilde g_{i,j})-\la\nabla f&,\nabla \tilde g_{i,j}\ra\div (h_{i,j}\tilde h_{i,j}\nabla g_{i,j})\\
& -h_{i,j}\tilde h_{i,j}\big<\nabla f,\nabla\la\nabla g_{i,j},\nabla \tilde g_{i,j}\ra\big>\,\d\mm,
\end{split}
\end{equation}
depends only on the value of
\begin{equation}
\label{eq:teatro2}
B:=\sum_{i,j}h_{i,j}\tilde h_{i,j}\nabla g_{i,j}\otimes\nabla\tilde g_{i,j}\quad \in  L^2(T^{\otimes 2}\X)
\end{equation}
 and not on the particular  representation of  $B$ in terms of the functions $g_{i,j},h_{i,j},\tilde g_{i,j},\tilde h_{i,j}$ as above. Indeed, if not we could find functions $g_{i,j}'\tilde g_{i,j}',h'_{i,j},h_{i,j}'$ such that $\sum_{i,j}h'_{i,j}\tilde h'_{i,j}\nabla g_{i,j}'\otimes\nabla \tilde g_{i,j}'=0$ for which the corresponding value in \eqref{eq:longexpr} is not zero. But then choosing appropriate multiples of these functions in \eqref{eq:claimed} we would obtain that the $\sup$ at the right hand side is $+\infty$, contradicting our assumption.

Therefore indeed if $f\in W^{1,2}(\X)$ is such that the $\sup$ in the right hand side of  formula \eqref{eq:claimed}, which we shall denote by $S$, is finite, then the value of \eqref{eq:longexpr} depends only on the tensor $B$  in \eqref{eq:teatro2}. In other words, letting $V\subset  L^2(T^{\otimes 2}\X)$ be the linear span of elements of the form $h_1h_2\nabla g_1\otimes\nabla g_2$ for $g_1,g_2,h_1,h_2\in\fsm$, the map $l:V \to\R$ given by
\[
B\qquad\mapsto\qquad\text{ the value of  \eqref{eq:longexpr} for $g_{i,j},h_{i,j},\tilde g_{i,j},\tilde h_{i,j}$ such that \eqref{eq:teatro2} holds,}
\]
is well defined. It is then obvious that it is linear and that it holds
\[
l(B)\leq S+\frac12\|B\|^2_{L^2(T^{\otimes 2}\X)},\qquad\forall B\in  V.
\]
Choosing  $\lambda B$ in place of $B$ in this inequality and then optimizing over $\lambda\in\R$ we deduce that
\[
|l(B)|\leq \|B\|_{L^2(T^{\otimes 2}\X)} \sqrt{2S},\qquad\forall B\in  V.
\]
Recalling the density of $V$ in  $L^2(T^{\otimes 2}\X)$ (see \eqref{eq:span2}), we deduce that  $l$ admits a unique extension to a linear and continuous map, still denoted by $l$, from $L^2(T^{\otimes 2}\X)$ to $\R$. In other words, such $l$  is an element of the dual $L^2(T^{\otimes 2}\X)'$ of $L^2(T^{\otimes 2}\X)$ as Banach space and being $L^2(T^{\otimes 2}\X)$ an Hilbert module, it has full dual and is reflexive (Propositions \ref{prop:hill2}, \ref{prop:fulllp} and Corollary \ref{cor:hilrefl}), and hence  there exists $A\in L^2((T^*)^{\otimes 2}\X)$ such that 
\[
l(B)=\int A(B)\,\d\mm,\qquad\forall B\in L^2(T^{\otimes 2}\X).
\]
By definition of $l$, we see that for such $A$ the identity \eqref{eq:defhess} holds for any $g_1,g_2,h\in\fsm$, i.e.\  $f\in W^{2,2}(\X)$ and $\H f=A$. Given that by construction we have $\|\H f\|_{L^2((T^*)^{\otimes 2}\X)}=\|l\|_{L^2(T^{\otimes 2}\X)'}\leq \sqrt{2S}$, the proof is completed.
\end{proof}
\begin{problem}{\rm
Extend $\ed$ to $L^2(\mm)$ by defining it to be $+\infty$ on $L^2\setminus W^{1,2}(\X)$. Is the resulting  functional is $L^2$-lower semicontinuous?
}\fr\end{problem}

\begin{remark}[Hessian and Laplacian]{\rm It is well known  that by no means we can expect the Laplacian to be the trace of the Hessian. This is due to the fact that the former is defined via integration by parts, and thus takes into account the reference measure, while the latter is a purely differential object. 

This can be better realized on a Riemannian manifold $M$ equipped with a weighted measure $\mm:=e^{-V}{\rm vol}$: here the Hessian of a smooth function is independent on the choice of $V$, while for the Laplacian we have the classical formula
\[
\Delta_\mm f=\Delta f-\la\nabla f,\nabla V\ra,
\]
linking the Laplacian  $\Delta$ on the unweighted manifold, defined by $\int g\Delta f\,\d{\rm vol}=-\int\la\nabla g,\nabla f\ra\,\d{\rm vol}$,  to that $\Delta_\mm$ of the weighted one, defined by $\int g\Delta_\mm f\,\d\mm=-\int\la\nabla g,\nabla f\ra\,\d\mm$.

At the level of speculation, we point out that given that the Laplacian is the trace of the Hessian on non-weighted Riemannian manifolds, and that for non-collapsing sequences of Riemannian manifolds with a uniform bound from below on the Ricci curvature the volume measures converge, one might expect this information to be linked to the notion of `non-collapsed $\RCD$ space', which is missing as of today.
}\fr\end{remark}
\begin{remark}{\rm
The duality formula \eqref{eq:claimed} for $\ed$ and the arguments used to prove it allow for an alternative equivalent definition of $W^{2,2}(\X)$: define it as $\{f:\ed(f)<\infty\}$, $\ed$ being defined by \eqref{eq:claimed}, and the corresponding Hessian via the duality arguments presented in the proof.
}\fr\end{remark}

\subsubsection{Why there are many $W^{2,2}$ functions}\label{se:whymanyw22}
In the last section we defined the space $W^{2,2}(\X)$ but as of now we don't know if it contains any non-constant function. In particular,  to notice that it is a priori not obvious that test functions are in $W^{2,2}(\X)$.

In the crucial lemma \ref{le:lemmachiave} below we prove the most important inequality of this paper, which among other things  will imply that $\fsm\subset W^{2,2}(\X)$. Read in the case of a smooth Riemannian manifold with Ricci curvature bounded from below by $K$ the lemma   tells that for a vector field $X$ and symmetric 2-tensor field $A$ we have
\begin{equation}
\label{eq:bochspiega}
|\nabla X:A|^2 \leq \Big(\Delta\frac{|X|^2}{2}+ \la X, (\Delta_\Ho X^\flat)^\sharp\ra-K|X|^2-|(\nabla X)_{\sf Asym}|_\HS^2\Big)|A|_\HS^2,
\end{equation}
where $\nabla X$ is the covariant derivative of $X$, $(\nabla X)_{\sf Asym}$ its antisymmetric part, and $\Delta_\Ho$ the Hodge Laplacian. Notice that optimizing in $A$ and writing  $(\nabla X)_{\sf Sym}:=\nabla X-(\nabla X)_{\sf Asym}$ for the symmetric part of the covariant derivative, we see that such inequality is equivalent to 
\[
|(\nabla X)_{\sf Sym}|_\HS^2\leq\Delta\frac{|X|^2}{2}+ \la X, (\Delta_\Ho X^\flat)^\sharp\ra-K|X|^2-|(\nabla X)_{\sf Asym}|_\HS^2,
\]
which is a way of writing Bochner inequality. 

In practice, we cannot yet state inequality \eqref{eq:bochspiega} as it is written because we still have to  introduce the various differential objects appearing there. We shall instead write \eqref{eq:bochspiega} somehow implicitly for a test vector field $X$ and $A$ of the form $\sum_i\nabla h_i\otimes \nabla h_i$ for test functions  $h_i$: compare \eqref{eq:bochspiega} with \eqref{eq:parteac1}.

\vspace{1cm}

We turn to the technical part and start with the following simple lemma concerning positive second order polynomials with coefficient measures. 
We remind that if $\mu,\nu\in \mes(\X)$ are two non-negative measures, the non-negative measure $\sqrt{\mu\nu}\in\mes(\X)$ is defined as
\[
\sqrt{\mu\nu}:=\sqrt{fg}\,\sigma,
\]
where $\sigma\in\mes(\X)$ is any non-negative measure such that $\mu,\nu\ll\sigma$ and $f,g$ are densities of $\mu,\nu$ w.r.t.\ $\sigma$ respectively. The definition is well posed because the right hand side of the above expression is independent on the particular choice of $\sigma$.

Recall also that for $\mu\in\mes(\X)$ the measure $|\mu|\in\mes(\X)$ is defined as $|\mu|:=\mu^++\mu^-$, where $\mu^+,\mu^-$ are the positive and negative parts, respectively,  of $\mu$ in its Jordan decomposition.

\begin{lemma}\label{le:lambda}
Let $\mu_1,\mu_2,\mu_3\in \mes(\X)$  and assume that 
\begin{equation}
\label{eq:lambda}
\lambda^2\mu_1+2\lambda\mu_2+\mu_3\geq 0,\qquad\forall\lambda\in\R.
\end{equation}
Then $\mu_1$ and $\mu_3$ are non-negative and the inequality
\begin{equation}
\label{eq:tesilambda}
|\mu_2|\leq\sqrt{\mu_1\mu_3},
\end{equation}
holds. In particular, $\mu_2\ll\mu_1$, $\mu_2\ll\mu_3$
\begin{equation}
\label{eq:mu2tv}
\|\mu_2\|_{\sf TV}\leq\sqrt{\|\mu_1\|_{\sf TV}\|\mu_3\|_{\sf TV}}
\end{equation}
and writing $\mu_i=\rho_i\mm+\mu_i^s$ with $\mu_i^s\perp\mm$, $i=1,2,3$ we have 
\begin{equation}
\label{eq:partesing}
\mu_1^s\geq 0,\qquad\mu_3^s\geq 0
\end{equation}
and
\begin{equation}
\label{eq:parteac}
|\rho_2|^2\leq \rho_1\rho_3,\qquad\mm\ae.
\end{equation}
\end{lemma}
\begin{proof} 
Taking $\lambda=0$ in \eqref{eq:lambda} yields $\mu_3\geq 0$ and dividing by $\lambda^2$ and letting $\lambda\to\infty$ gives $\mu_1\geq 0$. 

Now let $\nu:=|\mu_1|+|\mu_2|+|\mu_3|\in\mes(\X)$ so that $\mu_i\ll\nu$ for $1,2,3$ and put $\bar\eta_i:=\frac{\d\mu_i}{\d \nu}\in L^1(\nu)$, $i=1,2,3$. Inequality \eqref{eq:lambda} reads as: for every $\lambda \in\R$ it holds
\begin{equation}
\label{eq:passaggio}
\lambda^2\bar\eta_1+2\lambda\bar\eta_2+\bar\eta_3\geq 0,\qquad\nu\ae.
\end{equation}
Let $\bar E:=\{\bar \eta_1=0\}$ be defined up to $\nu$-negligible sets  and notice that the arbitrariness of $\lambda$ in \eqref{eq:passaggio} yields that $\bar\eta_2=0$ $\nu$-a.e.\ on $\bar E$. Then observe that for any 
partition $(\bar E_n)\subset \mathcal A$ of $\X\setminus \bar E$ and  sequence $(\lambda_n)\subset \R$, inequality \eqref{eq:passaggio} gives $\sum_n\nchi_{\bar E_n}\big(\lambda_n^2\bar\eta_1+2\lambda_n\bar\eta_2+\bar\eta_3\big)\geq 0$ $\nu$-a.e.,
so that choosing a sequences $(\bar E^m_n)$ and  $(\lambda^l_n)$ such that $\sum_n\nchi_{\bar E^l_n}\lambda^l_n\to -\frac{\bar\eta_2}{\bar\eta_1}$ $\nu$-a.e.\ on $\X\setminus \bar E$ as $l\to\infty$ and passing to the limit, we see that $-\frac{|\bar\eta_2|^2}{\bar \eta_1}+\bar\eta_3\geq 0$ $\nu$-a.e.\ on $\X\setminus \bar E$. Taking into account what we already proved on $\bar E$ we deduce that
\begin{equation}
\label{eq:rhotv}
|\bar\eta_2|^2\leq \bar\eta_1\bar\eta_3\qquad\nu\ae,
\end{equation}
which is \eqref{eq:tesilambda}. Then \eqref{eq:mu2tv} follows from \eqref{eq:rhotv} noticing that $\|\mu_i\|_{\sf TV}=\|\bar\eta_i\|_{L^1(\nu)}$ for $i=1,2,3$, \eqref{eq:partesing} comes from the non-negativity of $\mu_1,\mu_3$ and \eqref{eq:parteac} follows from \eqref{eq:rhotv} writing $\nu=\eta\mm+\nu^s$ with $\nu^s\perp\mm$ and observing that $\rho_i=\bar\eta_i\eta$ $\mm$-a.e.\ for $i=1,2,3$.
\end{proof}
We now turn to the main inequality. Recall that functions  $f\in\fsm$ have a Lipschitz continuous representative, uniquely defined on $\supp(\mm)$, which we denote by $\bar f$.
\begin{lemma}[Key inequality]\label{le:lemmachiave}
Let $n,m\in\N$ and $f_i,g_i,h_j\in\fsm$, $i=1,\ldots,n$, $j=1,\ldots,m$. Define the measure $\mu=\mu\big((f_i),(g_i)\big)\in \mes (\X)$ as
\[
\begin{split}
\mu\big((f_i),(g_i)\big):=&\sum_{i, i'}\bar g_i\bar g_{ i'}\big(\Ggamma_2(f_i,f_{ i'})-K\la\nabla f_i,\nabla f_{i'}\ra\mm\big)\\
&\qquad+\Big(2g_iH[f_i](f_{i'},g_{i'})+\frac{\la\nabla f_i,\nabla f_{i'}\ra\la\nabla g_i,\nabla g_{i'}\ra+\la\nabla f_i,\nabla g_{i'}\ra\la\nabla g_i,\nabla f_{i'}\ra}2\Big)\mm
\end{split}
\]
and write it as  $\mu=\rho \mm+\mu^s$ with $\mu^s\perp\mm$.

Then 
\begin{equation}
\label{eq:partesing1}
\mu^s\geq 0
\end{equation}
and
\begin{equation}
\label{eq:parteac1}
\bigg|\sum_{i,j}\la\nabla f_i,\nabla h_j\ra\la\nabla g_i,\nabla h_j\ra+g_iH[f_i](h_j, h_j)\bigg|^2\leq \rho \sum_{j, j'}|\la\nabla h_j,\nabla h_{j'}\ra|^2.
\end{equation}
\end{lemma}
\begin{proof}
Pick $\lambda\in\R$, ${\bf a}=(a_1,\ldots,a_n),{\bf b}=(b_1,\ldots,b_n)\in\R^n$ and ${\bf c}=(c_1,\ldots c_m)\in \R^m$ and define $\Phi:\R^{2n+m}\to\R$ as
\[
\Phi(x_1,\ldots,x_n,y_1,\ldots,y_n,z_1,\ldots,z_m):=\sum_i\lambda y_ix_i+a_ix_i-b_iy_i+\sum_j(z_j-c_j)^2,
\]
so that 
\[
\partial_{x_i}\Phi=\lambda y_i+a_i,\qquad\partial_{y_i}\Phi=\lambda x_i-b_i,\qquad\partial_{x_iy_i}\Phi=\lambda,\qquad\partial_{z_j}\Phi=2(  z_j-  c_j),\qquad\partial_{z_j z_j}\Phi=2,
\]
for every $i=1,\ldots,n$, $j=1,\ldots,m$ and all the other partial derivatives are 0.

Put   ${\mathbf f}:=(f_1,\ldots,f_n,g_1,\ldots,g_n,h_1,\ldots,h_m)$ and notice that by Theorem \ref{thm:BS} we have $\Phi({\mathbf f})\in \fsm$ and thus by \eqref{eq:gamma2} that
\begin{equation}
\label{eq:gamma2fil}
\Ggamma_2(\Phi({\mathbf f}),\Phi({\mathbf f}))-K|\nabla\Phi({\mathbf f})|^2\mm\geq 0.
\end{equation}
Denote by ${\bf A}(\lambda,{\mathbf a},{\mathbf b},{\mathbf c})$ the measure ${\bf A}(\Phi({\bf f}))$ defined in Theorem \ref{thm:BS} and similarly by $B(\lambda,{\mathbf a},{\mathbf b},{\mathbf c}),C(\lambda,{\mathbf a},{\mathbf b},{\mathbf c}),D(\lambda,{\mathbf a},{\mathbf b},{\mathbf c})$ the functions $B(\Phi({\bf f})),C(\Phi({\bf f})),D(\Phi({\bf f}))$.  

Then we have
\[
\begin{split}
{\bf A}(\lambda,{\mathbf a},{\mathbf b},{\mathbf c})&=\sum_{i, i'}(\lambda \bar g_i+a_i)(\lambda \bar g_{i'}+a_{i'})\Ggamma_2(f_i,f_{i'})+\text{other terms},\\
B(\lambda,{\mathbf a},{\mathbf b},{\mathbf c})&=2\sum_{i, i'}2(\lambda g_i+a_i)\lambda H[f_i](f_{i'},g_{i'})+2\sum_{i,j}(\lambda g_i+a_i)2H[f_i](h_j, h_j)+\text{other terms},\\
C(\lambda,{\mathbf a},{\mathbf b},{\mathbf c})&=\sum_{i, i'}2\lambda^2\big(\la\nabla f_i,\nabla f_{i'}\ra\la\nabla g_i,\nabla g_{i'}\ra+\la\nabla f_i,\nabla g_{i'}\ra\la\nabla g_i,\nabla f_{i'}\ra\big) \\
& \qquad \qquad+\sum_{i,j}8\lambda\la\nabla f_i,\nabla h_j\ra\la\nabla g_i,\nabla h_j\ra
+\sum_{j, j'}4|\la\nabla h_j,\nabla h_{ j'}\ra|^2+\text{other terms},\\
D(\lambda,{\mathbf a},{\mathbf b},{\mathbf c})&=\sum_{i, i'}(\lambda g_i+a_i)(\lambda g_{i'}+a_{i'})\la\nabla f_i,\nabla f_{i'}\ra+\text{other terms},
\end{split}
\]
where each `other term' contains either the factor $(\lambda\bar  f_i-b_i)$ or the factor $(\bar h_j-c_j)$ for some $i,j$.

Theorem \ref{thm:BS} and inequality \eqref{eq:gamma2fil} then gives that
\[
{\bf A}(\lambda,{\mathbf a},{\mathbf b},{\mathbf c})+\Big(B(\lambda,{\mathbf a},{\mathbf b},{\mathbf c})+C(\lambda,{\mathbf a},{\mathbf b},{\mathbf c})-KD(\lambda,{\mathbf a},{\mathbf b},{\mathbf c})\Big)\mm\geq 0,
\]
and being this inequality true for every choice of the constants ${\mathbf a},{\mathbf b}\in\R^n$ and ${\mathbf c}\in\R^m$, we deduce that for every partition $(\bar E_k)\subset\mathcal A$ of $\X$ and choice of constants ${\mathbf a}_k,{\mathbf b}_k\in\R^n$ and ${\mathbf c}_k\in\R^m$ we have
\begin{equation}
\label{eq:discorso}
\sum_{k}\nchi_{\bar E_k}\Big(
{\bf A}(\lambda,{\mathbf a_k},{\mathbf b_k},{\mathbf c_k})+\big(B(\lambda,{\mathbf a_k},{\mathbf b_k},{\mathbf c_k})+C(\lambda,{\mathbf a_k},{\mathbf b_k},{\mathbf c_k})- KD(\lambda,{\mathbf a_k},{\mathbf b_k},{\mathbf c_k})\big)\mm\Big)\geq0.
\end{equation}
Recalling that for arbitrary $f,g,h\in \fsm$ the measure $\Ggamma_2(f,g)$ has finite mass and $H[f](g,h)$ is in $L^1(\mm)$,  we see that given that the functions $f_i,g_i,h_j\in\fsm$ are fixed, the total variation norm of the measure in the left hand side of this last inequality can be bounded in terms of the $\sup$ of $|{\mathbf a}_k|,|{\mathbf b}_k|,|{\mathbf c}_k|$. Then  consider a sequence of partitions $(\bar E^l_k)\subset \mathcal A$ and of uniformly bounded constants ${\mathbf a}_k^l,{\mathbf b}_k^l\in \R^n$ and ${\mathbf c}^l_k\in\R^m$, $k,l\in\N$ so that
\[
\sum_k\nchi_{\bar E^l_k}{\mathbf a}^l_k\to \lambda (\bar g_1,\ldots,\bar g_n),\qquad\sum_k\nchi_{E^l_k}{\mathbf b}^l_k\to \lambda (\bar f_1,\ldots,\bar f_n),\qquad\sum_k\nchi_{E^l_k}{\mathbf c}^l_k\to (\bar h_1,\ldots,\bar h_m),
\]
everywhere in $\X$ as $l\to\infty$. With these choices and passing to the limit as $l\to\infty$, we see that the left hand side of \eqref{eq:discorso} converges in the total variation norm,   that all the `other terms' vanish and that the factors $\lambda \bar g_i+a_i$ converge to $2\lambda \bar g_i$. Thus, after a rearrangement and a division by 4, we obtain
\begin{equation}
\label{eq:EFG}
\lambda^2{\bf E}+2\lambda {\bf F}+{\bf G}\geq 0,
\end{equation}
for ${\bf E},{\bf F},{\bf G}\in\mes(\X)$ given by
\[
\begin{split}
{\bf E}&:=\sum_{i, i'}\bar g_i\bar g_{ i'}\big(\Ggamma_2(f_i,f_{ i'})-K\la\nabla f_i, \nabla f_{i'}\ra\mm\big)+2g_iH[f_i](f_{i'},g_{i'})\mm\\
&\qquad\qquad\qquad\qquad\qquad\qquad+\frac{\la\nabla f_i,\nabla f_{i'}\ra\la\nabla g_i,\nabla g_{i'}\ra+\la\nabla f_i,\nabla g_{i'}\ra\la\nabla g_i,\nabla f_{i'}\ra}2\mm,\\
{\bf F}&:=\Big(\sum_{i,j}\la\nabla f_i,\nabla h_j\ra\la\nabla g_i,\nabla h_j\ra+g_iH[f_i](h_j, h_j)\Big)\mm,\\
{\bf G}&:=\sum_{j, j'}|\la\nabla h_j,\nabla h_{j'}\ra|^2\mm.
\end{split}
\]
Inequality \eqref{eq:EFG} holds for any value of $\lambda\in\R$, hence the conclusions \eqref{eq:partesing1}, \eqref{eq:parteac1} follow from  \eqref{eq:partesing}, \eqref{eq:parteac} respectively in   Lemma \ref{le:lambda}.
\end{proof}
The first important consequence of this lemma is the following result:

\begin{theorem}\label{thm:key2}
Let $f\in\fsm$. Then $f\in W^{2,2}(\X)$, and writing 
\[
\Ggamma_2(f,f)=\gamma_2(f,f)\mm+\Ggamma_2^s(f,f)\qquad\text{ with }\quad\Ggamma_2^s(f,f)\perp\mm
\]
we have
\begin{equation}
\label{eq:hess34}
|\H f|^2_\HS\leq \gamma_2(f,f)-K|\nabla f|^2,\qquad\mm\ae,
\end{equation}
and moreover for every $g_1,g_2\in\fsm$ it holds
\begin{equation}
\label{eq:hess12}
H[f](g_1,g_2)=\H f(\nabla g_1,\nabla g_2),\qquad\mm\ae.
\end{equation}
\end{theorem}
\begin{proof} We apply Lemma \eqref{le:lemmachiave} with $n=1$ for given functions $f,g,h_j\in\fsm$, $j=1,\ldots,m$. In this case inequality \eqref{eq:parteac1}   reads as:
\[
\begin{split}
\bigg|&\sum_{j}\la\nabla f,\nabla h_j\ra\la\nabla g,\nabla h_j\ra+gH[f](h_j, h_j)\bigg|^2\\
&\ \leq\bigg(g^2(\gamma_2(f,f)-K|\nabla f|^2)+\frac{|\nabla f|^2| \nabla g|^2+|\la\nabla f,\nabla g\ra|^2}2+2gH[f](f,g)\bigg)\sum_{j, j'}|\la\nabla h_j,\nabla h_{j'}\ra|^2
\end{split}
\]
$\mm$-a.e..  Notice that directly by the definition  \eqref{eq:defhf} of $H[f]$ we have $2H[f](f,g)=\la\nabla |\nabla f|^2,\nabla g\ra$ and using this observation we see that both sides of this inequality  vary continuously in $L^1(\mm)$ as $g$ varies in $W^{1,2}(\X)$. Hence by approximation we can pick $g\in W^{1,2}(\X)$ identically 1 on some bounded set $B\subset \X$, so that we have $\nabla g=0$ and $H[f](f,g)=0$ $\mm$-a.e.\ on $B$, and by the arbitrariness of such $B$ and recalling the definition of pointwise norm on $L^2(T^{\otimes 2}\X)$ we deduce
\begin{equation}
\label{eq:perA}
\bigg|\sum_{j}H[f](h_j, h_j)\bigg|\leq \sqrt{\gamma_2(f,f)-K|\nabla f|^2}\,\Big|\sum_{j}\nabla h_j\otimes\nabla h_j\Big|_\HS,\qquad\mm\ae,
\end{equation}
Now notice that the symmetry and bilinearity of $H[f]$ as map from $[\fsm]^2$ to $L^2(\mm)$ gives  that for arbitrary $g_j,h_j,h'_j\in\fsm $ we have
\[
\sum_jg_jH[f](h_j, h'_j)=\frac12\sum_jg_j\Big(H[f](h_j+ h'_j,h_j+ h'_j)- H[f](h_j, h_j)-H[f](h'_j, h'_j)\Big).
\]
Therefore observing that
\[
\sum_jg_j\frac{\nabla (h_j+h_j')\otimes\nabla(h_j+ h_j')-\nabla h_j\otimes\nabla h_j-\nabla h_j'\otimes\nabla h_j'}2=\sum_jg_j\frac{\nabla h_j\otimes\nabla h_j'+\nabla h_j'\otimes\nabla h_j}2,
\]
from inequality \eqref{eq:perA} we see that
\begin{equation}
\label{eq:perA2}
\begin{split}
\Big|\sum_jg_jH[f](h_j, h'_j)\Big|&\leq \sqrt{\gamma_2(f,f)-K|\nabla f|^2}\,\bigg|\sum_jg_j\frac{\nabla h_j\otimes\nabla h_j'+\nabla h_j'\otimes\nabla h_j}2\bigg|_\HS\\
&\leq \sqrt{\gamma_2(f,f)-K|\nabla f|^2}\Big|\sum_jg_j\nabla h_j\otimes\nabla h_j'\Big|_\HS,
\end{split}
\end{equation}
having used the simple inequality $|A_{\sf Sym}|_\HS\leq|A|_\HS$ $\mm$-a.e.\ for $A:=\sum_jg_j\nabla h_j\otimes\nabla h_j'$ (recall  \eqref{eq:decompsym}) in the second step.

Consider now the space  $V\subset L^2(T^{\otimes 2}\X)$ made of linear combinations of elements of the form $g\nabla h\otimes\nabla h'$ for $g,h,h'\in\fsm$  and define the operator $A:V\to  L^0(\mm)$ by 
\[
A\Big(\sum_{j}g_j\nabla h_j\otimes\nabla h'_j\Big):=\sum_jg_jH[f](h_j, h'_j).
\]
Inequality \eqref{eq:perA2} shows that this is a good definition, i.e.\ that the right hand side depends only on $\sum_jg_j\nabla h_j\otimes\nabla h_j'$ and not on the particular way of writing such sum. Moreover, recalling that by \eqref{eq:gamma2} we have $\Ggamma_2^s(f,f)\geq 0$, we obtain
\[
\int\gamma_2(f,f)-K|\nabla f|^2\,\d\mm\leq \Ggamma_2(f,f)(\X)-K\int |\nabla f|^2\,\d\mm\stackrel{\eqref{eq:massaggamma2}}=\int (\Delta f)^2-K|\nabla f|^2\,\d\mm
\]
and thus integrating \eqref{eq:perA2} and using the Cauchy-Schwarz inequality we see that $A$ takes values in $L^1(\mm)$ with
 \[
 \|A(T)\|_{L^1(\mm)}\leq \sqrt{\int (\Delta f)^2-K|\nabla f|^2\,\d\mm}\ \|T\|_{L^2(T^{\otimes 2}\X)},\qquad\forall T\in V.
 \]
Since by \eqref{eq:span2} (and the fact that $\fsm$ is an algebra) $V$ is dense in $L^2(T^{\otimes 2}\X)$, we deduce that $A$ can be uniquely extended to a continuous linear operator from $L^2(T^{\otimes 2}\X)$ to $L^1(\mm)$ and by its very definition we see that $A(gT)=gA(T)$ for every $T\in L^2(T^{\otimes 2}\X)$ and $g\in L^\infty(\mm)$, i.e.\ $A$ is a module morphism. In other words,  $A\in L^2((T^*)^{\otimes 2}\X)$.

To conclude observe that for $g,h_1,h_2\in\fsm$ and recalling the definition \eqref{eq:defhf} of $H[f]$ we have
\[
\begin{split}
2&\int A(g\nabla h_1\otimes \nabla h_2)\,\d\mm\\
&=2\int gH[f](h_1,h_2)\,\d\mm\\
&=\int g\Big(\big<\nabla\la\nabla f,\nabla h_1\ra,\nabla h_2\big>+\big<\nabla\la\nabla f,\nabla h_2\ra,\nabla h_1\big>-\big<\nabla f,\nabla\la\nabla h_1,\nabla h_2\ra\big>\Big)\,\d\mm\\
&=\int -\la\nabla f,\nabla h_1\ra\div (g\nabla h_2)-\la\nabla f,\nabla h_2\ra\div (g\nabla h_1)-g\big<\nabla f,\nabla\la\nabla h_1,\nabla h_2\ra\big>\,\d\mm,
\end{split}
\]
i.e.\ $f\in W^{2,2}(\X)$ and $\H f=A$. This very same computation shows that \eqref{eq:hess12} holds. For \eqref{eq:hess34} notice that \eqref{eq:perA2} can be restated as
\[
|\H f(T)|\leq  \sqrt{\gamma_2(f,f)-K|\nabla f|^2} \,\,|T|_\HS,\qquad\forall T\in V,
\]
and use once again the density of $V$ in $L^2(T^{\otimes 2}\X)$ together with the definition \eqref{eq:normdual} of dual pointwise norm to conclude.
\end{proof}
In particular, the following useful corollary holds:
\begin{corollary}\label{cor:bello} We have $D(\Delta)\subset   W^{2,2}(\X)$ and
\begin{equation}
\label{eq:key}
\mathcal E_2(f)\leq \int(\Delta f)^2-K|\nabla f|^2\,\d\mm,\qquad\forall f\in D(\Delta).
\end{equation}
\end{corollary}
\begin{proof} For $f\in\fsm$ the inequality follows integrating \eqref{eq:hess34} and  noticing that \eqref{eq:gamma2} yields $\Ggamma_2^s(f,f)\geq 0$ and thus
\[
\int\gamma_2(f,f)\,\d\mm\leq \Ggamma_2(f,f)(\X)\stackrel{\eqref{eq:massaggamma2}}=\int (\Delta f)^2\,\d\mm.
\]
For the general case, let $f\in D(\Delta)$  and for $n\in\N$ and $t>0$ apply inequality \eqref{eq:key} to $\h_t(f_n)\in \fsm$, where $f_n:=\min\{\max\{f,-n\},n\}$. Recalling that $\Delta\h_t(f_n)\to \Delta\h_t(f)$ in $L^2(\mm)$ as $n\to\infty$ and that $\Delta\h_t(f)\to\Delta f$ in $L^2(\mm)$ as $t\downarrow0$ and similarly that $\h_t(f_n)\to \h_t(f)$ in $W^{1,2}(\X)$ as $n\to\infty$ and $\h_t(f)\to f$ and $\Delta\h_tf\to\Delta f$ in $W^{1,2}(\X)$ and $L^2(\mm)$ respectively as $t\downarrow0$, the conclusion follows by letting first $n\to\infty$, then $t\downarrow0$ and recalling the $W^{1,2}$-lower semicontinuity of $\mathcal E_2$.
\end{proof}

\begin{remark}\label{rem:bss}{\rm
The whole structure of both the statement and the proof of the key Lemma \ref{le:lemmachiave} is heavily inspired by the paper \cite{Bakry83} of  Bakry. He worked in the setting `abstract and smooth' of Dirichlet forms admitting an algebra of bounded functions dense in $L^p(\mm)$ for all $p<\infty$  and stable by the heat semigroup and the Laplacian (this latter condition is the one that appears out of reach in the non-smooth context) and obtained the bound
\begin{equation}
\label{eq:ophess}
H[f](g,h)\leq \sqrt{\gamma_2(f,f)-K|\nabla f|^2}\,|\nabla g|\,|\nabla h|.
\end{equation}
This corresponds to inequality \eqref{eq:parteac1} with $n=m=1$ and $g=1$ and can be read as   a bound on the operator norm of the Hessian.

Then, Savar\'e \cite{Savare13} was able to adapt Bakry's arguments in the non-smooth setting obtaining inequality \eqref{eq:ophess} under the same assumptions we are making now. 

In the very recent paper \cite{Sturm14}, Sturm, working in a setting somehow analogous to that of Bakry, realized that (in the notation of Lemma \ref{le:lemmachiave}) picking arbitrary $m\in\N$ leads to the improved inequality \eqref{eq:hess34} which gives a bound on the Hilbert-Schmidt norm of the Hessian.

If one is only interested in the definition of $W^{2,2}(\X)$ and to the definition of the Ricci curvature calculated along gradient vector fields, then Sturm's arguments in conjunction with Savar\'e's techniques are sufficient. Our formulation seems instead necessary to get the control of the whole Ricci curvature tensor and to write Bochner identity for general vector fields, see Lemma \ref{le:riscritto} and Theorem \ref{thm:ricci}.
}\fr\end{remark}

\subsubsection{Calculus rules}\label{se:calchess}
\paragraph{Some auxiliary Sobolev spaces}
In this section we introduce those auxiliary Sobolev spaces which are needed to develop calculus rules for functions in $W^{2,2}(\X)$ in a reasonable generality. Two things are worth of notice. 

The first is that although the various formulas   are easily established for functions in $\fsm$, we don't know whether these are dense in $W^{2,2}(\X)$ or not, thus in some instance we will have to deal with the space $H^{2,2}(\X)\subset W^{2,2}(\X)$ defined as the $W^{2,2}(\X)$-closure of $\fsm$ (see Definition \ref{def:h22}, Proposition \ref{prop:gradehess} and Remark \ref{re:h22}).

The second is that we don't have a meaningful definition for the space $W^{2,1}(\X)$, and this will force us to work with bounded functions when considering the Hessian of $fg$ for $f,g\in W^{2,2}(\X)$.  The problem in defining $W^{2,1}(\X)$ is that we do not have an $L^\infty$ control on the term $\nabla\la\nabla g_1,\nabla g_2\ra$, for whatever rich choice of $g_1,g_2$, appearing in the rightmost addend of the defining equation \eqref{eq:defhess} of the Hessian. The only information that we have comes from Proposition \ref{prop:regtest}, which grants $L^2$-integrability for such object, thus forcing the gradient of $f$ to be also in $L^2$. For this reason, in defining what it is an Hessian in $L^1((T^*)^{\otimes 2}\X)$ we cannot drop the assumption $\nabla f\in L^2(T\X)$ and thus we are obliged to  introduce a space somehow intermediate between $W^{2,1}(\X)$ and $W^{2,2}(\X)$  (see Definition \ref{def:w221} and Propositions \ref{prop:prodfunct}, \ref{prop:chainhess}).

\vspace{1cm}

We start with first order spaces. 
\begin{definition}[The spaces $W^{1,1}(\X,\sfd,\mm)$ and $H^{1,1}(\X,\sfd,\mm)$]\label{def:w11} The space $W^{1,1}(\X)\subset L^1(\mm)$ is the space of those functions $f\in L^1(\mm)$ such that there exists $\d_1 f\in L^1(T^*\X)$, called the differential of $f$, such that
\begin{equation}
\label{eq:perw112}
\int \d_1 f(\nabla g)h\,\d\mm=-\int f\div(h\nabla g)\,\d\mm,\qquad\forall g,h\in\fsm\text{ with }\Delta g\in L^\infty(\mm).
\end{equation}
On $W^{1,1}(\X )$ we put the norm
\[
\|f\|_{W^{1,1}(\X)}:=\|f\|_{L^1(\mm)}+\|\d_1 f\|_{L^1(T^*\X)},
\]
and we define $H^{1,1}(\X)$ as the closure of $\fsm\cap W^{1,1}(\X)$ in $W^{1,1}(\X)$.
\end{definition}
Recalling the property \eqref{eq:conlapinfty} we see that \eqref{eq:perw112} determines $\d_1 f$, so that if it exists it is unique. It is then clear that it linearly depends on  $f\in W^{1,1}(\X)$, that $W^{1,1}(\X)$  is a vector space and that $\|\cdot\|_{W^{1,1}(\X)}$ is a norm. Moreover,  the terms in \eqref{eq:perw112} are continuous in $f$ and $\d_1f$ w.r.t.\ convergence in $L^1(\mm)$, $L^1(T^*\X)$ respectively, from which it follows that $W^{1,1}(\X)$, and thus also $H^{1,1}(\X)$, is a Banach space. These spaces are also separable, because the map from $W^{1,1}(\X)$ to $L^1(\mm)\times L^1(T^*\X)$ sending $f$ to $(f,\d_1 f)$ is an isometry, the target space being endowed with the (separable, due to \eqref{eq:seplp}) norm $\|(f,\omega)\|:=\|f\|_{L^1(\mm)}+\|\omega\|_{L^1(T^*\X)}$.

We chose the notation $\d_1f$ to highlight that a priori this definition could provide a notion different from the one given by Definition \ref{def:diff}. With an approximation argument we now check that the two notions, whenever comparable, are equivalent.
\begin{proposition}[Compatibility of $\d_1 f$ and $\d f$]\label{prop:diffcomp}  Let $f\in W^{1,1}\cap \s^2(\X)$. Then $\d_1f=\d f$.
\end{proposition}
\begin{proof}
Let $f_n:=\min\{\max\{f,-n\},n\}$ and notice that by the chain rule for $\d$ we have $f_n\in \s^2(\X)$ and $\d f_n=\nchi_{\{|f|< n\}}\d f$. Moreover, $f_n\in L^1\cap L^\infty(\mm)\subset L^2(\mm)$ and thus $f_n\in W^{1,2}(\X)$ so that by definition of divergence we have
\[
\int f_n{\rm div}(h\nabla g)\,\d\mm=-\int \d f_n(\nabla g)h\,\d\mm=-\int_{\{|f|\leq n\}} \d f(\nabla g)h\,\d\mm,\qquad\forall n\in\N,
\]
for arbitrary $g,h\in\fsm$ with $\Delta g\in L^\infty(\mm)$. Letting $n\to\infty$ and using the $L^1(\mm)$-convergence of $f_n$ to $f$ and the $L^2(T^*\X)$-convergence of $\nchi_{\{|f|< n\}}\d f$ to $\d f$ we deduce
\[
\int f{\rm div}(h\nabla g)\,\d\mm=-\int  \d f(\nabla g)h\,\d\mm.
\]
On the other hand, by definition of $W^{1,1}(\X)$ and $\d_1 f$ we know that the left hand side of this last identity is equal to $-\int  \d_1f(\nabla g)h\,\d\mm$, so that the conclusion follows from  the density property \eqref{eq:conlapinfty}.
\end{proof}
Thanks to this identification we shall denote the differential of functions in $W^{1,1}(\X)$ simply as $\d f$, thus dropping the subscript 1. It might be worth to remark at this point that the product of two test functions is in $L^1(\mm)$ and, by the Leibniz rule for $\d$,  in $W^{1,2}(\X)$ with differential in $L^1(T^*\X)$. Hence  $H^{1,1}(\X)$ is a dense subspace of  $L^1(\X)$.

It will be useful to keep in mind the following:
\begin{proposition}\label{prop:w11w12}
Let  $f\in L^2\cap W^{1,1}(\X)$ is such that $\d f\in L^2(T^*\X)$. Then $f\in W^{1,2}(\X)$.
\end{proposition}
\begin{proof}
Start observing that thanks to the approximation property \eqref{eq:aplipt} we can freely pick $h$ Lipschitz with bounded support in the definition of $W^{1,1}(\X)$. Therefore picking as $h$ a 1-Lipschitz function with bounded support and identically 1 on a ball of radius $R$ and then letting $R\to\infty$, we see that for any $f\in W^{1,1}(\X)$ it must hold
\begin{equation}
\label{eq:sonnino}
\int \d f(\nabla g)\,\d\mm=-\int f\Delta g \,\d\mm,
\end{equation}
for every $g\in\fsm$  with $\Delta g\in L^\infty(\mm)$. Then observe that for a generic $g\in\fsm$ if we consider its evolution $\tilde\h_tg$ along the mollified heat flow defined in \eqref{eq:mollheat}, we have that $\d\tilde h_tg\to\d g$ in $L^2(T^*\X)$ and $\Delta\tilde\h_tg\to\Delta g$ in $L^2(\mm)$ as $t\downarrow0$ (these are easy to establish from the fact that $t\mapsto \E(\h_tg)$ and $t\mapsto\|\Delta\h_tg\|_{L^2(\mm)}$ are non-increasing). Thus since we assumed $|f|,|\d f|\in L^2(\mm)$,   we see that \eqref{eq:sonnino} holds for any $g\in\fsm$. With a similar approximation argument based on first truncating $g$ and then regularizing the truncation via the heat flow, we see that \eqref{eq:sonnino} holds for general $g\in D(\Delta)$.

Thus pick $g:=\h_{2t}f$ and notice that
\[
\begin{split}
2\E(\h_tf)&=\int \la\nabla\h_tf,\nabla \h_tf\ra\,\d\mm=-\int f\Delta\h_{2t}f\,\d\mm\stackrel{\eqref{eq:sonnino}}=\int\d f(\nabla \h_{2t}f)\,\d\mm\\
&\leq\|\d f\|_{L^2(T^*\X)}\|\nabla \h_{2t}f\|_{L^2(T\X)}\leq \|\d f\|_{L^2(T^*\X)} \sqrt{2\E( \h_{t}f)},
\end{split}
\]
where in the last inequality we used the fact that $(0,\infty)\ni t\mapsto \|\nabla \h_{2t}f\|_{L^2(T\X)}=\sqrt{2\E(\h_tf)}$ is non-increasing. Letting $t\downarrow 0$ and using the $L^2(\mm)$-lower semicontinuity of $\E$ we conclude that $\E(f)\leq \|\d f\|_{L^2(T^*\X)}<\infty$ which, by definition of $W^{1,2}(\X)$, gives that $f\in W^{1,2}(\X)$.
\end{proof}
Notice that in this last proposition we are not able to drop the assumption $f\in L^2(\mm)$, compare with point $(ii)$ in Proposition \ref{prop:baseh11}.

Despite the identification of differentials, it is not clear to us if the same calculus rules which are valid for the differential of functions in $\s^2(\X)$ are also valid for functions in $W^{1,1}(\X)$, the problem being the little flexibility offered by the defining property \eqref{eq:perw112}.  For instance, we don't know if \eqref{eq:BE1} holds for functions in $W^{1,1}(\X)$ nor whether $\h_tf\in W^{1,1}(\X)$ for given $f\in W^{1,1}(\X)$. Also,  the only form of locality that we are able to prove is 
\begin{equation}
\label{eq:localdw11}
\d f_1=\d f_2,\qquad\mm\ae\text{ on the interior of }\{f_1=f_2\}\text{ for every }f_1,f_2\in W^{1,1}(\X),
\end{equation}
where by `interior of $\{f_1=f_2\}$' we mean the union of all the open sets $\Omega\subset\X$ such that $f_1=f_2$ $\mm$-a.e.\ on $\Omega$. Evidently, property \eqref{eq:localdw11}  is  weaker than the analogous valid for functions in $\s^2(\X)$ provided by Theorem \ref{thm:diffloc}. To establish \eqref{eq:localdw11} just notice that, as already remarked, thanks to the approximation property \eqref{eq:aplipt} we can freely pick $h$ Lipschitz with bounded support in the definition of $W^{1,1}(\X)$, so that the claim follows picking generic such $h$'s with support contained in open sets where $f_1,f_2$ coincide $\mm$-a.e..

Such locality property seems not sufficient to deduce any form of Leibniz or chain rule as we did in Corollary \ref{cor:calcdiff}.

On the other hand, it is easy to see that $\mm$-a.e.\ we have
\begin{equation}
\label{eq:dh11}
\begin{array}{rlll}
\d(f_1f_2)&\!\!\!=f_2\d f_1+f_1\d f_2,&&\quad\forall f_1,f_2\in H^{1,1}\cap L^\infty(\X) \\
\d (\varphi\circ f)&\!\!\!=\varphi'\circ f\,\d f,&&\quad\forall f\in H^{1,1}(\X),\ \varphi\in{\rm LIP}(\R),\\
\d f&\!\!\!=0,\qquad\ \ \ \text{ on }f^{-1}(\mathcal N)&&\quad\forall f\in H^{1,1}(\X),\ \mathcal N\subset \R\text{ Borel and negligible}\\
\d f_1&\!\!\!=\d f_2,\qquad\text{ on }\{f_1=f_2\}&& \quad\forall f_1,f_2\in H^{1,1}(\X),
\end{array}
\end{equation}
where in the chain rule we also require $\varphi(0)=0$  if $\mm(\X)=\infty$ in order not to destroy the $L^1(\mm)$ integrability and as usual the term $\varphi'\circ f$ is defined arbitrarily on the preimage of the points of non-differentiability of $\varphi$. Here it  is  part of the statement that $\varphi\circ f$ - and similarly $f_1f_2$ in the Leibniz rule - belongs to $H^{1,1}(\X)$. 

The Leibniz and the chain  rule for $\varphi\in C^1(\R)$ can be directly established by first recalling their validity for functions in $\fsm$ and passing to the limit in the $W^{1,1}$-topology, then the locality property follows from the Leibniz with the very same arguments used in the proof of Theorem \ref{thm:der} and the full chain rule also comes from the very same approximation procedure used in the proof of Theorem \ref{thm:der}. We omit the details.

A simple consequence of the chain rule is that
\begin{equation}
\label{eq:tronch11}
\begin{split}
&\text{for }f\in H^{1,1}(\X),\ \text{putting }f_n:=\min\{n,\{\max\{f,-n\}\}\}\text{ we have }f_n\in H^{1,1}(\X)\\ 
&\text{with }\d f_n=\nchi_{\{f<|n|\}}\d f \text{ and in particular }f_n\to f\text{ in $H^{1,1}(\X)$ as $n\to\infty$.}
\end{split}
\end{equation}
We then have the following basic results about the space $H^{1,1}(\X)$: notice that point $(ii)$ below is the one which justifies the introduction of the space $H^{1,1}(\X)$ itself, as we don't know if the same conclusion holds for functions in $W^{1,1}(\X)$.
\begin{proposition}[Basic facts about $H^{1,1}(\X)$]\label{prop:baseh11} The following hold.
\begin{itemize}
\item[i)] Let $f\in L^1\cap \s^2(\X)$ be such that  $\d f\in L^1(T^*\X)$. Then $f\in H^{1,1}(\X)$.
\item[ii)] Let  $f\in H^{1,1}(\X)$ be such that  $\d f\in L^2(T^*\X)$. Then $f\in\s^2(\X)$.
\end{itemize} 
\end{proposition}
\begin{proof}$\ $\\
\noindent{$\mathbf{(i)}$} For $f\in L^1\cap \s^2(\X)$ define  $f_n:=\min\{\max\{f,-n\},n\}$ so that $f_n\in L^1\cap L^\infty(\mm)\subset L^2(\mm)$ and, from the chain rule in Corollary \ref{cor:calcdiff}, $f_n\in\s^2(\X)$ with $\d f_n=\nchi_{\{|f|<n\}} \d f$. From the assumption $\d f\in L^1(T^*\X)$, we deduce that $\d f_n\in L^1(T^*\X)$ as well with $\d f_n\to\d f$ in $L^1(T^*\X)$ as $n\to\infty$.

For $t>0$ the function $\h_tf_n$ is in $\fsm$ (recall \eqref{eq:prodtest}) and  by \eqref{eq:BE1} we have
\[
|\d\, \h_tf_n|\leq e^{-Kt}\h_t(|\d f_n|),\qquad\mm\ae,
\]
so that the family of functions $\{|\d\, \h_tf_n|\}_{t>0}$ is dominated in $L^1(\mm)$. From the fact that $f_n\in W^{1,2}(\X)$ it easily follows that $\h_tf_n\to f_n$ in $W^{1,2}(\X)$ as $t\downarrow0$ and thus that $\d(\h_tf_{n}-f_n)\to 0$ in $L^2(T^*\X)$. This latter fact together with the domination of $\{|\d\, \h_tf_n|\}_{t>0}$ in $L^1(\mm)$ implies that  $\d \h_tf_n\to \d f_n$ in $L^1(T^*\X)$, i.e.\ $f_n\in H^{1,1}(\X)$, which  is sufficient to conclude.

\noindent{$\mathbf{(ii)}$}  Assume at first that $f$ also belongs to $L^\infty(\mm)$, so that  $f\in L^2(\mm)$ as well. In this case the thesis comes from Proposition  \eqref{prop:w11w12}.

For the general case, let  $f\in H^{1,1}(\X)$ and for $n\in\N$ define the functions $f_n:=\min\{\max\{f,-n\},n\}$.  By \eqref{eq:tronch11} we have $(f_n)\subset  H^{1,1}(\X)$ with $f_n\to f$ in $H^{1,1}(\X)$ and the above argument yields $f_n\in\s^2(\X)$. Moreover, the identity $\d(f- f_n)=\nchi_{\{|f|\geq n\}} \d f$ grants that $\|\d f_n\|_{L^2(T^*\X)}$ is uniformly bounded, so that the stability property \eqref{eq:lscwug} grants that $f\in \s^2(\X)$.
\end{proof}
The usefulness of $H^{1,1}(\X)$ is due to the following result:
\begin{proposition}[Leibniz rule for functions in $W^{1,2}(\X)$]
Let $f_1,f_2\in W^{1,2}(\X)$. Then $f_1f_2\in H^{1,1}(\X)$ and the formula
\[
\d(f_1f_2)=f_1\,\d f_2+f_2\,\d f_1,
\]
holds.
\end{proposition}
\begin{proof}
For $n\in\N$ let $f_{i,n}:=\min\{\max\{f_i,-n\},n\}$. Then recalling the Leibniz rule for the differential in Corollary \ref{cor:calcdiff} we have that $f_{1,n}f_{2,n}\in L^1\cap \s^2(\X)$ with 
\[
\d(f_{1,n}f_{2,n})=f_{1,n}\,\d f_{2,n}+f_{2,n}\,\d f_{1,n}
\] 
and in particular by point $(i)$ of Proposition \ref{prop:baseh11} above that $f_{1,n}f_{2,n}\in H^{1,1}(\X)$. By construction it is clear that $f_{1,n}f_{2,n}\to f_1f_2$ and $f_{1,n}\,\d f_{2,n}+f_{2,n}\,\d f_{1,n}\to f_{1}\,\d f_{2}+f_{2}\,\d f_{1}$ in $L^1(\mm)$ and $L^1(T^*\X)$ respectively as $n\to\infty$, thus passing to the limit in the Leibniz rule for the truncated functions we conclude.
\end{proof}
We can read the above result as follows: the natural form of Leibniz rule for the product of two $W^{1,2}(\X)$ is always in place and the differential of the product obeys calculus rules \eqref{eq:dh11} perfectly in line with those available for functions $W^{1,2}(\X)$ (as opposed to lack of calculus capabilities for functions in $W^{1,1}(\X)$). 

\begin{remark}{\rm
The problem of whether $W^{1,2}(\X)=H^{1,1}(\X)$ or not is `one order simpler' than the other analogous problems we will encounter later on, yet it still seems quite delicate. Notice that to get a positive answer to this question, following the same arguments used in point $(ii)$ of Proposition \ref{prop:baseh11} it would be sufficient to prove that $W^{1,1}\cap L^\infty(\X)$ is dense in $W^{1,1}(\X)$.

We also remark that the space $H^{1,1}(\X)$ coincides with what in the literature on metric measure spaces has often been called $W^{1,1}(\X)$: see in particular \cite{Ambrosio-DiMarino} and \cite{GigliHan14}. Here we chose the different terminology to distinguish those functions which can be approximated by `smooth' ones from those for which integration by parts holds.
}\fr\end{remark}

\bigskip

We now pass to second order spaces.
\begin{definition}[The space $H^{2,2}(\X)$]\label{def:h22}
We define $H^{2,2}(\X)\subset W^{2,2}(\X)$ as the $W^{2,2}$-closure of $\fsm$. 
\end{definition}
The following is easily established:
\begin{proposition}
$H^{2,2}(\X)$ coincides with the $W^{2,2}$-closure of $D(\Delta)$.
\end{proposition}
\begin{proof}
Since $\fsm\subset D(\Delta)$, clearly $H^{2,2}(\X)$ is contained in the $W^{2,2}$-closure of $D(\Delta)$. For the other inclusion, pick $f\in  D(\Delta)$ and for  $n\in\N$ put $f_n:=\min\{\max\{f,-n\},n\}$ and for $t>0$ put $f_{n,t}:=\h_t(f_n)$. By \eqref{eq:prodtest} we know that $f_{n,t}\in\fsm$ for every $n\in\N$ and $t>0$. 

Inequalities \eqref{eq:altrafacile}, \eqref{eq:key} grant that for every $t>0$ the family $\{f_{n,t}\}_{n\in\N}$ is bounded in $W^{2,2}$. Since evidently $\h_t(f_n)\to \h_t(f)$ in $L^2(\mm)$ as $n\to\infty$  we deduce that   $(\h_t(f_{n}))$ converges to $\h_t(f)$ weakly in $W^{2,2}(\X)$ as $n\to\infty$. 

Now let $t\downarrow0$ and observe that from $\|\Delta\h_tf\|_{L^2(\mm)}=\|\h_t\Delta f\|_{L^2(\mm)}\leq\|\Delta f\|_{L^2(\mm)}$ and again the bound \eqref{eq:key} we get  that $(\h_t(f))$ is bounded in $W^{2,2}(\X)$. Since moreover $\h_tf\to f$ in   $W^{1,2}(\X)$ as $t\downarrow0$,  we get that $(\h_tf)$ weakly converges to $f$ in $W^{2,2}(\X)$ as $t\downarrow0$.

The conclusion follows observing that being $\fsm$ a vector space, its closure w.r.t. the weak topology of $W^{2,2}$ coincides with its closure w.r.t the strong topology of $W^{2,2}$.
\end{proof}
We conclude introducing a space of functions with Hessian in $L^1((T^*)^{\otimes 2}\X)$. As said at the beginning of the section we cannot really define the space $W^{2,1}(\X)$ due to integrability issues in the definition of distributional Hessian. We therefore restrict the attention to the space $W^{(2,2,1)}(\X)$ of functions in $W^{1,2}(\X)$ with Hessian in $L^1(\X)$ (the three indexes in parenthesis indicate the integrability of the function, the differential and the Hessian respectively, so that in this notation we would have $W^{(2,2,2)}(\X)=W^{2,2}(\X)$).

\begin{definition}[The space $W^{(2,2,1)}(\X,\sfd,\mm)$]\label{def:w221} The space $W^{(2,2,1)}(\X)\subset W^{1,2}(\X) $ is the space of those functions $f\in W^{1,2}(\X)$ such that there exists $A\in L^1((T^*)^{\otimes 2}\X)$ such that
\begin{equation}
\label{eq:perw212}
\begin{split}
2\int &hA(\nabla g_1,\nabla g_2)\,\d\mm\\
&=\int -\la\nabla f,\nabla g_1\ra \div(h\nabla g_2) -\la\nabla f,\nabla g_2\ra \div(h\nabla g_1)-h\big<\nabla f,\nabla\la\nabla g_1,\nabla g_2\ra\big>\,\d\mm,
\end{split}
\end{equation}
for every $ g_1,g_2,h\in\fsm$. Such $A$ will be called Hessian of $f$ and denoted by $\H f$. On $W^{(2,2,1)}(\X)$ we put the norm
\[
\|f\|_{W^{(2,2,1)}(\X)}:=\|f\|_{L^2(\mm)}+\|\d f\|_{L^2(T^*\X)}+\|\H f\|_{L^1((T^*)^{\otimes 2}\X)}.
\]
\end{definition}
Property \eqref{eq:conlapinfty2} ensures that there is at most one $A$ for which \eqref{eq:perw212} holds, so that the Hessian is uniquely defined and it is clear that for $f\in W^{2,2}\cap W^{(2,2,1)}(\X)$ the two notions of Hessian given here and in Definition \ref{def:w22} coincide.

The definition also ensures that the Hessian linearly depends on $f\in W^{(2,2,1)}(\X)$ and that $W^{(2,2,1)}(\X)$ is a normed space. Noticing that the left hand side of \eqref{eq:perw212} is continuous w.r.t.\ $L^1((T^*)^{\otimes 2}\X)$-convergence of $A$ and the right hand side continuous w.r.t.\ $W^{1,2}(\X)$-convergence of $f$ we see that in fact $W^{(2,2,1)}(\X)$ is a Banach space.

We conclude pointing out that $W^{(2,2,1)}(\X)$ is separable, which can be checked observing that the map 
\[
\begin{array}{ccc}
W^{(2,2,1)}(\X)&\quad\to\quad &L^2(\mm)\times L^2(T^*\X)\times L^1((T^*)^{\otimes 2}\X)\\
f&\quad\mapsto\quad&(f,\d f,\H f)
\end{array}
\]
is an isometry of $W^{(2,2,1)}(\X)$ with its image, the target space being endowed with the (separable, due to \eqref{eq:seplp} and \eqref{eq:seplp2}) norm $\|(f,\omega,A)\|:=\|f\|_{L^2(\mm)}+\|\omega\|_{L^2(T^*\X)}+\|A\|_{L^1((T^*)^{\otimes 2}\X)}$.
\paragraph{Statement and proofs of calculus rules}
Here we collect the basic calculus rules involving the Hessian.

\vspace{1cm}

We start underlying the following property,  which grants a bit more flexibility in the choice of the function $h$ when testing the definition of Hessian:
\begin{equation}
\label{eq:sceltah}
\begin{split}
&\text{for  $f\in W^{2,2}(\X)$, $g_1,g_2\in\fsm$ and $h\in W^{1,2}\cap L^\infty(\X)$ we have}\\
&2\int h\,\H f(\nabla g_1,\nabla g_2)\,\d\mm\\
&\qquad=\int -\la\nabla f,\nabla g_1\ra \div(h\nabla g_2) -\la\nabla f,\nabla g_2\ra \div(h\nabla g_1)-h\big<\nabla f,\nabla\la\nabla g_1,\nabla g_2\ra\big>\,\d\mm.
\end{split}
\end{equation}
To see why this holds, notice that the choice of $\h_th\in \fsm$, $t>0$, in place of  $h\in W^{1,2}\cap L^\infty(\X)$ is admissible by the very definition of Hessian. Then observing that $|\H f(\nabla g_1,\nabla g_2)|\leq|\H f|_\HS\,|\nabla g_1||\nabla g_2|\in L^2(\mm)$ we can pass to the limit as $t\downarrow0 $ in the left hand side using the $L^2(\mm)$-convergence of $\h_th$ to $h$. To pass to the limit in the right hand side use the convergence of $(\d \h_th)$ to $\d h$ in $L^2(T^*\X)$ (consequence of the fact that $\h_th\to h$ in $W^{1,2}(\X)$) and the weak$^*$-convergence in $L^\infty(\mm)$ of $(\h_th)$ to $h$ (consequence of the uniform bound $\|\h_th\|_{L^\infty(\mm)}\leq \|h\|_{L^\infty(\mm)}$ and of $L^2(\mm)$-convergence).

\begin{proposition}[Product rule for functions]\label{prop:prodfunct}
Let $f_1,f_2\in W^{2,2}\cap L^\infty(\X)$. Then $f_1f_2\in W^{(2,2, 1)}(\X)$ and the formula
\begin{equation}
\label{eq:leibhess}
\H{f_1f_2}=f_2\H {f_1}+f_1\H {f_2}+\d f_1\otimes\d f_2+\d f_2\otimes \d f_1,\qquad\mm\ae
\end{equation}
holds.
\end{proposition}
\begin{proof} It is obvious that $f_1f_2\in W^{1,2}(\X)$ and that  the right hand side of \eqref{eq:leibhess} defines an object in $L^1((T^*)^{\otimes 2}\X)$. Now  let $g_1,g_2,h\in\fsm$ be arbitrary and notice that
\[
\begin{split}
-\la\nabla(f_1f_2),\nabla g_1\ra\,\div (h\nabla g_2)&=-f_1\la\nabla f_2,\nabla g_1\ra\,\div (h\nabla g_2)-f_2\la\nabla f_1,\nabla g_1\ra \,\div (h\nabla g_2)\\
&=-\la \nabla f_2,\nabla g_1\ra\,\div(f_1h\nabla g_2)+h\la\nabla f_2,\nabla g_1\ra\la\nabla f_1,\nabla g_2\ra\\
&\qquad-\la\nabla f_1,\nabla g_1\ra\,\div(f_2h\nabla g_2)+h\la\nabla f_1,\nabla g_1\ra\la\nabla f_2,\nabla g_2\ra.
\end{split}
\]
Exchanging the roles of $g_1,g_2$, noticing that
\[
-h\big<\nabla(f_1f_2),\nabla\la \nabla g_1,\nabla g_2\ra\big>=-hf_1\big<\nabla f_2,\nabla\la\nabla g_1,\nabla g_2\ra\big>-hf_2\big<\nabla f_1,\nabla\la\nabla g_1,\nabla g_2\ra\big>,
\]
adding everything up and integrating, we conclude using property \eqref{eq:sceltah} for the Hessian of the function $f_1$ (resp. $f_2$) and the choice of  $f_2h$ (resp. $f_1h$) in place of $h$.
\end{proof}

\begin{proposition}[Chain rule]\label{prop:chainhess}
Let $f\in W^{2,2}(\X)$  and $\varphi:\R\to\R$ a $C^{1,1}$ function with uniformly bounded first and second derivative (and $\varphi(0)=0$ if $\mm(\X)=+\infty$). 

Then $\varphi\circ f\in W^{(2,2, 1)}(\X)$ and the formula
\begin{equation}
\label{eq:chainhess}
\H{\varphi\circ f}=\varphi''\circ f\, \d f\otimes\d f+\varphi'\circ f\,\H f,\qquad\mm\ae
\end{equation}
holds.
\end{proposition}
\begin{proof} It is obvious that $\varphi\circ f\in L^2(\mm)$, that the chain rule for the differential ensures that $\varphi\circ f\in W^{1,2}(\X)$ and that the right hand side of \eqref{eq:chainhess} defines an object in $L^1((T^*)^{\otimes 2}\X)$. Now let  $g_1,g_2,h\in\fsm$ be arbitrary and notice that
\[
\begin{split}
-\la\nabla(\varphi\circ f),\nabla g_1\ra\,\div(h\nabla g_2)&=-\varphi'\circ f\la\nabla f,\nabla g_1\ra\,\div(h\nabla g_2)\\
&=-\la\nabla f,\nabla g_1\ra\,\div(\varphi'\circ fh\nabla g_2)+h\varphi''\circ f\la\nabla f,\nabla g_1\ra\la\nabla f,\nabla g_2\ra.
\end{split}
\]
Similarly,
\[
-\la\nabla(\varphi\circ f),\nabla g_2\ra\,\div(h\nabla g_1)=-\la\nabla f,\nabla g_2\ra\,\div(\varphi'\circ fh\nabla g_1)+h\varphi''\circ f\la\nabla f,\nabla g_2\ra\la\nabla f,\nabla g_1\ra.
\]
and
\[
-h\big< \nabla(\varphi\circ f),\nabla\la \nabla g_1,\nabla g_2\ra\big>=-h\varphi'\circ f\big< \nabla  f,\nabla\la \nabla g_1,\nabla g_2\ra\big>.
\]
Adding up these three identities, integrating and applying \eqref{eq:sceltah} with $h\varphi'\circ f$ in place of $h$ we conclude.
\end{proof}

\begin{proposition}[Product rule for gradients]\label{prop:gradehess}
Let $f_1\in W^{2,2}(\X)$ and $f_2\in H^{2,2}(\X)$. Then $\la\nabla f_1,\nabla f_2\ra\in W^{1, 1}(\X)$ and 
\begin{equation}
\label{eq:gradprod1}
 \d\la\nabla f_1,\nabla f_2\ra=\H {f_1}(\nabla f_2,\cdot)+\H {f_2}(\nabla f_1,\cdot),\qquad\mm\ae.
\end{equation}
In particular, for $f\in W^{2,2}(\X)$ and $g_1,g_2\in H^{2,2}(\X)$ the identity
\begin{equation}
\label{eq:hessesteso}
2\H f(\nabla g_1,\nabla g_2)=\big< \nabla g_1,\nabla\la\nabla f,\nabla g_2\ra\big>+\big< \nabla g_2,\nabla\la\nabla f,\nabla g_1\ra\big>-\big< \nabla f,\nabla\la\nabla g_1,\nabla g_2\ra\big>
\end{equation}
holds $\mm$-a.e.\ (notice that the two sides of this expression are well defined elements of $L^0(\mm)$).

Moreover:
\begin{itemize}
\item[i)] if  $f_1\in W^{2,2}(\X)$ and $f_2\in H^{2,2}(\X)$ have both bounded gradients, then $\la \nabla f_1,\nabla f_2\ra \in W^{1,2}(\X)$,
\item[ii)] if $f_1,f_2\in H^{2,2}(\X)$, then  $\la\nabla f_1,\nabla f_2\ra\in H^{1, 1}(\X)$.
\end{itemize}

\end{proposition}
\begin{proof} 
Let $f_1\in W^{2,2}(\X)$ and $f_2,g,h\in \fsm$ with $\Delta g\in L^\infty(\mm)$. Then by definition of $\H {f_1}$ we have
\begin{equation}
\label{eq:mah1}
\begin{split}
&2\int h\H{f_1}(\nabla f_2,\nabla g)\,\d\mm\\
&=\int-\la\nabla f_1,\nabla f_2\ra\,\div (h\nabla g)-\la\nabla f_1,\nabla g\ra\div\, (h\nabla f_2)-h\big< \nabla f_1,\nabla\la\nabla f_2,\nabla g\ra\big>\,\d\mm.
\end{split}
\end{equation}
Now observe that for $g_1,g_2,g_3,g_4\in\fsm$, the integration by parts   
\[
\int\la \nabla g_1,\nabla g_2\ra\,\div (g_3\nabla g_4)\,\d\mm=-\int g_3\big<\nabla g_4,\nabla\la \nabla g_1,\nabla g_2\ra\big> \,\d\mm
\]
is justified by the fact that $\la\nabla g_1,\nabla g_2\ra\in W^{1,2}(\X) $, therefore for arbitrary $f\in\fsm$ we have
\[
\begin{split}
&2\int h\H{f_2}(\nabla f,\nabla g)\,\d\mm\\
&=\int-\la\nabla f_2,\nabla f\ra\,\div(h\nabla g)+h\big<\nabla\la\nabla {f_2},\nabla g\ra,\nabla f\big>+\div(h\nabla f_2)\la \nabla f,\nabla g\ra\,\d\mm.
\end{split}
\]
The expressions at both sides of the above identity are continuous in $f$ w.r.t.\ the $W^{1,2}(\X)$-topology, hence approximating our given $f_1\in W^{2,2}(\X)$ in the $W^{1,2}(\X)$-topology with functions in $\fsm$ (recall \eqref{eq:testw12}) we deduce that
\begin{equation}
\label{eq:mah3}
\begin{split}
&2\int h\H{f_2}(\nabla f_1,\nabla g)\,\d\mm\\
&=\int-\la\nabla f_2,\nabla f_1\ra\,\div(h\nabla g)+h\big<\nabla\la\nabla {f_2},\nabla g\ra,\nabla f_1\big>+\div(h\nabla f_2)\la \nabla f_1,\nabla g\ra\,\d\mm.
\end{split}
\end{equation}
Adding up \eqref{eq:mah1} and \eqref{eq:mah3} we obtain
\begin{equation}
\label{eq:mah4}
\int h\big(\H{f_1}(\nabla f_2,\nabla g)+\H{f_2}(\nabla f_1,\nabla g)\big)\,\d\mm=\int-\la\nabla f_2,\nabla f_1\ra\,\div (h\nabla g)\,\d\mm.
\end{equation}
Noticing that $\la\nabla f_2,\nabla f_1\ra\in L^1(\mm)$ and $\H {f_1}(\nabla f_2,\cdot)+\H {f_2}(\nabla f_1,\cdot)\in L^1(T^*\X)$, \eqref{eq:mah4} and the arbitrariness of $g,h\in\fsm$ with $\Delta g\in L^\infty(\mm)$ give that $\la\nabla f_2,\nabla f_1\ra \in W^{1,1}(\X)$ and the formula  \eqref{eq:gradprod1}. 

To drop the requirement $f_2\in\fsm$, notice that for given  $f_1\in W^{2,2}(\X)$ and $g,h\in\fsm$ with $\Delta g\in L^\infty(\mm)$, the expressions at both sides of \eqref{eq:mah4} are continuous in $f_2$ w.r.t.\ the $W^{2,2}$-topology. Hence from the validity of \eqref{eq:mah4} for $f_2\in\fsm$ we deduce its validity for $f_2\in H^{2,2}(\X)$.

To get \eqref{eq:hessesteso}, just write \eqref{eq:gradprod1} for the couple $f,g_1$ and multiply both sides by $\nabla g_2$, then swap the roles of $g_1,g_2$ and finally subtract \eqref{eq:gradprod1} written for $g_1,g_2$ and multiplied by $\nabla f$.

For point $(i)$ observe that if $\nabla f_1$ and $\nabla f_2$ are both bounded, then $\la\nabla f_!,\nabla f_2\ra\in L^2(\mm)$ and the right hand side of \eqref{eq:gradprod1} defines an element in $L^2(T^*\X)$. Thus the conclusion follows from Proposition \ref{prop:w11w12}.

For point $(ii)$, notice that if $f_1,f_2\in\fsm$ then point $(i)$ grants that $\la\nabla f_1,\nabla f_2\ra\in W^{1,2}(\X)$ and formula \eqref{eq:gradprod1} yields that $\d\la\nabla f_1,\nabla f_2\ra\in L^1(T^*\X)$, so that point $(i)$ of Proposition \ref{prop:baseh11} gives the thesis. The general case follows by approximation:  for $f_1,f_2\in H^{2,2}(\X)$ and sequences $(f_{i,n})\subset \fsm$ converging to $f_i$ in $H^{2,2}(\X)$, $i=1,2$, we have that $\la\nabla f_{1,n},\nabla f_{2,n}\ra\to \la\nabla f_1,\nabla f_2\ra$ in $L^1(\mm)$ and $\H {f_{1,n}}(\nabla f_{2,n},\cdot)+\H {f_{2,n}}(\nabla f_{1,n},\cdot)\to \H {f_1}(\nabla f_2,\cdot)+\H {f_2}(\nabla f_1,\cdot)$ in $L^1((T^*)^{\otimes 2}\X)$ as $n\to\infty$, thus giving the thesis.
\end{proof}
\begin{remark}\label{re:h22}{\rm
The difficulty in getting Sobolev regularity for $\la\nabla f_1,\nabla f_2\ra$ for generic $f_1,f_2\in W^{2,2}(\X)$ is due to the fact that the definition of  Hessian is given testing it against gradients of test functions, so that in the approximation argument we cannot really go further than $H^{2,2}(\X)$.

One encounters a similar problem in trying to obtain the classical Leibniz rule for Sobolev functions on the Euclidean space without using the fact that smooth functions are dense in the Sobolev spaces.
}\fr\end{remark}
We pass to the locality properties of the Hessian. In the statement below we shall refer to the `interior of $\{f_1=f_2\}$' for $f_1,f_2\in W^{2,2}(\X)$. This set is, by definition, the union of all the open sets $\Omega\subset \X$ such that $\{f_1=f_2\}$ $\mm$-a.e.\ on $\Omega$.
\begin{proposition}[Locality of the Hessian]\label{prop:lochess}
For given $f_1,f_2\in W^{2,2}(\X)$ we have
\begin{equation}
\label{eq:loch1}
\H {f_1}=\H{f_2},\qquad\mm\ae\ \text{\rm on the interior of } \{f_1=f_2\},
\end{equation}
and for $f_1,f_2\in H^{2,2}(\X)$ the finer property
\begin{equation}
\label{eq:loch2}
\H {f_1}=\H{f_2},\qquad\mm\ae \ \text{\rm on } \{f_1=f_2\},
\end{equation}
holds. 
\end{proposition}
\begin{proof} For \eqref{eq:loch1}, by linearity it is sufficient to prove that for every $f\in W^{2,2}(\X)$ we have
\[
\H f=0,\qquad\mm\ae \ \text{on the interior of }\{f=0\},
\]
which in turn is equivalent to 
\[
\H f(\nabla g_1,\nabla g_2)=0,\qquad\mm\ae \ \text{on the interior of }\ \{f=0\},
\]
for every $g_1,g_2\in\fsm$. This follows from formula \eqref{eq:hessesteso}, noticing that Proposition \ref{prop:gradehess} grants that $\la\nabla f,\nabla g_i\ra\in W^{1,1}(\X)$ and using the locality property \eqref{eq:localdw11}.

Property \eqref{eq:loch2} follows along the same lines noticing that this time Proposition  \ref{prop:gradehess} grants that $\la\nabla f,\nabla g_i\ra\in H^{1,1}(\X)$   and using the locality property in \eqref{eq:dh11}. 
%
\end{proof}

\begin{remark}[$H^{2,2}(\X)$ cut-off functions]{\rm
In applications it might be useful to know whether there are $H^{2,2}(\X)$ cut-off functions, i.e. whether 
\[
\begin{split}
&\text{for every $B\subset \X$ and $\Omega\subset \X$ open with $\sfd({B,\Omega^c})>0$ there is $\eta_{B,\Omega}\in H^{2,2}(\X)$}\\
&\text{with $\eta_{B,\Omega}=1$ $\mm$-a.e.\ on $B$ and $\eta_{B,\Omega}=0$ $\mm$-a.e.\ on $\Omega^c$.}
\end{split}
\]
The answer is affirmative. In \cite{AmbrosioMondinoSavare13-2} the authors built,  for $B,\Omega$ as above, a function in $\fsm$   identically 1 on $B$, 0 on $\Omega^c$ and with bounded Laplacian.  Their argument is based on the fact that for a bounded and continuous function $f$ the heat flow $\h_tf$ converges to $f$ uniformly as $t\downarrow0$, so that picking $t$ small enough and composing the resulting function with an appropriate map from $\R$ to $\R$ gives the claim.

A different construction comes from \cite{Gigli-Mosconi14}. In this case a finite dimensionality requirement is also necessary, but the construction gives a bit more flexibility in the choice of the cut-off allowing to produce for given $c$-concave functions $\varphi,\psi\in W^{1,2}(\X)$ such that $-\psi\leq\varphi$, a function $f\in D(\Delta)$ with bounded Laplacian and such that $-\psi\leq f\leq \varphi$. The argument is based on the Laplacian comparison estimates for the distance function \cite{Gigli12} and the abstract Lewy-Stampacchia inequality \cite{Gigli-Mosconi14}.
}\fr\end{remark}

\subsection{Covariant derivative}
\subsubsection{The Sobolev space $W^{1,2}_C(T\X)$}\label{se:w12c}
The definition of Hessian that we gave in the previous chapter was based on the identity
\[
2\H f(\nabla g_1,\nabla g_2)=\big<\nabla\la\nabla f,\nabla g_1\ra,\nabla g_2\big>+\big<\nabla\la\nabla f,\nabla g_2\ra,\nabla g_1\ra-\big<\nabla f,\nabla\la\nabla g_1,\nabla g_2\ra\big>,
\]
valid on a smooth Riemannian manifold and for $f,g_1,g_2$ smooth. We are now proceeding in a similar way to define the covariant derivative using instead the identity
\[
\la\nabla_{\nabla g_1} X, \nabla g_2\ra=\big<\nabla\la X,\nabla g_2\ra,\nabla g_1\big>-\H {g_2}(\nabla g_1,X).
\]
It is worth to  notice that these two equalities can be used on a smooth context as an alternative to Koszul's formula to introduce the covariant derivative in terms of the metric tensor only, without the use of Lie brackets. This is technically convenient because we cannot hope to define the Lie bracket for general vector fields without imposing any sort of regularity to them, but to impose such regularity we need to know in advance what the covariant derivative is.

We also remark that Sobolev regularity is the only kind of regularity that we have for vector fields, as we don't know what it is a Lipschitz or a continuous vector field.

\vspace{1cm}

We recall the the pointwise scalar product of two tensors $A,B\in L^2(T^{\otimes 2}\X)$ is denoted by $A:B$.
\begin{definition}[The Sobolev space $W^{1,2}_C(T\X)$]\label{def:w12c}
The Sobolev space  $W^{1,2}_C(T\X)\subset  L^2(T\X)$ is the space of all $X\in L^2(T\X)$ for which there exists $T\in L^2(T^{\otimes 2}\X)$ such that for every $g_1,g_2,h\in\fsm$ it holds
\begin{equation}
\label{eq:defcov}
\int h\, T:  (\nabla g_1\otimes \nabla g_2)\,\d\mm=\int-\la X,\nabla g_2\ra\,\div(h\nabla g_1)-h\H{ g_2}(X,\nabla g_1)\,\d\mm.
\end{equation}
In this case we shall call the tensor $T$ the covariant derivative of $X$ and denote it by $\nabla X$. We  endow $W^{1,2}_C(T\X)$ with the norm $\|\cdot\|_{W^{1,2}_C(T\X)}$ defined by
\[
\|X\|_{W^{1,2}_C(T\X)}^2:=\|X\|^2_{L^2(T\X)}+\|\nabla X\|_{L^2(T^{\otimes 2}\X)}^2.
\]
Also, we define the connection energy functional $\ec:L^2(T\X)\to[0,\infty]$ as
\[
\ec(X):=\left\{
\begin{array}{ll}
\displaystyle{\frac12\int|\nabla X|_\HS^2\,\d\mm},&\qquad\text{ if }X\in W^{1,2}_C(T\X),\\
+\infty,&\qquad\text{ otherwise}.
\end{array}
\right.
\]
\end{definition}
The basic properties of $W^{1,2}_C(T\X)$ are collected in the following theorem, whose proof closely follows the one of Theorem \ref{thm:basew22} and, in the last points, makes use of the calculus rules for the Hessian that we developed in Section \ref{se:calchess}.
\begin{theorem}[Basic properties of $W^{1,2}_C(T\X)$]\label{thm:basew12c}
The following holds.
\begin{itemize}
\item[i)] $W^{1,2}_C(T\X)$ is a separable Hilbert space.
\item[ii)] The covariant derivative is a closed operator, i.e. the set $\{(X,\nabla X) : X\in W^{1,2}_C(T\X)\}$ is a closed subset of $L^2(T\X)\times L^2(T^{\otimes 2}\X)$.
\item[iii)] The connection energy functional $\ec:L^2(T\X)\to [0,\infty]$ is lower semicontinuous and for every $X\in L^2(T\X)$ the energy $\ec(X)$ is equal to
\[
\begin{split}
\sup\Bigg\{ \sum_i\int -\la X, Z_i\ra\, \div (Y_i)- \Big(\sum_j\la\nabla g_{i,j}, Y_i\ra\la\nabla f_{i,j}, X\ra&+g_j\H{f_{i,j}}(Y_i,X)\Big)\,\d\mm\\
&-\frac12\Big\|\sum_iY_i\otimes Z_i\Big\|^2_{L^2(T^{\otimes 2}\X)}\Bigg\},
\end{split}
\]
where the $\sup$ is taken among all $n\in\N$, $Y_i,Z_i\in\vsm$, $i=1,\ldots,n$, and over all the finite collections of functions $f_{i,j},g_{i,j}\in\fsm$ such that $Z_i=\sum_jg_{i,j}\nabla f_{i,j}$ for any $i$.
\item[iv)] For $f\in W^{2,2}(\X)$ we have $\nabla f\in W^{1,2}_C(T\X)$ with $\nabla(\nabla f)=(\H f)^\sharp$.
\item[v)] We have $\vsm\subset W^{1,2}_C(T\X)$ with 
\begin{equation}
\label{eq:covvsm}
\nabla X=\sum_i\nabla g_i\otimes\nabla f_i+g_i(\H {f_i})^\sharp,\qquad\text{ for }\qquad X=\sum_ig_i\nabla f_i.
\end{equation}
In particular, $W^{1,2}_C(T\X)$ is dense in $L^2(T\X)$.
\end{itemize}
\end{theorem}
\begin{proof}
For given $g_1,g_2,h \in\fsm$ the left and right hand sides of  expression \eqref{eq:defcov} are continuous w.r.t.\ weak convergence of $X$ and $T$ in $L^2(T\X)$ and $L^2(T^{\otimes 2}\X)$ respectively, which gives point $(ii)$. The lower semicontinuity of $\ec$ then follows taking also into account that, being $L^2(T^{\otimes 2}\X)$ an Hilbert space, bounded sets are weakly relatively compact. Endowing $L^{2}(T\X)\times L^2(T^{\otimes 2}\X)$ with the norm $\|(X,T)\|^2:=\|X\|^2_{L^{2}(T\X)}+\|T\|^2_{L^2(T^{\otimes 2}\X)}$ we see that such space is Hilbert and separable (recall \eqref{eq:seplp} and \eqref{eq:seplp2}) and that the map
\[
W^{1,2}_C(\X)\ni X\qquad\mapsto\qquad (X,\nabla X)\in L^{2}(T\X)\times L^2(T^{\otimes 2}\X),
\]
is an isometry of $W^{1,2}_C(T\X)$ with its image. Hence from point $(ii)$ point $(i)$ follows as well.

To prove $(iv)$ recall that by Proposition \ref{prop:gradehess}, for every $g_1,g_2,h\in\fsm$ we have
\[
\int h\,\H{f}(\nabla g_1,\nabla g_2)\,\d\mm=\underbrace{\int h\big<\nabla\la\nabla f,\nabla g_2\ra,\nabla g_1\big>\,\d\mm}_{=\int -\la\nabla f,\nabla g_2\ra\,\div(h\nabla g_1)\,\d\mm}-\int h\H{g_2}(\nabla f,\nabla g_1)\,\d\mm,
\]
which is the claim. Point $(v)$ follows along the same lines. Indeed, notice that by linearity it is sufficient to consider $X=g\nabla f$ and in this case the claim is equivalent to the validity of
\[
\begin{split}
\int h\la\nabla g,\nabla g_1\ra\la\nabla f,\nabla g_2\ra&+hg\,\H f(\nabla g_1,\nabla g_2)\,\d\mm\\
&=\int -g\la\nabla f,\nabla g_2\ra\,\div(h\nabla g_1)-h g\,\H{g_2}(\nabla f,\nabla g_1)\,\d\mm,
\end{split}
\]
for any $g_1,g_2,h\in\fsm$. But this is a direct consequence of the identity  
\[
\int -g\la\nabla f,\nabla g_2\ra\div(h\nabla g_1)\,\d\mm=\int h \big<\nabla(g\la\nabla f,\nabla g_2\ra) , \nabla g_1\big>\,\d\mm,
\]
whose validity is easily established from the fact that all the functions involved are in $\fsm$, and of Proposition \ref{prop:gradehess} which, taking into account that $\nabla f,\nabla g\in L^\infty(T\X)$, ensures that $\la\nabla f,\nabla g\ra\in W^{1,2}(\X)$ so that formula \eqref{eq:gradprod1} and the Leibniz rule for the differential give
\[
\d \big(g\la\nabla f,\nabla g_2\ra\big)  =\d g\,\la  \nabla f,\nabla g_2\ra+g\, \H f( \nabla g_2,\cdot)+ g\,\H {g_2}(\nabla f,\cdot).
\]
It remains to prove the duality formula for $\ec$, which, thanks to \eqref{eq:covvsm}, can be rewritten as
\begin{equation}
\label{eq:ridual}
\ec(X)=\sup \Bigg\{ \sum_i\int -\la X, Z_i\ra\, \div( Y_i)- \nabla Z_i:(Y_i\otimes X)\,\d\mm-\frac12\Big\|\sum_iY_i\otimes Z_i\Big\|^2_{L^2(T^{\otimes 2}\X)}\Bigg\}.
\end{equation}
Notice that for $X\in W^{1,2}_C(T\X)$ from the very definition of $\nabla X$ and the identity \eqref{eq:covvsm} it follows that
\[
\int\nabla X:(Y\otimes Z)\,\d\mm=\int- \la X,Z\ra\div(Y)- \nabla Z:(Y\otimes X)\,\d\mm,\qquad\forall Y,Z\in\vsm,
\]
so that from the trivial identity
\[
\frac12\|T\|_{L^2(T^{\otimes 2}\X)}^2=\sup \int T:\sum_iY_i\otimes Z_i\,\d\mm-\frac12\Big\|\sum_iY_i\otimes Z_i\Big\|^2_{L^2(T^{\otimes 2}\X)},
\]
where the $\sup$ is taken among all finite choices of $Y_i,Z_i$ in $\vsm$ (recall \eqref{eq:span2}), we get  inequality $\geq$ in \eqref{eq:ridual}.

The opposite inequality then follows along the very same arguments used to proved the analogous inequality in the duality formula \eqref{eq:claimed} in Theorem \ref{thm:basew22}, starting from the observation that if $X\in L^2(T\X)$ is such that the $\sup$ in \eqref{eq:ridual} is finite, then the value of 
\[
 \sum_i\int -\la X, Z_i\ra\, \div( Y_i)- \nabla Z_i:(Y_i\otimes X)\,\d\mm
\]
depends only on $B=\sum_iY_i\otimes Z_i$ and not on the particular way of writing $B$ as  sum. We omit the details.
\end{proof}

\subsubsection{Calculus rules}\label{se:lc}
In this section we collect the basic calculus rules for the covariant derivative and show that, in the appropriate sense, it satisfies the axioms of the Levi-Civita connection. As for the Hessian, we shall start introducing a couple of auxiliary Sobolev spaces.

\vspace{1cm}

We know that $\vsm$ is contained in $W^{1,2}_C(T\X)$, but not if it is dense. Thus the following definition is meaningful:
\begin{definition}[The space $H^{1,2}_C(T\X)$]
We define $H^{1,2}_C(T\X)\subset W^{1,2}_C(T\X)$ as the $W^{1,2}_C(T\X)$-closure of $\vsm$.
\end{definition}
Much like we couldn't define the space $W^{2,1}(\X)$, we cannot define the space $W^{1,1}_C(T\X)$ because we don't have at disposal a large class of functions with bounded Hessian. Thus, in analogy with the definition $W^{(2,2,1)}(\X)$, we introduce the space $W^{(2,1)}_C(T\X)$ of $L^2$ vector fields have covariant derivative in $L^1$:
\begin{definition}[The space $W^{(2,1)}_C(T\X)$] The space $W^{(2,1)}_C(T\X)\subset L^2(T\X)$ is the space of $X\in L^2(T\X)$  such that there exists $T\in L^1(T^{\otimes 2}\X)$ for which the identity
\begin{equation}
\label{eq:w11c}
\int h\, T:  (\nabla g_1\otimes \nabla g_2)\,\d\mm=\int-\la X,\nabla g_2\ra\,\div(h\nabla g_1)-h\H{ g_2}(X,\nabla g_1)\,\d\mm.
\end{equation}
holds for any $g_1,g_2,h\in\fsm$. We shall call such tensor $T$, which is unique thanks to \eqref{eq:conlapinfty2}, the covariant derivative of $X$ and denote it by $\nabla X$. We endow $W^{(2,1)}_C(T\X)$ with the norm 
\[
\|X\|_{W^{(2,1)}_C(T\X)}:=\|X\|_{L^2(T\X)}+\|\nabla X\|_{L^1(T^{\otimes 2}\X)}.
\]
\end{definition}
Given that the two sides of \eqref{eq:w11c} are continuous w.r.t.\ convergence (weak, in fact) of $X,T$ in $L^2(T\X)$ and $L^1(T^{\otimes 2}X)$ respectively, we see  that $W^{(2,1)}_C(T\X)$ is complete. Moreover, the embedding $X\mapsto (X,\nabla X)$ of $W^{(2,1)}_C(T\X)$ into $L^2(T\X)\times L^1(T^{\otimes 2}\X)$ endowed with the norm $\|(X,T)\|:=\|X\|_{L^2(T\X)}+\|T\|_{L^1(T^{\otimes 2}\X)}$ is an isometry, showing that  $W^{(2,1)}_C(T\X)$ is separable. 

We also remark that, by the very definitions, for $X\in W^{(2,1)}_C\cap W^{1,2}_C(T\X)$ the two definitions of covariant derivative given in \eqref{eq:defcov} and \eqref{eq:w11c} coincide, so no ambiguity occurs.

\begin{proposition}[Leibniz rule]\label{prop:leibcov}
Let $X\in W^{1,2}_C(T\X)$ and $f\in L^\infty\cap W^{1,2}(\X)$.

Then  $fX\in W^{(2,1)}_C(T\X)$ and  
\begin{equation}
\label{eq:leibcov}
\nabla(fX)= \nabla f\otimes X+f\nabla X,\qquad\mm\ae.
\end{equation}
\end{proposition}
\begin{proof} Assume at first $f\in\fsm$, let $g_1,g_2,h\in\fsm$ be arbitrary, notice that $fh\in\fsm$  and use the definition of $\nabla X$ to obtain that 
\[
\int fh\nabla X:(\nabla g_1\otimes\nabla g_2)\,\d\mm=\int-\la X, \nabla g_2\ra\,\div(fh\nabla g_1)- fh\H{g_2}(X,\nabla g_1)\,\d\mm.
\]
Recalling that by \eqref{eq:leibdiv} we have $\div(fh\nabla g_1)=h\la \nabla f,\nabla g_1\ra+f\div(h\nabla g_1)$, the above yields that
\[
\begin{split}
\int h\la\nabla f,\nabla g_1\ra\,\la X,\nabla g_2\ra + & hf\nabla X:(\nabla g_1\otimes \nabla g_2)\,\d\mm\\
&=\int -\la fX,\nabla g_2\ra\,\div(h\nabla g_1)- h\H{g_2}(fX,\nabla g_1)\,\d\mm,
\end{split}
\]
which is the thesis.

The case of  general  $f\in L^\infty\cap W^{1,2}(\X)$ follows by approximation noticing that  $\h_t f\in \fsm$, that $\|\h_t f\|_{L^\infty(\mm)}\leq \|f\|_{L^\infty(\mm)}$ and that $\h_t f\to f$   in $W^{1,2}(\X)$: these are sufficient to deduce that $(\h_tfX)$ and $(\nabla \h_tf\otimes X+\h_tf\nabla X)$ converge to $fX$ and $ \nabla f\otimes X+f\nabla X$ in $L^2(T\X)$ and $L^1(T^{\otimes 2}\X)$ respectively as $t\downarrow0$. The thesis follows.
\end{proof}
In the statements below, for $X\in W^{1,2}_C(T\X)$ and $Z\in L^0(T\X)$ we shall indicate by $\nabla_ZX$  the vector field in $L^0(T\X)$ defined by
\begin{equation}
\label{eq:nablac}
\la \nabla_ZX, Y\ra:= \nabla X:(Z\otimes Y),\qquad\mm\ae,\qquad\forall Y\in L^0(T\X),
\end{equation}
where the right hand side is firstly defined for $Z,Y\in L^0(T\X)$ such that $Z\otimes Y\in L^2(T^{\otimes 2}\X)$ and then extended by continuity to a bilinear map from $[L^0(T\X)]^2$ to $L^0(\mm)$ (recall Proposition \ref{prop:dualm0}  to see that this really defines a vector field in $L^0(T\X)$).

\begin{proposition}[Compatibility with the metric]\label{prop:compmetr}
Let $X\in W^{1,2}_C(T\X)$ and $Y\in H^{1,2}_C(T\X)$. Then $\la X, Y\ra \in W^{1, 1}(\X)$ and 
\begin{equation}
\label{eq:compmetr}
\d\la X,  Y\ra(Z)=\la\nabla_Z X,  Y\ra+\la \nabla_ZY,X\ra,\qquad\mm\ae,
\end{equation}
for every $Z\in L^0(T\X)$. Moreover:
\begin{itemize}
\item[i)] if $X\in W^{1,2}_C(T\X)$ and $Y\in H^{1,2}_C(T\X)$ are both bounded, then $\la X,Y\ra\in W^{1,2}(\X)$,
\item[ii)] if  $X,Y\in H^{1,2}_C(T\X)$, then $\la X, Y\ra\in H^{1,1}(\X)$.
\end{itemize}
\end{proposition}
\begin{proof} Assume at first that $Y=\nabla f$ for some $f\in\fsm$. Then for arbitrary  $X\in W^{1,2}_C(T\X)$ and $g,h\in\fsm$, by the very definition of $\nabla X$  we have
\[
\int h\nabla X:(\nabla g\otimes\nabla f)\,\d\mm=\int -\la X,\nabla f\ra \,\div(h\nabla g)-h\H f(X,\nabla g)\,\d\mm,
\]
which, recalling the Definition  \ref{def:w11} of $W^{1,1}(\X)$ and that $(\H f)^\sharp=\nabla(\nabla f)$ by point $(iv)$ of Theorem \ref{thm:basew12c}, yields that $\la X,\nabla f\ra \in W^{1,1}(\X)$ with
\[
\d\la X,\nabla f\ra(h\nabla g)=\la\nabla_{h\nabla g}X,\nabla f\ra+\la\nabla_{h\nabla g}\nabla f,X\ra,\quad\mm\ae.
\]
The the density of $\vsm$ in $L^2(T\X)$ and the definition of $L^0(T\X)$ gives the claim.

The case of $Y\in \vsm$ then follows by the linearity of the covariant derivative, what just proved and Proposition \ref{prop:leibcov} above.

For general $Y\in H^{1,2}_C(TX)$ the result follows by approximation:  if $Y_n\to Y$ in $W^{1,2}_C(T\X)$ and $Z\in L^\infty(T\X)$ is arbitrary, then $(\la X, Y_n\ra)$ and $(\la \nabla_Z X, Y_n\ra+\la \nabla_Z Y_n,X\ra)$ converge to $\la X,  Y\ra$ and   $\la \nabla_Z X, Y\ra+\la \nabla_Z Y,X\ra$ respectively in $L^1(\mm)$ as $n\to\infty$. From the arbitrariness of $Z\in L^\infty(T\X)$ we see that we can pass to the limit in the definition of functions $W^{1,1}(\X)$ and of their differential and obtain the result.

Now point $(i)$ follows noticing that if $X,Y$ are bounded, then $\la X,Y\ra\in L^2(\mm)$ and formula \eqref{eq:compmetr} defines an object in $L^2(T^*\X)$, so that the conclusion follows from Proposition \ref{prop:w11w12}.

For point $(ii)$, notice that for $X,Y\in\vsm$ point $(i)$ grants that $\la X,Y\ra\in W^{1,2}(\X)$ while formula \eqref{eq:compmetr} yields that $\d\la X,Y\ra\in L^1(T^*\X)$, so that point $(i)$ of Proposition \ref{prop:baseh11} gives the thesis. The general case then follows by approximation. Indeed, for $X,Y\in H^{1,2}_C(T\X)$ we can find $(X_n),(Y_n)\subset \vsm$ converging to $X,Y$ respectively in $H^{1,2}_C(T\X)$ as $n\to\infty$. It is the clear that $\la X_n,Y_n\ra\to\la X_,Y\ra$ in $L^1(\mm)$ and that $\d\la X_n,Y_n\ra\to \d\la X,Y\ra$ in $L^1(T^*\X)$ as $n\to\infty$, thus giving the thesis. 
\end{proof}
We now pass to  the torsion-free identity, which, as in the smooth case, follows directly from the symmetry of the Hessian and the compatibility with the metric. For a function $f$ in $W^{1,2}(\X)$ or $W^{1,1}(\X)$ and a vector field $X\in L^2(T\X)$ we will write for brevity $X(f)$ in place of  $\d f(X)$. 
\begin{proposition}[Torsion free identity]
Let $f\in H^{2,2}(\X)$ and $X,Y\in W^{1,2}_C(T\X)$. Then $X(f),Y(f)\in W^{1, 1}(\X)$ and the identity
\begin{equation}
\label{eq:torsionfree}
X(Y(f))-Y(X(f))=\d f(\nabla_XY-\nabla_YX),\qquad\mm\ae,
\end{equation}
holds.
\end{proposition}
\begin{proof} Notice that for $f\in \fsm$ we have $\nabla f\in H^{1,2}_C(T\X)$ by the very definition of $H^{1,2}_C(T\X)$, hence by approximation we get that $\nabla f\in H^{1,2}_C(T\X)$ for every $f\in H^{2,2}(\X)$. Therefore under the current assumptions Proposition \ref{prop:compmetr} grants that  $Y(f)\in W^{1, 1}(\X)$ and the identity \eqref{eq:compmetr} yields 
\[
X(Y(f))=\nabla Y:(X\otimes \nabla f)+\H f(X,Y)=\d f(\nabla_XY)+\H f(X,Y).
\]
Subtracting the analogous expression for $Y(X(f))$ and using the symmetry of the Hessian we conclude.
\end{proof}
Since  $\fsm$ is dense in $W^{1,2}(\X)$, the same holds for $H^{2,2}(\X)$ and therefore taking into account the fact that the cotangent module $L^2(T^*\X)$ is generated, in the sense of modules, by the space $\{\d f:f\in W^{1,2}(\X )\}$ (Proposition \ref{prop:gencotan}), we see that $L^2(T^*\X)$ is also generated by $\{\d f:f\in H^{2,2}(\X )\}\subset L^\infty(T^*\X)$. It is then easy to realize that the vector field $\nabla_XY-\nabla_YX$ is the only one for which the identity \eqref{eq:torsionfree} holds for any $f\in H^{2,2}(\X)$. We can therefore give the following definition:
\begin{definition}[Lie bracket of Sobolev vector fields] Let $X,Y\in W^{1,2}_C(T\X)$. Then their Lie bracket $[X,Y]\in L^1(T\X)$ is defined as
\[
[X,Y]:=\nabla_XY-\nabla_YX.
\]
\end{definition}
We now discuss the locality properties of the covariant derivative. As we did when discussing the locality of the Hessian, by `the interior of $\{X_1=X_2\}$' for vector fields  $X_1,X_2$, we will mean the union of all the open sets $\Omega\subset \X$ such that $X_1=X_2$ $\mm$-a.e.\ on $\Omega$.
\begin{proposition}[Locality of the covariant derivative]\label{prop:loccov} For any $X_1,X_2\in W^{1,2}_C(T\X)$ we have
\begin{equation}
\label{eq:loccov1}
\nabla X_1=\nabla X_2,\qquad\mm\ae\ \text{\rm on the interior of }\{X_1=X_2\},
\end{equation}
and if $X_1,X_2\in H^{1,2}_C(T\X)$ the finer property
\begin{equation}
\label{eq:loccov2}
\nabla X_1=\nabla X_2,\qquad\mm\ae\ \text{\rm on }\{X_1=X_2\},
\end{equation}
holds.
\end{proposition}
\begin{proof} The argument is the same of Proposition \ref{prop:lochess}. For \eqref{eq:loccov1} it is sufficient to prove that for $X\in W^{1,2}_C(T\X)$ we have
\[
\nabla X:(\nabla g_1\otimes\nabla g_2)=0,\qquad\mm\ae \text{ on the interior of }\{X=0\},
\]
for any $g_1,g_2\in\fsm$. To this aim, recall that by Proposition \ref{prop:compmetr} we have $\la X, \nabla g_2\ra\in W^{1,1}(\X)$ and 
\[
\nabla X:(\nabla g_1\otimes\nabla g_2)=\big<\nabla\la X,\nabla g_2\ra,\nabla g_1\big>-\H {g_2}(X,\nabla g_1).
\]
Conclude recalling the locality property \eqref{eq:localdw11}. The second part of the statement follows analogously recalling that for $X\in H^{1,2}_C(T\X)$ we have $\la X, \nabla g_2\ra \in H^{1,1}(\X)$ (Proposition \ref{prop:compmetr}) and using the locality property in \eqref{eq:dh11}.
\end{proof}
\begin{remark}[Being local vs being a tensor]\label{rem:loctens}{\rm
In smooth Riemannian geometry one can recognize the fact that the covariant derivative $\nabla_YX$ of a smooth vector field $X$ along a smooth vector field $Y$ is a tensor in $Y$, by observing that for given $X$ the value of $\nabla_YX$   at a given point $p$ depends only on the value $Y(p)$ of $Y$ at $p$. This fact and linearity then grant that $\nabla_{fY}X=f\nabla_YX$ for any smooth function $f$.

The same pointwise property is certainly not true for $X$, and indeed the covariant derivative certainly does not satisfy $\nabla_Y(fX)=f\nabla_YX$ but fulfills instead the differentiation rule $\nabla_Y(fX)=f\nabla_YX+Y(f)X$.

In the current non-smooth setting we don't really know what is a vector field at one given point but only what is it's value $\mm$-a.e., in the sense made precise by the definition of $L^\infty$-module. Yet, passing from `everywhere' to `almost everywhere' destroys the difference outlined above as we certainly have
\[
\nabla_YX=\nabla_{\tilde Y}X,\qquad\mm\ae\text{  on}\ \{Y=\tilde Y\},
\]
and also
\[
\nabla_YX=\nabla_{Y}\tilde X,\qquad\mm\ae\text{  on}\ \{X=\tilde X\},
\]
at least for $X,\tilde X\in H^{1,2}_C(T\X)$, as shown by Proposition \ref{prop:loccov} above. Thus in our setting the quantity $\nabla_YX$ is `as local in $Y$ as it is in $X$', at least for $X\in H^{1,2}_C(T\X)$.

Perhaps not surprisingly, the property of being a tensor is recognized here by an inequality and precisely  by bounding the pointwise norm of the object investigated  in terms of the pointwise norm of the vector. Thus the trivial inequality
\[
|\nabla_YX|\leq|\nabla X|_\HS|Y|,\qquad\mm\ae,
\]
and linearity in $Y$ shows that $\nabla_{fY}X=f\nabla_YX$. On the other hand, for given $Y\neq0$ no bound of the form
\[
|\nabla_YX|\leq g|X|,\qquad\mm\ae
\]
holds, whatever function $g:\X\to\R$ we choose.
}\fr\end{remark}
\begin{remark}[The space $W^{1,2}_C(T\Omega)$]{\rm The locality property \eqref{eq:loccov1} and the Leibniz rule in Proposition \ref{prop:leibcov}  allow to introduce the space $W^{1,2}_C(T\Omega)$ of Sobolev vector fields on an open set $\Omega\subset X$. We can indeed say that $X\in L^2(T\X)$ belongs to $W^{1,2}_C(T\Omega)$ provided there exists $T\in L^2(T^{\otimes 2}\X )$ such that the following holds. For any Lipschitz function $\nchi:\X\to[0,1]$ with $\supp(\nchi)\subset\Omega$, the vector field $\nchi X$ belongs to $W^{1,2}_C(T\X)$ and the formula 
\[
\nabla(\nchi X)= T,\qquad\mm\ae  \text{ on the interior of }\{\nchi=1\},
\]
holds. In this case the tensor $T\in  L^2(T^{\otimes 2}\X)$ is uniquely defined on $\Omega$ and can be called covariant derivative of  $X$ on $\Omega$.

The role of the locality property \eqref{eq:loccov1} is to grant that the definition is meaningful, as it ensures  that for $\nchi_1,\nchi_2$ as above we have
\[
\nabla(\nchi_1 X)=\nabla(\nchi_2X) ,\qquad\mm\ae \text{ on the interior of }\{\nchi_1=\nchi_2\}.
\]
}\fr\end{remark}

\subsubsection{Second order differentiation formula}\label{se:secdiff} We introduced the concepts of Hessian and covariant differentiation using the identities
\[
\begin{split}
2\H f(\nabla g_1,\nabla g_2)&=\big<\nabla\la\nabla f,\nabla g_1\ra,\nabla g_2\big>+\big<\nabla\la\nabla f,\nabla g_2\ra,\nabla g_1\ra-\big<\nabla f,\nabla\la\nabla g_1,\nabla g_2\ra\big>,\\
\la\nabla_{\nabla g_1} X, \nabla g_2\ra&=\big<\nabla\la X,\nabla g_2\ra,\nabla g_1\big>-\H {g_2}(\nabla g_1,X),
\end{split}
\]
but in the smooth setting there is at least another, quite different, basic instance in which these objects appear:  for a smooth function $f$ and a smooth curve $\gamma$, the map $t\mapsto f(\gamma_t)$ is smooth and its second derivative is given by
\begin{equation}
\label{eq:sder}
\frac{\d^2}{\d t^2}f(\gamma_t)=\H f(\gamma_t',\gamma_t')+\la\nabla_{\gamma'_t}\gamma_t',\nabla f\ra.
\end{equation}
The question is then whether a similar formula holds also in our setting. Given that we don't have at disposal everywhere defined Hessian and covariant differentiation, we must formulate \eqref{eq:sder} in an appropriate $\mm$-a.e.\ sense.  We shall follow the same ideas used in Section \ref{se:normdistance} and pass from the `pointwise' formulation \eqref{eq:sder} to the `integrated' one which consists in looking at  $t\mapsto\int f\,\d\mu_t$ for a given curve $(\mu_t)\subset\probt\X$ satisfying suitable regularity requirements.

\vspace{1cm}

Our result is the following:
\begin{theorem}[Second order differentiation formula]\label{thm:secondder} Let $(\mu_t)\subset \probt X$ be a curve of bounded compression (Definition \ref{def:boundcompr}) solving the continuity equation
\[
\partial_t\mu_t+\nabla\cdot(X_t\mu_t)=0,
\]
(Definition \ref{def:solconteq}) for a family $(X_t)\subset L^2(T\X)$ of vector fields such that
\begin{itemize}
\item[i)] $\sup_t\|X_t\|_{W^{1,2}_C(T\X)}+\|X_t\|_{L^\infty(T\X)}<\infty$,
\item[ii)] $t\mapsto X_t\in L^2(T\X)$ is absolutely continuous.
\end{itemize}
Then for every $f\in H^{2,2}(\X)$ the map $[0,1]\ni t\mapsto\int f\,\d\mu_t$ is $C^{1,1}$ and the formula
\begin{equation}
\label{eq:secondorderderivative}
\frac{\d^2}{\d t^2}\int f\,\d\mu_t=\int \H f(X_t,X_t)+\la\nabla f, \partial_t X_t\ra+\la\nabla_{X_t}X_t, \nabla f\ra\,\d\mu_t,
\end{equation}
holds for a.e.\ $t\in[0,1]$. 
\end{theorem}
\begin{proof} Let $C>0$ be such that $\mu_t\leq C\mm$  for every $t\in[0,1]$. By the definition of solution of the continuity equation we know  that the map $t\mapsto\int f\,\d\mu_t$ is absolutely continuous and that the formula
\begin{equation}
\label{eq:firstorderderivative}
\frac\d{\d t}\int f\,\d\mu_t=\int \la \nabla f, X_t\ra\,\d\mu_t,
\end{equation}
holds for a.e.\ $t\in[0,1]$. The continuity of $t\mapsto X_t\in L^2(T\X)$ grants that $t\mapsto \la \nabla f, X_t\ra\in L^1(\mm)$ is continuous; on the other hand, the curve $t\mapsto\mu_t\in \prob \X$ is continuous w.r.t.\ convergence in duality with $C_b(\X)$ and thus, due to the uniform bound $\mu_t\leq C\mm$, also in duality with $L^1(\mm)$. It follows that the right hand side of \eqref{eq:firstorderderivative} is continuous, so that $t\mapsto\int f\,\d\mu_t$ is $C^1$ and formula \eqref{eq:firstorderderivative} holds for every $t\in[0,1]$.

Now assume for a moment that  $f\in\fsm$ and notice that the assumption $\sup_t\|X_t\|_{W^{1,2}_C(T\X)}+\|X_t\|_{L^\infty(T\X)}<\infty$  grants, together with Proposition \ref{prop:compmetr}, that $S:=\sup_{t\in[0,1]}\|\la \nabla f, X_t\ra\|_{W^{1,2}(\X)}<\infty$. In particular, for given $t_0<t_1\in[0,1]$ we can apply the definition of solution of continuity equation to the function $\la\nabla f, X_{t_0}\ra$ to deduce that
\[
\begin{split}
\left|\int \la\nabla f, X_{t_0}\ra\,\d(\mu_{t_1}-\mu_{t_0}) \right|&=\left|\int_{t_0}^{t_1}\int\big< \nabla\la\nabla f, X_{t_0})\ra,  X_t\big>\,\d\mu_t\,\d t\right|\leq (t_1-t_0)CS\sup_{t\in[0,1]}\|X_t\|_{L^2(T\X)}.
\end{split}
\]
Taking into account the absolute continuity of $t\mapsto X_t\in L^{2}(T\X)$ we also have
\[
\left|\int\la \nabla f, X_{t_1}- X_{t_0}\ra\,\d\mu_{t_1}\right|=\left|\int \int_{t_0}^{t_1}\la\nabla f, \partial_t X_{t}\ra\,\d t\,\d\mu_{t_1}\right|\leq C\|\nabla f\|_{L^2(T\X)}\int_{t_0}^{t_1}\|\partial_tX_t\|_{L^2(T\X)}\,\d t.
\]
These last two inequalities grant that $t\mapsto \int\la\nabla f, X_t\ra\,\d\mu_t$ is absolutely continuous and that
\[
\begin{split}
\lim_{h\to 0}\frac1h\left(\int\la \nabla f, X_{t+h}\ra\,\d\mu_{t+h}- \int\la \nabla f, X_t\ra\,\d\mu_t\right)=&\lim_{h\to 0}\frac1h\int\la\nabla f, X_{t}\ra\,\d(\mu_{t+h}-\mu_t)\\
&+\lim_{h\to 0}\frac1h\int\la\nabla f, X_{t+h}-X_t\ra\,\d \mu_t,
\end{split}
\]
for a.e.\ $t\in[0,1]$. Hence using formula \eqref{eq:firstorderderivative} with $\la \nabla f, X_t\ra$ in place of $f$ we deduce that for every point $t$ of differentiability of $s\mapsto X_s\in L^2(T\X)$ we have
\[
\lim_{h\to 0}\frac1h\left(\int\la \nabla f, X_{t+h}\ra\,\d\mu_{t+h}- \int\la\nabla f, X_t\ra\,\d\mu_t\right)=\int\big<\nabla\la\nabla f, X_t\ra, X_t\ra+\la\nabla f,\partial_tX_t\ra\,\d\mu_t,
\]
and the conclusion follows by expanding $\big<\nabla\la\nabla f, X_t\ra, X_t\ra$ via formula \eqref{eq:compmetr}.

The case of general $f\in H^{2,2}(\X)$ now follows by approximation. 
Indeed, for $(f_n)\subset\fsm$ converging to $f$ in $W^{2,2}(\X)$ we have that $\int f_n\,\d\mu_t\to\int f\,\d\mu_t$ for every $t\in[0,1]$ and the right hand sides of \eqref{eq:secondorderderivative} written for the functions $f_n$ are dominated in $L^1(0,1)$ and converge to the corresponding one for $f$ for a.e.\ $t$. This is sufficient to show that the second derivative of $t\mapsto\int f\,\d\mu_t$ in the sense of distribution is the $L^1(0,1)$-function in the right hand side of \eqref{eq:secondorderderivative}, which is equivalent to the thesis.
\end{proof}

\begin{problem}[The case of geodesics]{\rm

Theorem \ref{thm:secondder} does not cover the important case of geodesics, the problem being the lack of regularity of the vector fields and the difficulty in finding an appropriate approximation procedure. 

As a partial attempt in this direction, we remark that the results in \cite{GigliHan13} together with the abstract Lewy-Stampacchia inequality \cite{Gigli-Mosconi14} show that if $(\mu_t)\subset\probt\X$ is a $W_2$-geodesic  made of measures with uniformly bounded support and densities, then there are Lipschitz functions $\varphi_t\in D(\Delta)$ such that
\[
\partial_t\mu_t+\nabla\cdot(\nabla\varphi_t\mu_t)=0,
\]
and these functions can be chosen so that $\sup_{t\in[\eps,1-\eps]}\|\nabla\varphi_t\|_{L^\infty(T\X)} +\|\Delta\varphi_t\|_{L^\infty(\mm)}<\infty$ for every $\eps>0$.
Yet, it is not clear if it is possible to choose the $\varphi_t$'s so that $t\mapsto \nabla\varphi_t\in L^2(T\X)$ is absolutely continuous. If true, this latter regularity together with the fact that Kantorovich potentials along a geodesic evolve via the Hamilton-Jacobi equation, would grant that
\[
\partial_t\nabla\varphi_t+\nabla_{\nabla\varphi_t}\nabla\varphi_t=0,\qquad\mu_t\ae,
\]
for a.e.\ $t\in[0,1]$, which by \eqref{eq:secondorderderivative} would imply the very expected formula
\[
\frac{\d^2}{\d t^2}\int f\,\d\mu_t=\int  \H f(\nabla\varphi_t,\nabla\varphi_t)\,\d\mu_t.
\] 
}\fr\end{problem}

\begin{remark}[Regular Lagrangian flows]\label{rem:rlf}{\rm
In the same `Lagrangian' spirit of this section and in  connection with the Open Problem \ref{op:vc}, one might consider the following question. Find appropriate  conditions on a Borel map $[0,1]\ni t\mapsto X_t\in L^0(T\X)$ ensuring existence and uniqueness of a family of maps $F_t:\X\to \X$, $t\in[0,1]$, such that $F_0$ is the identity and 
\[
\frac\d{\d t}F_t=X_t\circ F_t,\qquad {\rm a.e. }\ t\in[0,1].
\]
Without paying too much attention to the technical details, this equation might be interpreted as: for a sufficiently large class of test functions $f$ the map $t\mapsto f\circ F_t\in L^1(\mm)$ is absolutely continuous and its derivative is given by $\frac{\d}{\d t}(f\circ F_t)=\d f(X_t)\circ F_t$.

In the smooth setting, we know by  the Cauchy-Lipschitz theorem that if the vector fields are Lipschitz then  existence and everywhere uniqueness are ensured. Still in $\R^n$, if one relaxes Lipschitz regularity to Sobolev one, then the correct notion of `a.e.\ solution' ensuring existence and uniqueness is that of  \emph{regular Lagrangian flow} introduced by Ambrosio in \cite{Ambrosio04}  in the study of the DiPerna-Lions theory \cite{DiPerna-Lions89}.

Given that we now have the notions of Sobolev vector field, one might wonder whether such theory can be developed in this more abstract context. The answer is affirmative, as proved in the very recent paper \cite{Ambrosio-Trevisan14}, where the concept of vector field is interpreted in terms of derivations of Sobolev functions. Starting from Theorem \ref{thm:dervf}, one can then verify that the theory developed in \cite{Ambrosio-Trevisan14}  can be fully read in terms of the language proposed here.
}\fr\end{remark}

\subsubsection{Connection Laplacian and heat flow of vector fields}\label{se:clap} The connection energy functional $\ec:L^2(T\X)\to [0,\infty]$ resembles a local Dirichlet form, being a quadratic form obtained by integrating the squared norm of a first-order differential object possessing the locality property expressed in Proposition \ref{prop:loccov}. 

The analogy is purely formal, given that the base space $L^2(T\X)$ is not really the $L^2$ space induced by some measure and, even in the smooth setting, there is no clear way of stating the Markov property for vector fields. Still, one can study the associated `diffusion operator' and the flow induced by this `form', which is what we briefly discuss here.

Actually, due to the fact that we don't know whether $W^{1,2}_C(T\X)=H^{1,2}_C(T\X)$, there is a choice to make: either to study the functional on the whole $W^{1,2}_C(T\X)$ or to concentrate the attention to the subspace $H^{1,2}_C(T\X)$. Here we adopt this second viewpoint, because the more general calculus rules available for vector fields in $H^{1,2}_C(T\X)$ will allow us to prove a sort of Bakry-\'Emery estimate for the induced flow, see Proposition \ref{prop:bev}

\vspace{1cm}

We start with the following definition:
\begin{definition}[Connection Laplacian]
The set $D(\Delta_C)\subset H^{1,2}_C(T\X)$ is the set of all $X\in H^{1,2}_C(T\X)$ such that there exists a vector field $Z\in L^2(T\X)$ satisfying
\[
\int \la Y, Z\ra\,\d\mm=-\int \nabla Y:\nabla X\,\d\mm,\qquad\forall Y\in H^{1,2}_C(T\X).
\]
The density of $H^{1,2}_C(T\X)$ in $L^2(T\X)$ grants that  $Z$ is uniquely identified by the above formula: it will be called connection Laplacian of $X$ and denoted by $\Delta_CX$. 
\end{definition}
The linearity of the covariant derivative ensures that  $D(\Delta_C)$ is a vector space and that $\Delta_C:D(\Delta_C)\to L^2(T\X)$ is linear as well.

Introducing the augmented energy functional $\ect:L^2(T\X)\to[0,\infty]$ as
\[
\ect(X):=\left\{
\begin{array}{ll}
{\displaystyle\frac12\int} |\nabla X|_\HS^2\,\d\mm,&\qquad\text{ if }X\in H^{1,2}_C(T\X),\\
&\\
+\infty,&\qquad\text{ otherwise},
\end{array}
\right.
\]
we can give an alternative description of the connection Laplacian. Notice that $\ect$ is convex, lower semicontinuous and with dense domain $D(\ect):=\{X:\ect(X)<\infty\}=H^{1,2}_C(T\X)$ in $L^2(T\X)$. Denoting by $\partial \ect(X)\subset L^2(T\X)$ the subdifferential of $\ect$ at $X\in D(\ect)$ and by $D(\partial\ect)$ its domain, i.e.\ the set of vector fields $X$ for which $\partial \ect(X)\neq\emptyset$, we see that $D(\partial\ect)=D(\Delta_C)$ and that for $X$ in these sets we have $\partial\ect(X)=\{-\Delta_CX\}$. Indeed, for $X\in D(\Delta_C)$ and $Y\in D(\ect)$, the convexity of $t\mapsto\ect((1-t)X+tY)$ yields
\[
\begin{split}
\ect(Y)-\ect(X)&\geq\frac{\ect(X+t(Y-X))-\ect(X)}t=\int\frac{|\nabla(X+t(Y-X))|_\HS^2-|\nabla X|_\HS^2}{2t}\,\d\mm\\
&=\int\nabla X:\nabla (Y-X)+\frac t2|\nabla(Y-X)|_\HS^2\,\d\mm\\
&=\int-\la (Y-X),\Delta_CX\ra+\frac t2|\nabla(Y-X)|_\HS^2\,\d\mm,\qquad\forall t\in(0,1],
\end{split}
\]
and thus letting $t\downarrow0$ from the arbitrariness of $Y\in D(\ect)$ we deduce that $-\Delta_CX\in\partial\ect(X)$. Conversely, for $Z\in\partial\ect(X)$ and arbitrary $Y\in D(\ect)$ and $t\in\R$ we have
\[
\begin{split}
\int\la Z,tY\ra\,\d\mm\leq \ect(X+tY)-\ect(X)=\int t\nabla X:\nabla Y+\frac{t^2}2|\nabla Y|_\HS^2\,\d\mm,
\end{split}
\]
and dividing by $t>0$ (resp. $t<0$) and letting $t\downarrow0$ (resp. $t\uparrow0$) we deduce that $X\in D(\Delta_C)$ and $-Z=\Delta_CX$, as claimed.

A direct consequence of this representation is that $\Delta_C$ is a closed operator, i.e.\ $\{(X,\Delta_CX)\ :\ X\in D(\Delta_C)\}$ is a closed subspace of $L^2(T\X)\times L^2(T\X)$.

Moreover, the standard theory of gradient flows of convex and lower semicontinuous functionals on Hilbert spaces, the linearity of $\Delta_C$ and the density of $H^{1,2}_C(T\X)$ in $L^2(T\X)$, grant that $D(\Delta_C)$ is a dense subset of $L^2(T\X)$ and the existence and uniqueness of a 1-parameter semigroup $(\h_{C,t})_{t\geq 0}$ of continuous linear operators from $L^2(T\X)$ into itself such that for any $X\in L^2(T\X)$ the curve $t\mapsto \h_{C,t}(X)\in L^2(T\X)$ is continuous on $[0,\infty)$, locally absolutely continuous on $(0,\infty)$ and fulfills
\[
\frac{\d}{\d t}\h_{C,t}(X)=\Delta_C\h_{C,t}(X),\qquad\forall t>0,
\]
where it is part of the statement the fact that $\h_{C,t}(X)\in D(\Delta_C)$ for every $X\in L^2(T\X)$, $t>0$. 

Then for given $X\in L^2(T\X)$ and putting $X_t:=\h_{C,t}(X)$, it is not hard to see that  $t\mapsto\frac12\int|X_t|^2\,\d\mm$ and $t\mapsto \ect(X_t)$ are locally absolutely continuous on $(0,\infty)$ with 
\[
\begin{split}
\frac\d{\d t}\frac12\int|X_t|^2\,\d\mm&=\int\la X_t, \Delta_CX_t\ra\,\d\mm=-\int |\nabla X_t|_\HS^2\,\d\mm\leq 0,\\
\frac\d{\d t}\ect(X_t)&=\int \nabla X_t:\nabla\Delta_CX_t\,\d\mm=-\int |\Delta_C X_t|^2\,\d\mm\leq 0,
\end{split}
\]
which shows that these two quantities are non-increasing and also lead to the standard a priori estimates:
\[
t\ect(X_t)\leq\int_0^t\ect(X_s)\,\d s=-\int_0^t\frac{\d}{\d s}\frac14\|X_s\|^2_{L^2(T\X)}\,\d s\leq \frac14\|X\|^2_{L^2(T\X)},
\]
and, taking into account that $t\mapsto\|\Delta_CX_t\|=\|\h_{C,t-\eps}\Delta_CX_\eps\|_{L^2(T\X)}$ is also non-increasing, that
\[
\begin{split}
\frac{t^2}2\|\Delta_CX_t\|^2_{L^2(T\X)}&\leq \int_0^ts\|\Delta_CX_s\|^2_{L^2(T\X)}\,\d s=\int_0^ts\frac\d{\d s}\ect(X_s)\,\d s\\
&=\int_0^t\ect(X_s)-\ect(X_t)\,\d s\leq \int_0^t\ect(X_s)\,\d s\\
&=-\int_0^t\frac{\d}{\d s}\frac14\|X_s\|^2_{L^2(T\X)}\,\d s\leq \frac14\|X\|^2_{L^2(T\X)}.
\end{split}
\]
We conclude this short discussion proving  the following sort of Bakry-\'Emery contraction rate estimate. Notice that the bound from below on the Ricci curvature does not appear in the inequality, but it is still necessary to work on an $\RCD(K,\infty)$ space for some $K\in\R$ in order to have the flow $\h_{C,t}$ to be well defined.
\begin{proposition}\label{prop:bev} For any $X\in L^2(T\X)$ and $t\geq 0$ we have
\[
|\h_{C,t}(X)|^2\leq \h_t(|X|^2),\qquad\mm\ae.
\]
\end{proposition}
\begin{proof} Taking into account the approximation result \eqref{eq:aplipt}, it is sufficient to prove that for any $t>0$ and non-negative $f\in\fsm$ with $\Delta f\in L^\infty(\mm)$ it holds  $\int f|\h_{C,t}(X)|^2\,\d\mm\leq \int f\h_t(|X|^2)\,\d\mm$.  Thus fix such $t,f$, and consider the map $F:[0,s]\to\R$ given by
\[
F(s):=\int f\h_{t-s}(|\h_{C,s}(X)|^2)\,\d\mm=\int \h_{t-s}(f) |\h_{C,s}(X)|^2\,\d\mm.
\]
Notice that $F$ is well defined because $|\h_{C,s}(X)|\in L^2(\mm)$ and $f\in L^\infty(\mm)$ and that  since $\frac{\d}{\d s}\h_sf=\Delta\h_sf=\h_s\Delta f$ the map $s\mapsto\h_sf\in L^\infty(\mm)$ is Lipschitz. Also, since the map $s\mapsto \h_{C,s}(X)\in L^2(T\X)$ is continuous on $[0,\infty)$ and locally absolutely continuous on $(0,\infty)$, the map  $s\mapsto |\h_{C,s}(X)|^2\in L^1(\mm)$ is continuous on $[0,\infty)$ and locally absolutely continuous on $(0,\infty)$.

Therefore $F:[0,t]\to\R$ is continuous and locally absolutely continuous on $(0,t]$ and in computing its derivative we can pass the derivative inside the integral. Thus writing $f_s$ for $\h_s(f)$ and $X_s$ for $\h_{C,s}(X)$, it is easy to see that we have
\[
F'(s)=\int -(\Delta f_{t-s} )|X_s|^2+2f_{t-s} \la X_s,\Delta_CX_s\ra\,\d\mm,\qquad \text{ a.e. }s\in[0,t],
\]
and therefore integrating by parts and recalling the Leibniz rule \eqref{eq:leibcov} we get
\[
\begin{split}
F'(s)&=\int \la \nabla f_{t-s},\nabla(|X_s|^2)\ra-2\nabla(f_{t-s}X_s):\nabla X_s\,\d\mm\\
&=\int  \la\nabla f_{t-s},\nabla(|X_s|^2)\ra - 2\nabla X_s:(\nabla f_{t-s}\otimes X_s) -f_{t-s}|\nabla X_s |_\HS^2\,\d\mm,\qquad \text{ a.e. }s\in[0,t].
\end{split}
\]
Now our choice of working with vector fields in $H^{1,2}_C(T\X)$ plays a role: we can apply formula  \eqref{eq:compmetr} (and recall the definition \eqref{eq:nablac})  to get
\[
\la\nabla(|X_s|^2),\nabla f_{t-s}\ra =2\nabla X_s:(\nabla f_{t-s}\otimes X_s),\qquad\mm\ae,
\]
and therefore $F'(s)=-\int f_{t-s}|\nabla X_s |_\HS^2\,\d\mm\leq 0$ for a.e.\ $s\in[0,t]$, which concludes the proof.
\end{proof}

\subsection{Exterior derivative}
\subsubsection{The Sobolev space $W^{1,2}_\d(\Lambda^kT^*\X)$}\label{se:extd}
In a smooth manifold, the chart-free definition of differential of exterior differential of a $k$-form reads as
\[
\begin{split}
\d\omega(X_0,\ldots,X_k)=&\sum_i(-1)^i\d\big(\omega(X_0,\ldots,\hat X_i,\ldots,X_k)\big)(X_i)\\
&+\sum_{i<j}(-1)^{i+j}\omega([X_i,X_j],X_0,\ldots,\hat X_i,\ldots,\hat X_j,\ldots,X_k),
\end{split}
\]
for any smooth vector fields $X_1,\ldots,X_k$.

Given that in the previous section we defined the Lie bracket for a large class of vector fields, and in particular for our test vector fields, we can use the above formula to define the exterior differential of forms. This is the aim of the current section.

\vspace{1cm}

Recall the notation introduced in Section \ref{se:test}, and in particular that the $k$-th exterior product of $L^2(T^*\X)$ with itself is denoted by  $L^2(\Lambda^kT^*\X)$, and for the special cases $k=1$ and $k=0$ we retain the notation $L^2(T^*\X)$ and $L^2(\mm)$ respectively. Similarly for $L^2(\Lambda^kT\X)$.

For a $k$-form  $\omega\in L^0(\Lambda^kT^*\X)$ and $X_0,\ldots,X_k\in L^0(T\X)$ we shall use the standard notation $\omega(X_0,\ldots,\hat X_i,\ldots,X_k)\in L^0(\mm)$, or just $\omega(\ldots,\hat X_i,\ldots)$ or $\omega(\hat X_i)$ for brevity, to indicate the application of $\omega$ to the sequence of vectors $X_0,\ldots,X_k$ except $X_i$. Similarly for $\omega(X_0,\ldots,\hat X_i,\ldots,\hat X_j,\ldots,X_k)$.

Notice  that for $X_i\in\vsm$ the fact that $X_i\in L^2\cap L^\infty(T\X)$ ensures that $|X_1\wedge\ldots\wedge X_n|\in L^2(\mm)$ and similarly from the very definition of Lie bracket we have that  $[X_i,X_j]\in L^2(T\X)$ and thus $|[X_i,X_j]\wedge  X_1\wedge\ldots\wedge X_n|\in L^2(\mm)$ as well.

 Then we give the following definition:
\begin{definition}[The space $W^{1,2}_\d(\Lambda^kT^*\X)$]
The space $W^{1,2}_\d(\Lambda^kT^*\X)\subset L^2(\Lambda^kT^*\X)$ is the space of $k$-forms $\omega$ such that there exists a $k+1$ form $\eta\in L^2(\Lambda^{k+1}T^*\X)$ for which the identity
\begin{equation}
\label{eq:defdext}
\begin{split}
\int \!\eta(X_0,\cdots, X_{k})\,\d\mm=& \int\sum_i (-1)^{i+1}\omega(X_0,\cdots,\hat X_i,\cdots, X_{k})\,\div( X_i)\,\d\mm\\
&+\int\sum_{i<j}(-1)^{i+j}\omega([X_i,X_j],X_0,\cdots ,\hat X_i ,\cdots ,\hat X_j ,\cdots, X_{k})\,\d\mm,
\end{split}
\end{equation}
holds for any $X_0,\ldots,X_{k}\in \vsm$. In this case $\eta$ will be called exterior differential of $\omega$ and denoted as $\d\omega$.

We endow $W^{1,2}_\d(\Lambda^kT^*\X)$ with the norm $\|\cdot\|_{W^{1,2}_\d(\Lambda^kT^*\X)}$ given by
\[
\|\omega\|_{W^{1,2}_\d(\Lambda^kT^*\X)}^2:=\|\omega\|^2_{L^2(\Lambda^kT^*\X)}+\|\d\omega\|^2_{L^2(\Lambda^{k+1}T^*\X)}
\]
and define the differential energy functional $\edd^k:L^2(\Lambda^{k}T^*\X)\to[0,\infty]$ as
\[
\edd^k(\omega):=\left\{
\begin{array}{ll}
\displaystyle{\frac12\int|\d\omega|^2\,\d\mm},&\qquad\text{ if }\omega\in W^{1,2}_\d(\Lambda^{k}T^*\X),\\
+\infty,&\qquad\text{ otherwise}.
\end{array}
\right.
\]
We shall often denote $\edd^k$ simply by $\edd$.
\end{definition}
The remarks made above  ensure that the integrands in formula \eqref{eq:defdext} are in $L^1(\mm)$, so that the expression makes sense. Then the density property \eqref{eq:denskvf} grants that  the exterior differential $\d\omega$ is unique. It is then clear that it linearly depends on $\omega$ and that $\|\cdot\|_{W^{1,2}_\d(\Lambda^kT^*\X)}$ is a norm.

It is also worth to underline that $W^{1,2}_\d(\Lambda^0T^*\X)=W^{1,2}(\X)$ and for a function in these spaces the definitions of differential as given above and in Definition \ref{def:diff} coincide. Such statement is indeed equivalent to the claim: for $f\in L^2(\mm)$ we have $f\in W^{1,2}(\X)$ if and only if there is $\omega\in L^2(T^*\X)$ such that $\int f\div(X)\,\d\mm=-\int\omega(X)\,\d\mm$ for any $X\in\vsm$ and in this case $\omega$ is the differential of $f$ as given by Definition \ref{def:diff}. Here the `only if' and the conclusion are obvious, for the `if' one can argue as in the proof of Proposition \ref{prop:w11w12}.

The following theorem collects the basic properties of $W^{1,2}_\d(\Lambda^kT^*\X)$. Its proof follows the same arguments already appeared in proving Theorems \ref{thm:basew22} and \ref{thm:basew12c}.
\begin{theorem}[Basic properties of $W^{1,2}_\d(\Lambda^kT^*\X)$]\label{thm:basew12d} For every $k\in\N$ the following holds.
\begin{itemize}
\item[i)]  $W^{1,2}_\d(\Lambda^kT^*\X)$ is a separable Hilbert space.
\item[ii)] The exterior differential is a closed operator, i.e.\ $\{(\omega,\d\omega):\omega\in W^{1,2}_\d(\Lambda^kT^*\X)\}$ is a closed subspace of $L^{2}(\Lambda^kT^*\X)\times L^{2}(\Lambda^{k+1}T^*\X)$.
\item[iii)] The differential energy $\edd^k:L^2(\Lambda^kT^*\X)\to[0,\infty]$ is lower semicontinuous and for every $\omega\in L^2(\Lambda^kT^*\X)$ the duality formula
\[
\begin{split}
\edd^k(\omega)=\sup\Bigg\{&\sum_l \int\sum_i (-1)^{i+1}\omega(X^l_0,\cdots,\hat X^l_i,\cdots, X^l_{k})\,\div(X^l_i)\,\d\mm\\
&+\sum_l \int\sum_{i<j}(-1)^{i+j}\omega([X^l_i,X^l_j],X^l_0,\cdots ,\hat X^l_i ,\cdots ,\hat X^l_j ,\cdots, X^l_{k})\,\d\mm\\
&\qquad\qquad\qquad\qquad\qquad\qquad\qquad-\frac12\bigg\|\sum_lX^l_0\wedge\cdots\wedge X^l_{k}\bigg\|^2_{L^2(\Lambda^{k+1}T\X)}\Bigg\},
\end{split}
\]
holds, where the  $\sup$ is taken among all  $n\in\N$, $X^l_i\in\vsm$, $i=1,\ldots,k+1$, $l=1,\ldots,n$.
\item[iv)] For  $f_0,\ldots,f_k\in\fsm$ we have $ f_0\d f_1\wedge\cdots\d f_k\in W^{1,2}_\d(\Lambda^kT^*X)$ and
\begin{equation}
\label{eq:df1fn}
\d( f_0\d f_1\wedge\cdots \wedge\d f_k)= \d f_0\wedge \d f_1\wedge\cdots \wedge\d f_k,
\end{equation}
and similarly $\d f_1\wedge\cdots\d f_k\in W^{1,2}_\d(\Lambda^kT^*X)$ with 
\begin{equation}
\label{eq:ddf1fn}
\d(\d f_1\wedge\cdots \wedge\d f_k)= 0.
\end{equation}
\item[v)]  We have $\ffsm k\subset W^{1,2}_\d(\Lambda^kT^*\X)$ and in particular $W^{1,2}_\d(\Lambda^kT^*\X)$ is dense in $L^{2}(\Lambda^kT^*\X)$.
\item[vi)] For given $\omega\in L^2(\Lambda^kT^*\X)$ we have $\omega\in W^{1,2}_\d(\Lambda^kT^*\X)$ with $\eta=\d \omega$ if and only if for every $X_0,\ldots,X_k\in\vsm$ and $f\in L^\infty\cap W^{1,2}(\X)$  it holds
\begin{equation}
\label{eq:altraext}
\begin{split}
\int f\eta(X_0,\cdots, X_{k})\,\d\mm=& \int\sum_i (-1)^{i+1}\omega(\hat X_i)\,\div(f X_i)\,\d\mm\\
&+\int\sum_{i<j}(-1)^{i+j}f\omega([X_i,X_j],\hat X_i ,\hat X_j)\,\d\mm.
\end{split}
\end{equation}
\end{itemize}
\end{theorem}
\begin{proof} Given that the two sides of \eqref{eq:defdext} are continuous w.r.t.\ weak convergence of $\omega$ and $\eta$ in $L^{2}(\Lambda^kT^*\X)$ and $L^{2}(\Lambda^{k+1}T^*\X)$ respectively, point $(ii)$ follows. The lower semicontinuity of $\edd^k$ comes from this continuity property and the relative weak compactness of bounded subsets of the Hilbert space $L^{2}(\Lambda^{k+1}T^*\X)$. From $(ii)$ it also follows that $W^{1,2}(\Lambda^{k}T^*\X)$ is complete, hence Hilbert given that the norm clearly satisfies the parallelogram rule. For separability, just notice that the map $\omega\to (\omega,\d\omega)$ is an isometry of $W^{1,2}(\Lambda^{k}T^*\X)$ with its image in $L^{2}(\Lambda^kT^*\X)\times L^{2}(\Lambda^{k+1}T^*\X)$, this latter space being endowed with the separable norm $\|(\omega,\eta)\|^2:=\|\omega\|^2_{L^{2}(\Lambda^{k}T^*\X)}+\|\eta\|^2_{L^{2}(\Lambda^{k+1}T^*\X)}$ (recall \eqref{eq:sepext}). Thus $(i)$ is addressed.

For $(iv)$ notice that   $ f_0\d f_1\wedge\cdots\d f_k\in L^{2}(\Lambda^kT^*\X)$  and that the right hand side of \eqref{eq:df1fn} is in $L^2(\Lambda^{k+1}T^*\X)$, so the statement makes sense. 

From  Proposition \ref{prop:compmetr} we  get that   for $f\in\fsm$ and $X\in\vsm $ it holds  $\d f(X)\in L^\infty\cap W^{1,2}(\X)$. Hence for   $X_1,\ldots,X_{k}\in \vsm $ we have $(\d f_1\wedge\ldots\wedge \d f_{k})(X_1,\ldots,X_{k})=\det\big(\d f_i(X_j)\big)\in W^{1,2}(\X)$. Now let for brevity $\omega:=\d f_1\wedge\cdots\wedge \d f_k$ and notice that for $X_0,\ldots,X_{k}\in\vsm$ we have
\begin{equation}
\label{eq:adria2}
\begin{split}
\sum_i(-1)^i\d\big(\omega&(\ldots,\hat X_i,\ldots)\big)(X_i)=-\sum_{i<j}(-1)^{i+j}\omega([X_i,X_j],\ldots,\hat X_i,\ldots, \hat X_j,\ldots),
\end{split}
\end{equation}
$\mm$-a.e., as  can be seen by direct computations recalling the definition of Lie bracket and the symmetry of the Hessian.

Then we have
\[
\begin{split}
\int&\sum_i(-1)^{i+1}f_0\omega(\hat X_i)\,\div (X_i)+\sum_{i<j}(-1)^{i+j}f_0\omega([X_i,X_j],\hat X_i,\hat X_j)\,\d\mm\\
&\stackrel{\phantom{\eqref{eq:adria2}}}=\int \sum_i(-1)^{i}\d f_0(X_i)\,\omega(\hat X_i)+f_0\,\d\big(\omega(\hat X_i)\big)(X_i)+\sum_{i<j}(-1)^{i+j}f_0\omega([X_i,X_j],\hat X_i,\hat X_j)\,\d\mm\\
&\stackrel{\eqref{eq:adria2}}=\int \sum_i(-1)^{i}\d f_0(X_i)\,\omega(\hat X_i)\,\d\mm=\int (\d f_0\wedge\omega)(X_0,\ldots,X_{k})\,\d\mm,
\end{split}
\]
which is \eqref{eq:df1fn}. Then \eqref{eq:ddf1fn} follows along the same lines picking $f_0$ identically 1 which, although not in $\fsm$, is still admissible in the computations done above.

Point $(v)$ is then a direct consequence of $(iv)$ and the very definition of $\ffsm k$, the density in $L^2(\Lambda^{k}T^*\X)$ being already remarked in  \eqref{eq:densetforms}.

The `if' part in point $(vi)$ follows by considering a sequence $(f_n)\subset W^{1,2}(\X)$ of uniformly Lipschitz and uniformly bounded functions converging to $1$ and applying the dominated convergence theorem. For the `only if' notice that up to approximation with the heat flow we can assume that $f\in\fsm$. In this case, the vector field $fX_0$ belongs to $\vsm$ and taking into account the identity
\[
[fX_0,X_j]=f[X_0,X_j]-X_0\,\d f(X_j),
\]
which follows from the definition of Lie bracket and the Leibniz rule \eqref{eq:leibcov},  we have
\[
\begin{split}
\int f\d\omega(X_0,\ldots,X_k)\,\d\mm&=\int \d\omega(fX_0,X_1\ldots,X_k)\,\d\mm\\
&=\int \sum_{i}(-1)^{i+1}f\omega(\hat X_i)\div(X_i)-\omega(\hat X_0)\d f(X_0)\,\d\mm\\
&+\int\sum_{i<j}(-1)^{i+j}f\omega([X_i,X_j],\hat X_i ,\hat X_j)+\sum_{j>0}(-1)^{j+1}\omega(\hat X_j)\d f(X_j) \,\d\mm,
\end{split}
\]
and the conclusion follows recalling that $\div(fX_i)=\d f(X_i)+f\div(X_i)$.

It remains to prove the duality formula for $\edd^k$: this follows the very same arguments used in Theorem \ref{thm:basew22} and hinted in Theorem \ref{thm:basew12c}, we omit the details.
\end{proof}
\begin{remark}[The role of the lower Ricci bound]{\rm
On an arbitrary infinitesimally Hilbertian space one can certainly consider the space $L^2(\Lambda^kT^*\X)$ as we defined it and declare that for Sobolev functions $f_0,\ldots,f_k$ with appropriate integrability the exterior differential of the form $f_0\d f_1\wedge\ldots\wedge \d f_k$ is $\d f_0\wedge \d f_1\wedge\ldots\wedge \d f_k$. 

Yet, without further assumptions it is not so clear if such differential is closable. In the current approach, this is granted by the assumption of lower bound on the Ricci which in turn ensures the existence of a large class of vector fields for which the Lie bracket is well defined. These permit to give the definition of exterior differentiation via integration by parts the way we did, thus directly leading to the closure of the differential.
}\fr\end{remark}

With the same computations one would do in the smooth setting, under suitable regularity assumptions it is possible to see that  the exterior differential satisfies the expected Leibniz rule:
\begin{proposition}[Leibniz rule]
Let $\omega\in W^{1,2}_\d(\Lambda^{k}T^*\X)$ and $\omega'\in\ffsm{k'}$. Then $\omega\wedge\omega'\in W^{1,2}_\d(\Lambda^{k+k'}T^*\X)$ with
\begin{equation}
\label{eq:leibext}
\d(\omega\wedge\omega')=\d\omega\wedge\omega'+(-1)^k\omega\wedge\d\omega'.
\end{equation}
\end{proposition}
\begin{proof} In the case $k'=0$ we pick  $\omega'=f\in\fsm=\ffsm 0$ and notice that the thesis reads as
\[
f\omega\in W^{1,2}_\d(\Lambda^{k}T^*\X)\qquad\text{and}\qquad \d(f\omega)=\d f\wedge\omega+f\d\omega.
\]
The definition of wedge product yields the identity
\[
(\d f\wedge \omega)(X_0.\ldots,X_k)= \sum_i(-1)^i\d f(X_i)\,\omega(\hat X_i),
\]
for any  $X_0,\ldots ,X_{k}\in\vsm$, and thus the conclusion follows by direct comparison of the formulas \eqref{eq:defdext} and \eqref{eq:altraext}.

We pass to the case $k'=1$ and $\omega'=\d f$ for $f\in\fsm$ where, taking into account that \eqref{eq:ddf1fn} gives $\d(\d f)=0$, the thesis reduces to
\[
\omega\wedge\d f\in   W^{1,2}_\d(\Lambda^{k+1}T^*\X) \qquad\text{and}\qquad \d(\omega\wedge\d f)=\d\omega\wedge\d f
\]
For $X_0,\ldots,X_{k+1}\in\vsm$ arbitrary we have $(\d\omega\wedge\d f)(X_0,\ldots,X_{k+1})=\sum_{i}(-1)^{i+k+1}\d\omega(\hat X_i)\d f(X_i)$, thus noticing that $\d f(X_i)\in L^\infty\cap W^{1,2}(\X)$, from \eqref{eq:altraext} we obtain
\[
\begin{split}
\int (\d\omega\wedge\d f)(X_0,\ldots,X_{k+1})\,\d\mm&=\int\sum_{i\neq j}a(i,j)\omega(\hat X_i,\hat X_j)\div\big(\d f(X_i)X_j\big)\\
&+\int \sum_{\natop{j\neq i\neq j'}{j<j'}}b(i,j,j')\omega([X_j,X_{j'}],\hat X_i,\hat X_j,\hat X_{j'})\d f(X_i)\,\d\mm,
\end{split}
\]
where 
\[
a(i,j):=\left\{\begin{array}{ll}
(-1)^{i+j+k+1}&\quad\text{ if }j<i\\
(-1)^{i+j+k}&\quad\text{ if }j>i
\end{array}\right.
\qquad
b(i,j,j'):=\left\{\begin{array}{ll}
(-1)^{i+j+j'+k+1}&\quad\text{ if $i\notin(j,j')$}\\
(-1)^{i+j+j'+k}&\quad\text{ if $i\in(j,j')$}
\end{array}\right.
\]
The claim then follows directly by the definition \eqref{eq:defdext} taking into account the identity
\[
\div\big(\d f(X_i)X_j\big)-\div\big(\d f(X_j)X_i\big)=\d f(X_i)\div(X_j)-\d f(X_j)\div(X_i)-\d f([X_i,X_j]).
\]
Now proceed by induction on $k'$, assume $\omega'=f_0\d f_1\wedge\ldots\wedge\d f_{k'}$ and put $\omega'':=\d f_2\wedge\ldots\wedge \d f_{k'}$. We have
\[
\begin{split}
\d(\omega\wedge\omega')&=\d\big((\omega\wedge f_0\d f_1)\wedge\omega'' \big)\\
\text{(by induction and \eqref{eq:ddf1fn})}\qquad\qquad&=\d \big((f_0\omega)\wedge \d f_1)\big)\wedge\omega''\\
\text{(by the cases analyzed)} \qquad \qquad&= (f_0\d\omega+\d f_0\wedge\omega )\wedge\d f_1\wedge\omega''\\
\text{(recalling \eqref{eq:df1fn})} \qquad \qquad&=\d \omega\wedge\omega'+(-1)^k\omega\wedge\d\omega',
\end{split}
\]
as desired. The case of general $\omega'\in\ffsm {k'}$ follows by linearity.
\end{proof}

\begin{definition}[The space $H^{1,2}_\d(\Lambda^kT^*\X)$]
We define $H^{1,2}_\d(\Lambda^kT^*\X)\subset W^{1,2}_\d(\Lambda^kT^*\X)$ as the $W^{1,2}_\d$-closure of $\ffsm k$.
\end{definition}
Notice that by point $(v)$ of Theorem \ref{thm:basew12d} we know that $H^{1,2}_\d(\Lambda^kT^*\X)$ is dense in  $L^{2}(\Lambda^kT^*\X)$. A key property of forms in $H^{1,2}_\d(\Lambda^kT^*\X)$ is:
\begin{proposition}[$\d^2=0$ for forms in  $H^{1,2}_\d(\Lambda^kT^*\X)$]\label{prop:dd}
Let  $\omega\in H^{1,2}_\d(\Lambda^kT^*\X)$. Then
\[
\d\omega\in H^{1,2}_\d(\Lambda^{k+1}T^*\X)\qquad\text{and}\qquad \d(\d\omega)=0.
\]
\end{proposition}
\begin{proof}
For forms in $\ffsm k$ the claim has been proved by point $(iv)$ in Theorem \ref{thm:basew12d}. The conclusion then comes from the closure of the exterior differential proved in point $(ii)$ of the same theorem.
\end{proof}
We conclude pointing out that the exterior differential is local in the following sense:
\begin{equation}
\label{eq:localext}
\forall\omega\in W^{1,2}_\d(\Lambda^kT^*\X)\qquad\text{we have }\qquad\d\omega=0\qquad\mm\ae\text{ on the interior of }\{\omega=0\},
\end{equation}
where the `interior of $\{\omega=0\}$' is by definition the union of all the open sets $\Omega\subset \X$ such that $\omega=0$ $\mm$-a.e.\ on $\Omega$. The proof of \eqref{eq:localext} follows easily using point $(iv)$ of Theorem \ref{thm:basew12d} and picking in formula \eqref{eq:altraext} arbitrary $f$'s with support contained in open sets where $\omega$ is 0 $\mm$-a.e..

It is worth underlying that we are not able to improve \eqref{eq:localext} for $\omega\in H^{1,2}_\d(\Lambda^kT^*\X)$ into `$\d\omega$ is 0 where $\omega$ is zero' as we did for the Hessian and the covariant derivative in Propositions \ref{prop:lochess}, \ref{prop:loccov}. The technical problem here is that for $\omega\in H^{1,2}_\d(\Lambda^kT^*\X)$ and $X_i\in \vsm$ it is unclear whether the function $\omega(X_1,\ldots,X_k)$ has any kind of Sobolev regularity.

\begin{remark}{\rm The same integrability issues that prevented us from defining the spaces $W^{2,1}(\X)$ and $W^{1,1}(T\X)$, do not allow for a definition of  $W^{1,1}_\d(\Lambda^kT^*\X)$. Here the problem appears on the Lie bracket in the defining formula \eqref{eq:defdext}, for which we have no better integrability other that $L^2$. 

It is for this reason that we are obliged to take one of the two forms in $\ffsm k$ in the Leibniz rule \eqref{eq:leibext}: even the choice of both forms in the a priori smaller space  $H^{1,2}_\d(\Lambda^kT^*\X)$ creates integrability problems.

A similar issue occurs in attempting to prove that $ \d(\d\omega)=0$ for generic $\omega\in  W^{1,2}_\d(\Lambda^kT^*\X)$: although the formal calculation lead to the desired result (evidently, because they are the same computations one would do in the smooth world), it is unclear whether they are really justifiable in this context.
}\fr\end{remark}

\subsubsection{de Rham cohomology  and Hodge theorem}\label{se:dr}
We defined the exterior differential and noticed that, at least for forms in $H^{1,2}_\d(\Lambda^{k-1}T^*\X)$, it squares to 0. Thus we can build a de Rham complex. In this section we study its basic properties.

\vspace{1cm}

We introduce the spaces $\closed k$ and $\exact k$ of closed and exacts $k$-forms as: 
\[
\begin{split}
\closed k&:=\ker\Big( \d: H^{1,2}_\d(\Lambda^kT^*\X)\to H^{1,2}_\d(\Lambda^{k+1}T^*\X)  \Big)=\big\{\omega\in H^{1,2}_\d(\Lambda^kT^*\X)\ :\ \d\omega=0\big\},\\
\exact k&:={\rm Im}\Big(\d: H^{1,2}_\d(\Lambda^{k-1}T^*\X)\to H^{1,2}_\d(\Lambda^{k}T^*\X)  \Big) =\big\{\d\omega\ :\ \omega\in H^{1,2}_\d(\Lambda^{k-1}T^*\X)\big\}.
\end{split}
\]
Notice that the closure of the differential ensures that $\closed k$ is a closed subspace of $L^2(\Lambda^kT^*\X)$. On the other hand, we don't know if $\exact k$ is closed, so we also introduce the space
\[
\exacto k:=\textrm{ $L^2(\Lambda^kT^*\X)$-closure of }\exact k.
\]
Since Proposition \ref{prop:dd} ensures that $\exact k\subset\closed k$, we also have $\exacto k\subset \closed k$. Thus we can give the definition:
\begin{definition}[de Rham cohomology]
For $k\in \N$ the vector space $H^k_{dR}(\X)$ is defined as the quotient
\[
H^k_{dR}(\X):=\frac{\closed k}{\exacto k}
\]
\end{definition}
Endowing $\closed k$ and $\exacto k$ of the $L^2(\Lambda^kT^*\X)$-norm, they become  Hilbert spaces and therefore  $H^k_{dR}(\X)$ comes with a canonical structure of Hilbert space as well. It is worth to underline that in our setup this Hilbert structure is, unlike the smooth case, entirely intrinsic because the existence of the pointwise scalar product in the tangent module  of our base space $(\X,\sfd,\mm)$ is part of our assumptions and is essential for defining both the concept of $k$-forms and of their differentials. 

The {\bf compatibility} of this definition with the standard one for compact smooth manifolds can be established calling into play Hodge theory, which also is based on the study of $L^2$ forms and produces the same cohomology groups as those obtained via smooth forms. In this direction it is worth to keep in mind that in the smooth case the equality $W^{1,2}_\d(\Lambda^kT^*\X)=H^{1,2}_\d(\Lambda^kT^*\X)$ holds, as can be seen using a partition of unit subordinate to a cover via charts, to reduce the problem of approximations of a Sobolev form with smooth ones to a problem in $\R^d$, where smoothing is easy and can be done componentwise.

We turn to {\bf functoriality}. Recall the notions introduced in Section \ref{se:mbd} and in particular that in Proposition \ref{prop:pull1form} we built, for a given map of bounded deformation $\varphi:\X_2\to\X_1$, the pullback of 1-forms $\varphi^*:L^2(T^*\X_1)\to L^2(T^*\X_2)$, this map being characterized by
\begin{subequations}
\begin{align}
\label{eq:fine}
\varphi^*(\d f)&=\d(f\circ\varphi)\\
\label{eq:modmorpb}
\varphi^*(g\omega)&=g\circ\varphi\,\varphi^*\omega\\
\label{eq:pullpunt}
|\varphi^*\omega|&\leq \lf(\varphi)\,|\omega|\circ\varphi
\end{align}
\end{subequations}
$\mm_2$-a.e.\ for every $f\in\s^2(\X_1)$, $\omega\in L^2(T^*\X_1)$ and $g\in L^\infty(\mm_1)$. We shall now extend the pullback operation to general forms in $L^2(\Lambda^kT^*\X_1)$ for $k\in\N$. To this aim start observing that from \eqref{eq:pullpunt} we see that $\varphi^*$ can be uniquely extended to a linear continuous map from $L^0(T^*\X_1)$ to $L^0(T^*\X_2)$ satisfying \eqref{eq:modmorpb} and  \eqref{eq:pullpunt} for arbitrary $g\in L^0(\mm_1)$ and $\omega\in L^0(T^*\X_1)$. 

To extend the pullback operation to $k$-forms, we start noticing that for every $\omega_1,\ldots,\omega_k\in L^2(T^*\X)$ we have
\begin{equation}
\label{eq:normk}
|(\varphi^*\omega_1)\wedge\ldots\wedge(\varphi^*\omega_k)|\leq \lf(\varphi)^k|\omega_1\wedge\ldots\wedge\omega_k|\circ\varphi,\qquad\mm_2\ae
\end{equation}
A way to establish this inequality is to use the structural characterization of Hilbert modules given in Theorem \ref{thm:structhil} and a density argument to reduce to the study of Hilbert spaces. Then noticing that the inequality only involves a finite number of vectors we can also reduce to finite dimensional Hilbert spaces. The question then is: given an $n$-dimensional Hilbert space $H$, an endomorphism $A$ of $H$ and the induced endomorphism $A^{\wedge^k}:\Lambda^kH\to\Lambda^kH$ defined by $A^{\wedge^k}(v_1\wedge\ldots\wedge v_k):=Av_1\wedge\ldots\wedge Av_k$ and extended by linearity, prove that the operator norm of $A^{\wedge^k}$ is bounded by the $k$-th power of the operator norm of $A$. To check this, introduce the symmetric and positively defined operator $B:=A^{\rm t}A$, let $\lambda_1\geq \ldots\geq\lambda_n\geq 0$ be its eigenvalues and recall that the operator norm of $A$ is equal to $\sqrt{\lambda_1}$. Let $v_1,\ldots,v_n\in H$ be orthogonal eigenvectors of $B$ and notice that  $v_{i_1}\wedge\ldots\wedge v_{i_k}\in \Lambda^kH$ is an eigenvector of  $(A^{\wedge^k})^{\rm t}A^{\wedge^k}=B^{\wedge^k}$ for any choice of distinct $i_1,\ldots,i_k$, the corresponding eigenvalue being $\prod_j\lambda_{i_j}$. Since these are $\binom nk$ independent elements of $\Lambda^kH$, we just found all the eigenvectors of $B^{\wedge^k}$ and seen that all of them are $\leq \lambda_1^k$, which gives the claim.

Then from \eqref{eq:normk} we deduce that there exists a unique linear map $\varphi^*:L^0(\Lambda^kT^*\X_1)\to L^0(\Lambda^kT^*\X_2)$ satisfying
\begin{equation}
\label{eq:defaextk}
\begin{split}
\varphi^*(\omega_1\wedge\cdots\wedge\omega_k)&=(\varphi^*\omega_1)\wedge\cdots\wedge(\varphi^*\omega_k),\\
|\varphi^*\omega|&\leq \lf(\varphi)^k|\omega|\circ\varphi,
\end{split}
\end{equation}
$\mm_2$-a.e.\ for every $\omega_1,\ldots,\omega_k\in L^0(T^*\X_1)$ and $\omega\in L^0(\Lambda^kT^*\X_1)$. 

In particular, $\varphi^*$ restricts to a map, still denoted by $\varphi^*$, from $ L^2(\Lambda^kT^*\X_1)$ to $ L^2(\Lambda^kT^*\X_2)$ satisfying
\[
\begin{split}
\varphi^*(f\omega)&=f\circ\varphi\,\varphi^*\omega,\\
|\varphi^*\omega|&\leq \lf(\varphi)^k|\omega|\circ\varphi,
\end{split}
\]
for every $\omega\in  L^2(\Lambda^kT^*\X_1)$ and $f\in L^\infty(\mm_1)$ (i.e., $(\varphi,\varphi^*)$ is a morphism in ${\bf Mod}_{2-L^\infty}$ - see Remark \ref{rem:catpullback}).
\begin{proposition}[Functoriality]
Let $(\X_1,\sfd_1,\mm_1),(\X_2,\sfd_2,\mm_2)$ be two $\RCD(K,\infty)$ spaces, $K\in\R$, and $\varphi:\X_2\to\X_1$ of bounded deformation. Then for every $k\in\N$ and $\omega\in H^{1,2}_\d(\Lambda^kT^*\X_1)$ we have $\varphi^*\omega\in H^{1,2}_\d(\Lambda^kT^*\X_2)$ and
\begin{equation}
\label{eq:functd}
\d(\varphi^*\omega)=\varphi^*\d\omega.
\end{equation}
In particular, $\varphi^*$ passes to the quotient and induces a linear continuous map from $H^k_{dR}(\X_1)$ to $H^k_{dR}(\X_2)$ with norm bounded by $\lf(\varphi)^k$.
\end{proposition}
\begin{proof} Let $f_0,\ldots,f_k\in \fsms{\X_1}$ and $\omega=f_0\d f_1\wedge\ldots\wedge\d f_k$. Then the definition of $\varphi^*$ and property \eqref{eq:fine} ensures that
\[
\varphi^*\omega=f_0\circ\varphi\,\d(f_1\circ\varphi)\wedge\ldots\wedge\d(f_k\circ\varphi).
\]
Notice that $f_i\circ\varphi\in W^{1,2}(\X_2)$ with $|\d(f_i\circ\varphi)|\leq \lf(\varphi)|\d f_i|\circ\varphi$ (recall \eqref{eq:sobpull}), so that  taking into account the approximation property \eqref{eq:aplipt} and \eqref{eq:df1fn} it  follows that $\varphi^*\omega\in H^{1,2}_\d(\Lambda^kT^*\X_2)$ with
\[
\d\varphi^*\omega=\d(f_0\circ\varphi)\wedge \d(f_1\circ\varphi)\wedge\ldots\wedge\d(f_k\circ\varphi)=\varphi^*\d\omega.
\]
Then the linearity, the continuity of $\d: H^{1,2}_\d(\Lambda^kT^*\X_2)\to L^2(\Lambda^{k+1}T^*\X_2)$ and the one of $\varphi^*$ yield the first claim.

For the second, notice that \eqref{eq:functd} ensures that $\varphi^*$ sends closed (resp.\ exact) forms in closed (resp.\ exact) forms and that its continuity ensures that forms in $\exactos{k}{\X_1}$ are sent in forms in $\exactos{k}{\X_2}$. Thus indeed $\varphi^*$ passes to the quotient. The bound on the norm is then a direct consequence of the inequality in \eqref{eq:defaextk}.
\end{proof}
\begin{remark}\label{rem:mv}{\rm
It is unclear to us whether a Mayer-Vietoris sequence can be built for the so-defined cohomology groups. It is not hard to adapt the above definitions to produce the cohomology groups of open subsets of an $\RCD$ space, but taking the closure of the space of exact forms  creates some problems with the classical diagram chasing arguments used to build the connecting morphism.
}\fr\end{remark}

We now turn to {\bf Hodge theory}, which quite directly fits in our framework, being it based on the notion of $L^2$ and Sobolev forms. We start with:
\begin{definition}[Codifferential]
The space $D(\delta_k)\subset L^2(\Lambda^kT^*\X)$ is the space of those forms $\omega$ for which there exists a form $\delta\omega\in L^2(\Lambda^{k-1}T^*\X)$, called codifferential of $\omega$, such that 
\begin{equation}
\label{eq:defcodiff}
\int \la \delta\omega,\eta\ra\,\d\mm=\int \la\omega,\d\eta\ra\,\d\mm,\qquad\forall\eta\in\ffsm{k-1}.
\end{equation}
In  the case $k=0$ we put $D(\delta_0):=L^2(\mm)$ and define the $\delta$ operator to be identically 0 on it.
\end{definition}
The density of $\ffsm{k-1}$ in $L^2(\Lambda^{k-1}T^*\X)$ ensures that the codifferential is uniquely defined and the definition also grants that $\delta$ is a closed operator in the sense that $\{(\omega,\delta\omega):\omega\in D(\delta_k)\}$ is a closed subspace of $ L^2(\Lambda^kT^*\X)\times  L^2(\Lambda^{k-1}T^*\X)$. 

For 1-forms, the operator $\delta$ is nothing but the opposite of the divergence, in the sense that $\omega\in D(\delta_1)$ if and only if $\omega^\sharp\in D(\div)$  and in this case
\begin{equation}
\label{eq:delta1form}
\delta\omega =-\div(\omega^\sharp),
\end{equation}
as can be checked by direct definition. Moreover, noticing that by approximation we can test the codifferential against forms in $H^{1,2}_\d(\Lambda^{k-1}T^*\X)$, and considering such forms concentrated on prescribed open sets, it is immediate to verify that for any $\omega\in D(\delta_k)$ we have
\begin{equation}
\label{eq:localdelta}
\delta\omega=0\qquad\mm\ae\text{ on the interior of }\{\omega=0\},
\end{equation}
where as usual this statement should be intended as `$\delta\omega=0$ $\mm$-a.e.\ on every open set $\Omega$ on which $\omega$ is $\mm$-a.e.\ $0$'.

We shall make use of the following computation:
\begin{proposition}\label{prop:contobrutto}
For $f_1,\ldots,f_k\in\fsm$ we have $\d f_1\wedge\ldots\d f_k\in D(\delta_k)$ with
\[
\begin{split}
\delta(\d f_1\wedge\ldots\d f_k)=&\sum_i(-1)^i\Delta f_i\,\d f_1\wedge\ldots \wedge\widehat{\d f_i} \wedge\ldots \wedge\d f_k\\
&+\sum_{i<j}(-1)^{i+j}[\nabla f_i,\nabla f_j]^\flat\wedge\ldots\wedge \widehat{\d f_i}\wedge\ldots\wedge\widehat{\d f_j}\wedge\ldots\wedge\d f_k
\end{split}
\]
\end{proposition}
\begin{proof} Put for brevity $\omega:=\d f_1\wedge\ldots\d f_k$ and notice that by linearity, it is sufficient to test condition \eqref{eq:defcodiff} for $\eta=g_1\d g_2\wedge\ldots\wedge \d g_k$ with $g_1,\ldots,g_k\in\fsm$. Thus pick such $\eta$, recall \eqref{eq:df1fn} and the definition of scalar product in $L^2(\Lambda^kT^*\X)$ to get
\[
\begin{split}
\int \la\omega,\d\eta\ra\,\d\mm&=\int\sum_{\sigma\in S_k}s(\sigma)\prod_{i}\la \nabla g_i,\nabla f_{\sigma(i)}\ra\,\d\mm\\
&=\int\sum_{\sigma\in S_k}s(\sigma)\la\nabla g_1,\nabla f_{\sigma(1)}\ra\prod_{i>1}\la \nabla g_i,\nabla f_{\sigma(i)}\ra\,\d\mm,
\end{split}
\]
where $S_k$ is the set of permutations of $\{1,\ldots,k\}$ and $s(\sigma)$ the sign of the permutation $\sigma\in S_k$.

Integrating by parts to get rid of the gradient of $g_1$ we get
\[
\begin{split}
\int \la \omega,\d\eta\ra\,\d\mm&=-\int\sum_{\sigma\in S_k}s(\sigma)\, g_1\,\div\Big(\nabla f_{\sigma(1)}\prod_{i>1}\la \nabla g_i,\nabla f_{\sigma(i)}\ra\Big)\,\d\mm\\
&=-\int\sum_{\sigma\in S_k}s(\sigma)g_1\bigg(\,\Delta f_{\sigma(1)}\prod_{i>1}\la \nabla g_i,\nabla f_{\sigma(i)}\ra \\
& \qquad \qquad \qquad \qquad \qquad +\sum_{j>1}\H{g_j}(\nabla f_{\sigma(1)},\nabla f_{\sigma(j)})\prod_{\natop{i>1}{i\neq j}}\la \nabla g_i,\nabla f_{\sigma(i)}\ra\\
&\qquad \qquad \qquad\qquad \qquad +\sum_{j>1}\H{f_{\sigma(j)}}(\nabla f_{\sigma(1)},\nabla g_{j })\prod_{\natop{i>1}{i\neq j}}\la \nabla g_i,\nabla f_{\sigma(i)}\ra\bigg)\,\d\mm.
\end{split}
\]
The terms containing $\H {g_j}$ disappear in the sum when adding the value given by a permutation with the one obtained swapping $\sigma(1)$ and $\sigma(j)$. The thesis then follows taking into account the identities
\[
\begin{split}
\sum_{\natop{\sigma\in S_k}{\sigma(1)=I}}s(\sigma)\prod_{i>1}\la \nabla g_i,\nabla f_{\sigma(i)}\ra&=(-1)^{1+I}\big<\d g_2\wedge\ldots\wedge\d g_k,\d f_1\wedge \ldots \wedge\widehat{\d f_I} \wedge\ldots\wedge\d f_k\big>,\\
\sum_{\natop{\sigma\in S_k}{\natop{\sigma(1)=I}{\sigma{(J)}=K}}}s(\sigma)\prod_{\natop{i>1}{i\neq J}}\la \nabla g_i,\nabla f_{\sigma(i)}\ra&=a(I,J,K)\big<\d g_2\ldots\widehat{\d g_J}\ldots\d g_k,\d f_1 \ldots\widehat{\d f_I}\ldots\widehat{\d f_K}\ldots\d f_k\big>,
\end{split}
\]
valid for any $I, K\in\{1,\ldots,k\}$, $I\neq K$, $J\in\{2,\ldots,k\}$, where
\[
a(I,J,K):=\left\{\begin{array}{ll}
(-1)^{1+I+J+K}&\qquad\text{ if }I<K,\\
(-1)^{I+J+K}&\qquad\text{ if }I>K,\\
\end{array}
\right.
\]
the fact that
\[
\H{f_K}(\nabla f_I,\cdot)-\H{f_I}(\nabla f_K,\cdot)=[\nabla f_I,\nabla f_K]^\flat,
\]
and the identity
\[
\begin{split}
\sum_{J>1}(-1)^J\la [\nabla f_I,\nabla f_K],\nabla g_J\ra& \big<\d g_2\ldots \widehat{\d g_J}\ldots\d g_k,\d f_1 \ldots\widehat{\d f_I}\ldots\widehat{\d f_K}\ldots\d f_k\big>\\
&=\big<\d g_2\wedge\ldots\wedge\d g_k, [\nabla f_I,\nabla f_K]\wedge\ldots\wedge \widehat{\d f_I}\wedge\ldots\wedge\widehat{\d f_K}\wedge\ldots\d f_k\big>.
\end{split}
\]
\end{proof}
We now introduce the `Hodge' Sobolev spaces:
\begin{definition}[The spaces $W^{1,2}_\Ho(\Lambda^kT^*\X)$ and $H^{1,2}_\Ho(\Lambda^kT^*\X)$]
The space $W^{1,2}_\Ho(\Lambda^kT^*\X)$ is defined as $W^{1,2}_\d(\Lambda^kT^*\X)\cap D(\delta_k)$ endowed with the norm
\[
\|\omega\|_{W^{1,2}_\Ho(\Lambda^kT^*\X)}^2:=\|\omega\|^2_{L^2(\Lambda^kT^*\X)}+\|\d\omega\|_{L^2(\Lambda^{k+1}T^*\X)}^2+\|\delta\omega\|^2_{L^2(\Lambda^{k-1}T^*\X)},
\]
and the space $H^{1,2}_\Ho(\Lambda^kT^*\X)$ as the $W^{1,2}_\Ho$-closure of $\ffsm k$ (which by point $(v)$ of Theorem \ref{thm:basew12d} and Proposition \ref{prop:contobrutto} is a subset of $W^{1,2}_\Ho(\Lambda^kT^*\X)$). The Hodge energy functional $\eh:L^2(\Lambda^kT^*\X)\to[0,\infty]$ is defined as
\[
\eh(\omega):=\left\{\begin{array}{ll}
\displaystyle{\frac12\int |\d \omega|^2+|\delta\omega|^2\,\d\mm},&\qquad\text{if }\omega\in W^{1,2}_\Ho(\Lambda^kT^*\X),\\
+\infty,&\qquad\text{otherwise}.
\end{array}\right.
\]
\end{definition}
Arguing as for Theorems \ref{thm:basew22}, \ref{thm:basew12c} and \ref{thm:basew12d}, we see that  $W^{1,2}_\Ho(\Lambda^kT^*\X)$ and  $H^{1,2}_\Ho(\Lambda^kT^*\X)$ are separable Hilbert spaces and $\eh$ is lower semicontinuous.

We can now give the definition of Hodge Laplacian and in analogy with what we did for the connection Laplacian in Section \ref{se:clap}, we shall restrict our attention to forms in $H^{1,2}_\Ho(\Lambda^kT^*\X)$:
\begin{definition}[Hodge Laplacian and harmonic forms]
Given $k\in\N$, the domain $D(\Delta_{\Ho,k})\subset H^{1,2}_\Ho(\Lambda^kT^*\X)$ of the Hodge Laplacian is the set of $\omega\in H^{1,2}_\Ho(\Lambda^kT^*\X)$   for which there exists $\alpha\in L^2(\Lambda^kT^*\X)$ such that
\[
\int\la\alpha,\eta\ra\,\d\mm=\int \la \d\omega,\d\eta\ra+\la \delta\omega,\delta\eta\ra\,\d\mm,\qquad\forall \eta\in H^{1,2}_\Ho(\Lambda^kT^*\X).
\]
In this case, the form $\alpha$ (which is unique by the density of $H^{1,2}_\Ho(\Lambda^kT^*\X)$ in $L^2(\Lambda^kT^*\X)$) will be called Hodge Laplacian of $\omega$ and denoted by $\Delta_\Ho\omega$.

The space $\harm k\subset D(\Delta_{\Ho,k})$ is the space of forms $\omega\in D(\Delta_{\Ho,k})$ such that $\Delta_\Ho\omega=0$.
\end{definition}
Recalling that for every function $f\in  L^2(\mm)$ we have $\delta f=0$, we see that $D(\Delta_{\Ho,0})=D(\Delta)$ with the usual unfortunate sign relation:
\[
\Delta_\Ho f=-\Delta f.
\]
The very definition also gives that
\begin{equation}
\label{eq:parth}
\eh(\omega)=\frac12\int \la\omega,\Delta_\Ho\omega\ra\,\d\mm,\qquad\forall \omega\in D(\Delta_{\Ho,k}).
\end{equation}
Arguing as for the connection Laplacian, it is easy to check that the Hodge Laplacian can also be seen as the only element of the subdifferential of the augmented Hodge energy $\eht:L^2(\Lambda^kT^*\X)\to[0,\infty]$ defined as
\begin{equation}
\label{eq:eht}
\eht(\omega):=\left\{\begin{array}{ll}
\displaystyle{\frac12\int |\d \omega|^2+|\delta\omega|^2\,\d\mm},&\qquad\text{if }\omega\in H^{1,2}_\Ho(\Lambda^kT^*\X),\\
+\infty,&\qquad\text{otherwise}.
\end{array}\right.
\end{equation}
This also shows that $\Delta_\Ho$ is a closed operator, i.e.\ $\{(\omega,\Delta_\Ho\omega):\omega\in D(\Delta_{\Ho,k})\}$ is a closed subset of $L^2(\Lambda^kT^*\X)\times L^2(\Lambda^kT^*\X)$ for every $k\in\N$.
Notice that in analogy with the smooth setting we have
\begin{equation}
\label{eq:farm}
\omega\in\harm k\qquad\Longleftrightarrow \qquad\omega\in H^{1,2}_\Ho(\Lambda^kT^*\X) \text{ with }\d\omega=0\text{ and }\delta\omega=0,
\end{equation}
indeed the implication $\Leftarrow$ is obvious, for the $\Rightarrow$ just notice that if $\omega\in \harm k$ we have $\la\omega,\Delta_\Ho\omega\ra=0$ $\mm$-a.e.\ and thus
\[
0=\int \la\omega,\Delta_{\Ho}\omega\ra\,\d\mm\stackrel{\eqref{eq:parth}}=\int|\d\omega|^2+|\delta\omega|^2\,\d\mm.
\]
The closure of the Hodge Laplacian grants that  $\harm k$ is a closed subspace of $L^2(\Lambda^kT^*\X)$ and thus an Hilbert space itself when endowed with the $L^2(\Lambda^kT^*\X)$-norm. We then have the following result:
\begin{theorem}[A version of the Hodge theorem]\label{thm:hodge} The map
\[
\harm k\ni \omega\qquad\mapsto\qquad[\omega]\in H^k_{dR}(\X)
\]
is an isomorphism of Hilbert spaces.
\end{theorem}
\begin{proof} The basic theory of Hilbert spaces grants that if $H$ is an Hilbert space, $V\subset H$ a subspace and $V^\perp$ its orthogonal complement in $H$, then the map
\[
V^\perp\ni w\quad\mapsto\quad w+\overline V\in H/\overline V
\]
is an isomorphism of Hilbert spaces.

To get the proof of the theorem, just apply such statement to the Hilbert space $\closed k$ endowed with the $L^2(\Lambda^kT^*\X)$-norm,  the subspace $\exact k$ and notice that  by the very definition of codifferential we have that $\omega\in D(\delta_k)$ with $\delta\omega=0$ if and only if $\omega$ is orthogonal to $\exact k$. Since we already know that $\d\omega=0$ for every $\omega\in \closed k$, by the characterization \eqref{eq:farm} we conclude.
\end{proof}
\begin{remark}{\rm
It should be noticed that our definition of the domain of the Hodge Laplacian, and thus that of harmonic forms, is tailored to get this version of Hodge theorem. Indeed, for what we know there might be forms in $W^{1,2}_\Ho(\Lambda^kT^*\X)\setminus H^{1,2}_\Ho(\Lambda^kT^*\X)$ with zero differential and codifferential.
}\fr\end{remark}

\subsection{Ricci curvature}\label{se:ricci}
We now use all the language developed so far to reformulate the content of the crucial Lemma \ref{le:lemmachiave} and define the `Ricci curvature tensor' on $\RCD(K,\infty)$ spaces. We then discuss its basic properties and some open problems concerning its structure.

\vspace{1cm}

We start with the following computation:
\begin{proposition}\label{prop:orly}
For every $f,g\in\fsm$ we have $f\d g\in D(\Delta_{\Ho,1})$ with
\begin{equation}
\label{eq:leibh1}
\Delta_\Ho(f\,\d g)=-f\,\d\Delta g-\Delta f\,\d g-2\H g(\nabla f,\cdot)
\end{equation}
and for every $f\in\fsm$ and $X\in\vsm$ we have $X^\flat,fX^\flat \in D(\Delta_{\Ho,1})$ with
\begin{equation}
\label{eq:leibhodge}
(\Delta_\Ho(fX^\flat))^\sharp=f(\Delta_\Ho X^\flat)^\sharp-\Delta fX-2\nabla_{\nabla f}X
\end{equation}

\end{proposition}
\begin{proof} The right hand sides of \eqref{eq:leibh1} and \eqref{eq:leibhodge} define forms in $L^2(T^*\X)$, so the statements make sense. From the definition of $\Delta_\Ho$ and a density argument, to prove  \eqref{eq:leibh1} it is sufficient to show that for any $\tilde f,\tilde g\in\fsm$ we have
\[
\int \la-f\,\d\Delta g-\Delta f\,\d g-2\H g(\nabla f,\cdot),\tilde f\d\tilde g\ra\,\d\mm=\int \la \d(f\d g),\d(\tilde f\d\tilde g)\ra+ \delta(f\d g)\delta(\tilde f\d\tilde g)\,\d\mm.
\]
To this aim, notice that from formula \eqref{eq:delta1form} and the fact that $(f\d g)^\sharp=f\nabla g$ we get $\delta(f\,\d g)=-\la\nabla f,\nabla g\ra-f\Delta g\in W^{1,2}(\X)$ and thus
\[
\begin{split}
\int \delta(f\d g)\delta(\tilde f\d\tilde g)\,\d\mm&=\int \la \d(-\la\nabla f,\nabla g\ra-f\Delta g),\tilde f\d\tilde g\ra\,\d\mm\\
&=-\int \la  \H f(\nabla g,\cdot)+\H g(\nabla f,\cdot )+\Delta g\,\d f+f\d\Delta g,\tilde f\nabla \tilde g\ra\,\d\mm.
\end{split}
\]
Similarly, we have $\d(f\d g)=\d f\wedge\d g$ and using Proposition \ref{prop:contobrutto} in the case $k=1$ we see that   $\d f\wedge\d g\in D(\delta_2)$ with
\[
\begin{split}
\int \la \d (f\d g),\d (\tilde f\d\tilde g)\ra\,\d\mm&=\int \la\delta(\d f\wedge\d g),\tilde f\d\tilde g\ra\,\d\mm\\
&=\int \la\Delta g\d f-\Delta f\d g-[\nabla f,\nabla g]^\flat,\tilde f\d\tilde g\ra\,\d\mm.
\end{split}
\]
Then \eqref{eq:leibh1} follows noticing  that $[\nabla f,\nabla g]^\flat=\H g(\nabla f,\cdot)-\H f(\nabla g,\cdot)$.

For \eqref{eq:leibhodge} notice that the fact that $X^\flat,fX^\flat$ are in $D(\Delta_{\Ho,1})$ follows by what we just proved and the fact that $\fsm$ is an algebra. Then by linearity it is sufficient to consider the case $X=g_1\nabla g_2$, $g_1,g_2\in\fsm$, and in this case from \eqref{eq:leibh1} we get
\[
\begin{split}
\Delta_\Ho(fg_1\,\d g_2)&=-fg_1\,\d\Delta g_2-\Delta(fg_1)\d g_2-2\H{g_2}(\nabla(fg_1),\cdot)\\
&=-fg_1\,\d\Delta g_2-f \Delta g_1\d g_2-g_1\Delta f\,\d g_2-2\la\nabla f,\nabla g_1\ra\d g_2\\
&\qquad\qquad\qquad\qquad\qquad\qquad-2 f\H{g_2}(\nabla g_1,\cdot)-2 g_1\H{g_2}(\nabla f,\cdot)\\
&=f\Delta_{\Ho}(g_1\d g_2)-\Delta f(g_1\d g_2)-2(\nabla_{\nabla f}(g_1\nabla g_2))^\sharp,
\end{split}
\]
having recalled that $\nabla(g_1\nabla g_2)=\nabla g_1\otimes\nabla g_2+g_1\H{g_2}^\sharp$ in the last step.
\end{proof}
We shall need the Leibniz rule for the measure valued Laplacian:
\begin{equation}
\label{eq:leibml}
\left.\begin{array}{ll}
&f\in\fsm\\
&g\in D(\bd)
\end{array}\right\}
\qquad\Rightarrow\qquad fg\in D(\bd)\quad\text{ and }\quad\bd(fg)=\bar f\bd g+\big(\Delta fg+2\la\nabla f,\nabla g\ra\big)\mm,
\end{equation}
where $\bar f$ is the continuous representative of $f$. This can be proved as \eqref{eq:leiblap}:  just notice that for any $\bar \varphi:\X\to\R$ Lipschitz with bounded support the function $\bar f\bar\varphi$ is also Lipschitz with bounded support and
\[
\begin{split}
-\int\la \nabla\varphi,\nabla(fg)\ra\,\d\mm&=-\int g\la\nabla\varphi,\nabla f\ra+f\la\nabla\varphi,\nabla g\ra\,\d\mm\\
&=\int- \la\nabla(\varphi g),\nabla f\ra-\la\nabla(\varphi f),\nabla g\ra+2\varphi\la\nabla f,\nabla g\ra\,\d\mm\\
&=\int \varphi g\Delta f+2\varphi\la\nabla f,\nabla g\ra\,\d\mm+\int\bar\varphi\bar f\,\d\bd g,
\end{split}
\]
which is the claim.

We shall also make use of the duality formula
\begin{equation}
\label{eq:dualsym}
|A_{\sf Sym}|_\HS^2=\esssup \left(2A:\sum_{i=1}^m\nabla h_i\otimes\nabla h_i-\left|\sum_{i=1}^n\nabla h_i\otimes\nabla h_i\right|_\HS^2\right)
\end{equation}
valid for every $A\in L^2(T^{\otimes 2}\X)$, where the $\esssup$ is taken among all $m\in\N$ and $h_1,\ldots,h_m\in\fsm$. This is an instance of formula \eqref{eq:normsym} (recall that $\{\d h:h\in\fsm\}$ is a space of bounded elements generating $L^2(T^*\X)$ in the sense of modules).

We now have all the ingredients needed to reinterpret the key Lemma \ref{le:lemmachiave} in terms of the differential calculus developed so far:
\begin{lemma}\label{le:riscritto}
Let $X\in\vsm$. Then $X^\flat\in D(\Delta_{\Ho,1})$, $|X|^2\in D(\bd)$ and the inequality
\begin{equation}
\label{eq:bochner}
\bd\frac{|X|^2}{2}\geq \Big(|\nabla X|_\HS^2-\la X, (\Delta_\Ho X^\flat)^\sharp\ra+K|X|^2\Big)\mm
\end{equation}
holds
\end{lemma}
\begin{proof} The fact that $X^\flat\in D(\Delta_{\Ho,1})$ comes from Proposition \ref{prop:orly}, while from the fact that for $f\in\fsm$ we have $|\nabla f|^2\in D(\bd)$ (Proposition \ref{prop:regtest}), so that the simple property \eqref{eq:leibml} gives that  $|X|^2\in D(\bd)$. Still \eqref{eq:leibml} gives, with little algebraic manipulation, that for $X=\sum_ig_i\nabla f_i$ we have 
\[
\begin{split}
\bd\frac{|X|^2}{2}=\sum_{i,j}&\Big(g_j\Delta g_i\la \nabla f_i,\nabla f_j\ra+\la\nabla g_i,\nabla g_j\ra\la\nabla f_i,\nabla f_j\ra\Big)\mm\\
&+\Big(2g_i\H{f_i}(\nabla f_j,\nabla g_j)+2g_i\H{f_j}(\nabla f_i,\nabla g_j)\Big)\mm+\frac12\overline g_i\overline g_j\bd\la\nabla f_i,\nabla f_j\ra,
\end{split}
\]
where $\overline g_i$ is the continuous representative of $g_i$. Writing $\bd\frac{|X|^2}{2}=\Delta_{ac}\frac{|X|^2}{2}\mm+\bd_{sing}\frac{|X|^2}{2}$ with $\bd_{sing}\frac{|X|^2}{2}\perp\mm$, the thesis will be achieved if we show that
\begin{equation}
\label{eq:thsing}
\bd_{sing}\frac{|X|^2}{2}\geq 0
\end{equation}
and
\begin{equation}
\label{eq:thac}
\Delta_{ac}\frac{|X|^2}{2}\geq  |\nabla X|_\HS^2-\la X, (\Delta_\Ho X^\flat)^\sharp\ra+K|X|^2,\qquad\mm\ae.
\end{equation}
By   Proposition \ref{prop:orly} we get  the formula
\[
\Delta_{\Ho}X^\flat=\sum_i-g_i\d\Delta f_i-\Delta g_i\d f_i-2\H {f_i}(\nabla g_i,\cdot)
\]
and therefore
\[
-\la X, (\Delta_\Ho X^\flat)^\sharp\ra=\sum_{i,j}g_ig_j\la\nabla \Delta f_i,\nabla f_j\ra+g_j\Delta g_i\la \nabla f_i,\nabla f_j\ra+2g_j\H{f_i}(\nabla g_i,\nabla f_j).
\]
Moreover, from formula \eqref{eq:covvsm} and recalling the definition of $A_{\sf Asym}$ in \eqref{eq:asim} we see that
\[
\begin{split}
(\nabla X)_{\sf Asym}&=\sum_{i}\frac{\nabla g_i\otimes\nabla f_i-\nabla f_i\otimes\nabla g_i}2
\end{split}
\]
and therefore
\[
|(\nabla X)_{\sf Asym}|_\HS^2=\sum_{i,j}\frac{\la \nabla f_i,\nabla f_j\ra\la\nabla g_i,\nabla g_j\ra-\la\nabla f_i,\nabla g_j\ra\la\nabla g_i,\nabla f_j\ra}2.
\]
Hence recalling the definition of the measure $\mu((f_i),(g_i))$ given in Lemma \ref{le:lemmachiave} we see that
\[
\mu\big((f_i),(g_i)\big)=\bd\frac{|X|^2}{2}+\Big(\la X, (\Delta_\Ho X^\flat)^\sharp\ra-K|X|^2-|(\nabla X)_{\sf Asym}|_\HS^2\Big)\mm.
\]
In particular, we see that the singular part of $\bd\frac{|X|^2}{2}$ w.r.t.\ $\mm$ coincides with the singular part of $\mu\big((f_i),(g_i)\big)$ w.r.t.\ $\mm$ and thus inequality \eqref{eq:partesing1} in Lemma \ref{le:lemmachiave}  gives \eqref{eq:thsing}.

On the other hand, inequality \eqref{eq:parteac1} in Lemma \ref{le:lemmachiave}  grants that for every $m\in\N$ and choice of $h_1,\ldots,h_m\in\fsm$ we have
\[
\left|\nabla X:\sum_{i=1}^m\nabla h_i\otimes\nabla h_i\right|\leq  \sqrt{ \Delta_{ac}\frac{|X|^2}{2}+ \la X, (\Delta_\Ho X^\flat)^\sharp\ra-K|X|^2-|(\nabla X)_{\sf Asym}|_\HS^2}\left|\sum_{i=1}^m\nabla h_i\otimes\nabla h_i\right|_\HS
\]
$\mm$-a.e., which after an application of Young's inequality at the right hand side gives
\[
2\nabla X:\sum_{i=1}^m\nabla h_i\otimes\nabla h_i-\left|\sum_{i=1}^m\nabla h_i\otimes\nabla h_i\right|_\HS^2\leq \Delta_{ac}\frac{|X|^2}{2}+  \la X, (\Delta_\Ho X^\flat)^\sharp\ra-K|X|^2-|(\nabla X)_{\sf Asym}|_\HS^2
\]
$\mm$-a.e.. Taking the essential supremum over all $m\in\N$ and choices of $h_1,\ldots,h_m\in\fsm$ and recalling the identity \eqref{eq:dualsym}  we obtain
\[
|(\nabla X)_{\sf Sym}|_\HS^2\leq \Delta_{ac}\frac{|X|^2}{2}+ X\cdot (\Delta_\Ho X^\flat)-K|X|^2-|(\nabla X)_{\sf Asym}|_\HS^2,\qquad\mm\ae,
\]
which by \eqref{eq:decompsym} is \eqref{eq:thac}.
\end{proof}
It might be worth to remark that for $X\in\vsm$ we have
\begin{equation}
\label{eq:integrale0}
\bd\frac{|X|^2}2(\X)=0.
\end{equation}
This can be easily checked picking a sequence $(\bar\varphi_n)$ of uniformly Lipschitz and uniformly bounded functions with bounded support and everywhere converging to 1, noticing that $\bd\frac{|X|^2}2(\X)$ is the limit of $\int\bar \varphi_n\,\d\bd\frac{|X|^2}{2}=-\int \nabla X:(\nabla \varphi_n\otimes X)\ra\,\d\mm$, that $|\nabla X:(\nabla \varphi_n\otimes X)|\leq |\nabla X|_\HS|X|\,|\nabla\varphi_n| $ and concluding by the dominate convergence theorem thanks to the fact that $ |\nabla X|_\HS|X|\in L^1(\mm)$.

In the foregoing discussion it will be useful to read the space $H^{1,2}_\Ho(T^*\X)$ in terms of vector fields rather than covector ones. Thus we give the following simple definition:
\begin{definition}
The space $H^{1,2}_\Ho(T\X)\subset L^2(T\X)$ is the space of vector fields $X$ such that $X^\flat\in H^{1,2}_\Ho(T^*\X)$ equipped with the norm $\|X\|_{H^{1,2}_\Ho(T\X)}:=\|X^\flat\|_{H^{1,2}_\Ho(T^*\X)}$.
\end{definition}

Lemma \ref{le:riscritto} directly gives the following inequality, which generalizes Corollary \ref{cor:bello} to   vector fields:
\begin{corollary}\label{cor:bello2}
We have $H^{1,2}_\Ho(T\X)\subset H^{1,2}_C(T\X)$ and
\begin{equation}
\label{eq:ehec}
\mathcal E_C(X)\leq \mathcal E_\Ho(X^\flat)-\frac K2\|X\|^2_{L^2(T\X)},\qquad\forall X\in H^{1,2}_\Ho(T\X).
\end{equation}
\end{corollary}
\begin{proof}
Let  $X\in\vsm$ and compute the measure of the whole space $\X$ in inequality \eqref{eq:bochner}. Taking into account the simple identity \eqref{eq:integrale0} we obtain
\[
\int|\nabla X|_\HS^2\,\d\mm\leq \int \la X,(\Delta_\Ho X^\flat)^\sharp\ra-K|X|^2\,\d\mm.
\]
Recalling \eqref{eq:parth}, \eqref{eq:ehec} is proved for $X\in\vsm$, The general case then follows approximating an arbitrary $X\in H^{1,2}_\Ho(T\X)$ in the $H^{1,2}_\Ho$-norm  with vectors in $\vsm$ and recalling the $L^2(T\X)$-lower semicontinuity of $\ec$.
\end{proof}
A simple consequence of this corollary is the following bound, which adapts to the current context a classical idea of Bochner, on the dimension of $H^{1}_{dR}(\X)$ on spaces with non-negative Ricci curvature:
\begin{proposition}[Bounding $\dim(H^1_{dR}(\X))$ on $\RCD(0,\infty)$ spaces]
Let $(X,\sfd,\mm)$ be an $\RCD(0,\infty)$ space and  $(E_i)_{i\in\N\cup\{\infty\}}$ the dimensional decomposition of $\X$ associated to the cotangent module $L^2(T^*\X)$ as given by Proposition \ref{prop:dimdec}. Denote by $\text{\sc n}_{\rm min}$ the minimal index $i$ in  $\N\cup\{\infty\}$  such that $\mm(E_i)>0$.

Then $\dim(H^1_{dR}(\X))\leq \text{\sc n}_{\rm min}$.
\end{proposition}
\begin{proof} By Theorem \ref{thm:hodge} we know that $\dim(H^1_{dR}(\X))=\dim(\harm 1)$, hence if $\harm 1=\{0\}$ there is nothing to prove. Thus assume that  $\dim(\harm 1)\geq 1$ and notice that for every  $\omega\in\harm 1$ inequality \eqref{eq:ehec} gives that
\[
\ec(\omega^\sharp)\leq \eh(\omega)=\int\la \omega,\Delta_H\omega\ra\,\d\mm=0.
\]
Therefore, by definition of $\ec$,  the covariant derivative of $\omega^\sharp$ is identically 0. 

It follows from point $(ii)$ in Proposition \ref{prop:compmetr}  that for  any two $\omega_1,\omega_2\in\harm 1$ the function $f:=\la\omega_1,\omega_2\ra=\la\omega_1^\sharp,\omega_2^\sharp\ra$ is in $H^{1,1}(\X)$ with $|\d f|=0$ $\mm$-a.e.. It is then clear that such $f$ must be $\mm$-a.e.\ constant: just recall that by the property \eqref{eq:tronch11} of the truncated functions $f_n:=\max\{\min\{f,n\},-n\}\in L^2(\mm)$ and Proposition \ref{prop:w11w12} we have that $f_n\in W^{1,2}(\X)$ so that from the Sobolev-to-Lipschitz property \eqref{eq:sobtolip} we infer that $f_n$ is $\mm$-a.e.\ constant. Being this true for every $n\in\N$, $f$ is $\mm$-a.e.\ constant as claimed.

Now let $\omega_1,\ldots,\omega_n\in (\harm 1,\|\cdot\|_{L^2(T^*\X)})$ orthonormal: since $\int\la\omega_i,\omega_j\ra\,\d\mm$ is equal to 1 if $i=j$ and to 0 otherwise, what we just proved implies that $\mm(\X)<\infty$ and 
\[
|\omega_i|^2=\frac1{\mm(X)},\quad\mm\ae\quad\forall i,\qquad\text{and}\qquad \la\omega_i,\omega_j\ra=0,\quad\mm\ae\quad\forall i\neq j.
\]
It is then clear that $\omega_1,\ldots,\omega_n$ are locally independent on $\X$ (Definition \ref{def:locind}) and in particular locally independent on $E_{\text{\sc n}_{\rm min}}$. By Proposition \ref{prop:localdimension} we deduce that $n\leq \text{\sc n}_{\rm min}$ and being this true for any choice of $n$ orthonormal elements of $\harm 1$, we deduce that $\dim(\harm 1)\leq \text{\sc n}_{\rm min}$, as desired.
\end{proof}

\begin{remark}{\rm

There is no finite dimensionality assumption in this last statement and thus we  might certainly have $ \text{\sc n}_{\rm min}=\infty$ so that taken as it is this proposition does not tell that much.   The related question is thus if one can say that on $\RCD(K,N)$ spaces one has  $ \text{\sc n}_{\rm min}\leq N$.

To answer this question is outside the scope of this work, but we remark that in the recent paper \cite{Mondino-Naber14} Mondino-Naber proved that $\RCD(K,N)$ spaces admit  biLipschitz charts - defined on Borel sets - where the corresponding target Euclidean spaces has dimension  bounded from above by $N$. Stated as it is, such result does not really allow to conclude due to the lack of information about the behavior of the measure on the chart, but apparently (private communication of the authors) with minor modification of the arguments one can also show that the image measure is absolutely continuous w.r.t.\ the Lebesgue one with bounded density. With this additional information and using the results of Section \ref{se:mbd} one could easily conclude that $ \text{\sc n}_{\rm min}\leq N$.

A different approach to the same result, more in line with the analysis carried out here, is to provide in the non-smooth setting a rigorous justification to the computations done by Sturm in \cite{Sturm14}: there the author showed how to deduce via purely algebraic means out of the Bochner inequality with the dimension term the fact that the abstract tangent space has dimension bounded by $N$.  Although the `smoothness' assumptions in \cite{Sturm14}  are not really justified in the general setting, the style of the arguments used seems suitable of adaptation in the language of $L^\infty(\mm)$-modules developed here.
}\fr\end{remark}

Another direct consequence of Lemma \ref{le:riscritto} is that it allows to define the Ricci  curvature:
\begin{theorem}[Ricci curvature]\label{thm:ricci}
There exists a unique continuous map $\ric: [H^{1,2}_\Ho(T\X)]^2\to \mes(\X)$ such that for every $X,Y\in\vsm$ it holds
\begin{equation}
\label{eq:defricci}
\ric(X,Y)=\bd\frac{\la X, Y\ra}2+\Big(\frac12\la X,(\Delta_\Ho Y^\flat)^\sharp\ra+ \frac12\la Y,(\Delta_\Ho X^\flat)^\sharp\ra-\nabla X:\nabla Y\Big)\mm.
\end{equation}
Such map is bilinear, symmetric and satisfies
\begin{align}
\label{eq:riccibound}
\ric(X,X)&\geq K|X|^2\mm,\\
\label{eq:riccitotal}
\ric(X,Y)(\X)&=\int\la\d X^\flat,\d Y^\flat\ra+\delta X^\flat\delta Y^\flat-\nabla X:\nabla Y\,\d\mm,\\
\label{eq:riccitv}
\|\ric(X,Y)\|_{\sf TV}&\leq2\sqrt{\eh(X^\flat)+K^-\|X\|_{L^2(T\X)}^2}\,\sqrt{\eh(Y^\flat)+K^-\|Y\|_{L^2(T\X)}^2},
\end{align}
for every $X,Y\in H^{1,2}_\Ho(T\X)$, where $K^-:=\max\{0,-K\}$.
\end{theorem}
\begin{proof} By Lemma \ref{le:riscritto} it directly follows by polarization that  $\la X, Y\ra\in D(\bd)$ for $X,Y\in\vsm$ so that \eqref{eq:defricci} makes sense. It is then clear that \eqref{eq:defricci} defines a map $\ric:[\vsm]^2\to\mes(\X)$ which is bilinear and symmetric.  The validity of \eqref{eq:riccibound} for $X\in\vsm$ is just a restatement of inequality \eqref{eq:bochner} and \eqref{eq:riccitotal} for $X,Y\in\vsm$ follows from \eqref{eq:integrale0} and the very definition of Hodge Laplacian.

Putting $\widetilde\ric(X,Y):=\ric(X,Y)-K\la X, Y\ra\mm$ for every $X,Y\in\vsm$, we have that $\widetilde\ric(X,X)\geq 0$ and thus also that
\begin{equation}
\label{eq:lezione}
\begin{split}
\|\widetilde\ric(X,X)\|_{\sf TV}=\widetilde\ric(X,X)(\X)&=2\eh(X^\flat)-2\ec(X)-K\|X\|_{L^2(T\X)}^2\\
&\leq2\eh(X^\flat)-K\|X\|_{L^2(T\X)}^2.
\end{split}
\end{equation}
Then for arbitrary $X,Y\in\vsm$ and $\lambda\in\R$ we have $\widetilde\ric(X+\lambda Y,X+\lambda Y)\geq 0$ and therefore by bilinearity and symmetry we get
\[
\lambda^2\widetilde\ric(X,X)+2\lambda\widetilde\ric(X,Y)+\widetilde\ric(Y,Y)\geq 0,
\]
so that the bound \eqref{eq:mu2tv} in Lemma \ref{le:lambda}  and inequality \eqref{eq:lezione} grant that
\[
\|\widetilde\ric(X,Y)\|_{\sf TV}\leq \sqrt{2\eh(X^\flat)-K\|X\|_{L^2(T\X)}^2}\sqrt{2\eh(Y^\flat)-K\|Y\|_{L^2(T\X)}^2}.
\]
Therefore since $\|\ric(X,Y)\|_{\sf TV}\leq \|\widetilde\ric(X,Y)\|_{\sf TV}+|K|\int|\la X, Y\ra|\,\d\mm$ we get
\[
\begin{split}
\|\ric(X,Y)\|_{\sf TV}&\leq \sqrt{2\eh(X^\flat)-K\|X\|_{L^2(T\X)}^2}\sqrt{2\eh(Y^\flat)-K\|Y\|_{L^2(T\X)}^2}+|K|\int |\la X, Y\ra|\,\d\mm \\
&\leq \sqrt{2\eh(X^\flat)+(|K|-K)\|X\|_{L^2(T\X)}^2}\sqrt{2\eh(Y^\flat)+(|K|-K)\|Y\|_{L^2(T\X)}^2},
\end{split}
\]
which is \eqref{eq:riccitv} for $X,Y\in\vsm$. This also proves that $\ric:[\vsm]^2\to\mes(\X)$ is continuous w.r.t.\ the $H^{1,2}_\Ho(T\X)$-norm, thus granting existence and uniqueness of the continuous extension of such map to $[H^{1,2}_\Ho(T\X)]^2$ and the conclusion. 
\end{proof}
We shall refer to the map $\ric:[H^{1,2}_\Ho(T\X)]^2\to\mes(\X)$ given by the previous theorem as to the {\bf Ricci curvature}. Notice that formula \eqref{eq:defricci} for $X=Y$ gives the familiar Bochner identity
\[
\bd\frac{|X|^2}{2}=\big(|\nabla X|_\HS^2-\la X,(\Delta_\Ho X^\flat)^\sharp\ra\big)\mm +\ric(X,X),\qquad\forall X\in\vsm,
\]
and in particular
\[
\bd\frac{|\nabla f|^2}{2}=\big(|\H f|_\HS^2+\la\nabla f,\nabla\Delta f\ra\big)\mm+\ric(\nabla f,\nabla f),\qquad\forall f\in\fsm.
\]
The expression \eqref{eq:defricci} cannot be written for generic $X,Y\in H^{1,2}_\Ho(T\X)$ because we don't know whether $\la X,Y\ra$ is in $D(\bd)$ nor if $X^\flat,Y^\flat$ are in $D(\Delta_{\Ho,1})$, but still explicit expressions   for the Ricci curvature  can be obtained   by `throwing some derivatives on the function one is integrating':
\begin{proposition}[Some representations of $\ric$]
For every $X,Y\in H^{1,2}_\Ho(T\X)$ and $f\in\fsm$ we have
\begin{equation}
\label{eq:riccialt}
\int \bar f\,\d\ric(X,Y)=\int \la \d X^\flat,\d(fY^\flat)\ra+ \delta X^\flat \delta(fY^\flat)-\nabla X:\nabla(fY) \,\d\mm
\end{equation}
and the more symmetric expression
\begin{equation}
\label{eq:ricsimm}
\begin{split}
\int \bar f\,\d\ric(X,Y)=\int &\H f(X,Y)+\d f\big(X \div( Y)+Y \div(X)\big)\\
&\qquad+f\big(\la \d X^\flat,\d Y^\flat\ra + \delta X^\flat \delta Y^\flat-\nabla X:\nabla Y\big)\,\d\mm,
\end{split}
\end{equation}
where $\bar f$ is the continuous representative of $f$.
\end{proposition}
\begin{proof} The two sides of both \eqref{eq:riccialt} and \eqref{eq:ricsimm} are continuous in $X,Y$ w.r.t.\ the $H^{1,2}_\Ho(T\X)$-topology, thus by approximation we can assume that $X,Y\in\vsm$. The proof then comes by direct computation. 

Fix $f\in\fsm$, $X,Y\in\vsm$ and call $A(X,Y)$ and $B(X,Y)$ the right hand sides of  \eqref{eq:riccialt} and \eqref{eq:ricsimm} respectively. We start claiming that $A(X,Y)=B(X,Y)$, which follows noticing that
\[
\begin{split}
\int \la\d X^\flat,\d(fY^\flat)\ra\,\d\mm&=\int\la \d X^\flat,\d f\wedge Y^\flat\ra+f\la \d X^\flat,\d Y^\flat\ra \,\d\mm\\
&=\int \d X^\flat(\nabla f,Y) +f\la \d X^\flat,\d Y^\flat\ra\,\d\mm\\
&=\int -\la X, Y\ra\Delta f+\la X, \nabla f\ra \div(Y)-\la X,[\nabla f,Y]\ra+f\la\d X^\flat,\d Y^\flat\ra\,\d\mm\\
&=\int -\la X, Y\ra\Delta f+\la X, \nabla f\ra \div(Y)\\
&\qquad\qquad-\nabla Y:(\nabla f\otimes X) +\H f(X,Y)+f\la\d X^\flat,\d Y^\flat\ra\,\d\mm\\
%
\int \delta X^\flat \delta(fY^\flat)\,\d\mm&=\int \div( X)\div(f Y)\,\d\mm=\int \div(X)\la\nabla f,  Y\ra+f\, \div(X)\,\div( Y)\,\d\mm\\
\int -\nabla X:\nabla(fY)\,\d\mm&=\int-\nabla X:(\nabla f\otimes Y)-f\nabla X:\nabla Y\,\d\mm
\end{split}
\]
adding everything up and using the identity
\[
\int -\la X, Y\ra\Delta f-\nabla Y:(\nabla f\otimes X)-\nabla X:(\nabla f\otimes Y)\,\d\mm=0.
\] 
Thus $A(X,Y)=B(X,Y)$ as claimed, which also ensures $A(X,Y)=A(Y,X)$. 

Now notice  that directly from \eqref{eq:defricci} we get
\[
\begin{split}
\int \bar f\,\d\ric(X,Y)&=\frac12\int -\big<\nabla f,\nabla\la X,Y\ra\big>+\la \d (fX^\flat),\d Y^\flat\ra +\delta(fX^\flat)\delta Y^\flat\\
&\qquad \qquad \qquad \qquad +\la  \d X^\flat,\d (fY^\flat)\ra +\delta X^\flat\delta (fY^\flat)-2f\nabla X:\nabla Y\,\d\mm,
\end{split}
\]
and therefore observing that
\[
\begin{split}
 -\big<\nabla f,\nabla\la X,Y\ra\big>-2f\nabla X:\nabla Y&=-\nabla X:(\nabla f\otimes Y)-\nabla Y:(\nabla f\otimes X)-2f\nabla X:\nabla Y\\
 &=-\nabla(fX):\nabla Y-\nabla X:\nabla (fY),
\end{split}
\]
we see that
\[
\int \bar f\,\d\ric(X,Y)=\frac12\big(A(X,Y)+A(Y,X)\big),
\]
which gives the thesis.
\end{proof}
Notice that in formula \eqref{eq:riccialt} if we knew that $X$ was both in $D(\Delta_C)$ and such that $X^\flat\in D(\Delta_{\Ho,1})$, then after an integration by parts we would get
\[
\int \bar f\,\d\ric(X,Y)=\int f\la Y,(\Delta_\Ho X^\flat)^\sharp+\Delta_CX\ra\,\d\mm,
\]
which is an instance of the classical Weitzenb\"ock identity. Yet, we don't really know if any non-zero $X$ as above exists, so we are obliged to formulate such identity as done in  \eqref{eq:riccialt}.

The Ricci curvature has the following tensor-like property:
\begin{proposition}\label{prop:riccifx}
For every $X,Y\in H^{1,2}_\Ho(T\X)$ and $f\in\fsm$ we have
\begin{equation}
\label{eq:riccixf}
\ric(fX,Y)=\bar f\,\ric(X,Y),
\end{equation}
where $\bar f$ is the continuous representative of $f$.
\end{proposition}
\begin{proof}
By the continuous dependence of $\ric$ on $X,Y\in H^{1,2}_\Ho(T\X)$ we can, and will, assume that $X,Y\in\vsm$ and with a density argument based on the approximation property \eqref{eq:aplipt} it is sufficient to show that for any $g\in \fsm$ we have
\[
\int \bar g\,\d\ric(f X,Y)=\int \bar g\bar f\,\d\ric(X,Y).
\]
By the defining property \eqref{eq:defricci} we have
\[
\begin{split}
\int \bar g\,\d\ric(f X,Y)&=\int \frac12\Big(f \Delta g \la X,Y\ra+fg\la X,(\Delta_\Ho Y^\flat)^\sharp\ra+g\la Y,(\Delta_\Ho(fX)^\flat)^\sharp\ra\Big)\\
&\qquad\qquad\qquad\qquad\qquad\qquad\qquad\qquad\qquad\qquad-g\nabla(fX):\nabla Y\,\d\mm,\\
\int \bar g\bar f\,\d\ric(X,Y)&=\int\frac12\Big( \Delta(fg)\la X,Y\ra+fg\la X,(\Delta_\Ho Y^\flat)^\sharp\ra+\la Y,(\Delta_\Ho X^\flat)^\sharp\ra\Big)\\
&\qquad\qquad\qquad\qquad\qquad\qquad\qquad\qquad\qquad\qquad-fg\nabla X:\nabla Y\,\d\mm,
\end{split}
\]
and the conclusion follows with a term-by-term comparison taking into account the identities
\[
\begin{split}
\Delta(fg)&=g\Delta f+f\Delta g+2\la\nabla f,\nabla g\ra\\
\nabla(fX):\nabla Y&=f \nabla X:\nabla Y+\nabla Y:(\nabla f\otimes X)\\
(\Delta_\Ho(fX)^\flat)^\sharp&=f(\Delta_\Ho X^\flat)^\sharp-\Delta fX-2\nabla_{\nabla f}X
\end{split}
\]
(see \eqref{eq:leiblap}, \eqref{eq:leibcov}, \eqref{eq:leibhodge}) and the fact that
\[
\int \la\nabla f,\nabla g\ra\la X,Y\ra\,\d\mm=-\int g\Big(\Delta f\la X,Y\ra +\nabla X:(\nabla f\otimes Y) +\nabla Y:(\nabla f\otimes X)\Big)\,\d\mm,
\]
which follows after an integration by parts taking into account formula \eqref{eq:compmetr}.
\end{proof}
As a direct consequence of these formulas and of the locality properties \eqref{eq:loccov2}, \eqref{eq:localext} and \eqref{eq:localdelta}, we see that the Ricci curvature is local in the following sense:
\[
\left.\begin{array}{rl}
X=0,&\ \mm\ae\text{ on }\Omega\\
Y=0,&\ \mm\ae\text{ on }\Omega'\\
\end{array}\right\}
\text{  for $\Omega,\Omega'\subset \X$ open with $\X=\Omega\cup\Omega'$}\qquad\Rightarrow\qquad \ric(X,Y)=0,
\]
and similarly
\[
X=Y\quad\mm\ae\text{ on an open set $\Omega\subset\X$}\qquad\Rightarrow\qquad \ric(X,X)\restr{\Omega}=\ric(Y,Y)\restr{\Omega}.
\]

\bigskip

To have at disposal the bound on the Ricci curvature tensor allows to generalize the Bakry-\'Emery contraction estimate to 1-forms. To state the inequality we first introduce the heat flow $(\h_{\Ho,t})$ on 1-forms as the gradient flow of the augmented Hodge energy functional $\eht:L^2(T^*\X)\to[0,\infty]$ defined in \eqref{eq:eht}. This means that for every  $\omega\in L^2(T^*\X)$ the curve $t\mapsto\h_{\Ho,t}(\omega)\in L^2(T^*\X)$ is the unique continuous curve on $[0,\infty)$ which is locally absolutely continuous on $(0,\infty)$ and fulfills $\h_{\Ho,t}(\omega)\in D(\Delta_{\Ho,1})$ and
\begin{equation}
\label{eq:hh}
\frac\d{\d t}\h_{\Ho,t}(\omega)=-\Delta_{\Ho}\h_{\Ho,t}(\omega),\qquad\forall t>0.
\end{equation}
Notice that for $f\in W^{1,2}(\X)$ we have
\[
\h_{\Ho,t}(\d f)=\d\h_t(f),\qquad\forall t\geq0,
\]
as can be checked for instance noticing that $t\mapsto \d\h_{t}(f)$ satisfies \eqref{eq:hh} and using the aforementioned uniqueness.

Then we have the following estimate:
\begin{proposition}
For every $\omega\in L^2(T^*\X)$ we have
\[
|\h_{\Ho,t}(\omega)|^2\leq e^{-2Kt}\h_t(|\omega|^2),\qquad\mm\ae\qquad\forall t\geq 0.
\]
\end{proposition}
\begin{proof}  The argument is similar to that for Proposition \ref{prop:bev}. By the approximation property \eqref{eq:aplipt} it is sufficient to prove that for any $t>0$ and non-negative $f\in\fsm$ with $\Delta f\in L^\infty(\mm)$ it holds  $\int f|\h_{\Ho,t}(\omega)|^2\,\d\mm\leq e^{-2Kt}\int f\h_t(|\omega|^2)\,\d\mm$.

Thus fix such $t$ and $f$ and consider the map $F:[0,s]\to\R$ given by
\[
F(s):=\int f\h_{t-s}(|\h_{\Ho,s}(\omega)|^2)\,\d\mm=\int \h_{t-s}(f) |\h_{\Ho,s}(\omega)|^2\,\d\mm.
\]
Then the very same arguments used in Proposition \ref{prop:bev} grant that $F:[0,t]\to\R$ is continuous and locally absolutely continuous on $(0,t]$ and in computing its derivative we can pass the derivative inside the integral. Thus putting $f_s:=\h_s(f)$ and $X_s:=(\h_{\Ho,s}(\omega))^\sharp$ we have
\[
\begin{split}
F'(s)=\int-\Delta f_{t-s}|X_{s}|^2-2f_{t-s}\la X_{s},(\Delta_\Ho X^\flat_{s})^\sharp\ra\,\d\mm.
\end{split}
\]
By Corollary \ref{cor:bello2}  we see that $X_{s}\in H^{1,2}_C(T\X)$ and thus by point $(ii)$ in Proposition \ref{prop:compmetr}  that $|X_{s}|^2\in H^{1,1}(\X)$ and
\[
\begin{split}
F'(s)&=\int \la\nabla f_{t-s},\nabla|X_{s}|^2\ra-2\la f_{t-s} X_{s},(\Delta_\Ho X^\flat_{s})^\sharp\ra\,\d\mm\\
&=2\int \nabla X_{s}:(\nabla f_{t-s}\otimes  X_{s})-\la \d (f_{t-s} X_{s})^\flat,  \d X^\flat_{s})\ra-\delta (f_{t-s} X^\flat_{s})\,\delta X^\flat_{s} \,\d\mm.
\end{split}
\]
Recalling the Leibniz rule \eqref{eq:leibcov} we get
\[
\nabla X_{s}:(\nabla f_{t-s}\otimes  X_{s})=\nabla X_{s}:\nabla(f_{t-s}X_{s})-f_{t-s}|\nabla X_{s}|_\HS^2\leq \nabla X_{s}:\nabla(f_{t-s}X_{s}),
\]
and thus from formula \ref{eq:riccialt} we obtain
\[
F'(s)\leq -2\int f_{t-s}\,\d\ric(X_{s},X_{s})\leq -2K\int f_{t-s}|X_{s}|^2\,\d\mm=-2K\,F(s),
\]
so that the conclusion follows by Gronwall's lemma.
\end{proof}

We built the Ricci curvature tensor on an $\RCD(K,\infty)$ space and used the weak curvature assumption to deduce that the Ricci curvature is indeed bounded from below by $K$, but one might also ask whether the viceversa holds, i.e.\ if having a lower bound on the Ricci curvature tensor gives any information in terms of synthetic treatment of  lower Ricci curvature bounds.

A result in this direction is provided by the following theorem, proved in \cite{AmbrosioGigliSavare12}:
\begin{theorem}[From Bochner inequality to $\RCD$]
Let $(\X,\sfd,\mm)$ be an $\RCD(K,\infty)$ space, $K'>K$ and assume that  the inequality 
\begin{equation}
\label{eq:bochbe}
\int \Delta\varphi\frac{|\nabla f|^2}2\,\d\mm\geq \int \varphi\la\nabla f,\Delta\nabla f\ra+K'\varphi|\nabla f|^2\,\d\mm
\end{equation}
holds for every $f,\varphi\in D(\Delta)\cap L^\infty(\mm)$ with $\Delta f,\Delta\varphi\in W^{1,2}\cap L^\infty (\X)$ and $\varphi\geq 0$. Then $(\X,\sfd,\mm)$ is an $\RCD(K',\infty)$ space.
\end{theorem}
\begin{proof} See Theorem 1.1 in \cite{AmbrosioGigliSavare12} and notice that all the topological assumptions are automatically satisfied because the space is assumed to be $\RCD(K,\infty)$ (making the proof much simpler in this case).
\end{proof}
With this theorem at disposal it is obvious that 
\[
\left.\begin{array}{l}
(\X,\sfd,\mm)\text{ is an $\RCD(K,\infty)$ space with}\\
\ric(X,X)\geq K'|X|^2\mm\quad\forall X\in H^{1,2}_\Ho(T\X)
\end{array}\right\}\qquad\Rightarrow\qquad  (\X,\sfd,\mm)\text{ is an $\RCD(K',\infty)$ space.}
\]
It is indeed sufficient to consider the inequality $\int\bar\varphi\,\d\ric(\nabla f,\nabla f)\geq K'\int\varphi|\nabla f|^2\,\d\mm $ for $f,\varphi\in\fsm$ and $\varphi\geq 0$ to recover \eqref{eq:bochbe} and thus the claim. This sort of statement is quite weak, because it gives a lower Ricci curvature bound for a space which already had one. It is much more interesting to obtain a curvature bound out of a purely differential information plus possibly some (minimal) global regularity assumption on the space. Results in this direction have been obtained in \cite{AmbrosioGigliSavare12}, where the lower bound on the Ricci is imposed via a weak formulation of the Bochner inequality.

\bigskip

We conclude the section discussing the limitations of the Ricci curvature so defined and the difficulties one encounters in trying to better understand it.

The first fact which is  important to realize is that  we might have $\ric\equiv 0$ on spaces which do not really look  Ricci-flat. The simplest example is the metric cone built over an $S^1$ of total length $<2\pi$ equipped with the 2-dimensional Hausdorff measure. This is an $\RCD(0,\infty)$ (in fact $\RCD(0,2)$) space, a fact which can be seen either proving that it is an Alexandrov space with non-negative curvature and then using Petrunin's result \cite{Petrunin11} or directly checking the curvature-dimension condition as done in \cite{BacherSturm14} (the fact that this space is infinitesimally Hilbertian is quite obvious). It is clear that this space without the vertex $O$ is a smooth and flat Riemannian manifold so that directly by the definition we get that $\ric(X,X)$ must vanish outside the vertex. Moreover, the set $C:=\{O\}$ has 0 capacity in the sense that
\begin{equation}
\label{eq:percap0}
\begin{split}
&\text{$\mm(C)=0$ and there exists a sequence $(f_n)\subset W^{1,2}(\X)$ of functions having}\\
&\text{Lipschitz  representatives $\bar f_n$ such that $\int\weakgrad {f_n}^2\,\d\mm\to 0$ and $\bar f_n(x)\to\nchi_C(x)$}\\
&\text{for every $x\in\X$ as $n\to\infty$.}
\end{split}
\end{equation}
Indeed, clearly $\mm(\{O\})=0$ and the functions $\bar f_n(x):=\max\{1-n\sfd(x,O),0\}$ are $1$ on  $O$, uniformly bounded in $W^{1,2}(\X)$ and pointwise converging to 0 on $\X\setminus\{O\}$. Thus $(f_n)$ weakly converges to 0 in $W^{1,2}(\X)$ so that by Mazur's lemma there is a sequence of convex combinations $W^{1,2}(\X)$-strongly converging to 0, which yields the property \eqref{eq:percap0}.

On the other hand, it is easy to see that the Ricci curvature $\ric$ never sees sets $C$ as in \eqref{eq:percap0}. To prove this claim, it is sufficient to check that $\ric(X,X)(C)=0$ for $X\in\vsm$, because $\vsm$ is dense in  $H^{1,2}_\Ho(T\X)$ and $\ric:[H^{1,2}_\Ho(T\X)]^2\to\mes(\X)$ is continuous, the target space being endowed with the total variation norm. Then observe that, by the very definition of measure valued Laplacian, for $C$ as in \eqref{eq:percap0} and $g\in D(\bd)$ we have $\bd g(C)=0$ and therefore by formula \eqref{eq:defricci} we see that for $X\in\vsm$ and $C$ as in \eqref{eq:percap0} we must have $\ric(X,X)(C)=0$, as claimed.

It follows  that for such cone the Ricci curvature must vanish identically, although intuitively one would expect to have as curvature some Dirac delta on the vertex. 

This example shows that the measure valued Ricci curvature as we defined it  is not suitable to recover any sort of Gauss-Bonnet formula.

\medskip

A different kind of problem that we have with the Ricci curvature as given by Theorem \ref{thm:ricci} is that it is not clear if it is really a tensor (see also Remark \ref{rem:loctens}). Here part of the issue is that it is defined only for Sobolev vector fields and not for generic $L^2$ ones and in particular we don't know if for some function $r:\X\to[0,\infty)$ we have
\[
\ric(X,X)\leq r|X|^2\mm,\qquad\mm\ae,
\]
for every $X\in H^{1,2}_\Ho(T\X)$. A basic question is then:
\[
\text{can we extend $\ric$ to a bilinear continuous map from $[L^2(T\X)]^2$ to $\mes(\X)$?}
\]
And more generally and somehow more vaguely:
\[
\text{what is the maximal subspace of $L^2(T\X)$  on which we can continuously extend $\ric$?}
\]
In this direction, notice that in inequality \eqref{eq:lezione} of the proof of Theorem \ref{thm:ricci} we thrown away the term $\ec(X)$, thus certainly losing some information. The result is that we obtain the inequality \eqref{eq:riccitv} which is evidently sub-optimal in the smooth case because it misses some crucial cancellations occurring in the smooth case. Yet,  we can't do anything better without information on the $L^2$-semicontinuity of the functional $\eh-\ec$. Here the problem can be isolated as follows: consider the functional $\mathcal E_{\rm diff}:L^2(T\X)\to [0,\infty]$ defined as
\[
\mathcal E_{\rm diff}(X):=\inf\limi_{n\to\infty}\eh(X_n^\flat)-\ec(X_n),
\]
where the $\inf$ is taken among all sequences $(X_n)\subset H^{1,2}_\Ho(T\X)$ converging to $X$ in $L^2(T\X)$,
\begin{equation}
\label{eq:q1}
\text{is it true that }\quad\mathcal E_{\rm diff}(X)= \eh(X)-\ec(X)\quad\text{ for $X\in H^{1,2}_\Ho(T\X)$ ?}
\end{equation}
A positive answer would allow to extend the definition of  $\ric(X,Y)$ to vector fields  $X,Y$ belonging to the space $D:=\{\mathcal E_{\rm diff}<\infty\}\subset L^2(T\X)$ equipped with the complete norm 
\[
\|X\|_D^2:=(1+\tfrac K2)\|X\|^2_{L^2(T\X)}+\mathcal E_{\rm diff}(X),
\]
as the very same arguments used in the proof of Theorem \ref{thm:ricci} would allow to improve  inequality \eqref{eq:riccitv} in
\[
\|\ric(X,Y)\|_{\sf TV}\leq2\sqrt{\mathcal E_{\rm diff}(X)+K^-\|X\|_{L^2(T\X)}^2}\,\sqrt{\mathcal E_{\rm diff}(Y)+K^-\|Y\|_{L^2(T\X)}^2}.
\]
Indeed, one would first introduce the abstract completion $V$ of the space $H^{1,2}_\Ho(T\X)$ w.r.t.\ the norm $X\mapsto \sqrt{(1+\tfrac K2)\|X\|^2_{L^2(T\X)}+\eh(X^\flat)-\ec(X)}\geq \|X\|_{L^2(T\X)}$, extends by continuity the Ricci curvature to $V$ and then use the positive answer to question \eqref{eq:q1} to show that the natural map from $V$ to $L^2(T\X)$ which assigns to a $V$-Cauchy sequence its $L^2(T\X)$-limit is injective, so that $V$ can be identified with the space $D$.

Notice that in the smooth world the Weitzenb\"ock identity ensures that $D$ coincides with $L^2(T\X)$.

\medskip

On a different direction, one might try to enlarge the domain of definition of $\ric$ using Proposition \ref{prop:riccifx}. Indeed, formula \eqref{eq:riccixf} suggests that for $X,Y\in H^{1,2}_\Ho(T\X)$ and $f,g\in L^\infty(\mm)$ with continuous representatives $\bar f,\bar g$ one might define $\ric (fX,gY):=\bar f\bar g\,\ric(X,Y)$. 

The problem in this attempt is that it is not really clear if it is consistent, the question being the following:
\[
\begin{split}
&\text{Let $X,Y,X_1,\ldots,X_n\in H^{1,2}_\Ho(T\X)$ and $f_1,\ldots,f_n\in L^\infty(\mm)$ with continuous representatives}\\
&\text{$\bar f_1,\ldots,\bar f_n$ such that $X=\sum_if_iX_i$. Is it true that $\sum_i\bar f_i \ric(X_i,Y)=\ric(X,Y)$ ?}
\end{split}
\]
This seems not a consequence of basic algebraic manipulations based on the formulas involving the Ricci curvature that we provided so far.

\def\cprime{$'$} \def\cprime{$'$}

\end{document}